\pdfoutput=1
\documentclass[11pt,reqno]{amsart}
\usepackage[margin=1.15in]{geometry}
\usepackage[T1]{fontenc}
\usepackage{lmodern,microtype}
\usepackage{amsmath,amssymb,amsthm,mathtools}
\usepackage{booktabs,array,longtable,enumitem,graphicx,xcolor,needspace}
\providecommand{\needspace}[1]{}
\usepackage[colorlinks=true,linkcolor=blue!60!black,citecolor=blue!60!black,urlcolor=blue!60!black]{hyperref}
\usepackage[capitalize,nameinlink]{cleveref}

\newcommand{\E}{\mathbb E}
\newcommand{\Prob}{\mathbb P}
\newcommand{\R}{\mathbb R}
\newcommand{\C}{\mathbb C}
\newcommand{\N}{\mathbb N}
\newcommand{\Z}{\mathbb Z}
\newcommand{\dd}{\,\mathrm d}
\newcommand{\ind}{\mathbf 1}
\DeclareMathOperator{\Var}{Var}
\DeclareMathOperator{\Cov}{Cov}
\DeclareMathOperator{\sgn}{sgn}
\DeclareMathOperator{\dist}{dist}
\renewcommand{\Re}{\operatorname{Re}}
\renewcommand{\Im}{\operatorname{Im}}
\newcommand{\nint}[1]{\|#1\|_{\R/\Z}}

\theoremstyle{plain}
\newtheorem{theorem}{Theorem}[section]
\newtheorem{lemma}[theorem]{Lemma}
\newtheorem{proposition}[theorem]{Proposition}
\newtheorem{corollary}[theorem]{Corollary}
\theoremstyle{definition}
\newtheorem{definition}[theorem]{Definition}
\theoremstyle{remark}
\newtheorem{remark}[theorem]{Remark}
\numberwithin{equation}{section}
\setlist{itemsep=2pt,topsep=4pt}

\makeatletter
\def\@secnumfont{\bfseries}
\renewcommand\section{\@startsection{section}{1}{\z@}
  {-3.2ex plus -1ex minus -.2ex}{1.6ex plus .2ex}{\normalfont\large\bfseries}}
\renewcommand\subsection{\@startsection{subsection}{2}{\z@}
  {-2.4ex plus -.8ex minus -.2ex}{1.0ex plus .2ex}{\normalfont\normalsize\bfseries}}
\renewcommand\subsubsection{\@startsection{subsubsection}{3}{\z@}
  {-1.8ex plus -.5ex minus -.2ex}{-.6em}{\normalfont\normalsize\itshape}}
\makeatother
\makeatletter
\def\LT@makecaption#1#2#3{%
  \LT@mcol\LT@cols c{\hbox to\z@{\hss\parbox[t]\textwidth{\centering
    \normalfont\normalsize\@captionheadfont#1{#2.}\@captionfont\upshape\enspace#3%
    \endgraf\vskip\belowcaptionskip}\hss}}}
\makeatother
\usepackage{etoolbox}
\makeatletter
\patchcmd{\@starttoc}{\linespacing\@plus\linespacing}{\z@}{}{}
\patchcmd{\@maketitle}{\dimen@34\p@}{\dimen@1.5\baselineskip}{}{}
\patchcmd{\@maketitle}{\global\topskip42\p@}{\global\topskip12\p@}{}{}
\makeatother
\makeatletter
\patchcmd{\@cflb}{\vskip \textfloatsep}{\vskip \textfloatsep \vskip \z@ \@plus 1fil}{}{\PackageWarningNoLine{preamble}{bottom-float patch failed}}
\makeatother
\title[Universal radial laws and non-universal fluctuations]{Universal Radial Laws for Random Polynomials: Erd\H{o}s Problem \#522 and Non-Universal Fluctuations}
\author{Sebastien Kawada}
\address{Massachusetts Institute of Technology, Cambridge, MA, USA}
\email{kawada@mit.edu}
\date{October 8, 2026}
\hypersetup{
 pdftitle={Universal Radial Laws for Random Polynomials: Erd\H{o}s Problem \#522 and Non-Universal Fluctuations},
 pdfauthor={Sebastien Kawada},
 pdfsubject={Mathematics Subject Classification 2020: Primary 30C15, 60F15; Secondary 30B20, 60F05, 60G15, 11C08, 68V20},
 pdfkeywords={Random polynomials, Kac polynomials, Littlewood polynomials, Erd\H{o}s problems, zeros near the unit circle, universality, law of the iterated logarithm, fourth cumulant, formal verification, Lean 4},
 pdflang={en-US},
 pdfdisplaydoctitle=true,
 bookmarksdepth=2
}
\subjclass[2020]{Primary 30C15, 60F15; Secondary 30B20, 60F05, 60G15, 11C08, 68V20}
\keywords{Random polynomials, Kac polynomials, Littlewood polynomials, Erd\H{o}s problems, zeros near the unit circle, universality, law of the iterated logarithm, fourth cumulant, formal verification, Lean 4}

\begin{document}
\begin{abstract}
We resolve Erd\H{o}s Problem \#522, posed in 1961, by proving that for
random Littlewood polynomials the proportion of zeros in the closed unit
disk converges almost surely to $1/2$. More generally, for partial sums
of random power series with i.i.d.\ coefficients that are bounded,
nondegenerate, and centrally symmetric, real Gaussian, or circular
complex Gaussian, we prove that the proportion of zeros in the disk of
radius $1+x/n$ converges almost surely to $\frac12(1+\coth x-1/x)$,
uniformly in $x\in\R$. Thus almost all zeros lie at distance $O(1/n)$
from the unit circle. For random Littlewood polynomials, even with
arbitrary dependence between degrees, the convergence at the unit circle
holds at the almost-sure rate $O(n^{-1/4}\sqrt{\log n})$. We establish a
law-of-the-iterated-logarithm criterion for dependent sequences and, for
real Gaussian coefficients, obtain the sharp law of the iterated
logarithm over all integer degrees, variance and cross-degree covariance
asymptotics, quantitative central limit theorems, joint Gaussian limits,
and increment bounds. For symmetric coefficient laws with bounded density
and all moments finite, the number of zeros in the closed unit disk has
variance $(c_G+\kappa_4/12)n+O_\xi(n^{399/400})$, and for random signs
the constant is $c_G-1/6$. Thus the radial law is universal, while the
fluctuations depend on the coefficient law through its fourth cumulant
$\kappa_4$. All our main results are formally verified in Lean~4, and
our proof of Erd\H{o}s Problem \#522 has been accepted as a solution to
its formal statement in Google DeepMind's Formal Conjectures project.
\end{abstract}
\maketitle
\setcounter{tocdepth}{1}
\tableofcontents
\clearpage

\section{Introduction}\label{sec:intro}

Let $\xi_0,\xi_1,\xi_2,\dots$ be independent, identically distributed
random variables, and let
\[
 f_n(z)=\sum_{j=0}^{n}\xi_jz^j\qquad(n\ge1)
\]
be the partial sums of the random power series $\sum_{j\ge0}\xi_jz^j$.
All degrees use the same coefficients, and $f_{n+1}$ is obtained from
$f_n$ by appending one term. We write $\nu_n(r)$ for the number of zeros of
$f_n$ in the closed disk $\{|z|\le r\}$, counted with multiplicity.

When the coefficients are random signs,
$\Prob(\xi_j=1)=\Prob(\xi_j=-1)=\frac12$, the polynomials $f_n$ are
random Littlewood polynomials. In 1961 Erd\H{o}s asked whether
$\nu_n(1)/n\to1/2$ almost surely \cite[p.~252]{Erdos1961}, that is,
whether asymptotically half of the zeros lie in the closed unit disk.
This is Problem \#522 in the Erd\H{o}s problems list
\cite{Erdos522Problem}. Yakir \cite{Yakir2022} proved the weaker
statement that $\nu_n(1)/n\to1/2$ in probability, posed as Problem~4.15
in Hayman's collection \cite{Hayman1967,HaymanLingham2018}.

We study the zeros of $f_n$ near the unit circle at two scales. To first
order, we prove an almost-sure radial law at distance of order $1/n$
from the circle. It holds for all degrees simultaneously and contains
the answer to Erd\H{o}s's question. To second order, we study the
discrepancy
\[
 D_n=\nu_n(1)-\frac n2 ,
\]
the excess of zeros in the closed unit disk over half the degree. We
describe its variance, its normal approximation, its dependence across
degrees and its almost-sure size.

\subsection{The radial law and Erd\H{o}s Problem \#522}
The natural variable for zeros near the unit circle is $x=n(|z|-1)$.
Our first theorem gives, almost surely and for all large $n$ at once,
the proportion of zeros of $f_n$ in each disk $\{|z|\le1+x/n\}$.
\begin{theorem}[Radial law]\label{thm:R}
Let $(\xi_j)_{j\ge0}$ be iid with one of the following laws:
\begin{enumerate}[label=\textup{(\roman*)}]
 \item a bounded law on $\R$ or $\C$ that is centrally symmetric ($\xi_0\overset{d}{=}-\xi_0$) and not concentrated at $0$, possibly with an atom at $0$, such as random signs ($\Prob(\xi_0=\pm1)=\frac12$) or the Steinhaus law (uniform on the unit circle);
 \item the real standard Gaussian law, or the standard circular complex Gaussian law.
\end{enumerate}
Let $f_n(z)=\sum_{j=0}^n\xi_jz^j$, with $\nu_n\equiv0$ when $f_n$ is the zero polynomial, and put
\[
 \Phi_{\rm rad}(x)=\frac{\int_0^1te^{2xt}\dd t}{\int_0^1e^{2xt}\dd t}=\frac12\Big(1+\coth x-\frac1x\Big).
\]
There are constants $C_A$, depending only on $A$ and on the law of $\xi_0$, and an event of probability one on which, for every $A\ge1$,
\[
 \sup_{|x|\le A}\left|\frac{\nu_n(1+x/n)}{n}-\Phi_{\rm rad}(x)\right|\le C_A\,n^{-1/34}(\log n)^{7/2}\qquad\text{for all sufficiently large }n.
\]
\end{theorem}
\begin{figure}[htb]
\centering
\includegraphics[width=\textwidth,alt={Three panels on the zeros of random Littlewood polynomials. (a) The zeros of one polynomial of degree ten thousand, plotted by angle and by scaled distance from the unit circle, with their histogram against the limiting density. (b) The empirical radial distribution functions at degrees one hundred, ten thousand and one million against the limiting profile. (c) The radial density of the zeros of polynomials of degree one million against the limiting density.}]{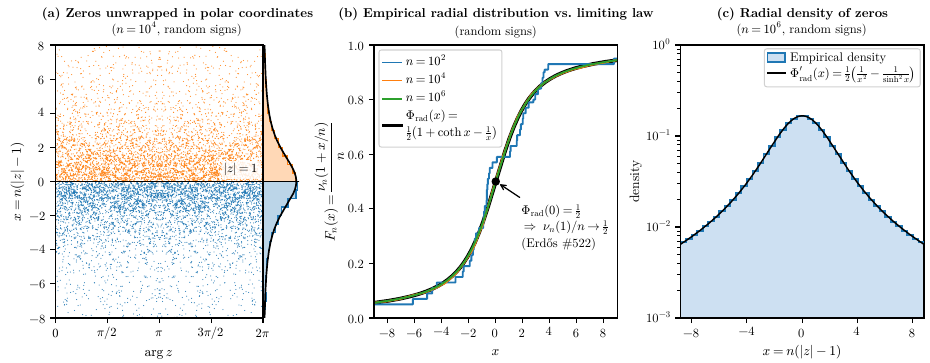}
\caption{\textbf{The radial law for random Littlewood polynomials.}
(a) The zeros of one polynomial of degree $n=10^4$ at angle $\arg z$ and
height $x=n(|z|-1)$, inside the unit disk in blue and outside in orange.
They lie in an $O(1)$ window of $x$, a band of width $O(1/n)$ around the
unit circle. At the right, the histogram of $x$ over these zeros
follows the limiting density $\Phi_{\rm rad}'$, which puts half of the zeros
on each side of the circle. (b) The profiles
$F_n(x)=\nu_n(1+x/n)/n$ for $n=10^2,10^4,10^6$ against $\Phi_{\rm rad}$.
The value $\Phi_{\rm rad}(0)=\frac12$ gives $\nu_n(1)/n\to\frac12$
(Erd\H{o}s Problem~\#522). (c) The density of $x$ over the zeros of $2048$
polynomials of degree $10^6$ against $\Phi_{\rm rad}'$.}
\label{fig:radial-law}
\end{figure}
The profile $\Phi_{\rm rad}$ increases from $0$ to $1$, so almost all
zeros lie at distance $O(1/n)$ from the unit circle, and in the variable
$x$ they have the limiting density $\Phi_{\rm rad}'$.
Figure~\ref{fig:radial-law} compares this density with the zeros of
sign polynomials. Since $\Phi_{\rm rad}(0)=1/2$, the case $x=0$ answers
Erd\H{o}s's question for the partial sums. The tail bound of
Theorem~\ref{thm:C}(i) below answers it with a rate, one degree at a
time, so it applies however the polynomials of different degrees depend
on one another.
\begin{corollary}[Erd\H{o}s Problem \#522]\label{thm:E}
For each $n\ge1$ let $f_n(z)=\sum_{j=0}^n\epsilon_{n,j}z^j$ be a random Littlewood polynomial of degree $n$, with independent random signs $\epsilon_{n,0},\dots,\epsilon_{n,n}$, all defined on one probability space and with arbitrary dependence between different degrees. For example, $f_n$ may be the partial sums of one sign sequence, or independent across $n$.
Then, almost surely, $D_n=O\bigl(n^{3/4}\sqrt{\log n}\bigr)$, and in particular
\[
 \lim_{n\to\infty}\frac{\nu_n(1)}{n}=\frac12 .
\]
\end{corollary}

The main difficulty is to pass from a sparse sequence of degrees, where
Borel--Cantelli arguments apply, to all degrees at once. Jensen's
formula expresses the radial counts of $f_N$ through the angular average
$J_N(r)$ of $\log|f_N(re^{i\theta})|$. This average concentrates, and
Borel--Cantelli gives the radial counts along the sparse degrees
$N=\lfloor j^{17/4}\rfloor$. We then show that $J_m(r)$ hardly moves as
$m$ runs through the block $N\le m<\lfloor(j+1)^{17/4}\rfloor$. Appending
the terms of degree above $N$ changes $f_N$ by a polynomial that is
uniformly small on the circle, so
$\log|f_m|-\log|f_N|$ is small except where $|f_N|$ itself is small.
That set has small measure, and logarithmic integrability controls its
contribution. The rate in
Theorem~\ref{thm:R} comes from this comparison. For signs at $x=0$,
Corollary~\ref{thm:E} has the much sharper rate
$O(n^{-1/4}\sqrt{\log n})$, which Theorem~\ref{thm:C} gives by a
different method.

\subsection{A law of the iterated logarithm for dependent sequences}
For the fluctuations we use the winding representation of $D_n$ from
Section~\ref{sec:prelim}. Up to a small error, $D_n$ is a signed sum
over the crossings of the real axis by the curve
$t\mapsto e^{-int/2}f_n(e^{it})$. There are of order $n$ crossings, and
cancellation among their signs suggests fluctuations of order $\sqrt n$.
The counts at different degrees are strongly dependent. The next theorem
isolates what a sharp almost-sure statement over all integer degrees
requires. Its four inputs are stated in Section~\ref{sec:framework}.
Input (I1) is a moderate-deviation estimate along geometric sequences of
degrees, (I2) and (I3) bound the increments over coarse and fine gaps,
and (I4) compares the sequence with independent proxies on geometric
blocks.
\begin{theorem}[Iterated-logarithm framework]\label{thm:A}
Let $c>0$ and let $(Z_n)_{n\ge1}$ be centered, square-integrable real
random variables on one probability space satisfying
Inputs~\ref{fw:I1}--\ref{fw:I4} of Section~\ref{sec:framework} with
constant $c$. Then, almost surely,
\[
 \limsup_{n\to\infty}\frac{Z_n}{\sqrt{2cn\log\log n}}=1,
 \qquad
 \liminf_{n\to\infty}\frac{Z_n}{\sqrt{2cn\log\log n}}=-1.
\]
\end{theorem}

\subsection{Gaussian coefficients}
For real Gaussian coefficients all four inputs of Theorem~\ref{thm:A}
hold, and this gives the sharp law of the iterated logarithm. The
variance of $D_n$ grows like $c_Gn$, where $c_G$ is given by an explicit
integral.
\begin{theorem}[Gaussian fluctuations]\label{thm:B}
Let $(\xi_j)_{j\ge0}$ be independent standard real Gaussian variables, $f_n(z)=\sum_{j=0}^n\xi_jz^j$ and $D_n=\nu_n(1)-n/2$. Put $r(u)=\sin(u/2)/(u/2)$, $r(0)=1$, and for $q\ge1$
\[
 a(q)=\frac1\pi\int_{\R}\frac{r'(u)^2}{q-r(u)^2}\dd u,\qquad c_G=a(1).
\]
Then $0<c_G<\infty$, and the following hold.
\begin{enumerate}[label=\textup{(\roman*)}]
 \item \emph{Moments.} $\E D_n=0$ for every $n$, and $\Var D_n=c_Gn+O(n^{99/100})$.
 \item \emph{Covariance across degrees.} If $N\to\infty$ and $M/N\to q>1$, then $N^{-1}\Cov(D_N,D_M)\to a(q)$. For all large $n$ and all $n<m\le2n$ with $\delta=(m-n)/n\ge n^{-1/1000}$,
 \[
 \bigl|n^{-1}\Cov(D_n,D_m)-a(1+\delta)\bigr|\le C\delta^{-100}n^{-1/100},
 \qquad c_G-a(1+\delta)=\sqrt{\delta/12}+O(\delta)\quad(\delta\downarrow0).
 \]
 \item \emph{Normal approximation.} $\sigma_n^2=\Var D_n>0$ for $n\ge1$, and $d_K(D_n/\sigma_n,\mathcal N(0,1))\le Cn^{-1/60}$ for all $n\ge1$.
 \item \emph{Joint limits.} For fixed $0<q_1<\dots<q_d$, the vector $n^{-1/2}(D_{\lfloor q_1n\rfloor},\dots,D_{\lfloor q_dn\rfloor})$ converges in distribution to the centered Gaussian vector with covariance $\Sigma_{ii}=c_Gq_i$ and $\Sigma_{ij}=q_ia(q_j/q_i)$ for $i<j$. If $l/m\to\alpha\in(0,1)$, then for large $m$ the standardized increment $(D_m-D_l)/\Var(D_m-D_l)^{1/2}$ is within Kolmogorov distance $C_\alpha m^{-1/60}$ of $\mathcal N(0,1)$.
 \item \emph{Increments.} Uniformly for $n\ge2$ and $1\le h\le n$, $\Var(D_{n+h}-D_n)\le C\sqrt{nh}\log(n+2)$, and for every $p>0$ there is $C_p$ with
 \[
 \Prob\Bigl(|D_{n+h}-D_n|>C_p\bigl[(nh)^{1/4}\log^{5/4}(n+2)+\log^2(n+2)\bigr]\Bigr)\le C_pn^{-p}.
 \]
 For each fixed $B>0$, $\Var(D_{n+h}-D_n)\asymp\sqrt{nh}$ uniformly for $n/(\log n)^B\le h\le n$ and large $n$. There is $\delta_0>0$ such that, uniformly for $\log^{-6}(n+2)\le h/n\le\delta_0$ and large $n$, the standardized increment $(D_{n+h}-D_n)/\Var(D_{n+h}-D_n)^{1/2}$ is within Kolmogorov distance $Cn^{-1/120}$ of $\mathcal N(0,1)$.
 \item \emph{Iterated logarithm.} Inputs \textup{(I1)--(I4)} of Theorem~\ref{thm:A} hold for $Z_n=D_n$ with $c=c_G$. Consequently, almost surely,
 \[
 \limsup_{n\to\infty}\frac{D_n}{\sqrt{2c_Gn\log\log n}}=1,\qquad
 \liminf_{n\to\infty}\frac{D_n}{\sqrt{2c_Gn\log\log n}}=-1,
 \]
 over all integers $n$.
\end{enumerate}
\end{theorem}
Numerically $c_G=0.3058977\ldots$. Parts (i) and (iii) give
Input~\ref{fw:I1}, part (v) gives Inputs~\ref{fw:I2} and~\ref{fw:I3},
and the block reversal of Corollary~\ref{pre:reversal-law} gives
Input~\ref{fw:I4}. Part (vi) then follows from Theorem~\ref{thm:A}. The
square-root cusp in (ii) means that an increment over a gap $h$ at
degree $n$ has standard deviation of order $(nh)^{1/4}$, and
Input~\ref{fw:I3} is formulated at this scale.

For Gaussian coefficients the mean radial counts are explicit at every
degree, and their profile is that of Theorem~\ref{thm:R} up to a bounded
error.
\begin{theorem}[Gaussian radial profile]\label{thm:G}
For $n\ge1$ and $r>0$ put $K_n(r)=\sum_{k=0}^nr^{2k}$, $m_n(r)=K_n(r)^{-1}\sum_{k=0}^nkr^{2k}$ and $p_n(r,\theta)=K_n(r)^{-1}\sum_{k=0}^nr^{2k}e^{2ik\theta}$. Let $(\xi_j)_{j\ge0}$ be independent coefficients, either all standard real Gaussian or all standard circular complex Gaussian ($\xi_j=(g_j+ig_j')/\sqrt2$ with $g_j,g_j'$ independent standard real Gaussians), and $f_n(z)=\sum_{j=0}^n\xi_jz^j$.
\begin{enumerate}[label=\textup{(\roman*)}]
 \item \emph{Exact radial means.} In the circular complex model $\E\nu_n(r)=m_n(r)$. In the real model
 \[
 \E\nu_n(r)=m_n(r)+\frac1{4\pi}\int_0^{2\pi}r\partial_r\log\Bigl(1+\sqrt{1-|p_n(r,\theta)|^2}\Bigr)\dd\theta,
 \]
 with an absolutely convergent integral.
 \item \emph{Mean profile.} In either model, for every $A<\infty$ there is $C_A$ such that $\sup_{|x|\le A}|\E\nu_n(1+x/n)-n\Phi_{\rm rad}(x)|\le C_A$ for all $n\ge\max(1,2A)$, where $\Phi_{\rm rad}$ is the radial profile of Theorem~\ref{thm:R}.
\end{enumerate}
\end{theorem}

\subsection{Littlewood polynomials}
Random sign coefficients differ from Gaussian ones in two ways. Their
concentration is weaker, and $f_n$ can vanish on the unit circle. We prove an almost-sure
bound for $D_n$ and control the number of zeros on the circle.
\begin{theorem}[Sign fluctuations and circle roots]\label{thm:C}
Let $(\epsilon_j)_{j\ge0}$ be independent random signs, $\Prob(\epsilon_j=1)=\Prob(\epsilon_j=-1)=\frac12$, and $f_n(z)=\sum_{j=0}^n\epsilon_jz^j$.
Let $B_n$ be the number of roots of $f_n$ on the unit circle, counted with multiplicity.
\begin{enumerate}[label=\textup{(\roman*)}]
 \item There are absolute constants $A,C$ such that, for every $n\ge1$,
 \[
  \Prob\bigl\{|D_n|>An^{3/4}\sqrt{\log(n+2)}\bigr\}\le Cn^{-3},
  \qquad
  \E D_n^2\le Cn^{3/2}\sqrt{\log(n+2)} .
 \]
 Almost surely, $D_n=O\bigl(n^{3/4}\sqrt{\log n}\bigr)$.
 \item The boundary count satisfies $\E B_n=O(1)$ and $\Var B_n=O(\sqrt n)$. Moreover, for each $\eta>0$, almost surely $B_n=O\bigl(\sqrt n\,(\log n)^{1+\eta}\bigr)$.
 \item The centered counts $Z_n=D_n-\E D_n$ satisfy Input~\ref{fw:I4}.
\end{enumerate}
\end{theorem}
\begin{remark}[A conditional law of the iterated logarithm]\label{thm:Cp}
By Theorem~\ref{thm:C}(iii) the sharp law for signs reduces to the first
three inputs. Let $c>0$, and suppose that $Z_n=D_n-\E D_n$ satisfies
Input~\ref{fw:I1} with constant $c$, together with Inputs~\ref{fw:I2}
and~\ref{fw:I3}. Since $\E D_n=\frac12\E B_n=O(1)$ by
\eqref{pre:centering} and Theorem~\ref{thm:C}(ii), Theorem~\ref{thm:A}
and Lemma~\ref{fw:grid-lower} then give, almost surely,
\[
 \limsup_{n\to\infty}\frac{D_n}{\sqrt{2cn\log\log n}}=1,
 \qquad
 \liminf_{n\to\infty}\frac{D_n}{\sqrt{2cn\log\log n}}=-1 ,
\]
over all integers and along every grid $\lfloor q^k\rfloor$ with $q>1$.
\end{remark}

\subsection{Non-universality}
For symmetric coefficient laws with a bounded density and all moments
finite, the variance of $D_n$ still grows linearly in $n$, but the
constant becomes $c_G+\kappa_4/12$, where $\kappa_4=\E\xi_0^4-3$ is the
fourth cumulant of a coefficient with unit variance. The radial profile
of Theorem~\ref{thm:R} is the same for every law it covers, whereas the
variance constant depends on the law.
\begin{theorem}[Fourth-cumulant variance formula]\label{thm:D}
Let $(\xi_j)_{j\ge0}$ be iid real coefficients whose common law is symmetric, has a bounded Lebesgue density, variance one and finite moments of every order, and put $\kappa_4=\E\xi_0^4-3$.
Let $f_n(z)=\sum_{j=0}^n\xi_jz^j$ and $D_n=\nu_n(1)-n/2$.
Then $\E D_n=0$ for every $n$, and
\[
 \Var D_n=\Big(c_G+\frac{\kappa_4}{12}\Big)n+O_\xi\big(n^{399/400}\big),
\]
where the implied constant and the starting degree depend on the law.
For example, $\kappa_4=-6/5$ for the uniform law on $[-\sqrt3,\sqrt3]$ and $\kappa_4=3$ for the Laplace law with density $2^{-1/2}e^{-\sqrt2|x|}$.
Moreover, with the positional quadratic statistic
\[
 H_n=\sum_{j=0}^n\frac{j-n/2}{n+1}\,(\xi_j^2-1),
\]
one has $\Cov(D_n-H_n,H_n)/n\to0$ and $\Var(D_n-H_n)/n\to c_G-1/6$.
\end{theorem}
The variance formula also holds for random signs.
\begin{theorem}[Variance for random signs]\label{thm:S}
Let $(\epsilon_j)_{j\ge0}$ be independent random signs, $f_n(z)=\sum_{j=0}^n\epsilon_jz^j$ and $D_n=\nu_n(1)-n/2$, and let $B_n$ be the number of roots of $f_n$ on the unit circle, counted with multiplicity. Then $\E D_n=\frac12\E B_n=O(1)$ and
\[
 \Var D_n=\Bigl(c_G-\frac16\Bigr)n+O\bigl(n^{799/800}\bigr).
\]
\end{theorem}
Theorem~\ref{thm:S} is the formula of Theorem~\ref{thm:D} with
$\kappa_4=\E\epsilon_0^4-3=-2$, for a law with no density. Since
$\epsilon_j^2=1$, the statistic $H_n$ vanishes for signs, so
Theorem~\ref{thm:S} contains the limit $\Var(D_n-H_n)/n\to c_G-\frac16$
of Theorem~\ref{thm:D}. It also gives $\E D_n^2\le Cn$, which sharpens
the second-moment bound of Theorem~\ref{thm:C}(i).

\subsection{Ideas of the proofs}
Erd\H{o}s Problem \#522 is resolved as a consequence of a universal
microscopic radial law, while a separate fluctuation theory reveals
precisely how the coefficient distribution influences the second-order
behavior.
The radial law rests on Jensen's formula, which turns zero counts into
angular averages of $\log|f_n|$, as explained after
Corollary~\ref{thm:E}. For the fluctuations of $D_n$ we use the argument
principle instead. It counts $D_n$ as the signed number of crossings of
the real axis by the curve $t\mapsto e^{-int/2}f_n(e^{it})$. Pairing the
coefficients $\xi_j$ and $\xi_{n-j}$ splits this curve into a cosine and
a sine coordinate. Once the sine coordinate is fixed, so are the
crossings, and $D_n$ becomes a sum of signs of the cosine coordinate. For
Gaussian coefficients the two coordinates are independent. For random
signs we condition on which pairs are equal and which are opposite, which
has the same effect, and Talagrand's convex concentration
inequality~\cite{Talagrand} controls the sum that remains. To pass from
these estimates to a law of the iterated logarithm,
Section~\ref{sec:framework} runs the classical proof with independence
replaced by Inputs~\ref{fw:I1}--\ref{fw:I4}.

To find the variance for other coefficient laws we compare with the
Gaussian case. We approximate $D_n$ in $L^2$ by a smooth function of the
coefficients, and Lindeberg replacement shows that the variance of this
function differs from its Gaussian value by $\kappa_4n/12$ up to lower
order. The approximation rests on small-ball bounds for sums of
coefficients and on the decorrelation of distant stretches of the curve.
The decorrelation comes from an Edgeworth expansion against test
functions with Gaussian Fourier decay, which holds wherever the
characteristic functions of these sums decay. For laws with a bounded
density they decay at every point, and this also gives the small-ball
bounds. For random signs they decay at points satisfying a Diophantine
condition, which carry almost all of the curve, and the small-ball bounds
of Section~\ref{sec:signs} complete the argument.

\subsection{Related work}
Yakir \cite{Yakir2022} proved that $|\nu_n(1)-n/2|<n^{9/10}$ with
probability tending to one, using concentration of the angular integral
of $\log|f_n|$ and Jensen's formula. Our proof of Theorem~\ref{thm:R}
starts from the same two tools.

In April 2026, Chojecki~\cite{Chojecki} and Kwon and Zou~\cite{KwonZou}
each posted an informal argument for the almost-sure statement for
random signs.
Our proof and its conceptual approach are original. Ours is the first
formally verified proof of Erd\H{o}s Problem \#522. Its Lean
formalization was completed on September 22, 2026, and every main result
of this paper is formally verified as well
(Subsection~\ref{sec:intro-formal}).

Do \cite{Do2024} proved an almost-sure law for the real zeros of nested
Kac polynomials, passing from sparse to all degrees by maximal
inequalities. Shepp and Vanderbei \cite{SheppVanderbei} computed the
expected zero distribution of real Gaussian polynomials. Buckley and
Feldheim \cite{BuckleyFeldheim} computed the winding variance of
stationary complex Gaussian processes, which yields an integral of the
form of $c_G$. Kiro and Nishry \cite{KiroNishry} studied
fluctuations of disk counts for Gaussian Taylor series, and Angst and
Poly \cite{AngstPoly} developed small-ball and Edgeworth methods for
real zeros of random trigonometric polynomials.

Do, Nguyen and Nguyen \cite{DoNguyenNguyen2022} proved that the variance
of the number of real zeros of a random trigonometric polynomial with iid
coefficients grows linearly in the degree, with a constant that depends on
the coefficient law only through its kurtosis. Their class of laws
includes discrete ones. This is the closest precedent for Theorems~\ref{thm:D}
and~\ref{thm:S}, and Section~\ref{sec:sign-variance} follows their inverse
argument for characteristic functions.

\subsection{Formal verification}\label{sec:intro-formal}
Every theorem in this introduction, and also Remark~\ref{thm:Cp}, is
proved in Lean~4 with Mathlib.
Appendix~\ref{sec:formal} pairs each result with its formal statement,
and every formal statement depends only on the axioms
\texttt{propext}, \texttt{Classical.choice} and \texttt{Quot.sound}.
Our Lean~4 proof of Corollary~\ref{thm:E} for the partial sums has been
accepted as a solution to the formal statement \texttt{Erdos522.erdos\_522} of
Problem~\#522 in Google DeepMind's Formal Conjectures
project~\cite{FormalConjectures}.

\subsection{Organization}
Section~\ref{sec:prelim} collects conventions, the winding
representation and Jensen's formula. Sections~\ref{sec:rc} and
\ref{sec:rm} prove Theorem~\ref{thm:R} and its uniform form on the line,
Corollary~\ref{rm:whole-line}. Section~\ref{sec:framework} proves
Theorem~\ref{thm:A}, Section~\ref{sec:gaussian} proves
Theorems~\ref{thm:B} and~\ref{thm:G}, and Section~\ref{sec:signs} proves
Theorem~\ref{thm:C} and Corollary~\ref{thm:E}. Section~\ref{sec:cumulant}
proves Theorem~\ref{thm:D}. It approximates $D_n$ by a smooth function of
the coefficients under hypotheses on small balls and on the decay of
characteristic functions, which hold for laws with a bounded density.
Section~\ref{sec:sign-variance} verifies these hypotheses for random signs
and proves Theorem~\ref{thm:S}. Appendix~\ref{sec:numerics} compares the
results with computation, and Appendix~\ref{sec:formal} lists the formal
statements.

Throughout, zeros are counted with multiplicity, the zero polynomial has
no zeros, and all almost-sure statements refer to a single event of
probability one for the coefficient sequence $(\xi_j)$. The letter $C$
denotes a constant that may change from line to line, and subscripts
indicate its dependence.
\section{Preliminaries}\label{sec:prelim}

Two identities turn zeros into analytic quantities. Jensen's formula
expresses the number of zeros in a disk through the mean of $\log|f|$ on
its boundary, and it drives the radial law. The argument principle
expresses $D_n$ through the winding of the curve
$t\mapsto e^{-int/2}f_n(e^{it})$. Pairing the coefficients $\xi_j$ and
$\xi_{n-j}$ splits this curve into a cosine and a sine coordinate, and
the winding becomes a signed count of the crossings of the real axis
(Proposition~\ref{pre:crossing}), which drives the fluctuation results.
Reversal of the coefficient vector makes the discrepancy antisymmetric,
and this gives the centering of $D_n$ and the block proxies of
Section~\ref{sec:framework}.

\subsection{Coefficients and counts}

\begin{definition}[Coefficient models]\label{pre:models}
Throughout, $(\xi_j)_{j\ge0}$ is one sequence of independent, identically
distributed real or complex random variables on a probability space
$(\Omega,\mathcal F,\Prob)$, and $f_n(z)=\sum_{j=0}^n\xi_jz^j$ for every
$n\ge0$. All degrees use this sequence, so that
$f_{n+1}=f_n+\xi_{n+1}z^{n+1}$. The laws of $\xi_0$ are the following.
\begin{enumerate}[label=\textup{(\alph*)}]
 \item \emph{Random signs}, $\Prob(\xi_0=\pm1)=\frac12$, in which case we
 write $\epsilon_j$ for $\xi_j$.
 \item \emph{Steinhaus}, uniform on the unit circle.
 \item \emph{Real Gaussian}, standard normal.
 \item \emph{Complex Gaussian}, with density $\pi^{-1}e^{-|z|^2}$ on $\C$.
 \item \emph{Bounded symmetric}, a bounded law on $\R$ or $\C$ with
 $\xi_0\overset d=-\xi_0$ and $\Prob(\xi_0=0)<1$.
 \item \emph{Absolutely continuous}, a law with a density with respect to
 Lebesgue measure on $\R$ or on $\C\cong\R^2$, such as (c) and (d).
\end{enumerate}
Multiplying all coefficients by one constant leaves the zeros in place,
so we normalize finite second moments to one.
\end{definition}

\begin{definition}[Counts]\label{pre:counts}
A \emph{polynomial of nominal degree $n$} is a coefficient vector
$(a_0,\ldots,a_n)\in\C^{n+1}$, with $f(z)=\sum_{j=0}^na_jz^j$. Its actual
degree may be smaller, and $f_n$ has nominal degree $n$. For $f\ne0$
and $r\ge0$ let $\nu_f(r)$ and $\nu_f^\circ(r)$ be the numbers of zeros
of $f$ in $\{|z|\le r\}$ and in $\{|z|<r\}$, and
$B_f(r)=\nu_f(r)-\nu_f^\circ(r)$ the number on $\{|z|=r\}$, always with
multiplicity and including zeros at the origin. For the zero polynomial,
and for $r<0$, all three counts are zero. Put $I(f)=\nu_f^\circ(1)$,
$B(f)=B_f(1)$ and, with the nominal degree $n$,
\[
 D(f)=\nu_f(1)-\frac n2,\qquad
 D^\circ(f)=D(f)-\frac{B(f)}2=I(f)+\frac{B(f)}2-\frac n2 .
\]
The \emph{balanced discrepancy} $D^\circ(f)$ gives each zero on the unit
circle weight one half. We write $\nu_n(r)$, $\nu_n^\circ(r)$, $B_n$,
$D_n$, $D_n^\circ$ for these quantities at $f_n$, with nominal degree
$n$. For $f\ne0$, $0\le\nu_f(r)\le\deg f\le n$, so $|D(f)|$ and
$|D^\circ(f)|$ are at most $n/2$.
\end{definition}

\begin{remark}[Measurability]\label{pre:as-convention}
For fixed $n$ and $r$, the counts $\nu_n(r)$, $\nu_n^\circ(r)$, $B_n$ are
Borel functions of $(\xi_0,\ldots,\xi_n)$. By \eqref{pre:secants} below,
$\nu_n(r)=\lim_kk\{J_{f_n}(re^{1/k})-J_{f_n}(r)\}$ on $\{f_n\ne0\}$ for
$r>0$, with $J_{f_n}(r)$ Borel by Tonelli's theorem, while
$\nu_n(0)=\lim_{s\downarrow0}\nu_n(s)$ and
$\nu_n^\circ(r)=\lim_{s\uparrow r}\nu_n(s)$.
\end{remark}

\subsection{Jensen's formula}

For a polynomial $f$ and $r>0$ put
\[
 J_f(r)=\int_0^{2\pi}\log|f(re^{i\theta})|\,\frac{\dd\theta}{2\pi}.
\]

\begin{lemma}[Jensen's formula]\label{pre:jensen}
Let $f$ be a nonzero polynomial with leading coefficient $c$ and zeros
$\alpha_1,\ldots,\alpha_d$, listed with multiplicity. For every $r>0$,
including radii carrying zeros,
\[
 J_f(r)=\log|c|+\sum_{j=1}^d\log\max(r,|\alpha_j|).
\]
Hence $s\mapsto J_f(e^s)$ is convex and piecewise linear, with right
derivative $\nu_f(e^s)$ and left derivative $\nu_f^\circ(e^s)$. In
particular, for $0<r<s$,
\begin{equation}\label{pre:secants}
 \nu_f(r)\le\frac{J_f(s)-J_f(r)}{\log(s/r)}\le\nu_f^\circ(s)\le\nu_f(s).
\end{equation}
\end{lemma}

\begin{proof}
Write $f(z)=c\prod_j(z-\alpha_j)$. For $\alpha\ne0$, Jensen's formula
\cite[Theorem~15.18]{Rudin1987}, which counts the zeros in the closed
disk, applied to $z-\alpha$ gives
$\int_0^{2\pi}\log|re^{i\theta}-\alpha|\,\dd\theta/2\pi=\log\max(r,|\alpha|)$,
and for $\alpha=0$ this is clear. Summing over the factors gives the
formula. Each term $\log\max(e^s,|\alpha_j|)$ is convex in $s$, with
right derivative $\ind\{|\alpha_j|\le e^s\}$ and left derivative
$\ind\{|\alpha_j|<e^s\}$. The secant of a convex function over
$[\log r,\log s]$ lies between its right derivative at $\log r$ and its
left derivative at $\log s$, which is \eqref{pre:secants}.
\end{proof}

\begin{lemma}[Local zero counts]\label{pre:local-jensen}
Let $F$ be analytic on a neighborhood of $\{|z-w|\le R\}$ with
$F(w)\ne0$, and let $0<\rho<R$. Then
\[
 \#\{z:F(z)=0,\ |z-w|\le\rho\}
 \le\frac{\log\max_{|z-w|=R}|F(z)|-\log|F(w)|}{\log(R/\rho)} .
\]
\end{lemma}

\begin{proof}
The function $z\mapsto F(w+z)$ is holomorphic on a disk $|z|<R'$ with
$R'>R$, and Jensen's formula \cite[Theorem~15.18]{Rudin1987} gives
\[
 \log|F(w)|+\sum_k\log\frac R{|z_k-w|}
 =\int_0^{2\pi}\log|F(w+Re^{i\theta})|\,\frac{\dd\theta}{2\pi},
\]
the sum running over the zeros $z_k$ in $|z-w|\le R$ with multiplicity.
Every term of the sum is nonnegative, each zero with $|z_k-w|\le\rho$
contributes at least $\log(R/\rho)$, and the right side is at most
$\log\max_{|z-w|=R}|F|$.
\end{proof}

\subsection{Reversal and zeros on the circle}

Reversing the coefficient vector maps each nonzero zero $\alpha$ to
$1/\alpha$ and exchanges the inside and the outside of the unit circle.
For a polynomial $f(z)=\sum_{j=0}^na_jz^j$ of nominal degree $n$, its
reversal is $f^*(z)=z^nf(1/z)=\sum_{j=0}^na_{n-j}z^j$, again of nominal
degree $n$, and $f^{**}=f$.

\begin{lemma}[Reversal]\label{pre:reversal}
Let $f\ne0$ have nominal degree $n$. For every $r>0$,
\begin{equation}\label{pre:reversal-radial}
 \nu_{f^*}(r)+\nu_f^\circ(1/r)=n,\qquad B_{f^*}(r)=B_f(1/r).
\end{equation}
In particular $B(f^*)=B(f)$, $\nu_f(1)+\nu_{f^*}(1)=n+B(f)$,
$D(f)+D(f^*)=B(f)$ and $D^\circ(f^*)=-D^\circ(f)$.
\end{lemma}

\begin{proof}
Let $d$ be the actual degree and $k$ the multiplicity of the zero at
the origin, and write $f(z)=cz^k\prod_{i=1}^{d-k}(z-\beta_i)$ with
$\beta_i\ne0$. Then
$f^*(z)=c\prod_i(-\beta_i)\,z^{n-d}\prod_i(z-1/\beta_i)$. Hence
$\nu_{f^*}(r)=(n-d)+\#\{i:|\beta_i|\ge1/r\}$ and
$\nu_f^\circ(1/r)=k+\#\{i:|\beta_i|<1/r\}$, whose sum is $n$, and
$|1/\beta_i|=r$ exactly when $|\beta_i|=1/r$. At $r=1$,
$\nu_f^\circ(1)=\nu_f(1)-B(f)$. The identities for $D$ and $D^\circ$
follow by subtracting $n/2$ twice.
\end{proof}

In the scaling $r=1+x/n$, \eqref{pre:reversal-radial} is the count-level
form of the identity $\Phi_{\rm rad}(x)+\Phi_{\rm rad}(-x)=1$.

\begin{corollary}[Reversal in law]\label{pre:reversal-law}
Let the coefficients be iid, real or complex, with
$\Prob(\xi_0=0)=0$, and fix $n\ge0$.
\begin{enumerate}[label=\textup{(\alph*)}]
 \item $(D_n^\circ,B_n)\overset d=(-D_n^\circ,B_n)$. Consequently
 \begin{equation}\label{pre:centering}
  \E D_n^\circ=0,\qquad \E D_n=\tfrac12\E B_n,\qquad
  D_n-\E D_n=D_n^\circ+\tfrac12(B_n-\E B_n).
 \end{equation}
 \item Let $0\le p<n$, $L=n-p-1$ and
 $P(z)=\sum_{i=0}^L\xi_{p+1+i}z^i$, of nominal degree $L$. Then $P$ is
 independent of $(\xi_0,\ldots,\xi_p)$, has the law of $f_L$, and
 \begin{equation}\label{pre:block-reversal}
  \bigl(D_n^\circ,\ B_n,\ D^\circ(P),\ B(P)\bigr)
  \overset d=\bigl(-D_n^\circ,\ B_n,\ -D_L^\circ,\ B_L\bigr).
 \end{equation}
\end{enumerate}
\end{corollary}

\begin{proof}
The block $P$ uses the coefficients $\xi_{p+1},\ldots,\xi_n$, so it is
independent of $(\xi_0,\ldots,\xi_p)$ and has the law of $f_L$. For the
symmetry, let $v=(\xi_0,\ldots,\xi_n)$ and $v^R=(\xi_n,\ldots,\xi_0)$.
Then $v^R\overset d=v$, and all polynomials below are nonzero almost
surely. The polynomial of $v^R$ is $f_n^*$, and its prefix of nominal degree $L$
has coefficients $(\xi_n,\ldots,\xi_{p+1})$, so it is $P^*$. By
Lemma~\ref{pre:reversal},
$(D_n^\circ,B_n,D^\circ(P),B(P))=(-a,b,-c,e)$, where $(a,b,c,e)$ is the
vector $(D_n^\circ,B_n,D_L^\circ,B_L)$ computed from $v^R$ in place of
$v$. Since $v^R\overset d=v$, this proves \eqref{pre:block-reversal}, and
its first two coordinates give the first assertion of (a). All these
variables are bounded by $n$, so the symmetry gives $\E D_n^\circ=0$,
and expanding
$D_n=D_n^\circ+B_n/2$ gives the rest of \eqref{pre:centering}.
\end{proof}

The symmetry acts on one coefficient vector at a time, since reversed
vectors of different degrees are not prefixes of one sequence.

\begin{lemma}[No zeros on a fixed circle]\label{pre:no-circle-roots}
Let $\xi_0$ have a density on $\R$ or on $\C$, and fix $r>0$. Almost
surely, simultaneously for all $n\ge0$, $\xi_n\ne0$ and $f_n$ has no
zero on $\{|z|=r\}$. In particular $B_n=0$ and $D_n=D_n^\circ$ for all
$n$ almost surely, and for every $n$, $D_n\overset d=-D_n$ and
$\E D_n=0$.
\end{lemma}

\begin{proof}
A density has no atoms, so $\xi_n\ne0$ almost surely. Fix $n\ge1$ and
let $F(z)=f_n(rz)$. Its coefficient vector $b=(\xi_jr^j)_{j\le n}$ again
has a density. Put $F^\#(z)=z^n\overline{F(1/\bar z)}=\sum_j\bar
b_{n-j}z^j$, which is $F^*$ for real coefficients. If $F(\zeta)=0$ with
$|\zeta|=1$, then $1/\bar\zeta=\zeta$ and
$F^\#(\zeta)=\zeta^n\overline{F(\zeta)}=0$. The Sylvester resultant
$R(b)$ of $F$ and $F^\#$ in formal degrees $(n,n)$ is a polynomial in
the real and imaginary parts of the $b_j$. It vanishes whenever $F$ and
$F^\#$ have a common zero, and not when $F$ has nonzero leading
coefficient and no zero in common with $F^\#$. At $b=(2,0,\ldots,0,1)$,
$F=z^n+2$ and $F^\#=2z^n+1$ have no common zero, so $\Re R$ or $\Im R$
is a nonzero real polynomial. A nonzero real polynomial
$\sum_kc_k(x')x_m^k$ on $\R^m$ has a Lebesgue-null zero set, since the
set where all $c_k$ vanish is null by induction on $m$, the other fibers
in $x_m$ are finite, and Fubini's theorem applies. So $f_n$ has a zero on
$|z|=r$ with probability zero. Take the union over $n$, and apply
Corollary~\ref{pre:reversal-law} at $r=1$.
\end{proof}

\subsection{Winding and signed crossings}

For the fluctuations we count zeros by the argument principle, along the
curve obtained by removing the mean rotation $e^{int/2}$ from
$f(e^{it})$. For a nonzero polynomial $f$ of nominal degree $n$ and
$r>0$ define the \emph{carrier curve} and the \emph{argument speed}
\[
 g_{f,r}(t)=e^{-int/2}f(re^{it}),\qquad
 a_{f,r}(t)=\Re\frac{re^{it}f'(re^{it})}{f(re^{it})}-\frac n2,
\]
the latter where $f(re^{it})\ne0$. On that set
$a_{f,r}=\Im(g_{f,r}'/g_{f,r})$ is the derivative of a continuous
argument of $g_{f,r}$. Write $g_f=g_{f,1}$ and $a_f=a_{f,1}$. For
$\gamma:[a,b]\to\C\setminus\{0\}$ continuous, $\Delta_{[a,b]}\arg\gamma$
is the increment of a continuous argument.

\begin{lemma}[Argument principle with half weights]\label{pre:argument}
Let $f\ne0$ have nominal degree $n$ and let $r>0$.
\begin{enumerate}[label=\textup{(\roman*)}]
 \item $a_{f,r}$ extends to a continuous, indeed real-analytic,
 $2\pi$-periodic function, and
 \begin{equation}\label{pre:balanced}
  \frac1{2\pi}\int_0^{2\pi}a_{f,r}(t)\,\dd t
  =\nu_f(r)-\frac{B_f(r)}2-\frac n2;\qquad\hbox{at }r=1:\quad
  \frac1{2\pi}\int_0^{2\pi}a_f(t)\,\dd t=D^\circ(f).
 \end{equation}
 \item If $f$ has no zero on $|z|=r$, then
 $\nu_f(r)-n/2=(2\pi)^{-1}\Delta_{[0,2\pi]}\arg g_{f,r}$. Since
 $g_{f,r}(t+2\pi)=(-1)^ng_{f,r}(t)$, this is half the winding number of
 $g_{f,r}$ over $[0,4\pi]$.
 \item If $f$ has real coefficients, then $a_{f,r}(-t)=a_{f,r}(t)$, and
 the integral in \eqref{pre:balanced} may be replaced by
 $\pi^{-1}\int_0^\pi a_{f,r}(t)\,\dd t$.
\end{enumerate}
\end{lemma}

\begin{proof}
Let $\alpha_1,\ldots,\alpha_d$ be the zeros, $w=re^{it}$. Then
$a_{f,r}(t)=\sum_j\Re\{w/(w-\alpha_j)\}-n/2$. If $|\alpha|=r$ and
$w\ne\alpha$, write $\alpha/w=e^{i\psi}$. Then
$w/(w-\alpha)=(1-e^{i\psi})^{-1}=\frac12+\frac i2\cot(\psi/2)$ has real
part $\frac12$. So each zero on the circle contributes the constant
$\frac12$, and the other terms are real-analytic, proving the extension.
If $|\alpha|<r$ then $w/(w-\alpha)=\sum_{k\ge0}(\alpha/w)^k$, with mean
$1$ over $t$. If $|\alpha|>r$ then $w/(w-\alpha)=-\sum_{k\ge1}(w/\alpha)^k$,
with mean $0$. Summing gives
$\nu_f^\circ(r)+\frac12B_f(r)-\frac n2$, which is \eqref{pre:balanced}.
Part (ii) follows since $a_{f,r}$ is the derivative of the argument of
$g_{f,r}$. For (iii) the zero set is invariant under conjugation,
and $\Re\{\bar w/(\bar w-\alpha)\}=\Re\{w/(w-\bar\alpha)\}$. An even
$2\pi$-periodic function has equal integrals over $[0,\pi]$ and
$[\pi,2\pi]$.
\end{proof}

For real coefficients we write the carrier curve at radius one in
coordinates, $g_f=U_f+iV_f$, so that
\[
 U_f(t)=\sum_{j=0}^na_j\cos\bigl((j-\tfrac n2)t\bigr),\qquad
 V_f(t)=\sum_{j=0}^na_j\sin\bigl((j-\tfrac n2)t\bigr).
\]
We call $U_f,V_f$ the \emph{carrier coordinates} and write
$U_n=U_{f_n}$, $V_n=V_{f_n}$. Pairing the coefficients $a_{n/2\pm\lambda}$
expresses them through the frequency set
$\Lambda_n=\{n/2-j:0\le j<n/2\}$, whose elements are positive integers for
even $n$ and positive half-integers for odd $n$, all at most $n/2$.

\begin{lemma}[Carrier coordinates]\label{pre:carrier}
Let $f$ have real coefficients $a_0,\ldots,a_n$. Then
\begin{equation}\label{pre:pairing}
\begin{aligned}
 U_f(t)&=\ind_{\{n\ \rm even\}}a_{n/2}
       +\sum_{\lambda\in\Lambda_n}(a_{n/2+\lambda}+a_{n/2-\lambda})\cos\lambda t,\\
 V_f(t)&=\sum_{\lambda\in\Lambda_n}(a_{n/2+\lambda}-a_{n/2-\lambda})\sin\lambda t,
\end{aligned}
\end{equation}
and $2U_f(t)=e^{-int/2}(f+f^*)(e^{it})$,
$2iV_f(t)=e^{-int/2}(f-f^*)(e^{it})$. Consequently:
\begin{enumerate}[label=\textup{(\alph*)}]
 \item $U_f$ is even and $V_f$ is odd, and
 $(U_f,V_f)(t+2\pi)=(-1)^n(U_f,V_f)(t)$.
 \item $V_f\equiv0$ if and only if $f=f^*$. Otherwise $V_f$ has at most
 $n$ zeros in every half-open interval of length $2\pi$, counted with
 multiplicity.
 \item If $e^{i\tau}$ is a zero of $f$ of multiplicity $m$, then $\tau$
 is a zero of order at least $m$ of both $U_f$ and $V_f$.
\end{enumerate}
\end{lemma}

\begin{proof}
Pair the indices $n/2\pm\lambda$ in $\sum_ja_je^{i(j-n/2)t}$. Real
coefficients give
$e^{-int/2}f^*(e^{it})=e^{int/2}f(e^{-it})=\overline{g_f(t)}$, which
proves the two identities, and (a) is read off from \eqref{pre:pairing}.
Since $t\mapsto e^{it}$ is locally biholomorphic, the order of vanishing
of $t\mapsto P(e^{it})$ at $\tau$ equals the multiplicity of $e^{i\tau}$
as a zero of $P$. Applied to the nonzero polynomial $P=f-f^*$, of degree
at most $n$, this proves (b). Applied to $P=f$, together with
$U_f=\Re g_f$ and $V_f=\Im g_f$, it proves (c).
\end{proof}

The sign of a crossing is recorded by the crossing index. For a real
function $V$ with finitely many zeros on an interval and an interior
point $t$, let $\sgn V(t\pm)$ be the sign of $V$ on $(t,t+\delta)$,
respectively $(t-\delta,t)$, for small $\delta>0$, and define the
\emph{crossing index}
\[
 \iota_V(t)=\tfrac12\bigl(\sgn V(t+)-\sgn V(t-)\bigr)\in\{-1,0,1\}.
\]
It vanishes unless $V(t)=0$. At a simple zero it equals $\sgn V'(t)$,
and at a zero of even order it is $0$. We use $\sgn0=0$.

\begin{lemma}[Signed half-turns]\label{pre:half-turns}
Let $\gamma=U+iV:[a,b]\to\C\setminus\{0\}$ be continuous with a
continuous argument $\theta$, and suppose $V$ has finitely many zeros in
$[a,b]$. With the sum over the zeros of $V$ in $(a,b)$,
\begin{equation}\label{pre:half-turn-identity}
 \frac{\theta(b)-\theta(a)}\pi
 =\sum_{a<t<b}\sgn U(t)\,\iota_V(t)+e(b)-e(a),
\end{equation}
where $e(a),e(b)\in[0,1]$. If $V(a)\ne0$ then
$e(a)=\theta(a)/\pi-\lfloor\theta(a)/\pi\rfloor$, and if $V(a)=0$ then
$e(a)=\frac12(1-\sgn U(a)\sgn V(a+))$. The same formulas hold at $b$,
with $\sgn V(b-)$ in place of $\sgn V(a+)$. In particular
$|(\theta(b)-\theta(a))/\pi-\sum_{a<t<b}\sgn U(t)\iota_V(t)|\le1$.
\end{lemma}

\begin{proof}
The function $k(t)=\lfloor\theta(t)/\pi\rfloor$ is constant on every
interval where $V\ne0$, because $\theta/\pi\in\Z$ exactly where $V=0$.
Let $V(t)=0$. Then $U(t)\ne0$ and $\theta(t)=m\pi$ with
$(-1)^m=\sgn U(t)$, and for $s$ near $t$,
$\sgn V(s)=\sgn\sin\theta(s)=\sgn U(t)\,\sgn(\theta(s)-m\pi)$. Hence on
either side of $t$, $k(s)=m-\frac12(1-\sgn U(t)\sgn V(s))$. This gives
the jump $k(t+)-k(t-)=\sgn U(t)\iota_V(t)$ at interior zeros and the
values of $e(a)=\theta(a)/\pi-k(a+)$ and $e(b)=\theta(b)/\pi-k(b-)$ at
endpoint zeros. At endpoints with $V\ne0$, $k$ is continuous. Writing
$\theta(b)/\pi-\theta(a)/\pi=k(b-)-k(a+)+e(b)-e(a)$ proves the identity.
\end{proof}

For real $f$ we write
\[
 S(f)=\sum_{0<t<\pi,\ V_f(t)=0}\sgn U_f(t)\,\iota_{V_f}(t),
\]
with $S(f)=0$ if $V_f\equiv0$. By Lemma~\ref{pre:carrier}(b) the sum has
at most $n$ nonzero terms. Zeros of $f$ on the circle contribute
nothing, since $U_f$ vanishes there.

\begin{figure}[t]
\centering
\includegraphics[width=\textwidth,alt={Three panels for one Littlewood polynomial of degree ten. (a) The sine coordinate of its carrier curve, with three zeros. (b) The carrier curve and its crossings of the real axis. (c) The ten zeros of the polynomial, six of them in the unit disk.}]{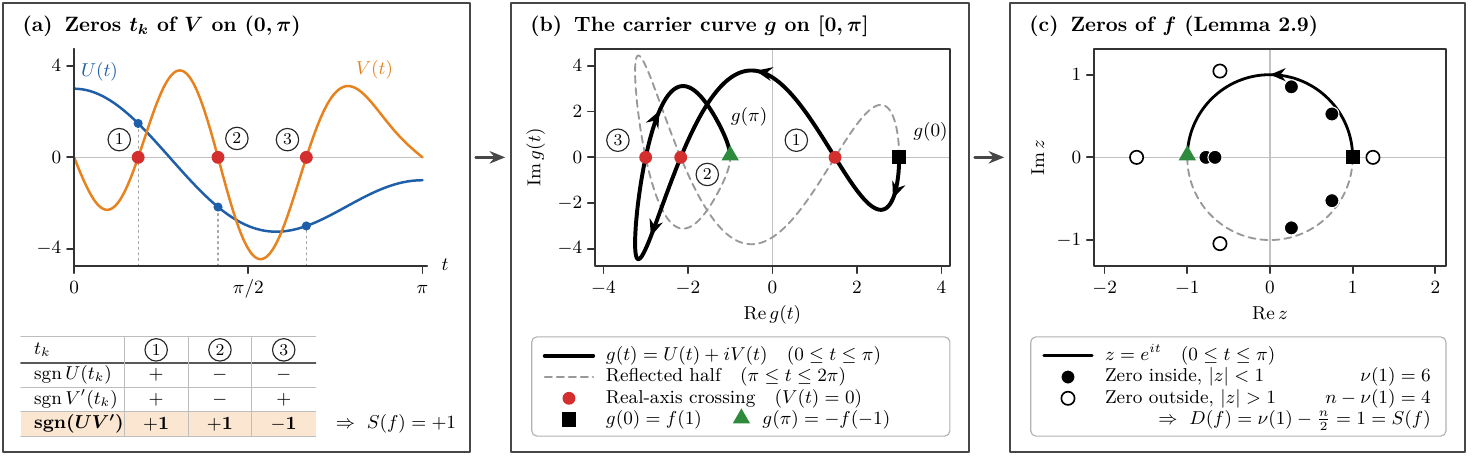}
\caption{Proposition~\ref{pre:crossing} for $f(z)=1+z-z^2+z^3+z^4-z^5+z^6+z^7+z^8-z^9-z^{10}$ ($n=10$),
whose carrier coordinates (Lemma~\ref{pre:carrier}) are $U(t)=-1+2\cos t+2\cos2t$ and $V(t)=2\sin3t-2\sin4t-2\sin5t$.
(a) $V$ has three zeros $t_k$ in $(0,\pi)$, and the signs of $U(t_k)$ and $V'(t_k)$ give $S(f)=+1$.
(b) The carrier curve $g=g_f=U+iV$ crosses the real axis at the points $U(t_k)$. The dashed curve is the reflected half, $\pi\le t\le 2\pi$.
(c) Six of the ten zeros of $f$ lie in the unit disk, so $D(f)=\nu(1)-n/2=1=S(f)$.}
\label{fig:winding}
\end{figure}
\begin{proposition}[$D_n$ as a signed crossing sum]\label{pre:crossing}
Let $f\ne0$ have real coefficients and nominal degree $n$, and write
$U=U_f$, $V=V_f$. Then
\begin{equation}\label{pre:crossing-bound}
 |D^\circ(f)-S(f)|\le1+\tfrac12B(f),\qquad |D(f)-S(f)|\le1+B(f).
\end{equation}
If $V\not\equiv0$, then for every set $A\subset(0,\pi)$,
\begin{equation}\label{pre:bulk-crossing}
 \Bigl|D(f)-\sum_{t\in A,\ V(t)=0}\sgn U(t)\iota_V(t)\Bigr|
 \le1+B(f)+\#\{t\in(0,\pi)\setminus A:V(t)=0\}.
\end{equation}
\end{proposition}

\begin{proof}
If $V\equiv0$ then $f=f^*$ by Lemma~\ref{pre:carrier}(b), so
$D^\circ(f)=0$ by Lemma~\ref{pre:reversal}, and
\eqref{pre:crossing-bound} holds. Let $V\not\equiv0$. We integrate the
argument speed $a_f$ between consecutive zeros of $f$ on the upper half
of the unit circle, and count half-turns on each piece. Since $V$ is
real-analytic, it vanishes identically on no interval. Let
$0=\tau_0<\tau_1<\cdots<\tau_{p+1}=\pi$, where $\tau_1,\ldots,\tau_p$ are
the points of $(0,\pi)$ at which $f(e^{it})=0$. Real coefficients make
$e^{\pm i\tau_l}$ distinct zeros, so $2p\le B(f)$. Fix $l\le p$ and let
$m,m'$ be the orders of vanishing of $g_f$ at $\tau_l,\tau_{l+1}$, zero
if $g_f$ does not vanish there. Then
$\gamma(t)=g_f(t)(t-\tau_l)^{-m}(\tau_{l+1}-t)^{-m'}$ is real-analytic and
zero-free on $[\tau_l,\tau_{l+1}]$. On the open interval it differs from
$g_f$ by a positive factor, so $\Re\gamma$ has the sign of $U$, $\Im\gamma$
has the zeros and crossing indices of $V$, and $a_f$ is the derivative
of a continuous argument of $\gamma$. Since $a_f$ is continuous on the
closed interval, Lemma~\ref{pre:half-turns} gives
\[
 \frac1\pi\int_{\tau_l}^{\tau_{l+1}}a_f(t)\,\dd t
 =\sum_{\tau_l<t<\tau_{l+1}}\sgn U(t)\,\iota_V(t)+\rho_l,
 \qquad|\rho_l|\le1.
\]
Summing over $l$, and using Lemma~\ref{pre:argument}(i),(iii) and
$U(\tau_l)=0$, gives $D^\circ(f)=S(f)+\sum_l\rho_l$, so
$|D^\circ(f)-S(f)|\le p+1\le1+B(f)/2$. Adding $D-D^\circ=B/2$ proves
\eqref{pre:crossing-bound}. Finally, \eqref{pre:bulk-crossing}
follows from \eqref{pre:crossing-bound} because each omitted term has
absolute value at most one.
\end{proof}

In applications $t$ is rescaled to $x=nt\in[0,\pi n]$, the set $A$ is a
bulk interval, and the last term of \eqref{pre:bulk-crossing} counts the
zeros of $V$ in the two endpoint windows.

\begin{figure}[t]
\centering
\includegraphics[width=\textwidth,alt={Eight panels, each showing the carrier curve of a Littlewood polynomial of degree between ten and forty with its signed crossings of the real axis, above a strip that shows the signs of its coefficients.}]{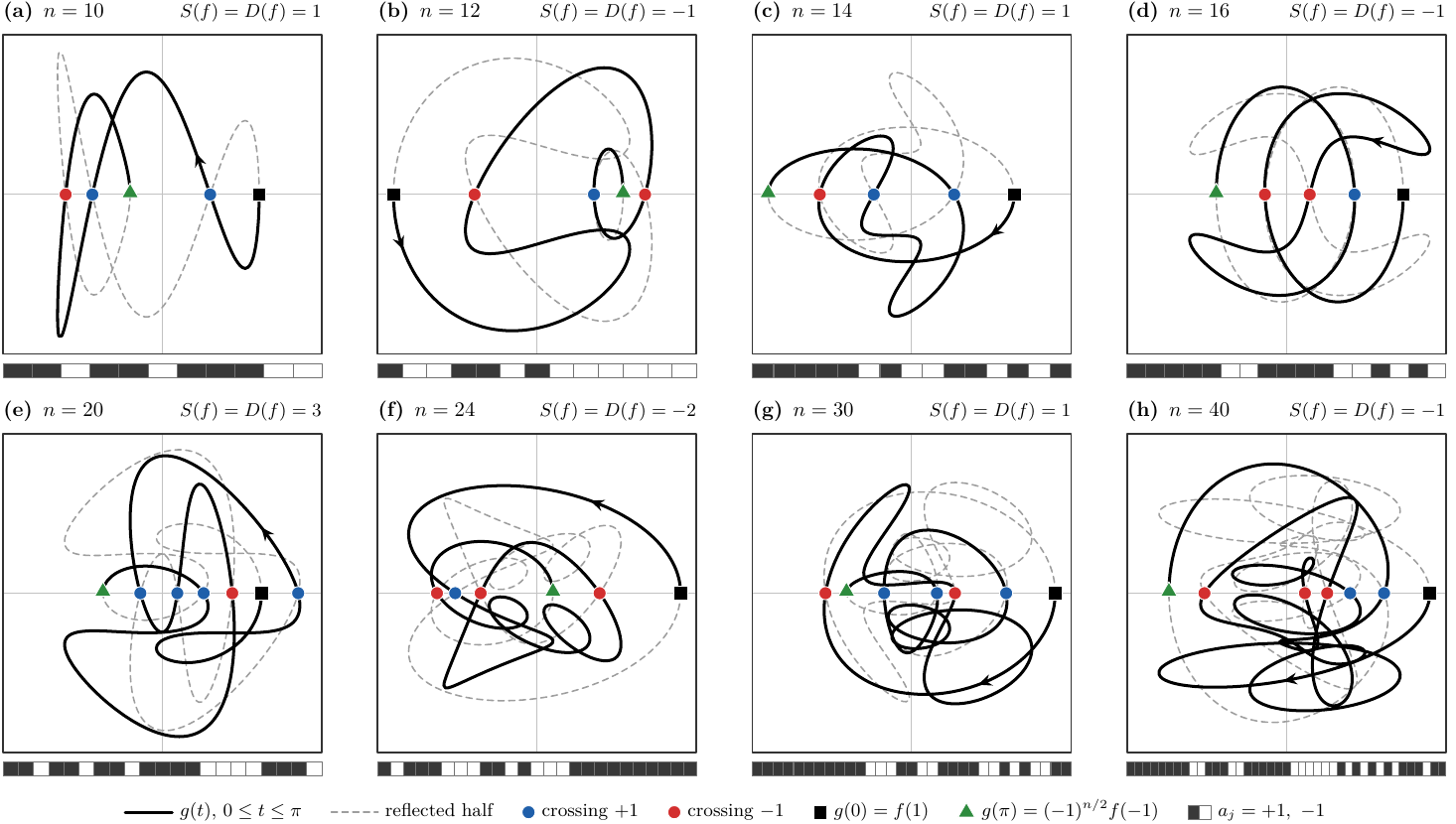}
\caption{Proposition~\ref{pre:crossing} for eight sign polynomials of degrees
$10\le n\le40$, each with $S(f)=D(f)$, where (a) is the polynomial of
Figure~\ref{fig:winding}. Each panel shows the carrier curve $g_f$ on
$[0,\pi]$, its reflected half dashed, and its crossings of the real axis,
counted $+1$ in blue and $-1$ in red. The strip below a panel shows the
coefficients $a_0,\dots,a_n$, dark for $+1$ and light for $-1$.}
\label{fig:gallery}
\end{figure}

\subsection{Trigonometric fields and Gaussian vectors}

Several sections sample the carrier coordinates in the scaled variable
$x=nt\in[0,\pi n]$. For $n\ge1$ and $b\in\R^{n+1}$ let
$\omega_j=(j-n/2)/n\in[-\frac12,\frac12]$ and
\[
 C_b(x)=n^{-1/2}\sum_{j=0}^nb_j\cos(\omega_jx),\qquad
 S_b(x)=n^{-1/2}\sum_{j=0}^nb_j\sin(\omega_jx).
\]
Thus $U_f(x/n)=n^{1/2}C_a(x)$ and $V_f(x/n)=n^{1/2}S_a(x)$, and
$C_b'=-S_{\omega b}$, $S_b'=C_{\omega b}$ with
$(\omega b)_j=\omega_jb_j$ and $|\omega b|\le|b|/2$.

\begin{lemma}[Energy and sampling]\label{pre:sampling}
Let $F\in\{C_b,S_b\}$.
\begin{enumerate}[label=\textup{(\roman*)}]
 \item $\int_0^{\pi n}(F^2+F'^2)\,\dd x\le2\pi|b|^2$.
 \item For every interval $I$ of length one and every $C^1$ function
 $F$, $\sup_IF^2\le2\int_I(F^2+F'^2)$.
 \item If $x_1,\ldots,x_q\in[0,\pi n]$, repetitions allowed, and every
 interval $[k,k+1)$, $k\in\Z$, contains at most $H$ of them, then
 \begin{equation}\label{pre:sampling-bound}
  \sum_{i=1}^qF(x_i)^2\le8\pi H|b|^2 .
 \end{equation}
\end{enumerate}
\end{lemma}

\begin{proof}
As in \eqref{pre:pairing},
\[
 C_b(x)=n^{-1/2}\Bigl\{c_0+\sum_{\lambda\in\Lambda_n}c_\lambda\cos\frac{\lambda x}n\Bigr\},\qquad
 S_b(x)=n^{-1/2}\sum_{\lambda\in\Lambda_n}s_\lambda\sin\frac{\lambda x}n,
\]
where $c_\lambda,s_\lambda=b_{n/2+\lambda}\pm b_{n/2-\lambda}$ and
$c_0=b_{n/2}$ for even $n$, $c_0=0$ for odd $n$. For distinct
$\lambda,\mu$ in $\Lambda_n$, or in $\Lambda_n\cup\{0\}$ when $n$ is
even, both $\lambda-\mu$ and $\lambda+\mu$ are nonzero integers, so
the functions $\cos\lambda t$ are orthogonal on
$[0,\pi]$, and so are the $\sin\lambda t$, with squared norms $\pi/2$
for $\lambda>0$ and $\pi$ for $\lambda=0$. Hence
\[
 \int_0^{\pi n}C_b^2=\pi c_0^2+\frac\pi2\sum_\lambda c_\lambda^2\le\pi|b|^2,
 \qquad
 \int_0^{\pi n}S_b^2=\frac\pi2\sum_\lambda s_\lambda^2\le\pi|b|^2.
\]
Applying
these to $\omega b$ bounds the derivative terms by $\pi|b|^2/4$, which
proves (i). For (ii), $F(x)^2\le2F(y)^2+2(\int_I|F'|)^2$ for $x,y\in I$.
Average over $y$ and use Cauchy--Schwarz. For (iii), assign $x_i$ to
$[\lfloor x_i\rfloor,\lfloor x_i\rfloor+1]$ if
$x_i<\lfloor\pi n\rfloor$, and to $[\pi n-1,\pi n]$ otherwise. Each of
these intervals receives at most $H$ points, and every point of
$[0,\pi n]$ lies in at most two of them. By (ii) and (i),
$\sum_iF(x_i)^2\le2H\cdot2\int_0^{\pi n}(F^2+F'^2)\le8\pi H|b|^2$.
\end{proof}

\begin{lemma}[Gaussian Poincar\'e inequality]\label{pre:gaussian}
Let $G$ be a standard Gaussian vector in $\R^m$, and let
$F\in C^1(\R^m)$ with $F$ and $\nabla F$ bounded by $C(1+|x|)^C$. Then
$\Var F(G)\le\E|\nabla F(G)|^2$ \cite[Theorem~3.20]{BLM2013}.
\end{lemma}
\section{Logarithmic concentration and radial counts}\label{sec:rc}

To count the zeros of $f_n$ at distance $O(1/n)$ from the unit circle we
use Jensen's formula. The angular logarithm $J_n(r)$, the mean of
$\log|f_n|$ over the circle of radius $r$, is a convex function of
$\log r$ whose slope is the number of zeros in the disk of radius $r$
(Lemma~\ref{pre:jensen}). Secants of $J_n$ therefore bound radial counts,
and it suffices to know $J_n(r)$ to within $o(1)$ at the radii
$r=1+O(1/n)$. This route goes back to Yakir~\cite{Yakir2022}. At a single
degree $n$ and with high probability, $J_n(r)$ lies within $o(1)$ of
$\log\sigma_n(r)-\gamma/2$, where $\sigma_n(r)^2=\sum_{k\le n}r^{2k}$ is
the variance of $f_n(re^{i\theta})$ and $\gamma$ is Euler's constant
(Theorem~\ref{rc:concentration}). Lemma~\ref{rc:secant} turns this into
radial counts.

This is the prediction for Gaussian values. If
$f_n(re^{i\theta})/\sigma_n(r)$ were a standard complex Gaussian variable
$G$ at every angle, $J_n(r)-\log\sigma_n(r)$ would have mean
$\E\log|G|=-\gamma/2$. Three facts make the prediction accurate. At two angles in general
position, the normalized values $f_n(re^{i\theta})/\sigma_n(r)$ are close
in law to two independent copies of $G$, so the fraction of the circle on
which $|f_n|\le s\sigma_n(r)$ concentrates at $\Prob(|G|\le s)$. The
logarithmic integrability of random Fourier series, due to Nazarov, Nishry
and Sodin~\cite{NNS2013}, controls the angles where $|f_n|$ is very small,
and Parseval's identity controls those where it is very large.
Section~\ref{sec:rm} combines these estimates across degrees.

\subsection{Admissible coefficients}

\begin{definition}[Admissible laws]\label{rc:class}
A real or complex random variable $\xi$ is \emph{admissible} if
$\xi\overset d=-\xi$, $\E|\xi|^2=1$, and either
\begin{enumerate}[label=\textup{(\alph*)}]
 \item $|\xi|\le B$ almost surely for a constant $B$, or
 \item $\xi$ is real standard Gaussian, or standard circular complex
 Gaussian $(g+ig')/\sqrt2$ with $g,g'$ independent real standard Gaussians.
\end{enumerate}
Put $B_\xi=B$ in case (a), so that $B_\xi\ge1$, and $B_\xi=2$ in case (b).
Put $c_\xi=1/(16B_\xi^4)$.
\end{definition}

With the normalization $\E|\xi_0|^2=1$ of Definition~\ref{pre:models},
every law of Theorem~\ref{thm:R} is admissible, and random signs and
Steinhaus coefficients have $B_\xi=1$. In this section and the next,
$\xi_0,\xi_1,\dots$ are iid admissible and
$f_n(z)=\sum_{k=0}^n\xi_kz^k$. The third absolute moment
$m_3=\E|\xi|^3$ and $\zeta=\E\xi^2$ satisfy
\begin{equation}\label{rc:moments}
 1\le m_3\le B_\xi,\qquad |\zeta|\le1,
\end{equation}
Here $|\zeta|=|\E\xi^2|\le\E|\xi|^2=1$, and the bounds on $m_3$ follow
from H\"older's inequality, from $|\xi|\le B$ in case (a), and from the
values $m_3=2\sqrt{2/\pi}$ and $m_3=3\sqrt\pi/4$ in case (b).

Let $\mu$ be the normalized arc measure $\dd\theta/2\pi$ on
$\mathbb T=\R/2\pi\Z$, and $d(t)=\dist(t,2\pi\Z)$. For $r>0$ put
\[
 \sigma_n(r)^2=\sum_{k=0}^nr^{2k},\qquad
 X_r(\theta)=\frac{f_n(re^{i\theta})}{\sigma_n(r)},\qquad
 J_n(r)=J_{f_n}(r)=\int_{\mathbb T}\log|f_n(re^{i\theta})|\dd\mu(\theta),
\]
with $J_f$ as in Jensen's formula (Lemma~\ref{pre:jensen}). Thus
$\E|X_r(\theta)|^2=1$ at every angle, and
$J_n(r)-\log\sigma_n(r)=\int\log|X_r|\dd\mu$. Both sides equal $-\infty$
exactly when $f_n\equiv0$. Let $G$ be a standard complex Gaussian
variable, with density $\pi^{-1}e^{-|z|^2}$, and
$F_G(s)=\Prob(|G|\le s)=1-e^{-s^2}$. Since $|G|^2$ is exponential with
mean one,
\[
 \E\log|G|=\frac12\int_0^\infty e^{-y}\log y\dd y=-\frac\gamma2 .
\]

The radii of interest form the interval $I_n(K)=[1-K/n,1+K/n]$, $K\ge1$.
For $n\ge2K$, $r,s\in I_n(K)$ and $0\le k\le n$ we have
$e^{-2K}\le r^k\le e^K$, since $\log(1-y)\ge-2y$ for $0\le y\le\frac12$.
Hence the weights $w_k(r)=r^{2k}/\sigma_n(r)^2$ satisfy
\begin{equation}\label{rc:weights}
 (n+1)e^{-4K}\le\sigma_n(r)^2\le(n+1)e^{2K},\qquad
 (rs)^k\le e^{2K},\qquad
 \sum_{k=0}^nw_k(r)=1,\qquad
 w_k(r)\le\frac{e^{6K}}{n+1}.
\end{equation}

Each estimate of this section concerns the law of $(\xi_0,\dots,\xi_n)$
for one $n$. Since $\xi\overset d=-\xi$, this is the law of
$(\epsilon_ka_k)_{k\le n}$, where the $\epsilon_k$ are independent random
signs and the $a_k$ are iid copies of $\xi$ independent of the signs. We
may therefore assume
\begin{equation}\label{rc:sym-rep}
 \xi_k=\epsilon_ka_k,\qquad 0\le k\le n.
\end{equation}

\subsection{Occupation of small disks}
For $s\ge0$ let
\[
 A_s=\mu\{\theta\in\mathbb T:|X_r(\theta)|\le s\}
\]
be the fraction of the circle of radius $r$ on which
$|f_n|\le s\sigma_n(r)$. If the values of $X_r$ at distinct angles were
independent copies of $G$, then $A_s$ would be $F_G(s)$ up to a small
random error. We prove this for admissible coefficients by a Gaussian
approximation of the values at one angle and at two angles.

Covariances of $f_n$ at two angles are oscillatory sums, small away from
the diagonals. The partial sums of $e^{ikt}$ have modulus at most
$1/|\sin(t/2)|\le\pi/d(t)$, so summation by parts gives, for complex
$b_0,\dots,b_N$ and $t\notin2\pi\Z$,
\begin{equation}\label{rc:abel}
 \Big|\sum_{k=0}^Nb_ke^{ikt}\Big|
 \le\pi\,\frac{|b_N|+\sum_{k<N}|b_{k+1}-b_k|}{d(t)} .
\end{equation}
If the $b_k$ are nonnegative and unimodal, that is, nondecreasing up to
some index and nonincreasing after it, with maximum $b_*$, the numerator
is $2b_*-b_0\le2b_*$.

A complex linear form $\sum_k\xi_kc_k$ is a vector in $\R^2$, and its
covariance involves both $\E|\xi|^2=1$ and $\zeta$. For $w\in\C$ let
$\mathcal M[w]$ and $\mathcal R[w]$ be the real matrices of the maps
$v\mapsto wv$ and $v\mapsto w\bar v$ on $\C=\R^2$,
\[
 \mathcal M[w]=\begin{pmatrix}\Re w&-\Im w\\ \Im w&\Re w\end{pmatrix},\qquad
 \mathcal R[w]=\begin{pmatrix}\Re w&\Im w\\ \Im w&-\Re w\end{pmatrix},
\]
both of operator norm $|w|$, and write $\vec y=(\Re y,\Im y)^T$ for
$y\in\C$. For $y=a+ib$ and $w=c+id$, both $\vec y\,\vec w^{\,T}$ and
$\frac12(\mathcal M[y\bar w]+\mathcal R[yw])$ equal
$\bigl(\begin{smallmatrix}ac&ad\\bc&bd\end{smallmatrix}\bigr)$. Taking
expectations with $\E\xi_k\bar\xi_l=\delta_{kl}$ and
$\E\xi_k\xi_l=\zeta\delta_{kl}$ gives, for $Y=\sum_k\xi_kc_k$ and
$W=\sum_k\xi_kd_k$ with deterministic $c_k,d_k\in\C$,
\begin{equation}\label{rc:real-cov}
 \E\big[\vec Y\vec W^{\,T}\big]
 =\frac12\Big(\mathcal M\Big[\sum_kc_k\bar d_k\Big]
   +\mathcal R\Big[\zeta\sum_kc_kd_k\Big]\Big).
\end{equation}

The Gaussian approximation itself is a multivariate Berry--Esseen bound.

\begin{theorem}[Rai\v{c}~\cite{Raic2019}, Theorem 1.1]\label{rc:raic}
Let $Y_0,\dots,Y_n$ be independent centered random vectors in $\R^d$ with
$\sum_k\Cov(Y_k)=I_d$, and let $Z$ be standard Gaussian in $\R^d$.
For every measurable convex set $D\subset\R^d$,
\[
 \Big|\Prob\Big(\sum_kY_k\in D\Big)-\Prob(Z\in D)\Big|
 \le(42d^{1/4}+16)\sum_k\E\|Y_k\|^3.
\]
\end{theorem}

\begin{lemma}[Occupation of small disks]\label{rc:occupation}
Let $K\ge1$, $n\ge1600e^{12K}$ and $r\in I_n(K)$, and put
\[
 H_K=C_0B_\xi e^{6K},\qquad C_0=6000.
\]
\begin{enumerate}[label=\textup{(\roman*)}]
 \item For every $s\ge0$,
 \[
  |\E A_s-F_G(s)|\le H_Kn^{-1/2},\qquad\Var(A_s)\le H_Kn^{-1/2}.
 \]
 \item If moreover $n\ge H_K^4$ and $0<t<1/8$, then
 $A_{n^{-t}}\le3n^{-2t}$ outside an event of probability at most
 $H_Kn^{-1/2+4t}$.
\end{enumerate}
\end{lemma}

\begin{proof}
(i) We approximate the law of $X_r$ at one angle and at two angles by
independent copies of $G$, after discarding the angles near the diagonals
where the covariances are not small. Let
$\Theta_1=\{\theta:d(2\theta)\ge n^{-1/2}\}$, and let $\Theta_2$ be the
set of pairs $(\theta,\phi)$ with
$d(2\theta),d(2\phi),d(\theta-\phi),d(\theta+\phi)\ge n^{-1/2}$. Each of
the four conditions removes a set of $\mu\otimes\mu$-measure
$n^{-1/2}/\pi$, so
\[
 \mu(\mathbb T\setminus\Theta_1)\le n^{-1/2},\qquad
 (\mu\otimes\mu)(\mathbb T^2\setminus\Theta_2)\le2n^{-1/2}.
\]

Fix $(\theta,\phi)\in\Theta_2$ and let $V\in\R^4$ be the vector
$(\vec X_r(\theta),\vec X_r(\phi))$. Put $\Sigma=\Cov(V)$ and
$S(u)=\sum_kw_k(r)e^{iku}$. The covariance identity \eqref{rc:real-cov}
gives
\[
 \Sigma=\tfrac12I_4+E,\qquad
 E=\frac12\begin{pmatrix}
 \mathcal R[\zeta S(2\theta)]&P\\ P^T&\mathcal R[\zeta S(2\phi)]
 \end{pmatrix},\qquad
 P=\mathcal M[S(\theta-\phi)]+\mathcal R[\zeta S(\theta+\phi)].
\]
The weights $w_k(r)$ are monotone in $k$, so \eqref{rc:abel} and
\eqref{rc:weights} give
$|S(u)|\le2\pi e^{6K}/((n+1)d(u))\le2\pi e^{6K}n^{-1/2}$ at the four
angles $u$ that occur. Bounding the diagonal and off-diagonal blocks
separately and using $|\zeta|\le1$,
\[
 \|E\|\le3\pi e^{6K}n^{-1/2}\le10e^{6K}n^{-1/2}\le\tfrac14 .
\]
So $\Sigma$ has its eigenvalues in $[1/4,3/4]$.

Write $\xi_k=\alpha_k+i\beta_k$. Then $V=\sum_kL_k(\alpha_k,\beta_k)^T$
with the $4\times2$ matrices
\[
 L_k=\frac{r^k}{\sigma_n(r)}\begin{pmatrix}\mathcal M[e^{ik\theta}]\\ \mathcal M[e^{ik\phi}]\end{pmatrix},
 \qquad\|L_k\|=\sqrt{2w_k(r)}.
\]
The vectors $Y_k=\Sigma^{-1/2}L_k(\alpha_k,\beta_k)^T$ are independent and
centered, $\sum_k\Cov(Y_k)=I_4$, and $\|\Sigma^{-1/2}\|\le2$. By
\eqref{rc:weights},
\[
 \sum_k\E\|Y_k\|^3\le8m_3\sum_k(2w_k)^{3/2}
 \le23m_3\Big(\max_kw_k\Big)^{1/2}\sum_kw_k\le23m_3e^{3K}n^{-1/2}.
\]
Let $D$ be a product of two closed disks centered at the origin. Then
$\Sigma^{-1/2}D$ is convex and $42\cdot4^{1/4}+16<76$, so
Theorem~\ref{rc:raic} gives
$|\Prob(V\in D)-N(0,\Sigma)(D)|\le1800m_3e^{3K}n^{-1/2}$.

It remains to replace $N(0,\Sigma)$ by $N(0,\frac12I_4)$, the law of two
independent copies of $G$. The eigenvalues of $2\Sigma$ are $1+u_j$ with
$|u_j|\le2\|E\|\le\frac12$, and $u-\log(1+u)\le u^2$ for
$|u|\le\frac12$. Hence the relative entropy satisfies
\[
 D\big(N(0,\Sigma)\,\|\,N(0,\tfrac12I_4)\big)
 =\frac12\sum_{j=1}^4\big(u_j-\log(1+u_j)\big)\le8\|E\|^2,
\]
and Pinsker's inequality gives the total variation bound
$\sup_D|N(0,\Sigma)(D)-N(0,\frac12I_4)(D)|\le2\|E\|\le20e^{6K}n^{-1/2}$.
Since $m_3\le B_\xi$ by \eqref{rc:moments}, we conclude that for all $s,s'\ge0$,
\[
 \big|\Prob(|X_r(\theta)|\le s,\ |X_r(\phi)|\le s')-F_G(s)F_G(s')\big|
 \le\delta:=1820B_\xi e^{6K}n^{-1/2}.
\]
The same argument in $\R^2$ gives
$|\Prob(|X_r(\theta)|\le s)-F_G(s)|\le\delta$ for $\theta\in\Theta_1$.

By Tonelli's theorem,
$\E A_s=\int\Prob(|X_r(\theta)|\le s)\dd\mu$ and
$\E A_s^2=\iint\Prob(|X_r(\theta)|\le s,|X_r(\phi)|\le s)\dd\mu\dd\mu$.
Bounding the integrands by one on the discarded sets,
\[
 |\E A_s-F_G(s)|\le\delta+n^{-1/2},\qquad
 |\E A_s^2-F_G(s)^2|\le\delta+2n^{-1/2}.
\]
Since $\E A_s$ and $F_G(s)$ lie in $[0,1]$,
\[
 \Var(A_s)\le|\E A_s^2-F_G(s)^2|+|F_G(s)^2-(\E A_s)^2|\le3\delta+4n^{-1/2},
\]
and both bounds are at most $C_0B_\xi e^{6K}n^{-1/2}$.

(ii) Put $\delta=H_Kn^{-1/2}$. Since $n^{1/4}\ge H_K$ and $t<1/8$, we have
$\delta\le n^{-1/4}\le n^{-2t}$. As $F_G(s)\le s^2$, part (i) gives
$\E A_{n^{-t}}\le n^{-2t}+\delta\le2n^{-2t}$ and $\Var(A_{n^{-t}})\le\delta$,
and Chebyshev's inequality gives
$\Prob(A_{n^{-t}}>3n^{-2t})\le\delta n^{4t}=H_Kn^{-1/2+4t}$.
\end{proof}

\subsection{Logarithmic moments}
The occupation fractions control $\log|X_r|$ only where $|X_r|$ is
neither very small nor very large. Small values are controlled by the
following theorem, and large values by the angular energy.

\begin{theorem}[Nazarov--Nishry--Sodin~\cite{NNS2013}, Corollary 1.2]\label{rc:nns}
There is an absolute constant $C_{\rm NNS}\ge1$ such that for all complex
$b_0,\dots,b_n$ with $\sum_k|b_k|^2=1$ and every real $p\ge1$,
\[
 \E\int_{\mathbb T}\Big|\log\Big|\sum_{k=0}^n\epsilon_kb_ke^{ik\theta}\Big|\Big|^p
 \dd\mu(\theta)\le(C_{\rm NNS}\,p)^{6p}.
\]
\end{theorem}

\begin{lemma}[Logarithmic moments and angular energy]\label{rc:log-moments}
Let $K\ge1$, $n\ge2K$ and $r\in I_n(K)$, and put
$\Lambda_\xi=2(C_{\rm NNS}+\log B_\xi+2)$ and
$\Lambda_\xi'=(4e\Lambda_\xi)^{12}$. There is an event
$\mathcal H=\mathcal H_n(r)$, with $\mathcal H=\Omega$ when $|\xi|=1$
almost surely, such that
\begin{enumerate}[label=\textup{(\roman*)}]
 \item $\Prob(\mathcal H^c)\le2\exp\big(-c_\xi(n+1)e^{-6K}\big)$;
 \item on $\mathcal H$, $f_n\not\equiv0$ and
 $\frac12\le\int|X_r|^2\dd\mu\le B_\xi^2$;
 \item for every real $p\ge1$,
 \[
  \E\Bigl[\ind_{\mathcal H}\int\bigl|\log|X_r|\bigr|^p\dd\mu\Bigr]\le(\Lambda_\xi p)^{6p};
 \]
 \item if $n\ge3$, the logarithmic norm
 $\mathcal L_n(r)=\int|\log|X_r||^2\dd\mu$ satisfies
 \[
  \Prob\big(\mathcal H\cap\{\mathcal L_n(r)\le\Lambda_\xi'(\log n)^{14}\}\big)
  \ge1-n^{-12}-2\exp\big(-c_\xi(n+1)e^{-6K}\big).
 \]
\end{enumerate}
\end{lemma}

\begin{proof}
Use \eqref{rc:sym-rep}. By Parseval's identity,
$V:=\int|X_r|^2\dd\mu=\sum_kw_k|a_k|^2$ with $w_k=w_k(r)$. Let
$\mathcal H=\{1/2\le V\le B_\xi^2\}$, an event of the amplitudes alone.
If $|\xi|=1$ then $V=1$ and $\mathcal H=\Omega$. The restriction to
$\mathcal H$ matters for laws with an atom at zero, for which
$f_n\equiv0$ with positive probability.

(i) In case (a), $V\le B^2$, and Hoeffding's inequality for
$V-1=\sum_kw_k(|a_k|^2-1)$, whose summands range over intervals of length
$w_kB^2$, gives
$\Prob(V<\frac12)\le\exp(-1/(2B^4\sum_kw_k^2))$.
In case (b), $Y_k=|a_k|^2$ satisfies $\E e^{\lambda Y_k}\le(1-2\lambda)^{-1/2}$
for $0\le\lambda<1/2$ and $\E e^{-\lambda Y_k}\le(1+2\lambda)^{-1/2}$ for
$\lambda\ge0$. These are equalities for real Gaussians, and for complex
Gaussians, where $Y_k$ is exponential with mean one, they follow from
$(1\mp\lambda)^2\ge1\mp2\lambda$. Let $w_*=\max_kw_k$ and
$\lambda=1/(4w_*)$. Using $-\log(1-y)\le y+y^2$ for $0\le y\le\frac12$ and
$\log(1+y)\ge y-y^2/2$ for $y\ge0$, the exponential Markov inequality
gives
\[
 \Prob(V>B_\xi^2)\le\Prob(V>2)\le e^{-2\lambda}\prod_k(1-2\lambda w_k)^{-1/2}
 \le e^{-\lambda+2\lambda^2w_*}=e^{-1/(8w_*)},
\]
\[
 \Prob(V<\tfrac12)\le e^{\lambda/2}\prod_k(1+2\lambda w_k)^{-1/2}
 \le e^{-\lambda/2+\lambda^2w_*}=e^{-1/(16w_*)}.
\]
In both cases $\sum_kw_k^2\le w_*\le e^{6K}/(n+1)$ by \eqref{rc:weights}.
Since $c_\xi\le1/(2B^4)$ in case (a) and $c_\xi\le1/16$ in case (b), this
proves (i).

(ii) On $\mathcal H$ some $a_k$ is nonzero, so $f_n\not\equiv0$, and the
energy bounds hold by definition.

(iii) Fix the amplitudes in $\mathcal H$ and put
$b_k=r^ka_k/(\sigma_n(r)\sqrt V)$. Then $\sum_k|b_k|^2=1$ and
$X_r=\sqrt V\sum_k\epsilon_kb_ke^{ik\theta}$. Since
$|\log\sqrt V|\le D:=\max(\frac12\log2,\log B_\xi)\le\log B_\xi+1$,
Theorem~\ref{rc:nns}, applied to the signs, which are independent of the
amplitudes, and $(x+y)^p\le2^{p-1}(x^p+y^p)$ give
\[
 \E\Big[\ind_{\mathcal H}\int|\log|X_r||^p\dd\mu\Big]
 \le2^{p-1}\big((C_{\rm NNS}p)^{6p}+D^p\big)
 \le\big(2(C_{\rm NNS}+D+1)p\big)^{6p}\le(\Lambda_\xi p)^{6p}.
\]

(iv) Let $m=\lceil\log n\rceil\le2\log n$. By Jensen's inequality on
$(\mathbb T,\mu)$, $\mathcal L_n(r)^m\le\int|\log|X_r||^{2m}\dd\mu$, so (iii) with
$p=2m$ gives
$\E[\ind_{\mathcal H}\mathcal L_n(r)^m]\le(2\Lambda_\xi m)^{12m}\le(4\Lambda_\xi\log n)^{12m}$.
By Markov's inequality,
\[
 \Prob\big(\mathcal H\cap\{\mathcal L_n(r)>\Lambda_\xi'(\log n)^{14}\}\big)
 \le\Big(\frac{(4\Lambda_\xi\log n)^{12}}{(4e\Lambda_\xi)^{12}(\log n)^{14}}\Big)^m
 \le e^{-12m}\le n^{-12},
\]
and (i) bounds $\Prob(\mathcal H^c)$.
\end{proof}

\subsection{Concentration of the angular logarithm}
Truncating $\log|X_r|$ at the level $t\log n$ combines the three
estimates.

\begin{theorem}[Logarithmic concentration]\label{rc:concentration}
Let $\xi$ be admissible.
There is a constant $C_\xi$, depending only on $B_\xi$ and $C_{\rm NNS}$, such
that the following holds for all $K\ge1$, $0<t<1/8$, all $n$ with
\[
 n\ge H_K^4,\qquad H_K=6000B_\xi e^{6K},
\]
and all $r\in I_n(K)=[1-K/n,1+K/n]$.
Let $e_n(t)=C_\xi n^{-t}(\log n)^7$.
Then
\[
 \Prob\big(|J_n(r)-\log\sigma_n(r)+\gamma/2|>e_n(t)\big)\le p_n(K,t),
\]
where
\begin{equation}\label{rc:pn}
 p_n(K,t)=2H_Kn^{-1/2+4t}+n^{-12}+2\exp\big(-c_\xi(n+1)e^{-6K}\big).
\end{equation}
Outside the exceptional event $f_n\not\equiv0$.
\end{theorem}

The deviation $e_n(t)$ is the same for every width $K$, which enters only
through the failure probability and the degree threshold. One may take
$C_\xi=2(4e\Lambda_\xi)^6+B_\xi^2+3$, with $\Lambda_\xi$ as in
Lemma~\ref{rc:log-moments}.

\begin{proof}
Write $X=X_r$. Since $H_K\ge6000$, we have
$n\ge H_K^4\ge\max(1600e^{12K},3)$, so Lemmas~\ref{rc:occupation}
and~\ref{rc:log-moments} apply and $\log n\ge1$. Put
\[
 T=t\log n,\qquad u=n^{-t},\qquad \delta=H_Kn^{-1/2}\le n^{-1/4}\le n^{-2t},
\]
and let $\ell_T(z)=\max(-T,\min(T,\log|z|))$, with $\ell_T(0)=-T$. The truncated
integral $L_T=\int\ell_T(X)\dd\mu$ is a functional of the occupation
fractions. We show that it is close to its mean, that the mean is close to
$\E\ell_T(G)$ and hence to $\E\log|G|=-\gamma/2$, and that truncation
changes $\int\log|X|\dd\mu$ little.

\emph{The truncated integral.} For every $z$,
$\ell_T(z)=T-\int_{-T}^T\ind\{|z|\le e^s\}\dd s$. By Fubini's theorem,
\[
 L_T=T-\int_{-T}^TA_{e^s}\dd s,\qquad
 \E\ell_T(G)=T-\int_{-T}^TF_G(e^s)\dd s .
\]
Lemma~\ref{rc:occupation}(i) gives $|\E L_T-\E\ell_T(G)|\le2T\delta$, and,
since $|\Cov(A_{e^s},A_{e^{s'}})|\le\delta$ by the Cauchy--Schwarz
inequality,
\[
 \Var(L_T)=\int_{-T}^T\!\int_{-T}^T\Cov(A_{e^s},A_{e^{s'}})\dd s\dd s'
 \le4T^2\delta .
\]
By Chebyshev's inequality, $|L_T-\E L_T|\le u$ outside an event of
probability at most
\[
 \frac{4T^2\delta}{u^2}=H_Kn^{-1/2+4t}(2t\log n)^2n^{-2t}\le H_Kn^{-1/2+4t},
\]
because $y^2e^{-y}\le4e^{-2}<1$ for $y=2t\log n$.

\emph{Large and small values.} Suppose that the events of
Lemma~\ref{rc:log-moments}(iv) and Lemma~\ref{rc:occupation}(ii) occur.
Then $f_n\not\equiv0$, so $\log|X|$ is integrable, and
\[
 \int\log|X|\dd\mu-L_T
 =\int(\log|X|-T)^+\dd\mu-\int(-T-\log|X|)^+\dd\mu .
\]
Chebyshev's inequality on $\mathbb T$ and the energy bound
$\int|X|^2\dd\mu\le B_\xi^2$ of Lemma~\ref{rc:log-moments}(ii) give
\[
 \int(\log|X|-T)^+\dd\mu=\int_T^\infty\mu\{|X|>e^s\}\dd s
 \le\int_T^\infty B_\xi^2e^{-2s}\dd s=\frac{B_\xi^2}2n^{-2t}.
\]
Since $e^{-T}=n^{-t}$, the Cauchy--Schwarz inequality and the bounds on
$\mathcal L_n(r)$ and $A_{n^{-t}}$ give
\[
 \int(-T-\log|X|)^+\dd\mu\le\int_{\{|X|<n^{-t}\}}|\log|X||\dd\mu
 \le\sqrt{\mathcal L_n(r)\,A_{n^{-t}}}\le\sqrt3\,(4e\Lambda_\xi)^6(\log n)^7n^{-t}.
\]

\emph{The Gaussian.} Since $\Prob(|G|>e^s)\le e^{-2s}$ and
$\Prob(|G|<e^{-s})\le e^{-2s}$, we have
$|\E\ell_T(G)-\E\log|G||\le e^{-2T}=n^{-2t}$.

\emph{Assembly.} The three exceptional events have total probability at
most $p_n(K,t)$. Outside them,
\begin{align*}
 \Big|\int\log|X|\dd\mu-\E\log|G|\Big|
 \le{}&\Big|\int\log|X|\dd\mu-L_T\Big|+|L_T-\E L_T|\\
 &+|\E L_T-\E\ell_T(G)|+|\E\ell_T(G)-\E\log|G||,
\end{align*}
and the four bounds above, with $\delta\le n^{-2t}\le n^{-t}$ and
$\E\log|G|=-\gamma/2$, give
\[
 \Big|\int\log|X|\dd\mu+\frac\gamma2\Big|
 \le\Big(1+2t\log n+\sqrt3(4e\Lambda_\xi)^6(\log n)^7+\frac{B_\xi^2}2+1\Big)n^{-t}
 \le e_n(t).
\]
Finally $\int\log|X|\dd\mu=J_n(r)-\log\sigma_n(r)$.
\end{proof}

\subsection{Jensen secants and radial counts}
By Jensen's formula, $J_n(e^\tau)$ is a convex function of $\tau$ whose
slope is the number of zeros in the disk of radius $e^\tau$. By
Theorem~\ref{rc:concentration} its secants over radii $1+O(1/n)$ are
close to those of $\log\sigma_n$, which follow the profile
\[
 F_{\rm rad}(x)=\frac12\log\int_0^1e^{2xs}\dd s,\qquad x\in\R .
\]
Differentiation under the integral gives $F_{\rm rad}'=\Phi_{\rm rad}$ and
$\Phi_{\rm rad}'(x)=2\Var(\tau_x)$, where $\tau_x$ has density proportional
to $e^{2xs}$ on $[0,1]$. Hence
\begin{equation}\label{rc:Phi-props}
 0\le\Phi_{\rm rad}\le1,\qquad 0\le\Phi_{\rm rad}'\le\tfrac12,\qquad
 \Phi_{\rm rad}(0)=\tfrac12,\qquad F_{\rm rad}(x)=x+F_{\rm rad}(-x),
\end{equation}
the last by the substitution $s\mapsto1-s$ in
$\int_0^1e^{2xs}\dd s=e^{2x}\int_0^1e^{-2xs}\dd s$.
In particular $\Phi_{\rm rad}\le\frac12$ on $(-\infty,0]$.

\begin{lemma}[The radial profile]\label{rc:profile}
For $K\ge1$, $n\ge4e^{4K}$ and $|x|\le K$,
\[
 \Big|\log\sigma_n(1+x/n)-\tfrac12\log n-F_{\rm rad}(x)\Big|\le\frac{3e^{4K}}n .
\]
\end{lemma}

\begin{proof}
Since $|\log(1+y)-y|\le y^2$ for $|y|\le\frac12$, replacing
$(1+x/n)^{2k}$ by $e^{2kx/n}$ for $0\le k\le n$ changes
$\log\sigma_n(1+x/n)^2$ by at most $2K^2/n$. The function
$s\mapsto e^{2xs}$ is monotone on $[0,1]$ with values in
$[e^{-2K},e^{2K}]$, so the Riemann sum over $1\le k\le n$ and the term
$k=0$ give
\[
 \Big|\frac1n\sum_{k=0}^ne^{2kx/n}-\int_0^1e^{2xs}\dd s\Big|
 \le\frac{2e^{2K}}n .
\]
The integral is at least $e^{-2K}$, so the relative error is at most
$2e^{4K}/n\le\frac12$, and $|\log(1+v)|\le2|v|$ for $|v|\le\frac12$ gives
$|\log\sigma_n(1+x/n)^2-\log n-2F_{\rm rad}(x)|\le(2K^2+4e^{4K})/n
\le6e^{4K}/n$. Halving proves the lemma.
\end{proof}

The passage from secants to counts is deterministic and works for any
smooth profile $F$. It uses only the secant inequality
\eqref{pre:secants}. In the application $F=F_{\rm rad}$ and
$c=\frac12\log n-\frac\gamma2$.

\begin{lemma}[Jensen secants]\label{rc:secant}
Let $K>0$, $n\ge4K$, $h>0$ and $|x|+h\le K$.
Let $F\in C^1[-K,K]$ satisfy $|F'|\le B_F$ and
$|F'(y)-F'(y')|\le M_F|y-y'|$, and let $c\in\R$ and $\eta\ge0$.
If a polynomial $f$ satisfies $|J_f(1+y/n)-c-F(y)|\le\eta$ for
$y\in\{x-h,x,x+h\}$, then
\[
 \Big|\frac{\nu_f(1+x/n)}n-F'(x)\Big|
 \le\frac{M_Fh}2+\frac{B_FK}n+\frac{3\eta}h .
\]
\end{lemma}

\begin{proof}
The hypothesis makes $J_f$ finite, so $f\not\equiv0$.
Let $\rho_y=1+y/n$ and
$\Delta_\pm=n|\log(\rho_{x\pm h}/\rho_x)|=\int_0^h\dd s/(1+(x\pm s)/n)$.
Since $|x\pm s|\le K$ in the integrand, $h/\Delta_\pm\in[1-K/n,1+K/n]$.
By \eqref{pre:secants} and the hypothesis,
\[
 \frac{F(x)-F(x-h)-2\eta}{\Delta_-}
 \le\frac{\nu_f(\rho_x)}n
 \le\frac{F(x+h)-F(x)+2\eta}{\Delta_+}.
\]
The ordinary secants $Q_\pm=(F(x\pm h)-F(x))/(\pm h)$ satisfy
$|Q_\pm-F'(x)|\le M_Fh/2$ and $|Q_\pm|\le B_F$.
Replacing the denominator $h$ by $\Delta_\pm$ changes $Q_\pm$ by at most
$B_F|h/\Delta_\pm-1|\le B_FK/n$.
Finally $2\eta/\Delta_\pm\le2\eta(1+K/n)/h\le3\eta/h$.
\end{proof}

Figure~\ref{rc:fig-secants} shows these secants for one sequence of random
signs, together with the concentration of Theorem~\ref{rc:concentration}
and a block of degrees from Section~\ref{sec:rm}.

\begin{figure}[htb]
\centering
\includegraphics[width=\textwidth,alt={Three panels. (a) Jensen secants for one sequence of random signs at degree one hundred thousand against the limiting curve. (b) The angular average of the logarithm at the unit circle for thousands of sign sequences at degrees from one thousand to one million, with the sample means and a band of two standard deviations. (c) The same quantity at every degree of one block of degrees for three sequences.}]{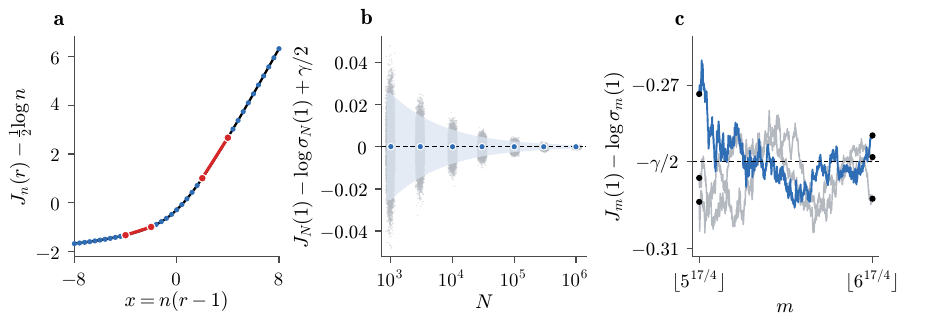}
\caption{(a) Jensen secants for one sequence of random signs at degree
$n=10^5$, with $x=n(r-1)$: the dots are $J_n(r)-\frac12\log n$, the curve
is $F_{\rm rad}(x)-\gamma/2$, and the chords run over $[r_1,r_2]$ and
$[r_3,r_4]$, where $r_1,\dots,r_4=1-4/n,\,1-2/n,\,1+2/n,\,1+4/n$. By
\eqref{pre:secants} the slope of each chord, measured in $n\log r$, lies
between the counts $\nu_n(r)/n$ at its two ends.
(b) The concentration of Theorem~\ref{rc:concentration} at $r=1$, for the
$16384$ sequences of random signs of Subsection~\ref{ap:radial} at
$N=10^3,\,3\cdot10^3,\,10^4,\,3\cdot10^4,\,10^5$ and the first $2048$ of them
at $N=3\cdot10^5,\,10^6$. The dots are the sample means and the band extends
two sample standard deviations to either side. The standard deviation is
$0.40/\sqrt N$ at every degree.
(c) $J_m(1)-\log\sigma_m(1)$ at every degree of the block between the
sparse degrees $\lfloor5^{17/4}\rfloor=934$ and $\lfloor6^{17/4}\rfloor=2028$
of Section~\ref{sec:rm}, for the sequence of (a) in blue and two further
sequences, with the values at the sparse degrees marked. The angular mean
barely moves through a block, which is how the proof passes from sparse
degrees to all degrees (Lemma~\ref{rm:block-increment}).}
\label{rc:fig-secants}
\end{figure}

\section{The radial law at all degrees}\label{sec:rm}\label{rm:sec}

Theorem~\ref{rc:concentration} controls the angular logarithm at one
degree, with a failure probability that is summable along the sparse
degrees $N_j=\lfloor j^{17/4}\rfloor$ but not over all degrees. To reach
every degree $m$ in a block $N_j\le m\le N_{j+1}$ it suffices to show that
$J_m(r)-J_N(r)$ is small, where $N=N_j$, because the secant bound of
Lemma~\ref{rc:secant} is deterministic and applies to $f_m$ at the scale
$N$. Appending the terms of degree above $N$ changes $f_N$ by a polynomial
of size about $\sqrt{M\log N}$ on the circle, where $M=N_{j+1}-N_j$, and
this is much smaller than $\sigma_N\approx\sqrt N$. So
$\log|f_m|-\log|f_N|$ is small except where $|f_N|$ is itself small. That
set has small measure, and the logarithmic norms of
Lemma~\ref{rc:log-moments} control its contribution. The coefficients are
iid admissible, as in Section~\ref{sec:rc}.

\subsection{Appended tails}
The first estimate bounds all the polynomials appended to $f_N$ within a
block at once.

\begin{lemma}[Appended tails]\label{rm:tail}
Let $K_0\ge0$, $N\ge\max(4,K_0)$ and $1\le M\le N$, and put
\[
 \rho=1+\frac{K_0}N,\qquad a=15B_\xi e^{2K_0}\sqrt{M\log N},\qquad
 G_m(z)=f_{N+m}(z)-f_N(z)=\sum_{k=N+1}^{N+m}\xi_kz^k .
\]
Outside an event of probability at most $2N^{-10}$,
\[
 \max_{0\le m\le M}\ \sup_{|z|\le\rho}|G_m(z)|<a .
\]
\end{lemma}

\begin{proof}
We bound each $G_m$ on a fine net of the circle $|z|=\rho$ by a
sub-Gaussian tail, and pass from the net to the circle with a derivative
bound. Linear forms in admissible coefficients have sub-Gaussian tails.
For deterministic $c_k\in\C$, finitely many of them nonzero, and $t\ge0$,
\begin{equation}\label{rm:subgauss}
 \Prob\Big\{\Big|\Re\sum_k\xi_kc_k\Big|\ge t\Big\}
 \le2\exp\Big[-\frac{t^2}{2B_\xi^2\sum_k|c_k|^2}\Big],
\end{equation}
and the same holds for the imaginary part. For bounded laws this is
Hoeffding's inequality for the signs in \eqref{rc:sym-rep}, applied
conditionally on the amplitudes. For Gaussian laws the real part is a
centered Gaussian variable with variance at most $\sum_k|c_k|^2$. Since
$|w|\ge t$ forces $|\Re w|\ge t/\sqrt2$ or $|\Im w|\ge t/\sqrt2$,
\begin{equation}\label{rm:subgauss-abs}
 \Prob\Big\{\Big|\sum_k\xi_kc_k\Big|\ge t\Big\}
 \le4\exp\Big[-\frac{t^2}{4B_\xi^2\sum_k|c_k|^2}\Big].
\end{equation}

For $|z|=\rho$ and $1\le m\le M$ we have
$\sum_{k=N+1}^{N+m}\rho^{2k}\le Me^{4K_0}$, since $k\le2N$. By
\eqref{rm:subgauss-abs},
\[
 \Prob\{|G_m(z)|\ge2a/3\}\le4\exp\Big[-\frac{(2a/3)^2}{4B_\xi^2Me^{4K_0}}\Big]
 =4N^{-25}.
\]
We use $\lceil N^4\rceil\le2N^4$ equally spaced points on $|z|=\rho\le2$,
hence at most $2N^5$ pairs of a point and an index $m$. The coefficients
are rarely large. The envelope event
\[
 \mathcal E^{\rm env}_N=\Big\{\max_{k\le2N}|\xi_k|\le B_\xi N\Big\}
 \qquad\text{satisfies}\qquad
 \Prob\big((\mathcal E^{\rm env}_N)^c\big)\le(2N+1)e^{-2N^2}.
\]
It holds surely for bounded laws, and for Gaussian laws
$\Prob(|\xi_0|>2N)\le e^{-2N^2}$. On $\mathcal E^{\rm env}_N$ we have
$|G_m'|\le2B_\xi N^2Me^{2K_0}$ on $|z|\le\rho$, so moving to the nearest
sample point changes $G_m$ by at most $4\pi B_\xi Me^{2K_0}/N^2<a/3$.
Outside an event of probability at most
$(2N+1)e^{-2N^2}+8N^{-20}\le2N^{-10}$, therefore, $|G_m|<a$ on the circle
for all $1\le m\le M$. The maximum principle gives the claim, and $G_0=0$.
\end{proof}

\subsection{The increment of the angular logarithm}
Where $|f_N|$ is much larger than the appended tail $T=f_m-f_N$, the
difference $\log|f_m|-\log|f_N|=\log|1+T/f_N|$ is small. The rest of the
circle, where $|f_N|$ is small, has small measure, and there the $L^2$
norms of $\log|f_N|$ and $\log|f_m|$ control the difference. The
comparison is deterministic.

\begin{lemma}[Increment of the angular logarithm]\label{rm:increment}
Let $f$ and $g$ be nonzero polynomials, $r>0$ and $\sigma_f,\sigma_g\ge1$.
Put $T=g-f$ and
\[
 A^f_u=\mu\{\theta:|f(re^{i\theta})|\le u\sigma_f\},\qquad
 \mathcal L_f=\int\Big|\log\frac{|f(re^{i\theta})|}{\sigma_f}\Big|^2\dd\mu,\qquad
 \mathcal L_g=\int\Big|\log\frac{|g(re^{i\theta})|}{\sigma_g}\Big|^2\dd\mu .
\]
Suppose that $|T|\le\tau$ on $|z|=r$, put $u_0=2\tau/\sigma_f$, and let
$u_1\ge u_0$. Then
\[
 |J_g(r)-J_f(r)|\le\frac{u_0}{u_1}+A^f_{u_1}
 +\sqrt{A^f_{u_1}}\,\Big(\sqrt{\mathcal L_f}+\sqrt{\mathcal L_g}+\log\sigma_f+\log\sigma_g\Big).
\]
\end{lemma}

\begin{proof}
All functions below are evaluated at $re^{i\theta}$. A nonzero polynomial
has finitely many zeros on the circle, so $\log|f|$ and $\log|g|$ are
square integrable, and $J_g(r)-J_f(r)=\int(\log|g|-\log|f|)\dd\mu$. Let
$E_1=\{\theta:f\ne0,\ |T|\le|f|/2\}$ and $E_2=\mathbb T\setminus E_1$.

On $E_1$ we write $g=f(1+T/f)$ with $|T/f|\le\frac12$. Since
$\log(1-y)\ge-2y$ and $\log(1+y)\le y$ for $0\le y\le\frac12$,
\[
 \big|\log|g|-\log|f|\big|=\big|\log|1+T/f|\big|\le2\frac{|T|}{|f|}\le1 .
\]
Where in addition $|f|\ge u_1\sigma_f$, the middle term is at most
$2\tau/(u_1\sigma_f)=u_0/u_1$. The rest of $E_1$ lies in
$\{|f|<u_1\sigma_f\}$, of measure at most $A^f_{u_1}$, and there we use
the bound $1$. Hence
$\int_{E_1}|\log|g|-\log|f||\dd\mu\le u_0/u_1+A^f_{u_1}$.

On $E_2$, either $f=0$ or $|f|<2|T|\le2\tau=u_0\sigma_f$, so
$E_2\subset\{|f|\le u_1\sigma_f\}$ and $\mu(E_2)\le A^f_{u_1}$. Since
$\log|f|=\log(|f|/\sigma_f)+\log\sigma_f$ with $\log\sigma_f\ge0$,
Minkowski's inequality gives
$\|\log|f|\|_{L^2(\mu)}\le\sqrt{\mathcal L_f}+\log\sigma_f$, and likewise for $g$.
The Cauchy--Schwarz inequality gives
\[
 \int_{E_2}\big|\log|g|-\log|f|\big|\dd\mu
 \le\sqrt{\mu(E_2)}\,\big(\|\log|f|\|_{L^2(\mu)}+\|\log|g|\|_{L^2(\mu)}\big).
\]
Adding the two parts proves the lemma.
\end{proof}

Over a block of degrees, Lemma~\ref{rm:tail} controls the appended tails,
Lemma~\ref{rc:occupation}(ii) the small values of $f_N$, and
Lemma~\ref{rc:log-moments}(iv) the logarithmic norms.

\begin{lemma}[Block increment]\label{rm:block-increment}
Let $K\ge1$, $0<t<1/8$, $N\ge H_K^4$, $1\le M\le N$, $|y|\le K$ and
$\rho=1+y/N$. Put
\[
 \tau=15B_\xi e^{2K}\sqrt{M\log N},\qquad u_0=\frac{2\tau}{\sigma_N(\rho)},
\]
and suppose that
\begin{equation}\label{rm:block-condition}
 900B_\xi^2e^{8K}M\log N\le N^{1-2t}.
\end{equation}
There is an event $\mathcal I_N(\rho)$ with
\[
 \Prob(\mathcal I_N(\rho)^c)\le2N^{-10}+H_KN^{-1/2+4t}
 +(M+1)\big(N^{-12}+2\exp(-c_\xi(N+1)e^{-12K})\big)
\]
on which, for every $N\le m\le N+M$,
\[
 |J_m(\rho)-J_N(\rho)|\le\varepsilon_N:=u_0N^t+3N^{-2t}
 +\sqrt3\,N^{-t}\Big(2\sqrt{\Lambda_\xi'}\,(\log2N)^7+\log(2N+1)+3K\Big).
\]
Moreover $u_0N^t\le30B_\xi e^{4K}\sqrt{M\log N/N}\,N^t\le1$.
\end{lemma}

\begin{proof}
By \eqref{rc:weights}, $\sigma_N(\rho)^2\ge(N+1)e^{-4K}\ge1$, so
$u_0\le30B_\xi e^{4K}\sqrt{M\log N/(N+1)}$, and \eqref{rm:block-condition}
gives $u_0\le N^{-t}$. For $N<m\le N+M$ we have $m\le2N$, hence
$|\rho-1|=|y|/N\le2K/m$, that is, $\rho\in I_m(2K)$.

Let $\mathcal I_N(\rho)$ be the intersection of the following events. The
first is the event of Lemma~\ref{rm:tail} with $K_0=K$, which applies
since $N\ge\max(4,K)$. The second and third are the events of
Lemma~\ref{rc:occupation}(ii) and Lemma~\ref{rc:log-moments}(iv) at
degree $N$ with width $K$. The others are the events of
Lemma~\ref{rc:log-moments}(iv) at each degree $N<m\le N+M$ with width
$2K$, which applies since $m\ge N\ge\max(4K,3)$. Each failure probability
of the last kind is at most $N^{-12}+2\exp(-c_\xi(N+1)e^{-12K})$, and all
of them together are at most the stated bound.

On $\mathcal I_N(\rho)$, every $f_m$ with $N\le m\le N+M$ is nonzero.
Apply Lemma~\ref{rm:increment} with $f=f_N$, $g=f_m$, $r=\rho$,
$\sigma_f=\sigma_N(\rho)$, $\sigma_g=\sigma_m(\rho)$ and $u_1=N^{-t}$,
where $\sigma_g\ge1$ by \eqref{rc:weights} at width $2K$.
Lemma~\ref{rm:tail} gives $|f_m-f_N|<\tau$ on $|z|=\rho\le1+K/N$. The
fraction $A^{f_N}_{u_1}$ is the occupation fraction $A_{N^{-t}}$ of $f_N$
on the circle of radius $\rho$, and $\mathcal L_{f_N}$ and $\mathcal L_{f_m}$ are the
logarithmic norms of $f_N$ and $f_m$ there. Hence
$A^{f_N}_{u_1}\le3N^{-2t}$,
$\sqrt{\mathcal L_{f_N}}\le\sqrt{\Lambda_\xi'}(\log N)^7$ and
$\sqrt{\mathcal L_{f_m}}\le\sqrt{\Lambda_\xi'}(\log m)^7\le\sqrt{\Lambda_\xi'}(\log2N)^7$.
By \eqref{rc:weights} at widths $K$ and $2K$,
\[
 \log\sigma_N(\rho)+\log\sigma_m(\rho)
 \le\tfrac12\log(N+1)+K+\tfrac12\log(m+1)+2K\le\log(2N+1)+3K .
\]
Inserting these bounds into Lemma~\ref{rm:increment} gives
$|J_m(\rho)-J_N(\rho)|\le\varepsilon_N$.
\end{proof}

\subsection{The radial law with a rate}
The proof of Theorem~\ref{thm:R} applies the block increment on a net of
radii, uses the secant bound at the net points, and fills in the radii
between them by the monotonicity of $\nu_m$.

\begin{proof}[Proof of Theorem~\ref{thm:R}]
Every law of Theorem~\ref{thm:R} is admissible, so the results of
Section~\ref{sec:rc} apply. Fix an integer $A\ge1$ and put
\[
 K=A+2,\qquad t=\tfrac1{17},\qquad N=N_j=\lfloor j^{17/4}\rfloor,\qquad
 M=N_{j+1}-N_j,\qquad h=N^{-1/34}(\log N)^{7/2}.
\]
By the mean value theorem $M\le\frac{17}4(j+1)^{13/4}+1\le5N^{13/17}$ for
$j\ge60$. Let $Y_j=h\Z\cap[-K,K]$ be the net of spacing $h$. It has
\[
 |Y_j|\le\frac{2K}h+1=2KN^{1/34}(\log N)^{-7/2}+1
\]
points. Write $\rho_y=1+y/N$. We consider $j\ge60$ so large that
\[
 N\ge\max\big(H_K^4,\,4e^{4K}\big),\qquad
 4500B_\xi^2e^{8K}\log N\le N^{2/17},\qquad h\le1.
\]
Since $H_K\ge6000$, we have $N^{4/17}\ge5$, so $M\le N$. The bound
$M\le5N^{13/17}$ and the second condition give \eqref{rm:block-condition},
since $1-2t=15/17$.

\emph{Bad events.} Let $\mathcal B_j$ be the union over $y\in Y_j$ of the
event $\{|J_N(\rho_y)-\log\sigma_N(\rho_y)+\gamma/2|>e_N(t)\}$ and the
event $\mathcal I_N(\rho_y)^c$ of Lemma~\ref{rm:block-increment}. Since
$-\frac12+4t=-\frac9{34}$, Theorem~\ref{rc:concentration} and
Lemma~\ref{rm:block-increment} give
\begin{align*}
 \Prob(\mathcal B_j)\le{}&\big(2KN^{1/34}(\log N)^{-7/2}+1\big)
 \Big[3H_KN^{-9/34}+3N^{-10}\\
 &\qquad+(5N^{13/17}+1)\big(N^{-12}+4e^{-c_\xi(N+1)e^{-12K}}\big)\Big].
\end{align*}
In the expanded product the term $6KH_KN^{-4/17}(\log N)^{-7/2}$
dominates. Since $N\ge j^{17/4}/2$ and $\log N\ge\log j$ for $j\ge2$, it
is $O(j^{-1}(\log j)^{-7/2})$, which is summable. The term
$3H_KN^{-9/34}$ is $O(j^{-9/8})$, and every other term is $O(j^{-2})$. By
the Borel--Cantelli lemma, almost surely $\mathcal B_j$ occurs for only
finitely many $j$.

\emph{Secants at the net points.} Suppose that $\mathcal B_j$ does not
occur, and let $N\le m\le N+M$. For every $y\in Y_j$,
\[
 |J_m(\rho_y)-\log\sigma_N(\rho_y)+\gamma/2|\le e_N(t)+\varepsilon_N=:\eta_N .
\]
With $M\le5N^{13/17}$ we have
$u_0N^t\le68B_\xi e^{4K}N^{-1/17}\sqrt{\log N}$, so
$\eta_N\le C_{A,\xi}N^{-1/17}(\log N)^7$. By the profile estimate of
Lemma~\ref{rc:profile},
$|\log\sigma_N(\rho_y)-\frac12\log N-F_{\rm rad}(y)|\le3e^{4K}/N$ for
$|y|\le K$. Let $y\in Y_j$ with $|y|\le A+1$. Then $y\pm h\in Y_j$ and
$|y|+h\le K$, so the secant bound of Lemma~\ref{rc:secant}, applied to
$f=f_m$ at the scale $N$ with $F=F_{\rm rad}$,
$c=\frac12\log N-\frac\gamma2$, $B_F=1$ and $M_F=\frac12$ from
\eqref{rc:Phi-props}, gives
\begin{equation}\label{rm:net-bound}
 \Big|\frac{\nu_m(\rho_y)}N-\Phi_{\rm rad}(y)\Big|
 \le\frac h4+\frac KN+\frac3h\Big(\eta_N+\frac{3e^{4K}}N\Big)=:\delta_N .
\end{equation}
The spacing $h$ balances the first and last terms, both
$O(N^{-1/34}(\log N)^{7/2})$. The other two terms are smaller, so
$\delta_N\le C'_{A,\xi}N^{-1/34}(\log N)^{7/2}$.

\emph{All coordinates.} Let $|x|\le A$ and put $y=xN/m$, so that
$1+x/m=\rho_y$ and $|y|\le A$. Let $y_-=h\lfloor y/h\rfloor$ and
$y_+=y_-+h$. Both lie in $Y_j$ and satisfy $|y_\pm|\le A+1$. Since
$\nu_m$ is nondecreasing in the radius and $0\le\Phi_{\rm rad}'\le\frac12$,
\eqref{rm:net-bound} gives
\[
 \Phi_{\rm rad}(y)-\frac h2-\delta_N\le\frac{\nu_m(\rho_{y_-})}N
 \le\frac{\nu_m(1+x/m)}N\le\frac{\nu_m(\rho_{y_+})}N
 \le\Phi_{\rm rad}(y)+\frac h2+\delta_N .
\]
Moreover
$|\Phi_{\rm rad}(y)-\Phi_{\rm rad}(x)|\le\frac12|x|(1-N/m)\le3AN^{-4/17}$,
and $0\le\nu_m\le m\le2N$ gives
$|\nu_m/m-\nu_m/N|\le2(m-N)/m\le10N^{-4/17}$. Hence
\[
 \sup_{|x|\le A}\Big|\frac{\nu_m(1+x/m)}m-\Phi_{\rm rad}(x)\Big|
 \le\delta_N+\frac h2+(3A+10)N^{-4/17}\le C_A\,m^{-1/34}(\log m)^{7/2},
\]
using $N\le m\le2N$. Every large $m$ lies in a block whose index $j$ is
beyond the last occurrence of $\mathcal B_j$, which proves the bound for
all large $m$. The intersection of these events over the integers
$A\ge1$ has probability one, and for real $A$ we use $\lceil A\rceil$.
\end{proof}

\begin{corollary}[Uniformity on the line]\label{rm:whole-line}
For every law in Theorem~\ref{thm:R}, almost surely
\[
 \lim_{n\to\infty}\sup_{x\in\R}
 \Big|\frac{\nu_n(1+x/n)}n-\Phi_{\rm rad}(x)\Big|=0 .
\]
\end{corollary}

\begin{proof}
We follow P\'olya's argument for distribution functions. The function
$F_n(x)=\nu_n(1+x/n)/n$ is nondecreasing with values in $[0,1]$, and
$\Phi_{\rm rad}$ is nondecreasing with limits $0$ and $1$ at $\mp\infty$.
Given $\varepsilon>0$, choose $A$ with $\Phi_{\rm rad}(-A)<\varepsilon$
and $\Phi_{\rm rad}(A)>1-\varepsilon$. For $x\le-A$ both $F_n(x)$ and
$\Phi_{\rm rad}(x)$ lie in $[0,\max(F_n(-A),\varepsilon)]$, and for
$x\ge A$ both lie in $[\min(F_n(A),1-\varepsilon),1]$. Hence
\[
 \sup_{x\in\R}|F_n(x)-\Phi_{\rm rad}(x)|
 \le\sup_{|x|\le A}|F_n(x)-\Phi_{\rm rad}(x)|+\varepsilon,
\]
and by Theorem~\ref{thm:R} the right side is at most $2\varepsilon$ for
all large $n$, almost surely. Taking $\varepsilon=1/k$ for $k\ge1$ proves
the corollary.
\end{proof}

\section{A probabilistic law of the iterated logarithm}\label{sec:framework}

The classical proof of the law of the iterated logarithm for sums of
independent variables has three steps. Along a geometric grid of indices,
moderate-deviation tails and the first Borel--Cantelli lemma give the
upper bound. Independent increments between grid points and the second
Borel--Cantelli lemma give the lower bound. A maximal inequality carries
the upper bound from the grid to every index. The counts $D_n-\E D_n$ of
Sections~\ref{sec:gaussian} and~\ref{sec:signs} have no independent
increments, since all degrees share one coefficient sequence. We run the
three steps with independence replaced by four inputs.
Input~\ref{fw:I1} supplies the moderate deviations, Input~\ref{fw:I4}
supplies independent proxies for the values on the grid, and
Inputs~\ref{fw:I2} and~\ref{fw:I3}, which bound increments over coarse and
fine gaps, take the place of the maximal inequality.

The inputs are tail bounds and couplings rather than moment bounds,
because second moments do not determine the constant. Multiplying
centered variables $Z_n$ by an independent factor $W>0$ with $\E W^2=1$ preserves all covariances
but multiplies both limits in Theorem~\ref{thm:A} by $W$.

Throughout, $(Z_n)_{n\ge1}$ are real random variables on one probability
space, $c>0$ is fixed, and
\begin{equation}\label{fw:lambda}
 \lambda_n=\sqrt{2cn\log\log n}\qquad(n\ge3),
\end{equation}
which is positive and nondecreasing in $n\ge3$.

\subsection{The inputs}

\begin{definition}\label{fw:geometric}
Let $q>1$. A sequence of nonnegative integers $(n_k)_{k\ge k_*}$ is
\emph{geometric of ratio $q$} if $q^{-k}n_k$ converges to a finite
positive limit.
\end{definition}

For such a sequence $\log n_k=k\log q+O(1)$, and therefore
\begin{equation}\label{fw:geometric-log}
 \frac{\log\log n_k}{\log k}\to1,\qquad
 \frac{\sqrt{2cn_k\log k}}{\lambda_{n_k}}\to1.
\end{equation}
For $\theta>1$ the grid $n_k=\lfloor\theta^k\rfloor$ is geometric of
ratio $\theta$, and so are its block lengths $L_k=n_k-n_{k-1}-1$, taken
from the first $k$ with $L_k\ge0$, since
$\theta^{-k}L_k\to1-\theta^{-1}$. As $n_k-L_k=n_{k-1}+1$,
\begin{equation}\label{fw:block-ratios}
 \frac{L_k}{n_k}\to1-\theta^{-1},\qquad
 \frac{n_k-L_k}{L_k}\to\frac1{\theta-1}.
\end{equation}

For probability laws $\mu,\nu$ on $\R$ with finite second moments, let
\begin{equation}\label{fw:W2}
 \mathcal W_2(\mu,\nu)=\inf_\pi
 \Bigl(\int_{\R^2}(x-y)^2\,\pi(\dd x,\dd y)\Bigr)^{1/2},
\end{equation}
the infimum over probability measures $\pi$ on $\R^2$ with marginals
$\mu$ and $\nu$. For square-integrable random variables $U$ and $V$,
possibly on different probability spaces, $\mathcal W_2(U,V)$ is the
distance between their laws.

The inputs, for the fixed constant $c>0$, are the following.
\begin{enumerate}[label=\textup{(I\arabic*)},ref=\textup{(I\arabic*)},leftmargin=*]
\item\label{fw:I1} \emph{Marginal moderate deviations.} For every $q>1$,
every sequence $(n_k)$ geometric of ratio $q$, every $u>0$ and every
$s\in\{-1,1\}$,
\begin{equation}\label{fw:moderate}
 \lim_{k\to\infty}
 \frac{\log\Prob\{sZ_{n_k}>u\sqrt{2cn_k\log k}\}}{\log k}=-u^2,
\end{equation}
that is, $\Prob\{sZ_{n_k}>u\sqrt{2cn_k\log k}\}=k^{-u^2+o(1)}$, where the
$o(1)$ may depend on $q$, $(n_k)$, $u$ and $s$.
\item\label{fw:I2} \emph{Coarse increment tails.} There are
$C_0,C_1,b\in(0,\infty)$, $\delta_0\in(0,1)$ and $n_0$ such that for all
$n\ge n_0$, all integers $h$ with $\log^{-6}(n+2)\le h/n\le\delta_0$ and
all $t>0$,
\begin{equation}\label{fw:coarse-tail}
 \Prob\{|Z_{n+h}-Z_n|>t\}\le
 C_0\exp\Bigl(-\frac{t^2}{C_1\sqrt{nh}}\Bigr)+C_0n^{-b}.
\end{equation}
\item\label{fw:I3} \emph{Fine increment tails.} For every $p>0$ there are
$C_p\in(0,\infty)$ and $n_p$ such that for all $n\ge n_p$ and all integers
$1\le h\le n$,
\begin{equation}\label{fw:fine-tail}
 \Prob\{|Z_{n+h}-Z_n|>C_p\varphi(n,h)\}\le C_pn^{-p},
 \qquad
 \varphi(n,h)=(nh)^{1/4}\log^{5/4}(n+2)+\log^2(n+2).
\end{equation}
\item\label{fw:I4} \emph{Independent block proxies.} There are
$C,\beta\in(0,\infty)$ and $\theta_0>1$ such that for every
$\theta\ge\theta_0$ the following holds, with $n_k=\lfloor\theta^k\rfloor$
and $L_k=n_k-n_{k-1}-1$. There are $k_0$ and independent square-integrable
real random variables $(\widetilde Z_k)_{k\ge k_0}$, on the probability
space of $(Z_n)$, such that for all $k\ge k_0$, $L_k\ge3$ and
\begin{equation}\label{fw:proxy}
 \mathcal W_2\bigl(\widetilde Z_k,Z_{L_k}\bigr)\le CL_k^{(1-\beta)/2},
 \qquad
 \mathcal W_2\bigl(Z_{n_k}-\widetilde Z_k,\,-(Z_{n_k}-Z_{L_k})\bigr)
 \le Cn_k^{(1-\beta)/2}.
\end{equation}
\end{enumerate}

In \ref{fw:I4} the constants $C$ and $\beta$ are the same for every
$\theta$, while $k_0$ and the proxies may depend on $\theta$. The proxies
are mutually independent and may depend on $(Z_n)$, and \eqref{fw:proxy}
bounds the law of each error $Z_{n_k}-\widetilde Z_k$ separately. For
root counts the proxy is the balanced
discrepancy of the polynomial formed by the last block of coefficients
(Lemma~\ref{fw:block-proxies}).

\subsection{Grid laws}

Input \ref{fw:I1} is used through the following dichotomy. If
$p_k\in(0,1]$ and $\log p_k/\log k\to-a$ with $a>0$, then for each
$\tau\in(0,a)$ we have $k^{-a-\tau}\le p_k\le k^{-a+\tau}$ for all large
$k$. Hence
\begin{equation}\label{fw:dichotomy}
 \sum_kp_k<\infty\ \text{ if } a>1,\qquad
 \sum_kp_k=\infty\ \text{ if } a<1.
\end{equation}

\begin{lemma}[Upper bound on geometric sequences]\label{fw:grid-upper}
Assume \ref{fw:I1}, and let $(n_k)$ be geometric of ratio $q>1$. Then
almost surely $\limsup_k|Z_{n_k}|/\lambda_{n_k}\le1$.
\end{lemma}

\begin{proof}
Let $u>1$. By \ref{fw:I1} and the dichotomy \eqref{fw:dichotomy} with
$a=u^2$, the probabilities $\Prob\{sZ_{n_k}>u\sqrt{2cn_k\log k}\}$ are
summable for $s=\pm1$. By the first Borel--Cantelli lemma, almost surely
$|Z_{n_k}|\le u\sqrt{2cn_k\log k}$ for all large $k$. Since
$\sqrt{2cn_k\log k}/\lambda_{n_k}\to1$ by \eqref{fw:geometric-log}, it
follows that $\limsup_k|Z_{n_k}|/\lambda_{n_k}\le u$. Intersect over
$u=1+1/j$, $j\in\N$.
\end{proof}

The lower bound comes from the independent proxies. Along a sparse grid
$\lfloor\theta^k\rfloor$, the last block carries a proportion
$1-\theta^{-1}$ of the degree. Its proxy has the marginal tails of
\ref{fw:I1}, and the coarse tail bound \ref{fw:I2} shows that $Z_{n_k}$
differs little from $Z_{L_k}$.

\begin{lemma}[Lower bounds on every grid]\label{fw:grid-lower}
Assume \ref{fw:I1}, \ref{fw:I2} and \ref{fw:I4}, and let $q>1$ and
$N_k=\lfloor q^k\rfloor$. Then almost surely
$\limsup_kZ_{N_k}/\lambda_{N_k}\ge1$ and
$\liminf_kZ_{N_k}/\lambda_{N_k}\le-1$.
\end{lemma}

\begin{proof}
Fix $w<1$ and $s\in\{-1,1\}$. We find a sparse subgrid of $(N_k)$ on which
$sZ_n>w\lambda_n$ infinitely often. Let $\delta_0$ and $C_1$ be the
constants of \ref{fw:I2}. Choose $u\in(\max\{w,0\},1)$ and
$\eta=(u-w)\sqrt{2c}/2$, so that $u-\eta/\sqrt{2c}=(u+w)/2>w$, and put
\[
 d_\eta=\min\Bigl\{\delta_0,\Bigl(\frac{\eta^2}{8C_1}\Bigr)^2\Bigr\}.
\]
For $\theta>1$, $n_k=\lfloor\theta^k\rfloor$ and $L_k=n_k-n_{k-1}-1$, the
limits \eqref{fw:geometric-log} and \eqref{fw:block-ratios} give
\begin{equation}\label{fw:lower-normalization}
 \frac{u\sqrt{2cL_k\log k}-\eta\sqrt{n_k\log k}}{\lambda_{n_k}}
 \to u\sqrt{1-\theta^{-1}}-\frac{\eta}{\sqrt{2c}}.
\end{equation}
As $\theta\to\infty$ this limit tends to $(u+w)/2>w$. Choose an integer
$r\ge1$ such that $\theta=q^r$ satisfies $\theta\ge\theta_0$,
$(\theta-1)^{-1}<d_\eta$ and $u\sqrt{1-\theta^{-1}}-\eta/\sqrt{2c}>w$.
Since $n_k=\lfloor q^{rk}\rfloor=N_{rk}$, by
\eqref{fw:lower-normalization} it suffices to show that almost surely
\begin{equation}\label{fw:block-excursion}
 sZ_{n_k}>u\sqrt{2cL_k\log k}-\eta\sqrt{n_k\log k}
 \quad\text{for infinitely many }k.
\end{equation}
Let $(\widetilde Z_k)_{k\ge k_0}$ be the proxies of \ref{fw:I4} for this
$\theta$.

\emph{Tail comparison.} If $U,V$ are square-integrable real random
variables, possibly on different probability spaces, with
$\mathcal W_2(U,V)\le\varrho$, then for all $t\in\R$ and $e>0$,
\begin{equation}\label{fw:coupling-tail-bounds}
 \Prob\{|U|>t+e\}\le\Prob\{|V|>t\}+\frac{\varrho^2}{e^2},
 \qquad
 \Prob\{sU>t\}\ge\Prob\{sV>t+e\}-\frac{\varrho^2}{e^2}.
\end{equation}
For a coupling $(X,Y)$ of the two laws,
$\{|X|>t+e\}\subset\{|Y|>t\}\cup\{|X-Y|>e\}$ and
$\{sY>t+e\}\subset\{sX>t\}\cup\{|X-Y|>e\}$. Both bounds follow by
Chebyshev's inequality for $X-Y$, since the infimum of $\E(X-Y)^2$ over
couplings is at most $\varrho^2$.

\emph{Step 1: excursions of the proxies.} Fix $v\in(u,1)$. The first
bound in \eqref{fw:proxy} and the second bound in
\eqref{fw:coupling-tail-bounds}, with margin
$e=(v-u)\sqrt{2cL_k\log k}$, give for $k\ge\max\{k_0,3\}$
\[
 \Prob\{s\widetilde Z_k>u\sqrt{2cL_k\log k}\}
 \ge\Prob\{sZ_{L_k}>v\sqrt{2cL_k\log k}\}
 -\frac{C^2L_k^{-\beta}}{2c(v-u)^2\log k}.
\]
Since $(L_k)$ is geometric of ratio $\theta$, \ref{fw:I1} and the
dichotomy \eqref{fw:dichotomy} with $a=v^2<1$ show that the first
probabilities on the right have infinite sum. The subtracted terms are
summable, so the left side has infinite sum. The events
$\{s\widetilde Z_k>u\sqrt{2cL_k\log k}\}$, $k\ge k_0$, are independent.
By the second Borel--Cantelli lemma, almost surely
$s\widetilde Z_k>u\sqrt{2cL_k\log k}$ for infinitely many $k$.

\emph{Step 2: the approximation error.} We show that almost surely
\begin{equation}\label{fw:proxy-close}
 |Z_{n_k}-\widetilde Z_k|\le\eta\sqrt{n_k\log k}
 \quad\text{for all large }k.
\end{equation}
Put $h_k=n_k-L_k$. By \eqref{fw:block-ratios},
$h_k/L_k\to(\theta-1)^{-1}\in(0,d_\eta)$. So for large $k$ we have
$L_k\ge n_0$ and $\log^{-6}(L_k+2)\le h_k/L_k\le d_\eta\le\delta_0$, and
\ref{fw:I2} applies with $n=L_k$ and $h=h_k$. Since
$\sqrt{L_kh_k}\le\sqrt{d_\eta}\,L_k\le\eta^2n_k/(8C_1)$, the exponent in
\eqref{fw:coarse-tail} at $t=\frac\eta2\sqrt{n_k\log k}$ is at least
$2\log k$, and
\[
 \Prob\bigl\{|Z_{n_k}-Z_{L_k}|>\tfrac\eta2\sqrt{n_k\log k}\bigr\}
 \le C_0k^{-2}+C_0L_k^{-b}.
\]
The first bound in \eqref{fw:coupling-tail-bounds}, applied to the second
bound in \eqref{fw:proxy} with $t=e=\frac\eta2\sqrt{n_k\log k}$, then
gives
\[
 \Prob\bigl\{|Z_{n_k}-\widetilde Z_k|>\eta\sqrt{n_k\log k}\bigr\}
 \le C_0k^{-2}+C_0L_k^{-b}+\frac{4C^2n_k^{-\beta}}{\eta^2\log k}.
\]
The right side is summable, and the first Borel--Cantelli lemma gives
\eqref{fw:proxy-close}.

\emph{Step 3: conclusion.} On the intersection of the events of Steps~1 and~2, the
bound $sZ_{n_k}\ge s\widetilde Z_k-|Z_{n_k}-\widetilde Z_k|$ gives
\eqref{fw:block-excursion}. Intersect over rational $w<1$ and $s=\pm1$.
\end{proof}

\subsection{Maximal summability}

To control the maximum over a block $[N_k,N_{k+1}]$, we compare $Z_n$
with $Z_{N_k}$ through a sparse set of checkpoints, using \ref{fw:I2}
from $N_k$ to each checkpoint and \ref{fw:I3} from the nearest checkpoint
to $n$. The ratio $q$ is chosen after $\varepsilon$, close enough to one
that the coarse tails beat the number of checkpoints.

\begin{lemma}[Coarse checkpoints and a fine remainder]\label{fw:bridge}
Assume \ref{fw:I2} and \ref{fw:I3}. Then for every $\varepsilon>0$
there is $q>1$ such that, with $N_k=\lfloor q^k\rfloor$,
\begin{equation}\label{fw:maximal}
 \sum_k\Prob\Bigl\{\max_{N_k\le n\le N_{k+1}}|Z_n-Z_{N_k}|
 >\varepsilon\sqrt{q^k\log\log q^k}\Bigr\}<\infty.
\end{equation}
\end{lemma}

\begin{proof}
Let $C_0,C_1,b,\delta_0,n_0$ be the constants of \ref{fw:I2}, and let
$C_6,n_6$ be those of \ref{fw:I3} with $p=6$. The parameters are
\[
 A=\frac\varepsilon4,\qquad
 d=\min\Bigl\{\delta_0,\frac12,\Bigl(\frac{A^2}{10C_1}\Bigr)^2\Bigr\},\qquad
 q=1+\frac d2,\qquad
 m=\Bigl\lfloor\frac N{k^5}\Bigr\rfloor+1,
\]
where $N=N_k=\lfloor q^k\rfloor$. Put $H=N_{k+1}-N_k$. As $k\to\infty$,
$H/N\to q-1<d$ and $\log(N+2)/k\to\log q$. Hence for all large $k$,
\begin{equation}\label{fw:bridge-ranges}
 H\le dN\le N/2,\qquad
 N\ge\max\{n_0,n_6,k^5\},\qquad
 k^{-5}\ge\log^{-6}(N+2),
\end{equation}
and we consider only such $k$. Then $N/k^5<m\le2N/k^5$. For $0\le h\le
H$ let $j=\lfloor h/m\rfloor$, so that
\begin{equation}\label{fw:checkpoint}
 |Z_{N+h}-Z_N|\le|Z_{N+jm}-Z_N|+|Z_{N+h}-Z_{N+jm}|.
\end{equation}

\emph{Coarse checkpoints.} Here $0\le j\le J:=\lfloor H/m\rfloor\le
Hk^5/N\le k^5$, and the first term in \eqref{fw:checkpoint} vanishes for
$j=0$. For $1\le j\le J$, the ranges \eqref{fw:bridge-ranges} give
\[
 \log^{-6}(N+2)\le k^{-5}<\frac mN\le\frac{jm}N\le\frac HN\le d\le\delta_0,
 \qquad
 \sqrt{N\,jm}\le\sqrt{NH}\le\sqrt d\,N\le\frac{A^2N}{10C_1}.
\]
So \ref{fw:I2} applies with $n=N$ and $h=jm$, and its exponent at
$t=A\sqrt{N\log k}$ is at least $10\log k$:
\[
 \Prob\bigl\{|Z_{N+jm}-Z_N|>A\sqrt{N\log k}\bigr\}
 \le C_0k^{-10}+C_0N^{-b}.
\]

\emph{Fine remainder.} The second term in \eqref{fw:checkpoint} is the
increment from $n'=N+jm\in[N,2N]$ over the gap $g=h-jm\in[0,m)$. It
vanishes for $g=0$. Otherwise $1\le g<m\le n'$, and \ref{fw:I3} with
$p=6$, together with the monotonicity of $\varphi$ in both arguments,
gives
\[
 \Prob\bigl\{|Z_{N+h}-Z_{N+jm}|>C_6\varphi(2N,m)\bigr\}\le C_6N^{-6}.
\]
The checkpoint spacing makes this threshold negligible. Since
$m\le2N/k^5$ and $\log(2N+2)/k\to\log q$,
\[
 \varphi(2N,m)\le\sqrt{2N}\Bigl(\frac{\log(2N+2)}k\Bigr)^{5/4}
 +\log^2(2N+2)=O(\sqrt N),
\]
so $C_6\varphi(2N,m)\le A\sqrt{N\log k}$ for large $k$.

\emph{Union bound.} There are at most $k^5$ checkpoint events and at most
$H+1\le2N$ remainder events, one for each $h$. Outside their union, an
event of probability at most
\[
 C_0k^{-5}+C_0k^5N^{-b}+2C_6N^{-5},
\]
the bound \eqref{fw:checkpoint} gives $|Z_{N+h}-Z_N|\le2A\sqrt{N\log k}
=\frac\varepsilon2\sqrt{N\log k}$ for every $0\le h\le H$. As
$N/q^k\to1$ and $\log\log q^k/\log k\to1$, eventually
$\frac\varepsilon2\sqrt{N\log k}\le\varepsilon\sqrt{q^k\log\log q^k}$.
Thus for large $k$ the $k$-th probability in \eqref{fw:maximal} is at
most the last display, which is summable because $N_k\ge q^k/2$.
\end{proof}

\subsection{All integers}\label{fw:all-integers}

\begin{proof}[Proof of Theorem~\ref{thm:A}]
The upper bound at every integer combines the grid bound of
Lemma~\ref{fw:grid-upper} with the maximal bound of Lemma~\ref{fw:bridge},
and the lower bound is that of Lemma~\ref{fw:grid-lower}. Let
$\varepsilon>0$, and let $q$ and $N_k=\lfloor q^k\rfloor$ be as in
Lemma~\ref{fw:bridge}. Since
$\sqrt{q^k\log\log q^k}/\lambda_{N_k}\to(2c)^{-1/2}$, the first
Borel--Cantelli lemma and \eqref{fw:maximal} show that almost surely, for
all large $k$,
\[
 \max_{N_k\le n\le N_{k+1}}|Z_n-Z_{N_k}|
 \le\varepsilon\sqrt{q^k\log\log q^k}
 \le\varepsilon\sqrt{2/c}\;\lambda_{N_k}.
\]
By Lemma~\ref{fw:grid-upper}, almost surely
$|Z_{N_k}|\le(1+\varepsilon)\lambda_{N_k}$ for all large $k$. For
$n\ge N_1$ let $k(n)=\max\{k:N_k\le n\}$. Then $N_{k(n)}\le
n<N_{k(n)+1}$ and $k(n)\to\infty$. As $\lambda$ is nondecreasing, for
all large $n$,
\[
 |Z_n|\le|Z_{N_{k(n)}}|+|Z_n-Z_{N_{k(n)}}|
 \le\bigl(1+\varepsilon(1+\sqrt{2/c})\bigr)\lambda_{N_{k(n)}}
 \le\bigl(1+\varepsilon(1+\sqrt{2/c})\bigr)\lambda_n.
\]
Taking $\varepsilon=1/j$, $j\in\N$, we obtain on one event of
probability one
\begin{equation}\label{fw:upper-all}
 \limsup_{n\to\infty}\frac{|Z_n|}{\lambda_n}\le1.
\end{equation}
On the event of Lemma~\ref{fw:grid-lower} for $q=2$, the integers
$N_k=2^k$ tend to infinity. So
$\limsup_nZ_n/\lambda_n\ge\limsup_kZ_{N_k}/\lambda_{N_k}\ge1$, and likewise
$\liminf_nZ_n/\lambda_n\le-1$. Together with \eqref{fw:upper-all} this
proves both limits.
\end{proof}

\subsection{Verifying the inputs}

The marginal and coarse inputs follow from quantitative normal
approximation, and \ref{fw:I4} follows for root counts from reversal of
the coefficient vector. We write $\Phi$ for the standard normal
distribution function, $\overline\Phi=1-\Phi$, and
$d_K(X,\mathcal N(0,1))=\sup_{x\in\R}|\Prob\{X\le x\}-\Phi(x)|$. For
$x\ge0$,
\begin{equation}\label{fw:mills}
 \frac{e^{-3/2}}{\sqrt{2\pi}}\,\frac{e^{-x^2/2}}{x+1}
 \le\overline\Phi(x)\le e^{-x^2/2}.
\end{equation}
The upper bound is Chernoff's. The lower bound integrates the normal
density over $[x,x+1/(x+1)]$, where $y^2\le x^2+3$.

\begin{lemma}[Quantitative normal approximation suffices]
\label{fw:normal-inputs}
Let $(Z_n)$ be centered and square integrable.
\begin{enumerate}[label=\textup{(\roman*)}]
\item If, for finite positive $c,\alpha,b_0,C$,
\[
 \Var Z_n=cn+O(n^{1-\alpha}),\qquad
 d_K\bigl(Z_n/\sqrt{\Var Z_n},\mathcal N(0,1)\bigr)\le Cn^{-b_0},
\]
then \ref{fw:I1} holds with constant $c$.
\item Suppose there are $\delta_0\in(0,1)$, $n_0$ and finite positive
$b_1,C$ such that for all $n\ge n_0$ and all integers $h$ with
$\log^{-6}(n+2)\le h/n\le\delta_0$, the variance
$\sigma_{n,h}^2=\Var(Z_{n+h}-Z_n)$ satisfies
$\sigma_{n,h}^2\le C\sqrt{nh}$ and, when $\sigma_{n,h}>0$,
$d_K\bigl((Z_{n+h}-Z_n)/\sigma_{n,h},\mathcal N(0,1)\bigr)\le Cn^{-b_1}$.
Then \ref{fw:I2} holds.
\end{enumerate}
\end{lemma}

\begin{proof}
(i) Let $(n_k)$ be geometric of ratio $q$, $u>0$ and $s=\pm1$. For large
$k$, $\sigma_k^2=\Var Z_{n_k}>0$. Put $x_k=u\sqrt{2cn_k\log k}/\sigma_k$.
Then $x_k^2=2u^2\log k\,(1+O(n_k^{-\alpha}))$ and
$n_k^{-\alpha}\log k\to0$, so $x_k^2/2=u^2\log k+o(1)$. The Kolmogorov
bound at $x_k$, and its left limit at $-x_k$ together with
$\Phi(-x)=\overline\Phi(x)$, give
\[
 \bigl|\Prob\{sZ_{n_k}>\sigma_kx_k\}-\overline\Phi(x_k)\bigr|
 \le Cn_k^{-b_0}.
\]
By the normal tail bounds \eqref{fw:mills},
$\log\overline\Phi(x_k)=-u^2\log k+O(\log\log k)$. Since $n_k^{-b_0}$
decays geometrically in $k$, it is $o(\overline\Phi(x_k))$, and
\eqref{fw:moderate} follows.

(ii) If $\sigma_{n,h}>0$, the Kolmogorov bound at $\pm t/\sigma_{n,h}$
(a left limit at $-t/\sigma_{n,h}$) and \eqref{fw:mills} give
\[
 \Prob\{|Z_{n+h}-Z_n|>t\}\le2\overline\Phi(t/\sigma_{n,h})+2Cn^{-b_1}
 \le2\exp\Bigl(-\frac{t^2}{2C\sqrt{nh}}\Bigr)+2Cn^{-b_1}.
\]
If $\sigma_{n,h}=0$, the centered increment vanishes almost surely.
Thus \ref{fw:I2} holds with $C_0=2\max\{1,C\}$, $C_1=2C$, $b=b_1$ and the
same $\delta_0$ and $n_0$.
\end{proof}

For root counts, \ref{fw:I4} reduces to a variance bound for the number
$B_n$ of zeros on the unit circle. This number vanishes for laws with a
density (Lemma~\ref{pre:no-circle-roots}), and its variance is
$O(\sqrt n)$ for random signs (Proposition~\ref{sg:boundary}). The proxy
is the balanced discrepancy $D^\circ=D-B/2$ of the last block of
coefficients, and the block reversal identity of
Corollary~\ref{pre:reversal-law} gives its law.

\begin{lemma}[Block proxies for root counts]\label{fw:block-proxies}
Let $(\xi_j)_{j\ge0}$ be iid real random variables with
$\Prob\{\xi_0=0\}=0$, and suppose that $\Var B_n\le C(n+1)^{1-\beta}$ for
all $n$, for some $C$ and $\beta>0$. Then the centered counts
$Z_n=D_n-\E D_n$ satisfy \ref{fw:I4} with this $\beta$ and $\theta_0=2$.
If $B_n=0$ for all $n$ almost surely, the proxies satisfy
$\mathcal L(\widetilde Z_k)=\mathcal L(Z_{L_k})$ and
$\mathcal L(Z_{n_k}-\widetilde Z_k)=\mathcal L(-(Z_{n_k}-Z_{L_k}))$, so
\ref{fw:I4} holds with every $\beta>0$.
\end{lemma}

\begin{proof}
Fix $\theta\ge2$, and let $n_k=\lfloor\theta^k\rfloor$ and
$L_k=n_k-n_{k-1}-1$. Let $P_k(z)=\sum_{i=0}^{L_k}\xi_{n_{k-1}+1+i}z^i$ be
the polynomial of nominal degree $L_k$ formed by the last block of
coefficients, and let $\widetilde Z_k=D^\circ(P_k)$ be its balanced
discrepancy. The polynomials $P_k$ use disjoint blocks of coefficients,
so the $\widetilde Z_k$ are independent. They are bounded, since every
count is at most the nominal degree. By the block reversal identity
\eqref{pre:block-reversal}, $\widetilde Z_k$ has the law of
$D_{L_k}^\circ$.

Put $R_n=(B_n-\E B_n)/2$. By \eqref{pre:centering}, $Z_n=D_n^\circ+R_n$,
and $\|R_n\|_2=\frac12(\Var B_n)^{1/2}\le\frac12C^{1/2}(n+1)^{(1-\beta)/2}$.
Since $\widetilde Z_k$ has the law of $D^\circ_{L_k}$, which lives on the
same space as $Z_{L_k}$,
\[
 \mathcal W_2(\widetilde Z_k,Z_{L_k})\le\|D_{L_k}^\circ-Z_{L_k}\|_2
 =\|R_{L_k}\|_2.
\]
For the second distance, \eqref{pre:block-reversal} with $n=n_k$ and
$p=n_{k-1}$ says that $(D_{n_k}^\circ,B_{n_k},\widetilde Z_k,B(P_k))$ has
the law of $(-D_{n_k}^\circ,B_{n_k},-D_{L_k}^\circ,B_{L_k})$. Applying the
map $(x_1,x_2,x_3,x_4)\mapsto x_1+(x_2-\E B_{n_k})/2-x_3$ to both
vectors gives
\[
 Z_{n_k}-\widetilde Z_k\ \overset{\mathrm d}{=}\
 -D_{n_k}^\circ+R_{n_k}+D_{L_k}^\circ
 =-(Z_{n_k}-Z_{L_k})+2R_{n_k}-R_{L_k}.
\]
Hence
$\mathcal W_2(Z_{n_k}-\widetilde Z_k,-(Z_{n_k}-Z_{L_k}))\le
2\|R_{n_k}\|_2+\|R_{L_k}\|_2$.

For $\theta\ge2$ we have $n_{k-1}\le\theta^k/2\le(n_k+1)/2$, so
$n_k/4\le L_k\le n_k$ and $L_k\ge3$ for all large $k$. Then $L_k+1$ and
$n_k+1$ lie between $n_k/4$ and $2n_k$, and the two bounds above are at
most $C'L_k^{(1-\beta)/2}$ and $C'n_k^{(1-\beta)/2}$, with $C'$ depending
only on $C$ and $\beta$. This is \ref{fw:I4}. If $B_n=0$ for all $n$,
then $R_n=0$ and both comparisons are equalities in law.
\end{proof}

\section{Gaussian coefficients}\label{sec:gaussian}

By Proposition~\ref{pre:crossing}, $D_n$ counts with signs the crossings of the real axis by the carrier curve. For Gaussian coefficients the two coordinates of this curve are independent Gaussian fields. Given the imaginary coordinate the crossings are fixed, and $D_n$ becomes a sum of signs of the real coordinate at these points. A Gaussian H\"older inequality makes this sum sub-Gaussian, and a smoothed version of it carries the variance and the normal approximation of Theorem~\ref{thm:B}. The curves of two degrees share coefficients, so their coordinates are correlated, and a covariance of log-moduli takes the place of independence. Increments are controlled by pairing the crossings of two nearby curves. Together these estimates verify the inputs of Theorem~\ref{thm:A}. The last subsection computes the mean radial counts of Theorem~\ref{thm:G}.

Except in Subsection~\ref{ga:radial-sub}, the coefficients $(\xi_j)_{j\ge0}$ are independent standard real Gaussian variables. Constants $C,c>0$ are absolute unless their dependence is indicated. A bound stated for all $n\ge1$ with an unspecified constant is proved for $n\ge n_0$. The finitely many smaller $n$ are absorbed in the constant, because $|D_n|\le n/2$ and $0\le\nu_n\le n$, and probabilities and Kolmogorov distances are at most one.

\subsection{Fields and zeros}\label{ga:fields-sub}
Fix $n\ge1$, and let $\omega_j$ be the frequencies of Lemma~\ref{pre:sampling}, so that $\omega_{n-j}=-\omega_j$. In the scaled angle $x=nt$ let
\begin{equation}\label{ga:fields-def}
 X_n(x)+iY_n(x)=n^{-1/2}e^{-ix/2}f_n(e^{ix/n})=n^{-1/2}\sum_{j=0}^n\xi_je^{i\omega_jx}.
\end{equation}
These are the cosine and sine fields $C_\xi,S_\xi$ of Lemma~\ref{pre:sampling}, and for deterministic $b\in\R^{n+1}$ we write $X[b]=C_b$ and $Y[b]=S_b$. All of them are entire functions of $x$. By Lemma~\ref{pre:no-circle-roots}, almost surely no $f_n$ has a zero on the unit circle, and $\E D_n=0$ for every $n$. On this event the argument principle (Lemma~\ref{pre:argument}) gives, for all $n$ simultaneously,
\begin{equation}\label{ga:winding}
 D_n=\frac1\pi\,\Delta_{[0,\pi n]}\arg\bigl(X_n+iY_n\bigr).
\end{equation}
By the pairing \eqref{pre:pairing}, $X_n$ depends only on the cosine coordinates $\xi_j+\xi_{n-j}$ and $Y_n$ only on the sine coordinates $\xi_j-\xi_{n-j}$. These Gaussian variables are uncorrelated, so the processes $X_n$ and $Y_n$ are independent. With the real kernel $R_n(v)=n^{-1}\sum_{j=0}^ne^{i\omega_jv}$,
\begin{equation}\label{ga:kernel}
 \E X_n(x)X_n(y)=\tfrac12\{R_n(x-y)+R_n(x+y)\},\qquad
 \E Y_n(x)Y_n(y)=\tfrac12\{R_n(x-y)-R_n(x+y)\}.
\end{equation}
Write $v_n^\pm(x)$ for the two variances and $\rho_\pm(x,y)$ for the correlation functions. In this section $U_n=X_n/\sqrt{v_n^+}$ and $V_n=Y_n/\sqrt{v_n^-}$ denote the normalized fields, where the variances are positive. They differ from the carrier coordinates $U_{f_n}(x/n)$ and $V_{f_n}(x/n)$ of Section~\ref{sec:prelim} by positive factors. The reflected terms $R_n(x+y)$ in \eqref{ga:kernel} are small when $x+y$ is far from $0$ and $2\pi n$, so we work in a bulk interval. We fix a large $L\ge3$, put $I_n=[L,\pi n-L]$, and call an interval $[k,k+1)$ with $k\in\Z$ a unit cell. The kernel and its derivatives decay by the summation-by-parts bound \eqref{rc:abel}, in which $d(t)=\dist(t,2\pi\Z)$.

\begin{lemma}[Kernels, correlations and frames]\label{ga:fields}
There are $L,n_0$ and $c,C>0$ such that for $n\ge n_0$ the following hold.
\begin{enumerate}[label=\textup{(\alph*)}]
\item $|R_n^{(k)}(v)|\le C_k/(1+\dist(v,2\pi n\Z))$ for $k\ge0$ and real $v$. On $[L-3,\pi n-L+3]$ we have $\frac14\le v_n^\pm\le2$, and the derivatives of order at most three of $(v_n^\pm)^{\pm1/2}$ are bounded by $C$.
\item For $x,y\in I_n$, $1-\rho_\pm(x,y)^2\ge c\min(|x-y|^2,1)$. For $x\ne y$ in $I_n$,
\begin{equation}\label{ga:arcsine-derivative}
 |\partial_x\arcsin\rho_\pm(x,y)|+|\partial_y\arcsin\rho_\pm(x,y)|\le C\min(1,|x-y|^{-1}),
\end{equation}
and for $0\le\kappa<1$ the same bound holds for $\arcsin(\kappa\rho_\pm)$ on all of $I_n^2$.
\item \emph{(Frame bound.)} Let $b\in\R^{n+1}$, and let $x_1,\dots,x_q\in I_n$, repetitions allowed, with at most $\Theta$ of them in each unit cell. Then
\begin{equation}\label{ga:frame}
 \sum_{i=1}^q\ \sup_{|y-x_i|\le2}\bigl(X[b]^2+X[b]'^2+Y[b]^2+Y[b]'^2\bigr)(y)\le C\Theta\|b\|^2.
\end{equation}
\end{enumerate}
\end{lemma}

\begin{proof}
(a) The weights in $R_n^{(k)}(v)=n^{-1}e^{-iv/2}\sum_j(i\omega_j)^ke^{ijv/n}$ have modulus at most one and total variation at most two, so \eqref{rc:abel} gives $|R_n^{(k)}(v)|\le3\pi/\dist(v,2\pi n\Z)$, and trivially $|R_n^{(k)}|\le2$. Since $v_n^\pm(x)=\frac12\{R_n(0)\pm R_n(2x)\}$ with $R_n(0)=(n+1)/n$, and $\dist(2x,2\pi n\Z)\ge2(L-3)$ on the stated interval, the other bounds hold for large $L$.

(b) Write $X_n(x)=\langle w(x),\xi\rangle$ with $w(x)=n^{-1/2}(\cos\omega_jx)_j$, and $u=w/\|w\|$, so that $\rho_+(x,y)=\langle u(x),u(y)\rangle$. Since $\|w'\|^2=\frac12\{-R_n''(0)+R_n''(2x)\}$, $\langle w,w'\rangle=\frac12R_n'(2x)$ and $-R_n''(0)=(n+1)(n+2)/(12n^2)$, part (a) gives
\[
 \|u'(x)\|^2=\frac{\|w'\|^2\|w\|^2-\langle w,w'\rangle^2}{\|w\|^4}=\frac{n+2}{12n}+O(L^{-1})\in[\tfrac1{24},1],\qquad\|u''(x)\|\le C.
\]
The same holds for the sine field.

For the correlation gap let $\tau=|x-y|$. For $\tau\le\tau_0$, $1-\rho=\frac12\|u(x)-u(y)\|^2\ge\frac12\tau^2(\|u'(x)\|-C\tau)^2\ge\tau^2/200$. For $\tau\ge T_0$, (a) gives $|\rho_\pm|\le\frac12$. In between, Riemann sums give $|\rho_\pm(x,y)-r(x-y)|\le C/L+CT_0/n$ with the function $r$ of Theorem~\ref{thm:B}, and $|r|<1$ off the origin, so $|\rho_\pm|\le1-c$ once $L$ and then $n_0$ are large.

Since $u'\perp u$, $\partial_x\rho=\langle u'(x),u(y)-\rho u(x)\rangle$ and $\|u(y)-\rho u(x)\|^2=1-\rho^2$, so $|\partial_x\rho|\le\sqrt{1-\rho^2}$. For $\tau\ge1$, (a) gives $|\partial_x\rho|\le C/(1+\tau)$ while $1-\rho^2\ge c$. This proves \eqref{ga:arcsine-derivative}. For $\kappa<1$, $|\partial_x\arcsin(\kappa\rho)|\le|\partial_x\rho|(1-\rho^2)^{-1/2}$ off the diagonal, and the derivative vanishes on it as $\langle u'(x),u(x)\rangle=0$. At $\kappa=1$ the derivative jumps across the diagonal, since $\arcsin\rho_\pm(x,x+\tau)=\frac\pi2-\|u'(x)\||\tau|+o(\tau)$.

(c) If $F\in\{X[b],Y[b]\}$, then $F'$ is a field of the other kind with coefficients $\pm\omega_jb_j$. Lemma~\ref{pre:sampling}, applied to $F$ and $F'$, gives $\int_0^{\pi n}(F^2+2F'^2+F''^2)\le3\pi\|b\|^2$ and $\sup_J(F^2+F'^2)\le2\int_J(F^2+2F'^2+F''^2)$ for every unit interval $J\subset[0,\pi n]$. As $L\ge3$, each $[x_i-2,x_i+2]$ lies in five such intervals with integer endpoints, and each of these serves at most $5\Theta$ points.
\end{proof}

\begin{lemma}[Tails of local zero counts]\label{ga:jensen}
Let $F$ be a random entire function, $x_0\in\R$ and $\varrho>0$. Suppose that $\E|F(w)|^2\le K$ for $|w-x_0|\le3\varrho$ and that $F(x_0)$ is centered Gaussian with variance $s^2>0$. Then the number $N$ of zeros of $F$ in $\{|w-x_0|\le\varrho\}$ satisfies $\Prob(N>t)\le(9K+s^{-1})2^{-t/2}$ for $t\ge0$.
\end{lemma}

\begin{proof}
By the mean value inequality for $|F|^2$ on disks of radius $\varrho$, $S=\sup_{|w-x_0|\le2\varrho}|F|$ has $\E S^2\le9K$. Lemma~\ref{pre:local-jensen} with radii $\varrho<2\varrho$ gives $N\log2\le\log_+S+\log_-|F(x_0)|$. Here $\Prob(\log_+S>a)\le9Ke^{-2a}$ by Markov's inequality and $\Prob(\log_-|F(x_0)|>a)\le s^{-1}e^{-a}$. Take $a=\frac12t\log2$.
\end{proof}

The hypothesis is checked with $|\sin z|,|\cos z|\le e^{|\Im z|}$ and $|\sin z/z|=|\int_0^1\cos(sz)\dd s|\le e^{|\Im z|}$. If $F(w)=n^{-1/2}\sum_kb_k\sin(\lambda_kw)$ with independent standard Gaussian $b_k$, at most $2n+1$ terms and $|\lambda_k|\le\frac32$, then $\E|F(w)|^2$ and $\E|F(w)/w|^2$ are at most $7e^{3|\Im w|}$. The same holds for cosine fields.

\begin{lemma}[Crossing sums and zero tails]\label{ga:zeros}
\textup{(a)} Let $\Lambda$ be a nonempty finite set of positive numbers, all in $\Z$ or all in $\Z+\frac12$, and let $(c_\lambda)_{\lambda\in\Lambda}$ have a density on $\R^\Lambda$. Then almost surely $\sum_\lambda c_\lambda\sin(\lambda t)$ has only simple zeros in $(0,\pi)$.

\textup{(b)} Almost surely, for every $n\ge1$, $Y_n\not\equiv0$, its zeros in $(0,\pi n)$ are simple, $Y_n(L)Y_n(\pi n-L)\ne0$, and $f_n$ has no zero on the unit circle. On this event, for $n\ge n_0$, the crossing sum
\begin{equation}\label{ga:crossing}
 Q_n=\sum_{x\in I_n:\,Y_n(x)=0}\sgn X_n(x)\sgn Y_n'(x)\quad\text{satisfies}\quad|D_n-Q_n|\le1+N_n^{\rm end},
\end{equation}
where $N_n^{\rm end}$ counts the zeros of $Y_n$ in $(0,\pi n)\setminus I_n$. Let $\Theta_n$ be the largest number of zeros of $Y_n$ in a unit cell meeting $I_n$. Then for $t\ge0$ and $p\ge1$,
\begin{equation}\label{ga:occupancy}
 \Prob(N_n^{\rm end}>t)\le Ce^{-ct},\qquad\Prob(\Theta_n>t)\le Cne^{-ct},\qquad\|\Theta_n\|_p\le C(\log(n+2)+p).
\end{equation}
\end{lemma}

\begin{proof}
(a) For $\lambda\in\Z$, $\sin(\lambda t)=\sin t\,U_{\lambda-1}(\cos t)$ with the Chebyshev polynomial $U_{\lambda-1}$, and for $\lambda=k+\frac12$, $\sin(\lambda t)=\sin(t/2)(1+2\sum_{i\le k}\cos it)$. So $\sin(\lambda t)=s(t)p_\lambda(\cos t)$ with $s=\sin t$ or $\sin(t/2)$, positive on $(0,\pi)$, and a polynomial $p_\lambda$ of degree $d_\lambda=\lceil\lambda\rceil-1$. The sum is $s(t)P(\cos t)$ with $P=\sum_\lambda c_\lambda p_\lambda$, and a multiple zero gives a multiple root of $P$. Let $\lambda^*=\max\Lambda$ and $d=d_{\lambda^*}$. As the $d_\lambda$ are distinct, $P$ has degree $d$ when $c_{\lambda^*}\ne0$, and for $d\ge1$ its discriminant is a polynomial in $(c_\lambda)$, nonzero at $c_\lambda=\ind\{\lambda=\lambda^*\}$, where $P=p_{\lambda^*}$ has $d$ distinct roots $\cos(j\pi/\lambda^*)$, $1\le j\le d$. Off the null zero sets of $c_{\lambda^*}$ and the discriminant, $P$ has only simple roots.

(b) By \eqref{pre:pairing}, $Y_n(nt)=n^{-1/2}\sum_{\lambda}(\xi_{n/2+\lambda}-\xi_{n/2-\lambda})\sin\lambda t$ is a sum as in (a), and $Y_n(L)$, $Y_n(\pi n-L)$ are nondegenerate Gaussian variables. Intersect over $n$ and with the event of Lemma~\ref{pre:no-circle-roots}. The bulk crossing bound \eqref{pre:bulk-crossing} with $A=n^{-1}I_n$ is then \eqref{ga:crossing}, since the crossing index at a simple zero is the sign of the derivative.

For the tails we use Lemma~\ref{ga:jensen}. For $\Theta_n$ take $F=Y_n$, $\varrho=1$ and $x_0\in I_n$, so that $s^2\ge\frac14$, and a union bound over at most $4n$ centers. Near $0$ take $F(w)=Y_n(w)/w$, $x_0=0$ and $\varrho=L$, with $s^2=\Var Y_n'(0)=(n+1)(n+2)/(12n^2)\ge\frac1{12}$. Near $\pi n$, since $\omega_j\pi n$ lies in $\pi\Z$ for even $n$ and in $\pi(\Z+\frac12)$ for odd $n$, $Y_n(\pi n-\cdot)$ has the law of $-Y_n$ or of $X_n$, and $\Var X_n(0)=(n+1)/n$. The moments follow by integrating the tail of $\Theta_n$ above $A\log(n+2)$.
\end{proof}

\subsection{Gaussian tools}\label{ga:tools-sub}
Sums of signs of a Gaussian field are controlled by a H\"older inequality for functions of overlapping blocks of a Gaussian vector. It makes such sums sub-Gaussian, at the cost of the frame constant $H$ only. Throughout, $G$ denotes a standard Gaussian vector of the dimension at hand.

\begin{lemma}[Gaussian block H\"older inequality]\label{ga:holder}
Let $B_i:\R^m\to\R^{d_i}$, $1\le i\le k$, be linear maps with $B_iB_i^*=\mathrm{Id}$ and $\sum_i\|B_ix\|^2\le H\|x\|^2$, where $H\ge1$. Let $f_i:\R^{d_i}\to[a_i,b_i]$, $0<a_i\le b_i$, be Borel functions continuous almost everywhere. If $G$ and $G_i$ are standard Gaussian vectors in $\R^m$ and $\R^{d_i}$, then
\begin{equation}\label{ga:holder-ineq}
 \E\prod_if_i(B_iG)^{1/H}\le\prod_i\bigl(\E f_i(G_i)\bigr)^{1/H}.
\end{equation}
\end{lemma}

\begin{proof}
The vector $X=(B_1G,\dots,B_kG)$ is centered Gaussian with covariance matrix $T=\mathcal B\mathcal B^*$, where $\mathcal Bx=(B_ix)_i$. Its diagonal blocks are $B_iB_i^*=\mathrm{Id}$, and $T\le\|\mathcal B\|^2\mathrm{Id}\le H\,\mathrm{Id}$, since $\|\mathcal Bx\|^2=\sum_i\|B_ix\|^2\le H\|x\|^2$. This is the hypothesis $T\le\mathrm{diag}(HT_{11},\dots,HT_{kk})$ of the Gaussian H\"older inequality of Chen, Dafnis and Paouris \cite[Theorem~1(i)]{ChenDafnisPaouris2015}, with every exponent equal to $H$. That inequality holds for all nonnegative Borel functions. Applied to the functions $f_i^{1/H}$, it gives \eqref{ga:holder-ineq}, because each $B_iG$ is a standard Gaussian vector.
\end{proof}

\begin{corollary}[Block Bernstein bound and signs]\label{ga:bernstein}
\textup{(a)} In the setting of Lemma~\ref{ga:holder}, let $\phi_i:\R^{d_i}\to\R$ be bounded Borel functions continuous almost everywhere, with $m_i=\E\phi_i(G_i)$, $|\phi_i-m_i|\le b$ and $V=\sum_i\Var\phi_i(G_i)$. Then for $s>0$
\begin{equation}\label{ga:block-bernstein}
 \Prob\Bigl(\Bigl|\sum_i(\phi_i(B_iG)-m_i)\Bigr|>s\Bigr)\le2\exp\Bigl(-\frac{s^2}{2H(V+bs/3)}\Bigr).
\end{equation}
\textup{(b)} Let $G_1,\dots,G_k$ be centered jointly Gaussian with unit variances and correlation matrix of operator norm at most $M\ge1$. For real $w_i$, $\alpha_i$ and $\lambda$,
\begin{equation}\label{ga:sign-mgf}
 \E\exp\Bigl(\lambda\sum_iw_i\{\sgn(G_i+\alpha_i)-\E\sgn(G_i+\alpha_i)\}\Bigr)\le\exp\Bigl(\frac{M\lambda^2}2\sum_iw_i^2\Bigr).
\end{equation}
\end{corollary}

\begin{proof}
(a) A centered variable $Z$ with $|Z|\le b$ satisfies $\E e^{\lambda Z}\le\exp\{\lambda^2\E Z^2/(2(1-b|\lambda|/3))\}$ for $b|\lambda|<3$, since $|\E Z^j|\le b^{j-2}\E Z^2$ and $j!\ge2\cdot3^{j-2}$. Apply \eqref{ga:holder-ineq} to $f_i=\exp(H\lambda(\phi_i-m_i))$ and Markov's inequality with $\lambda=s/(H(V+bs/3))$, and the same to $-\phi_i$.

(b) Write $G_i=\langle r_i,G\rangle$ with unit rows $r_i$. Then $\sum_i\langle r_i,x\rangle^2\le M\|x\|^2$, because the matrix $R$ with rows $r_i$ satisfies $RR^*\le M$. By Hoeffding's lemma, a centered variable with values in an interval of length two has $\log\E e^{\mu Z}\le\mu^2/2$. Apply \eqref{ga:holder-ineq} with $H=M$ to $f_i(z)=\exp(M\lambda w_i\{\sgn(z+\alpha_i)-\E\sgn(G_i+\alpha_i)\})$.
\end{proof}

\begin{lemma}[Lipschitz concentration]\label{ga:interpolation}
If $F:\R^d\to\R$ is $\Lambda$-Lipschitz, then for real $\lambda$ and $s>0$
\[
 \E e^{\lambda(F(G)-\E F(G))}\le e^{\pi^2\lambda^2\Lambda^2/8},\qquad
 \Prob\bigl(|F(G)-\E F(G)|>s\bigr)\le2e^{-2s^2/(\pi^2\Lambda^2)}.
\]
\end{lemma}

\begin{proof}
The Gaussian concentration inequality \cite[Theorems~5.5 and~5.6]{BLM2013} gives the stronger bounds $e^{\lambda^2\Lambda^2/2}$ and, applied to $F$ and to $-F$, $2e^{-s^2/(2\Lambda^2)}$, since $\frac12<\frac{\pi^2}8$.
\end{proof}

We also use the Poincar\'e inequality $\Var F(G)\le\E\|\nabla F(G)\|^2$ of Lemma~\ref{pre:gaussian}.

\begin{lemma}[Second-order normal approximation]\label{ga:chatterjee}
\textup{(a)} \cite[Theorem~2.2]{Chatterjee} Let $g:\R^d\to\R$ be $C^2$, with $g(G)$, $\|\nabla g(G)\|$ and $\|\nabla^2g(G)\|_{\rm op}$ in $L^4$ and $s^2=\Var g(G)>0$. Then
\begin{equation}\label{ga:chatterjee-bound}
 d_{\rm TV}\bigl(g(G),\mathcal N(\E g(G),s^2)\bigr)\le\frac{2\sqrt5}{s^2}\,\|\|\nabla g(G)\|\|_4\,\|\|\nabla^2g(G)\|_{\rm op}\|_4 .
\end{equation}
\textup{(b)} If $X,Y\in L^2$ and $t>0$, then
\[
 d_K\bigl(X,\mathcal N(0,1)\bigr)\le d_K\bigl(Y,\mathcal N(0,1)\bigr)+\frac t{\sqrt{2\pi}}+\frac{\E(X-Y)^2}{t^2}.
\]
If $X,Y$ are centered with positive variances, then $|\sqrt{\Var X}-\sqrt{\Var Y}|\le\|X-Y\|_2$ and $\|X/\sqrt{\Var X}-Y/\sqrt{\Var Y}\|_2\le2\|X-Y\|_2/\sqrt{\Var X}$.
\end{lemma}

\begin{proof}
(b) $\Prob(X\le x)\le\Prob(Y\le x+t)+\Prob(|X-Y|>t)\le\Phi(x)+t/\sqrt{2\pi}+d_K(Y,\mathcal N(0,1))+\E(X-Y)^2/t^2$, and the lower bound is symmetric. The rest is the triangle inequality in $L^2$.
\end{proof}

\subsection{The smoothed current and the variance}\label{ga:variance-sub}
The crossing sum $Q_n$ of \eqref{ga:crossing} is a discontinuous function of the coefficients. To compute its variance, and later to apply the second-order Poincar\'e bound, we replace its signs by smoothed signs. For $0<\varepsilon\le1$ put $h_\varepsilon(z)=2\Phi(z/\varepsilon)-1=\E\sgn(z+\varepsilon G)$, with $G$ standard Gaussian, and
\begin{equation}\label{ga:current-def}
 F_{n,\varepsilon}=\frac12\int_{I_n}h_\varepsilon(U_n(x))\,h_\varepsilon'(V_n(x))\,V_n'(x)\dd x .
\end{equation}
Near a simple zero of $V_n$ the factor $h_\varepsilon'(V_n)V_n'$ is close to twice the Dirac mass with orientation $\sgn V_n'$, so $F_{n,\varepsilon}$ is a smoothed form of $Q_n$, and it is a smooth function of $(\xi_0,\dots,\xi_n)$. Let $P_tf(z)=\E f(e^{-t}z+\sqrt{1-e^{-2t}}\,G)$ be the Ornstein--Uhlenbeck semigroup acting on functions of this vector, and write $I_n=[a,b]$.

\begin{lemma}[Smoothed current]\label{ga:current}
Let $n\ge n_0$, $0<\varepsilon\le1$, $\kappa=(1+\varepsilon^2)^{-1}$, $A_\kappa=\arcsin(\kappa\rho_+)$ and $B_\kappa=\arcsin(\kappa\rho_-)$.
\begin{enumerate}[label=\textup{(\alph*)}]
\item $\E F_{n,\varepsilon}=\E Q_n=0$, $|F_{n,\varepsilon}|\le6n+1$, and $F_{n,\varepsilon}\to Q_n$ almost surely and in $L^2$ as $\varepsilon\downarrow0$.
\item The arcsine identity
\begin{equation}\label{ga:arcsine}
 \Var F_{n,\varepsilon}=\frac1{\pi^2}\Bigl\{\int_a^b\bigl[A_\kappa\partial_yB_\kappa\bigr]_{x=a}^{x=b}\dd y-\iint_{I_n^2}\partial_xA_\kappa\,\partial_yB_\kappa\dd x\dd y\Bigr\}
\end{equation}
holds, and it holds for $\Var Q_n$ with $\kappa=1$, the integrands being defined off the diagonal.
\item $F_{n,\varepsilon}=P_tQ_n$ with $t=\frac12\log(1+\varepsilon^2)$. In particular $\|Q_n-F_{n,\varepsilon}\|_2^2\le\Var Q_n-\Var F_{n,\varepsilon}$.
\end{enumerate}
\end{lemma}

\begin{proof}
(a) The reversal $\xi_j\mapsto\xi_{n-j}$, $j\le n$, preserves the law, fixes $X_n$ and changes the sign of $Y_n$, hence of $Q_n$ and, as $h_\varepsilon'$ is even, of $F_{n,\varepsilon}$. On the event of Lemma~\ref{ga:zeros}, $h_\varepsilon'(V_n)\to0$ uniformly away from the zeros of $V_n$, and near a zero $x_k$ the substitution $w=V_n(x)$ turns the integral into $\frac12\sgn V_n'(x_k)\int h_\varepsilon(U_n(x(w)))h_\varepsilon'(w)\dd w\to\sgn U_n(x_k)\sgn V_n'(x_k)$, as $\int h_\varepsilon'=2$. For domination, $V_n'$ has the sign of $2v_n^-Y_n'-Y_n(v_n^-)'$, a trigonometric polynomial in $t=x/n$ with frequencies in $\frac12\Z$ of size at most $\frac32n$, not identically zero since $V_n(L)\ne V_n(L+1)$ almost surely. So $V_n'$ has at most $6n$ zeros in $I_n$, $|F_{n,\varepsilon}|$ is at most half the total variation $2(6n+1)$ of $h_\varepsilon\circ V_n$, and bounded convergence gives the $L^2$ limit.

(b) Let $G'$ be an independent copy of $G$. Since $U_n(x)+\varepsilon G$ and $U_n(y)+\varepsilon G'$ have correlation $\kappa\rho_+(x,y)$, Sheppard's formula gives the first identity below. The second is the same for $V_n$, and differentiation under the expectation gives the third:
\begin{align*}
 \E h_\varepsilon(U_n(x))h_\varepsilon(U_n(y))&=\tfrac2\pi A_\kappa(x,y),\\
 \E h_\varepsilon(V_n(x))h_\varepsilon(V_n(y))&=\tfrac2\pi B_\kappa(x,y),\\
 \E\bigl[(h_\varepsilon'(V_n)V_n')(x)\,(h_\varepsilon'(V_n)V_n')(y)\bigr]&=\tfrac2\pi\,\partial_x\partial_yB_\kappa(x,y).
\end{align*}
By the independence of $X_n$ and $Y_n$, $\E F_{n,\varepsilon}^2=\pi^{-2}\iint A_\kappa\partial_x\partial_yB_\kappa$, and an integration by parts in $x$ gives \eqref{ga:arcsine}. As $\varepsilon\downarrow0$ the left side tends to $\Var Q_n$ by (a). By Lemma~\ref{ga:fields}(b) the boundary integrand is at most $C\min(1,|b-y|^{-1})+C\min(1,|a-y|^{-1})$ and the double integrand at most $C\min(1,|x-y|^{-2})$, uniformly in $\kappa$, so dominated convergence gives the identity at $\kappa=1$.

(c) Let $\xi'$ be an independent copy and $\sigma=\sqrt{1-e^{-2t}}$. At fixed $x$ the noises $\langle u(x),\xi'\rangle$, $\langle v(x),\xi'\rangle$ and $\langle v'(x),\xi'\rangle$, where $u,v$ are the rows of $U_n,V_n$, are independent, because cosine and sine rows are orthogonal and $v'\perp v$. Averaging $h_\eta(U_n)h_\eta'(V_n)V_n'$ over the noise therefore gives $h_\alpha(U_n)h_\alpha'(V_n)V_n'$ with $\alpha^2=e^{2t}(\eta^2+\sigma^2)$, that is, $P_tF_{n,\eta}=F_{n,\alpha}$. As $\eta\downarrow0$, $\alpha\to\varepsilon$, and (a) with the $L^2$ contractivity of $P_t$ gives $P_tQ_n=F_{n,\varepsilon}$. Finally $\langle Q_n,P_tQ_n\rangle=\|P_{t/2}Q_n\|_2^2\ge\|P_tQ_n\|_2^2$.
\end{proof}

\begin{proposition}[Variance]\label{ga:variance}
The constant $c_G$ of Theorem~\ref{thm:B} is finite and positive, and
\begin{equation}\label{ga:var-rate}
 |\Var D_n-c_Gn|\le Cn^{99/100}\qquad(n\ge1).
\end{equation}
\end{proposition}

\begin{proof}
The integrand $k(u)=r'(u)^2/(1-r(u)^2)$ of $c_G$ is positive and continuous off the origin, $k(u)\to1/12$ as $u\to0$ since $r(u)=1-u^2/24+O(u^4)$, and $k(u)=O(u^{-2})$ at infinity. So $0<c_G<\infty$, and quadrature gives $c_G\simeq0.3058977938$.

In the arcsine identity \eqref{ga:arcsine} at $\kappa=1$ the boundary term is $O(\log n)$ and the double integral is $O(n)$, so $\Var Q_n\le Cn$. By \eqref{ga:crossing} and \eqref{ga:occupancy}, $\|D_n-Q_n\|_2\le C$, so $|\Var D_n-\Var Q_n|\le C\sqrt n$ and $\Var D_n\le Cn$.

For the rate substitute $x=ns$, $y=ns+u$. The double integral becomes $n\int_0^\pi\int_{\R}g_n(s,u)\dd u\dd s$ with $g_n(s,u)=(\partial_xA_1\partial_yB_1)(ns,ns+u)$ when both points lie in $I_n$ and $g_n=0$ otherwise, and $|g_n|\le C\min(1,u^{-2})$. Its limit is $g(u)=-k(u)$, and $-\frac n{\pi^2}\int_0^\pi\int_{\R}g=c_Gn$. Let $\varrho=R^{-1}=n^{-1/100}$. The regions $|u|<\varrho$, $|u|>R$ and $s\notin[2\varrho,\pi-2\varrho]$ contribute at most $C\varrho$ to $\iint|g_n-g|$. On the remaining rectangle, at the point $(ns,ns+u)$,
\[
 |\rho_\pm-r(u)|+|\partial_x\rho_++r'(u)|+|\partial_y\rho_--r'(u)|\le e_n:=C\Bigl(\frac{1+R}n+\frac1{n\varrho}\Bigr).
\]
Indeed, by Lemma~\ref{ga:fields}(a) the reflected terms $R_n^{(k)}(2ns+\cdot)$ are $O(1/(n\varrho))$, and so are the deviations of the variances from $\frac12$ and their derivatives, while Riemann sums give $|R_n^{(k)}(v)-r^{(k)}(v)|\le C(1+|v|)/n$. Now $\partial_xA_1\partial_yB_1=\Pi(\rho_+,\rho_-,\partial_x\rho_+,\partial_y\rho_-)$ and $g=\Pi(r,r,-r',r')$ with $\Pi(v,w,p,q)=pq\{(1-v^2)(1-w^2)\}^{-1/2}$. Between the two arguments $1-v^2,1-w^2\ge c\varrho^2-3e_n\ge c\varrho^2/2$, where the first derivatives of $\Pi$ are $O(\varrho^{-4})$. So $|g_n-g|\le C\varrho^{-4}e_n$ on the rectangle, which contributes $CR\varrho^{-4}e_n\le Cn^{-94/100}$. Together with the boundary term,
\[
 |\Var Q_n-c_Gn|\le C\log(n+2)+Cn\bigl(n^{-1/100}+n^{-94/100}\bigr)\le Cn^{99/100}.
\]
The bound for $\Var D_n$ follows.
\end{proof}

\subsection{Covariance across degrees}\label{ga:covariance-sub}
For $n<m$ the curves of $D_n$ and $D_m$ share coefficients, and the coordinates of one curve are correlated with those of the other. Covariances across degrees are computed from the radial current $\mathcal J_l(s)=\Re\{e^{is}f_l'(e^{is})/f_l(e^{is})\}=a_{f_l}(s)+l/2$, where $a_{f_l}$ is the argument speed of Lemma~\ref{pre:argument}. Since almost surely no $f_l$ vanishes on the unit circle, \eqref{pre:balanced} and Lemma~\ref{pre:argument}(iii) give
\begin{equation}\label{ga:current-winding}
 D_l=\frac1\pi\int_0^\pi\Bigl(\mathcal J_l(s)-\frac l2\Bigr)\dd s .
\end{equation}
The current $\mathcal J_l(s)$ is also the radial derivative of $\log|f_l|$ on the unit circle, so covariances of currents are derivatives of covariances of logarithms.

Identify $\C^2$ with $\R^4$ through $\langle(z,w),(z',w')\rangle=\Re(z\bar z'+w\bar w')$. A centered Gaussian vector $Z=(Z_1,Z_2)$ in $\C^2$ has the real covariance operator $S$ given by $\langle\zeta,S\zeta'\rangle=\E\langle\zeta,Z\rangle\langle\zeta',Z\rangle$, which determines its law. For positive definite $S$ let $\mathcal G(S)=\Cov(\log|Z_1|,\log|Z_2|)$ with $Z$ centered Gaussian of covariance $S$. Write $S=\Delta+\Gamma$ with $\Delta$ block diagonal (the covariances of $Z_1$ and of $Z_2$) and $\Gamma$ the cross part, and let $\mathcal S_\mu=\{S:\mu\,\mathrm{Id}\le S\le2\,\mathrm{Id}\}$ for $0<\mu\le1$.

\begin{lemma}[Log-modulus covariance]\label{ga:logcov}
\begin{enumerate}[label=\textup{(\alph*)}]
\item $\mathcal G$ is smooth, and $\|D^k\mathcal G(S)\|\le A_k\mu^{-2k-2}$ for $S\in\mathcal S_\mu$ and $k\le4$.
\item $\mathcal G(\Delta)=0$ and $\mathcal G(\Delta+\Gamma)=\mathcal G(\Delta-\Gamma)$. If $\Delta+\Gamma\in\mathcal S_\mu$, then $\Delta+a\Gamma\in\mathcal S_\mu$ for $|a|\le1$.
\item Let $S=\Delta+\Gamma\in\mathcal S_\mu$ and $M_k=A_k\mu^{-2k-2}$. For block-diagonal $K,K'$ and cross operators $\Gamma',\Gamma''$,
\[
\begin{gathered}
 |D\mathcal G(S)[\Gamma']|\le M_2\|\Gamma\|\|\Gamma'\|,\qquad |D^2\mathcal G(S)[K,\Gamma']|\le M_3\|\Gamma\|\|K\|\|\Gamma'\|,\\
 |D^2\mathcal G(S)[K,K']|\le\tfrac12M_4\|\Gamma\|^2\|K\|\|K'\|,\qquad |D^2\mathcal G(S)[\Gamma',\Gamma'']|\le M_2\|\Gamma'\|\|\Gamma''\|.
\end{gathered}
\]
\end{enumerate}
\end{lemma}

\begin{proof}
(a) Let $\ell(z)=\log|z|$, let $p_S$ be the centered Gaussian density on $\R^4$ with covariance $S$, and let $S_{11},S_{22}$ be the diagonal blocks. Then
\[
 \mathcal G(S)=\int\ell(\zeta_1)\ell(\zeta_2)\,p_S(\zeta)\dd\zeta-\int\ell\,p_{S_{11}}\int\ell\,p_{S_{22}}.
\]
Derivatives of $\log p_S=-\frac12\log\det S-\frac12\langle\zeta,S^{-1}\zeta\rangle+\mathrm{const}$ in $S$ are traces and quadratic forms built from $S^{-1}$ and the directions, so $|D^kp_S[H_1,\dots,H_k](\zeta)|\le C_k\mu^{-2k}(1+|\zeta|^{2k})p_S(\zeta)\prod\|H_i\|$. On $\mathcal S_\mu$, $p_S(\zeta)\le(2\pi)^{-2}\mu^{-2}e^{-|\zeta|^2/4}$, which is integrable against $(1+|\zeta|^{2k})|\ell(\zeta_1)\ell(\zeta_2)|$. This justifies differentiation under the integral and gives the bound, and the marginal integrals are treated in the same way.

(b) If $\Gamma=0$, $Z_1$ and $Z_2$ are independent. The vector $(Z_1,-Z_2)$ has covariance $\Delta-\Gamma$ and the same moduli. With $\mathcal R(z,w)=(z,-w)$, $\Delta+a\Gamma=\frac{1+a}2S+\frac{1-a}2\mathcal RS\mathcal R$ is a convex combination of elements of $\mathcal S_\mu$.

(c) By (b), for block-diagonal $\Delta'$ near $\Delta$ and cross $\Gamma'$, $a\mapsto\mathcal G(\Delta'+a\Gamma')$ is even and $\mathcal G(\Delta')=0$. Hence $D\mathcal G(\Delta)[\Gamma']=0$, $D^2\mathcal G(\Delta)[K,\Gamma']=0$, $D^3\mathcal G(\Delta)[K,K',\Gamma']=0$ and $D^2\mathcal G(\Delta)[K,K']=0$. Taylor's formula along $a\mapsto\Delta+a\Gamma$, $0\le a\le1$, gives the first three bounds, and the fourth is (a).
\end{proof}

\begin{lemma}[Radial exchange]\label{ga:exchange}
Let $P,Q$ be random polynomials whose coefficients are linear functions of one standard Gaussian vector. Fix $z_0,w_0\ne0$ and put $\zeta(x)=P(z_0e^x)$, $\omega(y)=Q(w_0e^y)$. Suppose that for $|x|,|y|\le x_0$ the real covariance of $(\zeta(x),\omega(y))$ is positive definite, and that almost surely $P$ has no zero on $\{z_0e^x:|x|\le x_0\}$ and $Q$ none on $\{w_0e^y:|y|\le x_0\}$. Then $\Phi(x,y)=\Cov(\log|\zeta(x)|,\log|\omega(y)|)$ is $C^2$, the variables $\Re(\zeta'/\zeta)(x)$ and $\Re(\omega'/\omega)(y)$ and their product are integrable uniformly for $|x|,|y|\le x_0$, and
\[
 \partial_x\partial_y\Phi(0,0)=\Cov\bigl(\Re(\zeta'/\zeta)(0),\Re(\omega'/\omega)(0)\bigr),\qquad
 \partial_x\E\log|\zeta(x)|=\E\Re(\zeta'/\zeta)(x).
\]
\end{lemma}

\begin{proof}
We bound the logarithmic derivatives by Gaussian regression, and integrate them to recover $\Phi$. Gaussian regression writes $(\zeta',\omega')=\mathcal A(\zeta,\omega)+\mathcal R$ with $\mathcal A=\mathcal A_{x,y}$ real linear and $\mathcal R$ centered Gaussian independent of the values, both continuous in $(x,y)$. So there is $K$, depending on the smallest eigenvalue of the value covariance on the square, with
\begin{equation}\label{ga:conditional-moment}
 \E\bigl[|\zeta'||\omega'|\bigm|\zeta=z,\omega=w\bigr]\le K(1+|z|^2+|w|^2),\qquad
 \E\bigl[|\zeta'|\bigm|\zeta=z,\omega=w\bigr]\le K(1+|z|+|w|),
\end{equation}
and with value density at most $Ke^{-(|z|^2+|w|^2)/K}$. As $(1+|z|^2)(1+\log^2|z|)|z|^{-1}e^{-|z|^2/K}$ is integrable on $\R^2$, the expectations of $|\zeta'\omega'/(\zeta\omega)|$, $|(\zeta'/\zeta)\log|\omega||$ and $|\log|\zeta|\log|\omega||$ are bounded on the square. Given the values $(z,w)$, the conditional expectation of $\Re(\zeta'\bar z)\Re(\omega'\bar w)$ is a polynomial in $(z,w)$ with coefficients continuous in $(x,y)$, so $k(x,y)=\Cov(\Re(\zeta'/\zeta)(x),\Re(\omega'/\omega)(y))$ is continuous by dominated convergence. Integrating $\log|\zeta(x)|-\log|\zeta(0)|=\int_0^x\Re(\zeta'/\zeta)$ and using Fubini's theorem,
\[
 \Phi(x,y)-\Phi(x,0)-\Phi(0,y)+\Phi(0,0)=\int_0^x\!\!\int_0^yk(a,b)\dd b\dd a .
\]
Since $\Phi=\mathcal G(S(x,y))$ with $S(x,y)$ the smooth positive definite value covariance, $\Phi$ is $C^2$ by Lemma~\ref{ga:logcov}(a), and dividing by $xy$ gives the first identity. The second is the one-variable case.
\end{proof}

The no-zero hypothesis holds for $P=f_n$ with real or circular complex Gaussian coefficients and $z_0=\varrho_0e^{is}$, where $s\notin\pi\Z$ in the real case. Indeed, given $\xi_1,\dots,\xi_n$, a zero at $z$ forces $\xi_0=-T(z)$ with $T(z)=\sum_{j\ge1}\xi_jz^j$. In the real case $\Im T(\varrho e^{is})$ is a polynomial in $\varrho$ with linear coefficient $\xi_1\sin s\ne0$, so $T$ is real at finitely many points of the segment. In the complex case $-T$ maps the segment onto a $C^1$ curve of planar measure zero.

\begin{lemma}[Proper Gaussian logarithms]\label{ga:dilog}
\textup{(a)} Let $(Z,W)$ be centered circular complex Gaussian with $\E|Z|^2=\E|W|^2=1$, $\E Z\bar W=c$ and $v=|c|^2<1$. Then
\[
 \Cov(\log|Z|,\log|W|)=\tfrac14\operatorname{Li}_2(v)=\frac14\int_0^1\!\!\int_0^1\frac{v\dd a\dd b}{1-vab}.
\]
\textup{(b)} For $q\ge1$ and real $x,y,u$ put $c_q(x,y,u)=I_1(x+y+iu)/\sqrt{I_1(2x)I_q(2y)}$ with $I_a(z)=\int_0^ae^{zv}\dd v$, and let $S_c$ be the covariance operator of the pair in \textup{(a)}. Then $|c_q|^2<1$ if $q>1$ or $u\ne0$, and then, with $\Psi(v)=\int_0^1(1-va)^{-1}\dd a$ and $r=r(u)$,
\begin{equation}\label{ga:local-kernel}
 K_q(u):=\partial_x\partial_y\,\mathcal G\bigl(S_{c_q(x,y,u)}\bigr)\Big|_{x=y=0}=\frac{\Psi(r^2/q)\,(r'^2-rr'')}{2q},\qquad
 \int_{\R}K_q(u)\dd u=\pi a(q).
\end{equation}
Moreover $|K_q(u)|\le C(q-1)^{-1}(1+|u|)^{-2}$ for $q>1$, and $|K_1(u)|\le C(1+|\log|u||)(1+|u|)^{-2}$.
\end{lemma}

\begin{proof}
(a) Write $W=\bar cZ+\sqrt{1-v}\,Z'$ with $Z'$ an independent copy of $Z$. For $\sigma,\tau\ge0$, conditioning on $Z$ gives
\[
 \E e^{-\sigma|Z|^2-\tau|W|^2}=\frac1{(1+\sigma)(1+\tau)-v\sigma\tau},\qquad
 \frac{\Cov(e^{-\sigma|Z|^2},e^{-\tau|W|^2})}{\sigma\tau}=\frac{v}{\Lambda(\Lambda-v\sigma\tau)}\ge0,
\]
with $\Lambda=(1+\sigma)(1+\tau)$. The Frullani integrals $T_{\epsilon,R}(U)=\int_\epsilon^R(e^{-\sigma}-e^{-\sigma U})\sigma^{-1}\dd\sigma$ tend to $\log U$ as $\epsilon\to0$, $R\to\infty$, and $|T_{\epsilon,R}(U)|\le|\log U|$. Since $\log|Z|$, $\log|W|$ and their product are integrable, dominated convergence and Fubini's theorem give
\[
 \Cov(\log|Z|^2,\log|W|^2)=\int_0^\infty\!\!\int_0^\infty\frac{v\dd\sigma\dd\tau}{\Lambda(\Lambda-v\sigma\tau)}=\int_0^1\!\!\int_0^1\frac{v\dd a\dd b}{1-vab},
\]
with $a=\sigma/(1+\sigma)$, $b=\tau/(1+\tau)$. Expanding the geometric series gives $\operatorname{Li}_2(v)$.

(b) By the Cauchy--Schwarz inequality $|I_1(x+y+iu)|^2\le I_1(2x)I_1(2y)\le I_1(2x)I_q(2y)$, strictly if $u\ne0$, respectively $q>1$. Note $\operatorname{Li}_2'=\Psi$ and $\Psi(v)+v\Psi'(v)=(1-v)^{-1}$. Put $F=|I_1(x+y+iu)|^2$ and $\mathcal D=I_1(2x)I_q(2y)$. From $I_1(iu)=e^{iu/2}r$, $I_1'(iu)=e^{iu/2}(r/2-ir')$ and $I_1''(iu)=e^{iu/2}(r/4-ir'-r'')$, at $x=y=0$,
\[
 F=F_x=F_y=r^2,\quad F_{xy}=r^2+2(r'^2-rr''),\qquad \mathcal D=\mathcal D_x=q,\quad \mathcal D_y=\mathcal D_{xy}=q^2.
\]
So $V=F/\mathcal D$ has $V=r^2/q$, $V_x=0$ and $V_{xy}=2(r'^2-rr'')/q$. By (a), $\mathcal G(S_{c_q})=\frac14\operatorname{Li}_2(V)$, and since $V_x=0$ the chain rule gives \eqref{ga:local-kernel}. With $\mathcal L_q=\Psi(r^2/q)rr'/(2q)$, the identity for $\Psi$ gives $K_q+\mathcal L_q'=r'^2/(q-r^2)$. Since $r^{(k)}(u)=\Re\int_0^1(i(v-\frac12))^ke^{iu(v-1/2)}\dd v$, one integration by parts gives $|r^{(k)}(u)|\le(3+k)/(1+|u|)$, and $|r'(u)|\le|u|/12$. Integrating over $|u|\ge\epsilon$ and letting $\epsilon\to0$ gives $\int K_q=\int r'^2/(q-r^2)=\pi a(q)$. For $q=1$ the boundary terms vanish because $0\le\Psi(r^2)\le1+\log(1-r^2)^{-1}\le C(1+|\log|u||)$, as $1-r^2\ge c\min(u^2,1)$. The same bounds and $\Psi(r^2/q)\le q/(q-1)$ bound $K_q$.
\end{proof}

To compare two degrees, fix $1\le n<m$, write $q=m/n$, and use the radial scale $1/n$ for both. For real $x,y$ and angles $s,t$ put
\[
 W_n(x,s)=\frac{f_n(e^{x/n+is})}{\sqrt{K_n(e^{x/n})}},\qquad W_m(y,t)=\frac{f_m(e^{y/n+it})}{\sqrt{K_m(e^{y/n})}},
\]
where $K_l(r)=\sum_{j=0}^lr^{2j}$ as in Theorem~\ref{thm:G}. Then $\E|W_n|^2=\E|W_m|^2=1$, and the real covariance operator of $(W_n,W_m)$ is
\begin{equation}\label{ga:S}
 S(z,w)=\tfrac12\bigl(z+p_n\bar z+\gamma w+\eta\bar w,\ \bar\gamma z+\eta\bar z+w+p_m\bar w\bigr),
\end{equation}
with $p_n=\E W_n^2$, $p_m=\E W_m^2$, $\gamma=\E W_n\bar W_m$ and $\eta=\E W_nW_m$. Explicitly, $p_n$ and $p_m$ are the values $p_n(e^{x/n},s)$ and $p_m(e^{y/n},t)$ of the functions of Theorem~\ref{thm:G}, and
\[
 \gamma,\eta=\sum_{j=0}^n\frac{e^{j(x+y)/n+ij(s\mp t)}}{\sqrt{K_n(e^{x/n})K_m(e^{y/n})}} .
\]
Deterministic factors do not change covariances of logarithms, so $\mathcal G(S(x,y;s,t))$ is the covariance of $\log|f_n(e^{x/n+is})|$ and $\log|f_m(e^{y/n+it})|$, and $\partial_x\log|f_n(e^{x/n+is})|=\mathcal J_n(s)/n$ at $x=0$.

\begin{lemma}[Cross-degree geometry]\label{ga:cross-geometry}
Let $q=m/n\le q_1$, $\beta=1-\sqrt{(n+1)/(m+1)}$, $T\ge4\pi/\beta$ with $2T<\pi n$, and $I^T=[T/n,\pi-T/n]$, and put
\[
 K_{n,m}(s,t)=n^{-2}\Cov\bigl(\mathcal J_n(s),\mathcal J_m(t)\bigr).
\]
For $s,t\in I^T$ and $x=y=0$ we have $S\in\mathcal S_{\beta/4}$ and $K_{n,m}(s,t)=\partial_x\partial_y\mathcal G(S(x,y;s,t))$. The block-diagonal part satisfies $\Delta_{xy}=0$ and $\|\Delta_x\|+\|\Delta_y\|\le C_{q_1}$, and
\[
 \|\Gamma\|+\|\Gamma_x\|+\|\Gamma_y\|+\|\Gamma_{xy}\|\le\frac{C_{q_1}}{1+n|s-t|},\qquad|K_{n,m}(s,t)|\le\frac{C_{q_1}\beta^{-10}}{(1+n|s-t|)^2}.
\]
\end{lemma}

\begin{proof}
At $x=y=0$ the Cauchy--Schwarz inequality gives $|\gamma|\le\sqrt{(n+1)/(m+1)}$, so the part $\frac12(z+\gamma w,\bar\gamma z+w)$ of $S$ is at least $\beta/2$. By \eqref{rc:abel}, $|p_n|,|p_m|,|\eta|\le\pi/(2T)$, using $\dist(2s,2\pi\Z)\ge2T/n$ and $(s+t)/2\in I^T$. So the remaining part of $S$ has norm at most $\pi/T\le\beta/4$. Also $S\ge0$ and $\operatorname{tr}S=2$, so $S\le2\,\mathrm{Id}$. The diagonal blocks depend on $x$ and $y$ separately, so $\Delta_{xy}=0$, and $|\partial_xp_n|+|\partial_yp_m|\le8q_1$ from the formulas. Each of $\gamma,\eta$ and their derivatives $\partial_x,\partial_y,\partial_x\partial_y$ at the origin is a sum $\sum_{j=0}^nP(j/n)e^{ij\vartheta}/\sqrt{(n+1)(m+1)}$ with $\vartheta=s\mp t$ and $P$ of degree at most two with coefficients bounded by $C_{q_1}$, for instance $P(v)=(v-\frac12)(v-\frac q2)$ for $\partial_x\partial_y\gamma$. By \eqref{rc:abel} these sums are at most $C_{q_1}\min(1,(n|s-t|)^{-1})$, since $\dist(s\mp t,2\pi\Z)\ge|s-t|$ for $s,t\in[0,\pi]$. The identity for $K_{n,m}$ is Lemma~\ref{ga:exchange}, whose hypotheses hold by the continuity of $S$ and the remark after it. Finally $K_{n,m}=D^2\mathcal G(S)[\Delta_x+\Gamma_x,\Delta_y+\Gamma_y]+D\mathcal G(S)[\Gamma_{xy}]$, and Lemma~\ref{ga:logcov}(c) with $\mu=\beta/4$ bounds the five resulting terms.
\end{proof}

For a cutoff $T$ as in Lemma~\ref{ga:cross-geometry} and $l\in\{n,m\}$ put $A_l=\frac1\pi\int_{I^T}(\mathcal J_l(s)-\frac l2)\dd s$. By \eqref{ga:winding} and Lemma~\ref{pre:half-turns} on the two end intervals, $|D_l-A_l|\le2+N_l$, where $N_l$ counts the zeros of $Y_l$ in $(0,T_l)\cup(\pi l-T_l,\pi l)$, $T_l=Tl/n$. Lemma~\ref{ga:jensen} with $\varrho=T_l$, applied as in the proof of Lemma~\ref{ga:zeros}, gives $\Prob(N_l>t)\le Ce^{9T_l}2^{-t/2}$, so the angular cutoff costs
\begin{equation}\label{ga:cutoff-error}
 \E(D_l-A_l)^2\le C(1+T_l)^2\qquad(1\le T_l\le\pi l/4).
\end{equation}

\begin{proposition}[Covariance across degrees]\label{ga:covariance}
\begin{enumerate}[label=\textup{(\alph*)}]
\item For every $q_1>1$ there is $C_{q_1}$ such that for all $n\ge1$ and $n<m\le q_1n$ with $\delta=(m-n)/n\ge n^{-1/1000}$,
\begin{equation}\label{ga:cov-rate}
 \bigl|n^{-1}\Cov(D_n,D_m)-a(1+\delta)\bigr|\le C_{q_1}\delta^{-100}n^{-1/100}.
\end{equation}
\item The function $a$ is continuous on $[1,\infty)$ with $0\le a\le c_G$, and $c_G-a(1+\delta)=\sqrt{\delta/12}+O(\delta)$ as $\delta\downarrow0$.
\item Put $\Lambda(d)=c_Gd+2(c_G-a(1+d))$. Then $c\sqrt d\le\Lambda(d)\le C\sqrt d$ for $0<d\le1$, $\Lambda(d)/\sqrt d\to1/\sqrt3$ as $d\downarrow0$, and for $1\le h\le n$ with $d=h/n\ge n^{-1/1000}$,
\begin{equation}\label{ga:increment-variance}
 \Bigl|\frac{\Var(D_{n+h}-D_n)}n-\Lambda(d)\Bigr|\le Cd^{-100}n^{-1/100}.
\end{equation}
\end{enumerate}
\end{proposition}

\begin{proof}
\emph{Representation.} Let $T$ be as in Lemma~\ref{ga:cross-geometry}. The constant in \eqref{ga:conditional-moment} is uniform over $(I^T)^2$, because the value covariance is at least $\beta/4$ there, so $\E\iint_{(I^T)^2}|\mathcal J_n(s)\mathcal J_m(t)|\dd s\dd t<\infty$. Fubini's theorem, Lemma~\ref{ga:cross-geometry} and the substitution $t=s-u/n$ give
\begin{equation}\label{ga:cov-representation}
 \frac{\Cov(A_n,A_m)}n=\frac1{\pi^2}\int_0^\pi\!\!\int_{\R}\ind_{\Omega_n}(s,u)\,K_{n,m}\Bigl(s,s-\frac un\Bigr)\dd u\dd s,
 \qquad\Omega_n=\Bigl\{s,s-\frac un\in I^T\Bigr\}.
\end{equation}
By the cutoff bound \eqref{ga:cutoff-error}, the bound $\Var D_l\le Cl$ from \eqref{ga:var-rate} and the Cauchy--Schwarz inequality,
\begin{equation}\label{ga:cov-endpoint}
\begin{aligned}
 |\Cov(D_n,D_m)-\Cov(A_n,A_m)|&\le\|D_n-A_n\|_2\|D_m\|_2+\|A_n\|_2\|D_m-A_m\|_2\\
 &\le C(1+qT)\bigl(\sqrt{qn}+qT\bigr).
\end{aligned}
\end{equation}

(b) The integrands satisfy $0\le r'^2/(q-r^2)\le r'^2/(1-r^2)$, which gives the bounds and, by dominated convergence, continuity. Put $b=1-r^2$ and $k=r'^2/b$, so that $c_G-a(1+\delta)=\pi^{-1}\int\delta k/(\delta+b)$. Near zero $b=u^2/12+O(u^4)$ and $k=1/12+O(u^2)$, so on $|u|\le u_0$ the integrand differs from $\delta/(12\delta+u^2)$ by at most $C\delta(u^2\delta+u^4)/(\delta+u^2)^2\le C\delta$. On $|u|\ge u_0$, $b$ is bounded below and $k$ is integrable, so that part is $O(\delta)$. Finally
\[
 \frac1\pi\int_{\R}\frac{\delta\dd u}{12\delta+u^2}=\sqrt{\frac\delta{12}},
\]
and the part of this integral over $|u|\ge u_0$ is $O(\delta)$.

(a) Let $q=m/n=1+\delta$. The left side of \eqref{ga:cov-rate} is at most $B_0=C\sqrt{q_1}+c_G$ by Proposition~\ref{ga:variance} and (b). If $\delta\le n^{-1/10000}$, then $\delta^{-100}n^{-1/100}\ge1$, and the bound holds with $C_{q_1}\ge B_0$. Let $\delta\ge n^{-1/10000}$. Then $\beta=(m-n)/(\sqrt{m+1}(\sqrt{m+1}+\sqrt{n+1}))\ge\delta/(4q_1)$. Take $T=16\pi q_1/\delta$, so that $S\in\mathcal S_{\mu}$ on $(I^T)^2$ with $\mu=\delta/(16q_1)$, and $\varrho=R^{-1}=n^{-1/50}$, so that $T/n\le\varrho$ for large $n$. Let $s\in[2\varrho,\pi-2\varrho]$, $|u|\le R$ and $t=s-u/n$, so that $s,t\in[\varrho,\pi-\varrho]\subset I^T$. The entry $\gamma$ of \eqref{ga:S} and its derivatives at the origin are Riemann sums,
\[
 \frac1n\sum_{j\le n}P(j/n)e^{iju/n}=\int_0^1P(v)e^{iuv}\dd v+O_{q_1}\Bigl(\frac{1+|u|}n\Bigr),\qquad\frac n{\sqrt{(n+1)(m+1)}}=q^{-1/2}+O\Bigl(\frac1n\Bigr),
\]
whose limits are the corresponding entries of $S_0=S_{c_q(x,y,u)}$ from Lemma~\ref{ga:dilog}(b). The entries $p_n,p_m,\eta$ have phases $2s,2t,s+t$ at distance at least $2\varrho$ from $2\pi\Z$, so by \eqref{rc:abel} they and their derivatives are $O(1/(n\varrho))$. Hence, with all derivatives taken at the origin,
\[
 \|S-S_0\|+\|S_x-S_{0,x}\|+\|S_y-S_{0,y}\|+\|S_{xy}-S_{0,xy}\|\le C_{q_1}e_n,\qquad e_n=\frac{1+R}n+\frac1{n\varrho}.
\]
Since $|c_q|^2\le q^{-1}$ and $1-q^{-1/2}\ge\delta/(2q_1)$, $S_0\in\mathcal S_{\delta/(4q_1)}$, and the segment from $S$ to $S_0$ lies in $\mathcal S_\mu$. There Lemma~\ref{ga:logcov}(a) bounds $D\mathcal G,D^2\mathcal G,D^3\mathcal G$ by $C_{q_1}\delta^{-8}$, so $|K_{n,m}(s,s-u/n)-K_q(u)|\le C_{q_1}\delta^{-8}e_n$ on the rectangle. Outside it, Lemmas~\ref{ga:cross-geometry} and~\ref{ga:dilog} bound both kernels by $C_{q_1}\delta^{-10}(1+|u|)^{-2}$. As $\int K_q=\pi a(q)$, \eqref{ga:cov-representation} gives
\[
 \Bigl|\frac{\Cov(A_n,A_m)}n-a(q)\Bigr|\le C_{q_1}\delta^{-10}\bigl(\varrho+R^{-1}+Re_n\bigr)\le C_{q_1}\delta^{-10}n^{-1/50}.
\]
By \eqref{ga:cov-endpoint} with $qT\le16\pi q_1^2/\delta$, the endpoint error divided by $n$ is at most $C_{q_1}\delta^{-2}n^{-1/2}$. Altogether the error is at most $C_{q_1}\delta^{-10}n^{-1/50}\le C_{q_1}\delta^{-100}n^{-1/100}$ for large $n$.

(c) By (b) there are $\delta_1\in(0,1]$ and $K$ with $\sqrt{d/48}\le c_G-a(1+d)\le\sqrt{d/12}+Kd$ for $d\le\delta_1$, while $c_G\sqrt{\delta_1d}\le c_Gd\le\Lambda(d)\le3c_G$ for $\delta_1\le d\le1$. This gives the two-sided bound and the limit $2/\sqrt{12}=1/\sqrt3$. Expanding $\Var(D_{n+h}-D_n)$, the variance rate \eqref{ga:var-rate} and \eqref{ga:cov-rate} with $q_1=2$ give \eqref{ga:increment-variance}.
\end{proof}

For a stationary circular complex Gaussian process with covariance $r$, Buckley and Feldheim \cite[Theorems~1 and~3, Remark~1.3]{BuckleyFeldheim} give the asymptotic variance $\int r'^2/(1-r^2)=\pi c_G$ of the argument increment per unit length. The real polynomial has in addition the reflected terms $R_n(x+y)$ in \eqref{ga:kernel}, which the estimates above control.

\subsection{Normal approximation}\label{ga:normal-sub}
The smoothed current is a smooth function of the coefficients, close to $D_n$ in $L^2$, with derivatives of controlled size. This is what Chatterjee's bound \eqref{ga:chatterjee-bound} requires.

\begin{proposition}[Smoothing]\label{ga:smoothing}
For $n\ge n_0$ and $0<\varepsilon\le1$, $\|D_n-F_{n,\varepsilon}\|_2^2\le C(n\varepsilon+\log(n+2))$ and
\[
 \bigl\|\|\nabla F_{n,\varepsilon}\|\bigr\|_4\le C\varepsilon^{-2}\sqrt n,\qquad
 \bigl\|\|\nabla^2F_{n,\varepsilon}\|_{\rm op}\bigr\|_4\le C\varepsilon^{-3}\sqrt{\log(n+2)}.
\]
\end{proposition}

\begin{proof}
\emph{Approximation.} By Lemma~\ref{ga:current}(c) it suffices to bound $\Var Q_n-\Var F_{n,\varepsilon}$ through the arcsine identity \eqref{ga:arcsine}. The boundary terms are $O(\log n)$ uniformly in $\kappa$. Put $\Delta=1-\kappa^2\le2\varepsilon^2$. For $x\ne y$ and $\rho=\rho_\pm(x,y)$, the ratio of $\partial\arcsin(\kappa\rho)$ to $\partial\arcsin\rho$ is $(1+\Delta/(\kappa^2(1-\rho^2)))^{-1/2}\in[0,1]$, and its deficit from one is at most $\min(1,2\Delta/(1-\rho^2))$. With \eqref{ga:arcsine-derivative} and the correlation gap, the double integrands at $\kappa$ and at $1$ differ by at most $C\min(1,\Delta|x-y|^{-2})$ for $|x-y|\le1$ and by $C\Delta|x-y|^{-2}$ for $|x-y|\ge1$. Integrating,
\[
 |\Var Q_n-\Var F_{n,\varepsilon}|\le C\log(n+2)+Cn\Bigl(\int_0^1\min\Bigl(1,\frac\Delta{u^2}\Bigr)\dd u+\int_1^\infty\frac{\Delta\dd u}{u^2}\Bigr)\le C\bigl(\log(n+2)+n\sqrt\Delta\bigr),
\]
and $\sqrt\Delta\le2\varepsilon$. With $\|D_n-Q_n\|_2\le C$ this gives the first bound.

\emph{Derivatives.} Let $u(x),v(x)$ be the rows of $U_n(x),V_n(x)$, so $V_n'(x)=\langle v'(x),\xi\rangle$, and put $\psi(a,b,w)=\frac12h_\varepsilon(a)h_\varepsilon'(b)w$. Since $|h_\varepsilon^{(k)}|\le C_k\varepsilon^{-k}$, the first and second derivatives of $\psi$ are at most $C\varepsilon^{-2}(1+|w|)$ and $C\varepsilon^{-3}(1+|w|)$. For $\beta\in\R^{n+1}$ let $\ell_\beta(x)=(\langle u(x),\beta\rangle,\langle v(x),\beta\rangle,\langle v'(x),\beta\rangle)$. By Lemmas~\ref{ga:fields}(a) and~\ref{pre:sampling}, $\int_{I_n}|\ell_\beta|^2\le C\|\beta\|^2$. Differentiating under the integral,
\[
 |\nabla F_{n,\varepsilon}\cdot\beta|\le C\varepsilon^{-2}\int_{I_n}(1+|V_n'|)|\ell_\beta|,\qquad
 |\langle\nabla^2F_{n,\varepsilon}\beta,\beta'\rangle|\le C\varepsilon^{-3}\Bigl(1+\sup_{I_n}|V_n'|\Bigr)\int_{I_n}|\ell_\beta||\ell_{\beta'}|,
\]
so $\|\nabla F_{n,\varepsilon}\|\le C\varepsilon^{-2}(n+\int_{I_n}V_n'^2)^{1/2}$ and $\|\nabla^2F_{n,\varepsilon}\|_{\rm op}\le C\varepsilon^{-3}(1+\sup_{I_n}|V_n'|)$. Minkowski's inequality and $\E V_n'(x)^4\le C$ give the first moment bound. For a unit interval $J$ with left endpoint $x_J$, $\sup_J|V_n'|\le|V_n'(x_J)|+\int_J|V_n''|$ has $L^p$ norm at most $C\sqrt p$. Covering $I_n$ by $\pi n$ such intervals, $\|\sup_{I_n}|V_n'|\|_p\le(\pi n)^{1/p}C\sqrt p$, and $p=\max(4,\log(n+2))$ gives the second.
\end{proof}

All normal approximations in Theorem~\ref{thm:B} follow from one estimate for linear combinations of discrepancies.

\begin{lemma}[Normal approximation]\label{ga:normal}
For each $d\ge1$ there is $C_d$ such that the following holds. Let $n_0\le n_1,\dots,n_d\le N$, $\lambda\in[-1,1]^d$ and $S=\sum_i\lambda_iD_{n_i}$, and suppose $s^2=\Var S\ge N^{39/40}$. Then
\begin{equation}\label{ga:normal-bound}
 d_K\bigl(S/s,\mathcal N(0,1)\bigr)\le C_d\Bigl(\frac{N^{3/4}\sqrt{\log(N+2)}}{s^2}+\Bigl(\frac{N^{19/20}}{s^2}\Bigr)^{1/3}\Bigr).
\end{equation}
\end{lemma}

\begin{proof}
Let $\varepsilon=N^{-1/20}$ and $F=\sum_i\lambda_iF_{n_i,\varepsilon}$, a smooth centered function of $(\xi_0,\dots,\xi_N)$. By Proposition~\ref{ga:smoothing}, $\|S-F\|_2^2\le C_dN^{19/20}$, $\|\|\nabla F\|\|_4\le C_d\varepsilon^{-2}\sqrt N$ and $\|\|\nabla^2F\|_{\rm op}\|_4\le C_d\varepsilon^{-3}\sqrt{\log(N+2)}$. For $N\ge N_d$, Lemma~\ref{ga:chatterjee}(b) gives $|\sqrt{\Var F}-s|\le\|S-F\|_2\le s/2$, so $\Var F\ge s^2/4$, and the bound \eqref{ga:chatterjee-bound} gives
\[
 d_{\rm TV}\bigl(F/\sqrt{\Var F},\mathcal N(0,1)\bigr)\le\frac{C_d}{s^2}\varepsilon^{-5}\sqrt{N\log(N+2)}=C_d\frac{N^{3/4}\sqrt{\log(N+2)}}{s^2}.
\]
Lemma~\ref{ga:chatterjee}(b) with $\|S/s-F/\sqrt{\Var F}\|_2^2\le4\|S-F\|_2^2/s^2$ and $t=(N^{19/20}/s^2)^{1/3}$ gives \eqref{ga:normal-bound}. For $N<N_d$ the right side is at least one if $C_d$ is large, because $s^2\le C d^2N$.
\end{proof}

\subsection{Increments: pairing crossings across degrees}\label{ga:increments-sub}
The increment $D_{n+h}-D_n$ has size at most $(nh)^{1/4}$ up to logarithms, much less than the size $\sqrt n$ of each term, because the curves of degrees $n$ and $n+h$ have nearly the same crossings. Write the polynomial of degree $n+h$ on the carrier of degree $n$. Up to boundary terms and the shift $h/2$, $D_n$ and $D_{n+h}$ then count with signs the crossings of the real axis by two curves $X+iY$ and $(X+X^h)+i(Y+Y^h)$, where $X^h$ and $Y^h$ are the fields of the $h$ added coefficients and have size $\sqrt{h/n}$. The two sine coordinates have opposite signs only where $|Y|<|Y^h|$, a set of measure about $\sqrt{nh}$.

Pairing the crossings of the two curves (Lemma~\ref{ga:canonical}) writes the difference of the two crossing sums as a sum of differences of signs of $X$ at nearby points. Given $Y$ and the added coefficients, this sum is sub-Gaussian, and since two nearby signs rarely differ, its variance is about $\sqrt{nh}$ (Lemma~\ref{ga:paired}). Its conditional mean $J$ fluctuates about $\E J$ on the scale $\sqrt h$, which is at most $(nh)^{1/4}$. Given the added coefficients it concentrates around its average, which is a smooth function of them (Lemma~\ref{ga:conditional-mean}). With the difference $Q$ of the crossing sums and the boundary term $E$ of Lemma~\ref{ga:common-carrier}, the increment splits as
\begin{equation}\label{ga:fine-split}
 D_{n+h}-D_n=(Q-J)+(J-\E J)+(E-\E E),
\end{equation}
and Propositions~\ref{ga:uniform-increment} and~\ref{ga:fine} bound the three terms.

Fix $n\ge n_0$ and $1\le h\le n$, and write $X=X_n$, $Y=Y_n$, $I=I_n$ and $\ell=\log(n+2)$. The added terms form the tail fields
\[
 X^h(x)+iY^h(x)=n^{-1/2}e^{-ix/2}(f_{n+h}-f_n)(e^{ix/n})=n^{-1/2}\sum_{k=n+1}^{n+h}\xi_ke^{i\omega_kx},\qquad \widetilde Y=Y+Y^h,
\]
with $\omega_k=k/n-\frac12\in(\frac12,\frac32]$.

Let $T=(\xi_{n+1},\dots,\xi_{n+h})$ and $\theta=(Y,T)$. The field $X$ depends on the cosine coordinates of $(\xi_0,\dots,\xi_n)$, and $\theta$ on the sine coordinates and $T$, so $X$ is independent of $\theta$. We write $\Prob_x$ and $\E_x$ for probability and expectation over the cosine coordinates with $\theta$ fixed.

Each of $X^h(x),Y^h(x),(X^h)'(x),(Y^h)'(x)$ is centered Gaussian with variance at most $\frac94h/n$. In $t=x/n$ the frequencies $k-n/2$ are distinct and all in $\Z$ or all in $\Z+\frac12$, so Parseval's identity over a period $4\pi$ gives the energy bound
\begin{equation}\label{ga:tail-energy}
 \int_I((X^h)^2+(X^h)'^2+(Y^h)^2+(Y^h)'^2)\dd x\le C\|T\|^2.
\end{equation}

\begin{lemma}[Common carrier]\label{ga:common-carrier}
\textup{(a)} Almost surely the zeros of $Y$ and of $\widetilde Y$ in $(0,\pi n)$ are simple, at most $n$ and $2n$ in number, and disjoint, $Y$ and $\widetilde Y$ do not vanish at $L$ and $\pi n-L$, and $f_n$, $f_{n+h}$ have no zero on the unit circle.

\textup{(b)} On this event,
\begin{equation}\label{ga:increment-identity}
 D_{n+h}-D_n+\frac h2=Q+E,\quad
 Q=\sum_{\substack{z\in I\\\widetilde Y(z)=0}}\sgn(X+X^h)(z)\sgn \widetilde Y'(z)-\sum_{\substack{y\in I\\Y(y)=0}}\sgn X(y)\sgn Y'(y),
\end{equation}
with $|E|\le6+N^{\rm end}$, where $N^{\rm end}$ counts the zeros of $Y$ and $\widetilde Y$ in $(0,\pi n)\setminus I$.

\textup{(c)} Let $\Theta$ be the largest number of zeros of $Y$ and $\widetilde Y$ together in a unit cell meeting $I$, and $N_c$ the number of zeros of $\widetilde Y$ in a unit cell $c$. Then $\Prob(N^{\rm end}>t)\le Ce^{-ct}$, $\Prob(\Theta>t)\le Cne^{-ct}$, $\E(\Theta+1)^2\le C\ell^2$ and $\E N_c^2\le C$.
\end{lemma}

\begin{proof}
(a) By \eqref{pre:pairing}, $\widetilde Y(nt)$ is a sine sum with distinct positive frequencies at most $\frac32n$, all in $\Z$ or all in $\Z+\frac12$, and independent nondegenerate Gaussian coefficients. Lemma~\ref{ga:zeros}(a) gives simplicity, and the representation $s(t)P(\cos t)$ in its proof bounds the numbers of zeros. A common zero of $Y$ and $\widetilde Y$ is a common zero of $Y$ and $Y^h$. Given $T$, $Y^h$ has finitely many zeros in $(0,\pi n)$ almost surely, and at each of them the independent variable $Y(z)$ is nondegenerate, so Fubini's theorem gives disjointness. The rest is as in Lemma~\ref{ga:zeros}.

(b) The curve $X+X^h+i\widetilde Y$ differs from the curve of degree $n+h$ only by a rotation and a change of scale:
\[
 X+X^h+i\widetilde Y=e^{ihx/(2n)}\sqrt{\frac{n+h}n}\,\bigl(X_{n+h}+iY_{n+h}\bigr)\Bigl(\frac{(n+h)x}n\Bigr).
\]
The factor $e^{ihx/(2n)}$ turns by $\pi h/2$ over $[0,\pi n]$, so the winding identity \eqref{ga:winding} at degree $n+h$ gives $\pi^{-1}\Delta_{[0,\pi n]}\arg(X+X^h+i\widetilde Y)=D_{n+h}+h/2$. Apply Lemma~\ref{pre:half-turns} to both curves on $I$ and on the two end intervals.

(c) Apply Lemma~\ref{ga:jensen} as in the proof of Lemma~\ref{ga:zeros}. The field $\widetilde Y$ is a sine field with at most $2n+1$ terms and frequencies at most $\frac32$, $\Var \widetilde Y(x_0)\ge\Var Y(x_0)\ge\frac14$ for $x_0\in I$, and $\Var \widetilde Y'(0)\ge\frac1{12}$. At $\pi n$, $\widetilde Y(\pi n-\cdot)$ has the law of $-\widetilde Y$ for even $n$ and of $n^{-1/2}\sum_{k\le n+h}\xi_k\cos(\omega_k\cdot)$ for odd $n$.
\end{proof}

The next lemma pairs the crossings of $Y$ with those of $\widetilde Y$, deterministically. Each pair contributes a difference of two signs of $X$ at nearby points, which is small in mean square.

\begin{lemma}[Canonical pairing]\label{ga:canonical}
Let $a<b$, let $y_0,y_1\in C^1[a,b]$ have finitely many zeros, all simple and in $(a,b)$, with disjoint zero sets, and let $x,u$ be continuous on $[a,b]$. Put $S_0=\sum_{y_0(z)=0}\sgn x(z)\sgn y_0'(z)$, $S_1=\sum_{y_1(z)=0}\sgn(x+u)(z)\sgn y_1'(z)$ and $\mathcal E=\{y_0y_1<0\}$. For $0<d_0\le1$ there are pairs $(s_i,t_i)$ in $[a,b]$ with $0<|s_i-t_i|\le d_0$, signs $\epsilon_i\in\{\pm1\}$ and shifts $\alpha_i\in\{0,u(s_i)\}$, $\beta_i\in\{0,u(t_i)\}$ such that
\[
 S_1-S_0=\sum_i\epsilon_i\bigl\{\sgn(x(s_i)+\alpha_i)-\sgn(x(t_i)+\beta_i)\bigr\}+\varrho,\qquad|\varrho|\le2,
\]
$\sum_i|s_i-t_i|=|\mathcal E|$, $\sum_i(|\alpha_i|+|\beta_i|)=\sum_{y_1(z)=0}|u(z)|$, and every unit cell contains at most $N+d_0^{-1}+2$ of the points $s_i$, where $N$ is the largest number of zeros of $y_0$ and $y_1$ together in a unit cell.
\end{lemma}

\begin{proof}
We write $S_1-S_0$ as a sum over the components of $\mathcal E$ and telescope within each component. The function $g=\ind\{y_1>0\}-\ind\{y_0>0\}$ vanishes off $\mathcal E$ and equals a constant $\epsilon_K\in\{\pm1\}$ on each component $K$ of $\mathcal E$. Each zero of $y_0$ or $y_1$ is an endpoint of exactly one component, since it is simple and the other function does not vanish there, and each endpoint of a component in $(a,b)$ is such a zero. For an endpoint $p$ put $\sigma(p)=\sgn(x+u)(p)$ if $y_1(p)=0$ and $\sigma(p)=\sgn x(p)$ otherwise. At a left endpoint $c$ of $K$, $g$ jumps by $\epsilon_K$, so $\epsilon_K=\sgn y_1'(c)$ if $y_1(c)=0$ and $\epsilon_K=-\sgn y_0'(c)$ if $y_0(c)=0$. At a right endpoint the jump is $-\epsilon_K$. In all four cases the endpoint contributes $\epsilon_K\sigma(c)$, respectively $-\epsilon_K\sigma(d)$, to $S_1-S_0$, so
\[
 S_1-S_0=\sum_K\epsilon_K\{\sigma(c_K)-\sigma(d_K)\}+\varrho,
\]
where $\sigma(a)=\sgn x(a)$, $\sigma(b)=\sgn x(b)$, and $\varrho$ collects the terms of the at most two components containing $a$ or $b$. Let $m_1<\dots<m_r$ be the points of $d_0\Z$ in the interior of $K$, put $\sigma(m_l)=\sgn x(m_l)$, and telescope $\sigma(c_K)-\sigma(d_K)$ over $c_K,m_1,\dots,m_r,d_K$. The pairs have positive length at most $d_0$ and total length $|K|$, each zero of $y_1$ occurs in exactly one pair with shift $u$, and all other points carry shift zero. The first points of the pairs are left endpoints, at most $N+1$ per cell, and points of $d_0\Z$, at most $d_0^{-1}+1$ per cell.
\end{proof}

\begin{lemma}[Paired samples]\label{ga:paired}
There are $d_0\in(0,1]$ and $c>0$ such that for $n\ge n_0$ the following holds. Let $(s_i,t_i)$ be deterministic pairs in $I_n$ with $0<|s_i-t_i|\le d_0$ and at most $H\ge1$ of the $s_i$ in each unit cell, and let $|\epsilon_i|\le1$ and $\alpha_i,\beta_i\in\R$. Put $Z_i=\epsilon_i\{\sgn(X_n(s_i)+\alpha_i)-\sgn(X_n(t_i)+\beta_i)\}$ and $\Delta'=\sum_i\min(1,|s_i-t_i|+|\alpha_i|+|\beta_i|)$. Then for $s>0$
\begin{equation}\label{ga:paired-tail}
 \Prob\Bigl(\Bigl|\sum_i(Z_i-\E Z_i)\Bigr|>s\Bigr)\le2\exp\Bigl(-\frac{cs^2}{H(\Delta'+s)}\Bigr),\qquad
 \E\Bigl(\sum_i(Z_i-\E Z_i)\Bigr)^2\le C(H\Delta'+H^2).
\end{equation}
\end{lemma}

\begin{proof}
Each $Z_i$ is a function of the two-dimensional block $(X_n(s_i),X_n(t_i))$. We show that these blocks satisfy the frame condition of Lemma~\ref{ga:holder} with constant $CH$, and that $\Var Z_i\le C\min(1,|s_i-t_i|+|\alpha_i|+|\beta_i|)$ because two nearby signs rarely differ. The block Bernstein bound then gives the tail.

Let $w(x)$ be the row of $X_n(x)$. As in the proof of Lemma~\ref{ga:fields}(b), the Gram matrix of $w(s),w'(s)$ is within $C/L+C/n$ of ${\rm diag}(\frac12,\frac1{24})$, and $\|w''\|\le\frac12$. The difference quotient $\mathrm q=(w(t)-w(s))/(t-s)$ satisfies $\|\mathrm q-w'(s)\|\le|t-s|/4$. For $L,n_0$ large and $d_0$ small, the Gram matrix of $w(s),\mathrm q$ has eigenvalues in $[\frac1{48},1]$ whenever $|t-s|\le d_0$. Gram--Schmidt applied to $w(s_i),\mathrm q_i$ gives orthonormal rows $e_{i,0},e_{i,1}$ with bounded coefficients, and we put $B_i\beta=(\langle e_{i,0},\beta\rangle,\langle e_{i,1},\beta\rangle)$. Then $X_n(s_i)$ and $X_n(t_i)=X_n(s_i)+(t_i-s_i)\langle\mathrm q_i,\xi\rangle$ are nonzero linear functions of $B_i\xi$, and $Z_i=\phi_i(B_i\xi)$ with $\phi_i$ bounded and continuous off two lines. By the mean value theorem $\|B_i\beta\|^2\le C\sup_{|y-s_i|\le1}(X[\beta]^2+X[\beta]'^2)(y)$, so the frame bound \eqref{ga:frame} gives $\sum_i\|B_i\beta\|^2\le CH\|\beta\|^2$.

If the signs of $X_n(s)+\alpha$ and $X_n(t)+\beta$ differ, then the signs of $X_n(s),X_n(t)$ differ, or $|X_n(s)|\le|\alpha|$, or $|X_n(t)|\le|\beta|$. By Sheppard's formula the first event has probability $\pi^{-1}\arccos\rho_+(s,t)\le\sqrt{(1-\rho_+)/2}\le|s-t|$, as $1-\rho_+=\frac12\|u(s)-u(t)\|^2$ and $\|u'\|\le1$, and the densities of the values are at most one. Hence $\Var Z_i\le C\min(1,|s_i-t_i|+|\alpha_i|+|\beta_i|)$, while $|Z_i-\E Z_i|\le4$. Corollary~\ref{ga:bernstein}(a) gives the tail bound, and integrating it, with $e^{-cs^2/(H(\Delta'+s))}\le e^{-cs^2/(2H\Delta')}+e^{-cs/(2H)}$, gives the second moment.
\end{proof}

We apply Lemma~\ref{ga:canonical} on $I$ with $y_0=Y$, $y_1=\widetilde Y$, $x=X$, $u=X^h$ and the $d_0$ of Lemma~\ref{ga:paired}. Then $S_1-S_0=Q$, the first points are at most $H=\Theta+d_0^{-1}+2$ per cell, and the cost of the pairing is
\begin{equation}\label{ga:cost}
 \Delta=|\{x\in I:Y(x)\widetilde Y(x)<0\}|+\sum_{z\in I:\,\widetilde Y(z)=0}|X^h(z)|.
\end{equation}

The pairing is measurable. Almost surely the zeros of $Y$ and $\widetilde Y$ in $I$ are simple and avoid $\partial I\cup d_0\Z$, so by the implicit function theorem they move continuously with $\theta$, and their number and order relative to $d_0\Z$ are locally constant. Hence the pairs, signs and shifts are Borel in $\theta$ on countably many Borel pieces.

\begin{lemma}[Integrated signs and occupation]\label{ga:mesh-signs}
Let $n\ge n_0$ and let $w,a$ be continuous deterministic functions on $I=I_n$.
\begin{enumerate}[label=\textup{(\alph*)}]
\item For real $\lambda$,
\[
 \E\exp\Bigl(\lambda\int_Iw\,\bigl\{\sgn(Y+a)-\E\sgn(Y+a)\bigr\}\dd x\Bigr)\le\exp\Bigl(C\lambda^2\int_Iw^2\dd x\Bigr).
\]
\item For $\tau>0$ and $u\ge1$, the occupation time $O_\tau=|\{x\in I:|Y(x)|\le\tau\}|$ satisfies $\Prob(O_\tau>C(n\min(\tau,1)+u))\le e^{-u}$.
\end{enumerate}
\end{lemma}

\begin{proof}
Both parts follow from Corollary~\ref{ga:bernstein} on a mesh. Let $x_j=L+j\eta$, $0\le j<N$, with $\eta=|I|/N\le1$. By the frame bound \eqref{ga:frame}, the correlation matrix of $(V_n(x_j))_j$ has norm at most $C/\eta$.

(a) Corollary~\ref{ga:bernstein}(b) with weights $\eta w(x_j)$ and shifts $a(x_j)/\sqrt{v_n^-(x_j)}$ bounds the moment generating function of the Riemann sum by $\exp(\frac C2\lambda^2\sum_j\eta w(x_j)^2)$. The zero set of $Y+a$ is almost surely null, so the Riemann sums converge, and bounded convergence gives (a).

(b) Here $p_j=\Prob(|Y(x_j)|\le\tau)\le\min(1,2\tau)$. Corollary~\ref{ga:bernstein}(a) with $\phi_j=\eta\ind\{|Y(x_j)|\le\tau\}$, frame constant $C/\eta$, $b=\eta$ and $V\le\eta\sum_j\eta p_j$ gives $\Prob(\sum_j\eta(\ind\{|Y(x_j)|\le\tau\}-p_j)>s)\le\exp(-s^2/(C(n\min(\tau,1)+s)))$. The level sets $\{Y=\pm\tau\}$ are almost surely null, so the Riemann sums converge to $O_\tau-\E O_\tau$, and Fatou's lemma passes the bound to the limit. Take $s=C(n\min(\tau,1)+u)$.
\end{proof}

Averaging the crossing sum $Q$ over the cosine field leaves a functional $J$ of $\theta$ alone. The average $G(T)$ of $J$ given $T$ is smooth in the added coefficients, and given $T$ the functional $J$ concentrates around $G(T)$.

\begin{lemma}[Conditional mean]\label{ga:conditional-mean}
Let $b_+=2\Phi(X^h/\sqrt{v_n^+})-1$ and $b_-=2\Phi(Y^h/\sqrt{v_n^-})-1$ on $I$. Almost surely
\begin{equation}\label{ga:J}
 J:=\E_xQ=\tfrac12\bigl[b_+\sgn \widetilde Y\bigr]_{\partial I}-\tfrac12\int_Ib_+'\sgn \widetilde Y\dd x,\quad
 G(T):=\E(J\mid T)=\tfrac12\bigl[b_+b_-\bigr]_{\partial I}-\tfrac12\int_Ib_+'b_-\dd x.
\end{equation}
Moreover $G(0)=0$, $\|\nabla G(T)\|\le C(1+\|T\|)$, $\Var J\le C(h+1)$, and $\Prob(|J-G(T)|>2+s\mid T)\le2e^{-s^2/(C\|T\|^2)}$ for $s>0$.
\end{lemma}

\begin{proof}
Since $X(z)$ is centered with variance $v_n^+(z)$ and independent of $\theta$, $\E_x\sgn(X+X^h)(z)=b_+(z)$ and $\E_x\sgn X(y)=0$, so $J=\sum_{\widetilde Y(z)=0}b_+(z)\sgn \widetilde Y'(z)$. At a simple zero $z$ the function $\sgn \widetilde Y$ jumps by $2\sgn \widetilde Y'(z)$, and integration by parts against $b_+$ gives the formula for $J$. Given $T$, $Y(x)$ is centered with variance $v_n^-(x)$, so $\E(\sgn \widetilde Y(x)\mid T)=b_-(x)$, and Fubini's theorem gives $G$. At $T=0$, $X^h=Y^h=0$ and $b_+=0$.

By Lemma~\ref{ga:fields}(a), $|b_+'|\le C(|X^h|+|(X^h)'|)$ and $|b_-'|\le C(|Y^h|+|(Y^h)'|)$. In a direction $\dot T$ the variations satisfy $|\dot b_+|\le C|X^h[\dot T]|$ and $|\dot b_-|\le C|Y^h[\dot T]|$, where $X^h[\dot T],Y^h[\dot T]$ are the tail fields with coefficients $\dot T$. Differentiating $G$ and integrating the term $\int\dot b_+'b_-$ by parts,
\[
 \dot G=\tfrac12\bigl[b_+\dot b_-\bigr]_{\partial I}+\tfrac12\int_I(b_-'\dot b_+-b_+'\dot b_-)\dd x .
\]
The energy bound \eqref{ga:tail-energy} for $T$ and $\dot T$ bounds the integral by $C\|T\|\|\dot T\|$, and $|Y^h[\dot T]|\le\|\dot T\|$. Hence $\|\nabla G(T)\|\le C(1+\|T\|)$, and the Poincar\'e inequality gives $\Var G(T)\le C(1+h)$.

Given $T$,
\[
 J-G=\tfrac12\bigl[b_+(\sgn \widetilde Y-b_-)\bigr]_{\partial I}-\tfrac12\int_Ib_+'\bigl(\sgn(Y+Y^h)-b_-\bigr)\dd x.
\]
The boundary term is at most $2$, and Lemma~\ref{ga:mesh-signs}(a) with $w=b_+'$, $a=Y^h$ and \eqref{ga:tail-energy} bound the conditional moment generating function of the integral by $\exp(C\lambda^2\|T\|^2)$. This gives the conditional tail and $\E((J-G)^2\mid T)\le C(1+\|T\|^2)$, so $\Var J=\E\Var(J\mid T)+\Var G\le C(h+1)$.
\end{proof}

\begin{proposition}[Uniform increment variance]\label{ga:uniform-increment}
For all $n\ge2$ and $1\le h\le n$, $\Var(D_{n+h}-D_n)\le C\sqrt{nh}\log(n+2)$.
\end{proposition}

\begin{proof}
We may take $n\ge n_0$. By the increment identity \eqref{ga:increment-identity} and Lemma~\ref{ga:common-carrier}(c), $\Var(D_{n+h}-D_n)\le2\Var Q+C$, and $\Var Q=\E\E_x(Q-J)^2+\Var J$. The pairing writes $Q-J$ as a centered paired sum plus a term of size at most four, with $\Delta'\le\Delta$, so by Lemma~\ref{ga:paired}
\[
 \E_x(Q-J)^2\le C\bigl((\Theta+1)\Delta+(\Theta+1)^2\bigr).
\]

The expected pairing cost is $\E\Delta\le C\sqrt{nh}$. Indeed, $Y(x)\widetilde Y(x)<0$ forces $|Y(x)|<|Y^h(x)|$, and as $Y(x)$ is independent of $Y^h(x)$ with density at most one, this has probability at most $2\E|Y^h(x)|\le3\sqrt{h/n}$. Integrating over $I$, of length at most $\pi n$, bounds the first term of $\E\Delta$:
\[
 \E\bigl|\{x\in I:Y(x)\widetilde Y(x)<0\}\bigr|\le3\pi\sqrt{nh}.
\]
For the second term, $\sum_z|X^h(z)|\le\sum_cN_cS_c$ over the at most $\pi n+1$ unit cells $c$ meeting $I$, with $S_c=\sup_{c\cap I}|X^h|$. Lemma~\ref{pre:sampling} gives $S_c^2\le2\int_{\bar c}((X^h)^2+(X^h)'^2)$, so $\E S_c^2\le7h/n$, and with $\E N_c^2\le C$ the Cauchy--Schwarz inequality gives $\E N_cS_c\le C\sqrt{h/n}$, and we sum over the cells.

Deterministically $\Theta\le3n$ and $\Delta\le\bar\Delta=\pi n+2\sqrt n\sum_k|T_k|$, with $\E\bar\Delta^2\le Cn^3$. For a large constant $A$, $\Prob(\Theta+1>A\ell)\le n^{-8}$, so
\[
 \E(\Theta+1)\Delta\le A\ell\,\E\Delta+4n\,(\E\bar\Delta^2)^{1/2}\Prob(\Theta+1>A\ell)^{1/2}\le C\sqrt{nh}\,\ell.
\]
With $\E(\Theta+1)^2\le C\ell^2$ and $\Var J\le C(h+1)$, $\Var(D_{n+h}-D_n)\le C(\sqrt{nh}\,\ell+\ell^2+h)\le C\sqrt{nh}\,\ell$.
\end{proof}

\begin{proposition}[Fine increments]\label{ga:fine}
Let $\varphi$ be as in \ref{fw:I3}. For every $p>0$ there is $C_p$ such that for all $n\ge2$ and $1\le h\le n$,
\begin{equation}\label{ga:fine-tail}
 \Prob\bigl(|D_{n+h}-D_n|>C_p\varphi(n,h)\bigr)\le C_pn^{-p}.
\end{equation}
\end{proposition}

\begin{proof}
Fix $p>0$, let $n\ge n_0$, and let $C_p$ denote constants depending on $p$. Since $\E D_k=0$ and $\E J=\E Q$, the increment identity \eqref{ga:increment-identity} gives the split \eqref{ga:fine-split}, and $\Prob(|E-\E E|>C_p\ell)\le n^{-p}$ by Lemma~\ref{ga:common-carrier}(c). We bound the other two terms at the scale $\varphi$.

\emph{Tail fields.} Every value of $X^h$ and $Y^h$ has variance at most $\frac94h/n$. A union bound over a grid of mesh $n^{-2}$ in $I$ gives $\max_{\rm grid}(|X^h|+|Y^h|)\le C_p\sqrt{h\ell/n}$ outside probability $n^{-p}$. Let $p_0=p+10$ and $R^2=8(h+p_0\ell)$. Since $\E e^{\|T\|^2/4}=2^{h/2}$, $\Prob(\|T\|>R)\le(n+2)^{-2p_0}$. On $\{\|T\|\le R\}$, $|(X^h)'|+|(Y^h)'|\le3R\sqrt{h/n}$, so the grid error is below $\sqrt{h/n}$. Hence $\sup_I(|X^h|+|Y^h|)\le C_p\sqrt{h\ell/n}$ outside probability $2n^{-p}$.

\emph{The paired sum.} Take $\tau=C_p\sqrt{h\ell/n}$. If $\sup_I|Y^h|\le\tau$, then $\{Y\widetilde Y<0\}\cap I\subset\{|Y|\le\tau\}$, and as $\widetilde Y$ has at most $2n$ zeros, $\Delta\le O_\tau+2n\sup_I|X^h|$. Lemma~\ref{ga:mesh-signs}(b) with $u=p\ell$ and Lemma~\ref{ga:common-carrier}(c) give a $\theta$-measurable event $\mathcal A$ with $\Prob(\mathcal A^c)\le C_pn^{-p}$ on which
\[
 H\le C_p\ell,\qquad\Delta\le C_p(n\tau+\ell)+2nC_p\sqrt{h\ell/n}\le C_p\sqrt{nh\ell}.
\]
On $\mathcal A$, the pairing and the paired tail \eqref{ga:paired-tail} give $\Prob_x(|Q-J|>4+s)\le2\exp(-cs^2/(H(\Delta+s)))$, and $s=C_p[(nh)^{1/4}\ell^{5/4}+\ell^2]$ makes the exponent at least $(p+1)\ell$. Integrating over $\theta\in\mathcal A$,
\[
 \Prob\bigl(|Q-J|>C_p\varphi(n,h)\bigr)\le C_pn^{-p}.
\]

\emph{The conditional mean.} The radial projection $T\mapsto T\min(1,R/\|T\|)$ is $1$-Lipschitz onto the ball of radius $R$, on which $G$ is $C(1+R)$-Lipschitz. So $G_R=G(T\min(1,R/\|T\|))$ is $C(1+R)$-Lipschitz on $\R^h$, and Lemma~\ref{ga:interpolation} gives $\Prob(|G_R-\E G_R|>s)\le2\exp(-cs^2/(1+R)^2)$. Here $s=C_p(\sqrt{h\ell}+\ell)$ makes the exponent at least $p\ell$, since $(1+R)^2\le C_p(h+\ell)$. From $G(0)=0$ and the gradient bound, $|G(T)|\le C(\|T\|+\|T\|^2)$, so $|\E G_R-\E G|\le C(1+h)(n+2)^{-p_0}\le1$. On $\{\|T\|\le R\}$, Lemma~\ref{ga:conditional-mean} gives $\Prob(|J-G|>2+C_pR\sqrt\ell\mid T)\le2(n+2)^{-p}$. As $\E J=\E G$ and $R\sqrt\ell\le C_p(\sqrt{h\ell}+\ell)\le C_p\varphi(n,h)$,
\[
 \Prob\bigl(|J-\E J|>C_p\varphi(n,h)\bigr)\le C_pn^{-p}.
\]
Insert the three bounds in \eqref{ga:fine-split}.
\end{proof}

\subsection{Proof of Theorem~\ref{thm:B}}\label{ga:inputs-sub}

\begin{proof}[Proof of Theorem~\ref{thm:B}]
Write $\sigma_n^2=\Var D_n$.

(i) $\E D_n=0$ by Lemma~\ref{pre:no-circle-roots}, and Proposition~\ref{ga:variance} gives $0<c_G<\infty$ and the variance rate.

(ii) The rate is Proposition~\ref{ga:covariance}(a) with $q_1=2$, and the cusp is Proposition~\ref{ga:covariance}(b). For the limit, if $M/N\to q>1$, then eventually $M\le(q+1)N$ and $M/N-1\ge(q-1)/2\ge N^{-1/1000}$, so \eqref{ga:cov-rate} with $q_1=q+1$ and the continuity of $a$ give $N^{-1}\Cov(D_N,D_M)\to a(q)$.

(iii) For $n\ge1$, $\sigma_n^2>0$, because the polynomials $z^n+2$ and $2z^n+1$ have $\nu_n(1)=0$ and $\nu_n(1)=n$, zeros depend continuously on the coefficients when the leading one is nonzero, and open sets of coefficient vectors have positive probability. For large $n$, $\sigma_n^2\ge c_Gn/2$, and Lemma~\ref{ga:normal} with $d=1$ gives
\[
 d_K\bigl(D_n/\sigma_n,\mathcal N(0,1)\bigr)\le C\bigl(n^{-1/4}\sqrt{\log(n+2)}+n^{-1/60}\bigr)\le Cn^{-1/60}.
\]

(iv) By scaling it suffices to take $\lambda\in[-1,1]^d$. Put $S_n=\sum_i\lambda_iD_{n_i}$, $n_i=\lfloor q_in\rfloor$. By (i) and (ii), $n^{-1}\Var S_n\to\lambda^{\rm T}\Sigma\lambda$. If this limit is zero, $S_n/\sqrt n\to0$ in $L^2$. Otherwise Lemma~\ref{ga:normal} with $N=n_d$ and $s^2\ge cn$ gives $S_n/\sqrt n\to\mathcal N(0,\lambda^{\rm T}\Sigma\lambda)$, and the Cram\'er--Wold device gives the joint limit. If $l/m\to\alpha\in(0,1)$, expanding the variance and using (ii) with $N=l$ and $M=m$ gives
\[
 m^{-1}\Var(D_m-D_l)\to c_G(1+\alpha)-2\alpha a(1/\alpha)\ge c_G(1-\alpha)>0,
\]
by $a\le c_G$. So $\Var(D_m-D_l)\ge c_\alpha m$ for large $m$, and Lemma~\ref{ga:normal} with $d=2$ and $N=m$ gives the Kolmogorov bound $C_\alpha m^{-1/60}$.

(v) The variance bound is Proposition~\ref{ga:uniform-increment} and the tail bound is \eqref{ga:fine-tail}. If $n/(\log n)^B\le h\le n$ and $n$ is large, then $d=h/n\ge n^{-1/1000}$ and $d^{-100}n^{-1/100}/\sqrt d\le(\log n)^{101B}n^{-1/100}\to0$, so the increment variance expansion \eqref{ga:increment-variance} and $\Lambda(d)\asymp\sqrt d$ give $\Var(D_{n+h}-D_n)\asymp n\sqrt d=\sqrt{nh}$. For the coarse normal approximation let $\log^{-6}(n+2)\le d\le1$ and $n$ be large. With $B=7$, $s^2=\Var(D_{n+h}-D_n)\ge cn\sqrt d\ge cn\log^{-3}(n+2)$, and Lemma~\ref{ga:normal} with $d=2$ and $N=n+h\le2n$ gives
\[
 d_K\le C\bigl(n^{-1/4}\log^{7/2}(n+2)+n^{-1/60}\log(n+2)\bigr)\le Cn^{-1/120}.
\]

(vi) We verify the inputs for $Z_n=D_n$ and $c=c_G$. Input~\ref{fw:I1} follows from (i), (iii) and Lemma~\ref{fw:normal-inputs}(i). For~\ref{fw:I2} take $\delta_0=\frac12$. On its range of gaps, (v) gives $\Var(D_{n+h}-D_n)\le C\sqrt{nh}$ and the Kolmogorov bound $Cn^{-1/120}$ for large $n$, and Lemma~\ref{fw:normal-inputs}(ii) applies. Input~\ref{fw:I3} is \eqref{ga:fine-tail}. Finally $\Prob(\xi_0=0)=0$ and, by Lemma~\ref{pre:no-circle-roots}, almost surely no $f_n$ has a zero on the unit circle, so Lemma~\ref{fw:block-proxies} gives~\ref{fw:I4}, with proxies that have exactly the required laws. Theorem~\ref{thm:A} then gives the two limits.
\end{proof}

\subsection{Radial profiles}\label{ga:radial-sub}
For Theorem~\ref{thm:G} we average the argument principle at radius $r$, which writes $\nu_n(r)$ as an angular mean of $\Re(r\partial_rf_n/f_n)$, and compute the mean of this ratio exactly at each angle. In this subsection the coefficients are independent standard real Gaussians, or independent standard circular complex Gaussians $\xi_j=(g_j+ig_j')/\sqrt2$. We use $K_n(r)=\sum_kr^{2k}$, $m_n(r)$ and $p_n(r,\theta)$ from Theorem~\ref{thm:G}. By \eqref{rc:Phi-props} the profile $\Phi_{\rm rad}$ is $\frac12$-Lipschitz.

\begin{proof}[Proof of Theorem~\ref{thm:G}\textup{(i)}]
Both laws have densities, so by Lemma~\ref{pre:no-circle-roots} $f_n$ almost surely has no zero on $\{|z|=r\}$. Put $Z=f_n(re^{i\theta})$ and $\dot Z=r\partial_rZ=\sum_kk\xi_kr^ke^{ik\theta}$. By the argument principle \eqref{pre:balanced},
\begin{equation}\label{ga:radial-winding}
 \nu_n(r)=\frac1{2\pi}\int_0^{2\pi}\Re\frac{\dot Z}Z\,\dd\theta .
\end{equation}
\emph{Circular case.} The pair $(Z,\dot Z)$ is circular complex Gaussian, so $\dot Z=\mu Z+W$ with $W$ centered and independent of $Z$, and
\[
 \E\Re\frac{\dot Z}Z=\mu=\frac{\E\dot Z\bar Z}{\E|Z|^2}=\frac{\sum_kkr^{2k}}{\sum_kr^{2k}}=m_n(r),
\]
and since $\E|Z|^{-1}<\infty$ bounds $\E|\dot Z/Z|$ uniformly in $\theta$, Fubini's theorem gives $\E\nu_n(r)=m_n(r)$.

\emph{Real case.} For $\theta\notin\pi\Z$ the real covariance of $Z$ is nondegenerate, since the rows $(r^k\cos k\theta)_k$ and $(r^k\sin k\theta)_k$ are independent. Lemma~\ref{ga:exchange} in one variable, with the remark after it, gives $\E\Re(\dot Z/Z)=\partial_u\E\log|f_n(e^{u+i\theta})|$ at $u=\log r$. The eigenvalues $\lambda_1\ge\lambda_2>0$ of the covariance of $Z$ satisfy $\lambda_1+\lambda_2=\E|Z|^2=K_n(r)$ and $\lambda_1-\lambda_2=|\E Z^2|=K_n(r)|p_n(r,\theta)|$, so $|Z|$ has the law of $|\sqrt{\lambda_1}G_1+i\sqrt{\lambda_2}G_2|$. In polar coordinates $G_1+iG_2=\varrho e^{it}$, $\sqrt{\lambda_1}\cos t+i\sqrt{\lambda_2}\sin t=\mathcal Ae^{it}+\mathcal Be^{-it}$ with $\mathcal A=\frac12(\sqrt{\lambda_1}+\sqrt{\lambda_2})>|\mathcal B|$, and $\log|1+(\mathcal B/\mathcal A)e^{-2it}|$ has mean zero in $t$. Since $(\sqrt{\lambda_1}+\sqrt{\lambda_2})^2=K_n(r)(1+\sqrt{1-|p_n|^2})$,
\[
 \E\log|Z|=\E\log\varrho-\log2+\tfrac12\log K_n(r)+\tfrac12\log\bigl(1+\sqrt{1-|p_n(r,\theta)|^2}\bigr),
\]
and $\partial_u\frac12\log K_n(e^u)=m_n(e^u)$. So $\E\Re(\dot Z/Z)$ is $m_n(r)$ plus half the integrand of Theorem~\ref{thm:G}(i).

It remains to show $\E|\dot Z/Z|\le C_{n,r}(1+|\log|\sin\theta||)$, which justifies Fubini's theorem in \eqref{ga:radial-winding} and gives absolute convergence. Away from $\pi\Z$ this follows from nondegeneracy, as in the circular case. Near $\theta=0$ write $Z=Z_1+i\sin\theta\,Z_2$ and $\dot Z=\dot Z_1+i\sin\theta\,\dot Z_2$, with $Z_2=\sum_k\xi_kr^k\sin(k\theta)/\sin\theta$ and $\dot Z_2$ extended continuously. At $\theta=0$ the rows of $Z_1,Z_2$ are $(r^k)_k$ and $(kr^k)_k$, which are independent, so near $0$ the covariance of $(Z_1,Z_2)$ is uniformly positive definite and the regression of $(\dot Z_1,\dot Z_2)$ on $(Z_1,Z_2)$ has bounded coefficients. With $s=|\sin\theta|$,
\[
 \E\Bigl|\frac{\dot Z}Z\Bigr|\le C\iint_{\R^2}\frac{1+|a|+|b|}{\sqrt{a^2+s^2b^2}}e^{-c(a^2+b^2)}\dd a\dd b\le C\int_0^{2\pi}\frac{\dd t}{\sqrt{\cos^2t+s^2\sin^2t}}\le C(1+|\log s|).
\]
The point $\theta=\pi$ is the same with $(-r)^k$ in place of $r^k$.
\end{proof}

\begin{proof}[Proof of Theorem~\ref{thm:G}\textup{(ii)}]
We first show that at exponential radii, for every $A$ and all $n\ge1$,
\begin{equation}\label{ga:exp-profile}
 \sup_{|a|\le A}\bigl|\E\nu_n(e^{a/n})-n\Phi_{\rm rad}(a)\bigr|\le C_A .
\end{equation}
Let $|a|\le A$ and $w_k=e^{2ak/n}$. Riemann sums for $e^{2av}$ and $ve^{2av}$, and $n^{-1}\sum_kw_k\ge e^{-2A}$, give
\[
 m_n(e^{a/n})=n\,\frac{n^{-1}\sum_{k=0}^n(k/n)w_k}{n^{-1}\sum_{k=0}^nw_k}=n\Phi_{\rm rad}(a)+O_A(1),
\]
which is the circular case. In the real case put $q=|p_n(e^{a/n},\theta)|^2$ and $t=n\dist(\theta,\pi\Z)$. Since $r\partial_r=n\partial_a$, by part (i) it suffices to show
\begin{equation}\label{ga:real-correction}
 \int_0^{2\pi}\bigl|\partial_a\log(1+\sqrt{1-q})\bigr|\dd\theta\le\frac{C_A}n .
\end{equation}
The weights $w_k$ and $(2k/n)w_k$ have bounded supremum and total variation, so \eqref{rc:abel} gives $|p_n|+|\partial_ap_n|\le C_A/(1+t)$. Near $\theta_0\in\{0,\pi,2\pi\}$ write $\theta=\theta_0\pm t/n$, and let $T,T'$ be independent with $\Prob(T=k/n)=w_k/\sum_jw_j$. Then $q=\E\cos(2t(T-T'))$ and $\partial_aq=2\E[(T+T'-2\E T)(\cos(2t(T-T'))-1)]$, so $|\partial_aq|\le8t^2$. The weights are within the factor $e^{4A}$ of uniform, so $\E(T-T')^2\ge e^{-4A}/6$, and $1-q\ge c_At^2$ for $t\le t_0$.

Choose $L_A$ with $|p_n|\le\frac12$ for $t\ge L_A$. For $t_0\le t\le L_A$, $p_n$ converges uniformly to $\int_0^1e^{2av+2itv}\dd v/\int_0^1e^{2av}\dd v$, whose modulus is less than one, so $1-q\ge c_A$ for large $n$. Since $|\partial_a\log(1+\sqrt{1-q})|\le|\partial_aq|/(2\sqrt{1-q})$ and $|\partial_aq|\le2|p_n||\partial_ap_n|$, the integrand is at most $C_At$, $C_A$ and $C_At^{-2}$ on the three ranges, and $\dd\theta=\dd t/n$ gives \eqref{ga:real-correction} for large $n$.

For linear radii let $a_n(x)=n\log(1+x/n)$. For $n\ge2A$ and $|x|\le A$, $|a_n(x)|\le2A$ and $|a_n(x)-x|\le A^2/n$, so $n|\Phi_{\rm rad}(a_n(x))-\Phi_{\rm rad}(x)|\le A^2/2$.
\end{proof}

\section{Littlewood polynomials}\label{sec:signs}\label{sg:section}

For Gaussian coefficients the two coordinates of the curve
$t\mapsto e^{-int/2}f_n(e^{it})$ are independent, and $D_n$ is a sum of
signs of one field at the zeros of the other. For random signs the same
structure appears after conditioning on the coefficient pairs
$(\epsilon_j,\epsilon_{n-j})$. Write
\[
 e^{-int/2}f_n(e^{it})=U_n(t)+iV_n(t)
\]
with the carrier coordinates of Lemma~\ref{pre:carrier}. An equal pair
contributes only to $U_n$, and an opposite pair only to $V_n$. Conditioning
on the type of every pair and on the signs of the opposite pairs therefore
fixes $V_n$ and its zeros, while $U_n$ remains a sum of independent signs.
By Proposition~\ref{pre:crossing}, $D_n$ is, up to
boundary roots, a signed sum of $\sgn U_n$ over the zeros of $V_n$ in
$(0,\pi)$. We show that these zeros are well spread, that $U_n$ is rarely
small at them, and that the signed sum concentrates, which gives
Theorem~\ref{thm:C}. The same conditioning, with a divisibility argument
for roots of unity, shows that roots on the unit circle are rare.

Throughout this section $(\epsilon_j)_{j\ge0}$ are independent random
signs, $f_n(z)=\sum_{j=0}^n\epsilon_jz^j$ and $\ell_n=\log(n+2)$. Each
almost-sure statement below follows from a summable tail bound at fixed
degree by the first Borel--Cantelli lemma, so at a fixed degree we may
realize the signs in any convenient way. Pair $j$ with $n-j$ for
$0\le j<n/2$. There are $p_n=\lfloor(n+1)/2\rfloor$ pairs, the pair $j$
has frequency $\lambda_j=n/2-j$, and for even $n$ there is a middle index
$n/2$.

\begin{lemma}[Pair activity]\label{sg:activity}
Let $\chi_j$ \textup($0\le j<p_n$\textup) be Bernoulli variables with parameter
$\frac12$, and let $\sigma_j,\tau_j$ \textup($0\le j<p_n$\textup) and $\varsigma$
be random signs, all independent. Put
\[
 \epsilon_j=\chi_j\sigma_j-(1-\chi_j)\tau_j,\qquad
 \epsilon_{n-j}=\chi_j\sigma_j+(1-\chi_j)\tau_j,\qquad
 \epsilon_{n/2}=\varsigma\ \ (n\ \text{even}).
\]
Then $\epsilon_0,\dots,\epsilon_n$ are independent random signs. With
$f_n^*(z)=z^nf_n(1/z)$,
\begin{align*}
 f_n+f_n^*&=2\sum_{j<p_n}\chi_j\sigma_j\,(z^j+z^{n-j})+2\varsigma z^{n/2}\ind_{n\ \rm even},&
 f_n-f_n^*&=2\sum_{j<p_n}(1-\chi_j)\tau_j\,(z^{n-j}-z^j),\\
 U_n(t)&=\varsigma\ind_{n\ \rm even}+2\sum_{j<p_n}\chi_j\sigma_j\cos(\lambda_jt),&
 V_n(t)&=2\sum_{j<p_n}(1-\chi_j)\tau_j\sin(\lambda_jt).
\end{align*}
\end{lemma}

\begin{proof}
If $\chi_j=1$ then $(\epsilon_j,\epsilon_{n-j})=(\sigma_j,\sigma_j)$, and if
$\chi_j=0$ then $(\epsilon_j,\epsilon_{n-j})=(-\tau_j,\tau_j)$. Each of the four
sign pairs has probability $\frac14$, and different pairs and the middle sign
use disjoint sets of independent variables. Since
$\epsilon_j+\epsilon_{n-j}=2\chi_j\sigma_j$ and
$\epsilon_{n-j}-\epsilon_j=2(1-\chi_j)\tau_j$, the formulas follow from
$f_n^*=\sum_j\epsilon_{n-j}z^j$ and the pairing formula \eqref{pre:pairing}.
\end{proof}

We call the pair $j$ \emph{active} if $\chi_j=1$. Conditionally on $(\chi,\tau)$,
the imaginary coordinate $V_n$ is fixed and $U_n$ is a linear form in the
independent signs $\sigma_j$ of the active pairs and $\varsigma$.
Conditionally on $(\chi,\sigma,\varsigma)$ the roles are exchanged. The
representation is special to the unit circle. For $r\ne1$ the imaginary
coordinate of $e^{-int/2}f_n(re^{it})$ determines every pair
$(\epsilon_j,\epsilon_{n-j})$, because $r^{n-j}\epsilon_{n-j}-r^j\epsilon_j$
takes four distinct values when $r^{n-j}\ne r^j$.

\subsection{Boundary roots}\label{sg:boundary-section}
Let $B_n$ be the number of roots of $f_n$ on the unit circle, with
multiplicity, and write $B_n=B_n^{\rm nc}+B_n^{\rm cyc}$, where
$B_n^{\rm cyc}$ counts the roots that are roots of unity. Roots of unity
are controlled by divisibility. The other roots are rare because a sum of
signs with distinct weights rarely takes a given value.

\begin{lemma}[Distinct complex weights]\label{sg:ss}
There is an absolute constant $C_{\rm SS}$ such that for $m\ge1$, distinct
$w_1,\dots,w_m\in\C$ and independent random signs $\sigma_1,\dots,\sigma_m$,
\[
 \sup_{x\in\C}\Prob\Bigl\{\sum_{j=1}^m\sigma_jw_j=x\Bigr\}\le C_{\rm SS}m^{-3/2}.
\]
\end{lemma}

\begin{proof}
For distinct integers this is the Sárközy--Szemerédi theorem
\cite[Theorem~1.2]{NguyenVu2011}. In general fix $x$, and let $V\subset\C$
be the $\mathbb Q$-linear span of $w_1,\dots,w_m,x$. A vector space over an
infinite field is not a finite union of proper subspaces. Applied to the
dual of $V$ and its proper subspaces $\{\phi:\phi(w_j-w_k)=0\}$, $j\ne k$,
this gives a $\mathbb Q$-linear $\phi:V\to\mathbb Q$ with distinct values at
$w_1,\dots,w_m$. Choose an
integer $N\ge1$ with $N\phi(w_j)\in\Z$ for all $j$ and $N\phi(x)\in\Z$. Then
$\sum_j\sigma_jw_j=x$ implies $\sum_j\sigma_jN\phi(w_j)=N\phi(x)$, and the
integer case proves the bound.
\end{proof}

\begin{proposition}[Boundary roots of Littlewood polynomials]\label{sg:boundary}
There is an absolute constant $C$ with the following properties.
\begin{enumerate}[label=\textup{(\alph*)}]
 \item For every $n\ge0$,
 \begin{equation}\label{sg:boundary-moments}
 \E B_n^{\rm nc}\le C(n+1)^{-1/2},\qquad
 \E(B_n^{\rm cyc})^2\le C,\qquad
 \E B_n\le C,\qquad
 \Var B_n\le C\sqrt{n+1}.
 \end{equation}
 \item For every $A>0$ there is $C_A$ such that
 $\Prob\{B_n^{\rm cyc}>C_A\ell_n^5\}\le C_A(n+2)^{-A}$ for every $n\ge0$.
 \item Almost surely, simultaneously for every $\eta>0$,
 \begin{equation}\label{sg:all-degree}
  B_n=O_\eta\bigl(\sqrt n\,\ell_n^{1+\eta}\bigr)\qquad(n\to\infty).
 \end{equation}
\end{enumerate}
\end{proposition}

\begin{proof}
\emph{Noncyclotomic roots.} Let $n\ge1$, since $f_0$ has no roots, and
realize the signs as in Lemma~\ref{sg:activity}. Let $p=p_n\ge n/2$ and let
$N_1=\sum_j\chi_j$ and $N_0=p-N_1$ be the numbers of active and inactive pairs.
If $|\zeta|=1$ and $\zeta$ is a root of $f_n$ of multiplicity $m$, then so is
$\bar\zeta=1/\zeta$, because the coefficients are real, and therefore
$\zeta$ is a root of $f_n^*$ of multiplicity $m$. Thus $(z-\zeta)^m$ divides
both $f_n+f_n^*$ and $f_n-f_n^*$. Let $g$ be one of these two polynomials,
assume $g\neq0$, and let $Z(g)$ be the set of roots of $g$ on the unit circle
that are not roots of unity, with multiplicities $m_\zeta(g)$. Then
\begin{equation}\label{sg:dominate}
 B_n^{\rm nc}\le\sum_{\zeta\in Z(g)}m_\zeta(g)\,\ind\{f_n(\zeta)=0\},
 \qquad \sum_{\zeta\in Z(g)}m_\zeta(g)\le\deg g\le n .
\end{equation}
For $\zeta\in Z(g)$ the numbers $\zeta^j+\zeta^{n-j}$, $0\le j<n/2$, are
distinct, and so are the numbers $\zeta^{n-j}-\zeta^j$. For
$j<k<n/2$ the equality $\zeta^j+\zeta^{n-j}=\zeta^k+\zeta^{n-k}$ reads
$\zeta^j(1-\zeta^{k-j})=\zeta^{n-k}(1-\zeta^{k-j})$ and gives
$\zeta^{n-j-k}=1$ with $n-j-k>0$. The second equality gives
$\zeta^{n-j-k}=-1$ in the same way. Both contradict the fact that $\zeta$ is
not a root of unity.

Suppose first that $N_0\ge1$ and $N_1\ge p/2$. Then $g=f_n-f_n^*\neq0$ is a
function of $(\chi,\tau)$. At $\zeta\in Z(g)$, $f_n(\zeta)=0$ forces
$(f_n+f_n^*)(\zeta)=0$, that is,
\[
 \sum_{\chi_j=1}\sigma_j(\zeta^j+\zeta^{n-j})=-\varsigma\zeta^{n/2}\ind_{n\ \rm even}.
\]
Conditionally on $(\chi,\tau,\varsigma)$ the $N_1$ signs $\sigma_j$ with
$\chi_j=1$ are independent and uniform, and their weights are distinct, so
Lemma~\ref{sg:ss} bounds the conditional probability of this equation by
$C_{\rm SS}(p/2)^{-3/2}$. By \eqref{sg:dominate},
$\E[B_n^{\rm nc}\mid\chi,\tau,\varsigma]\le C_{\rm SS}\,n(p/2)^{-3/2}$ on this
event. If instead $N_1\ge1$ and $N_0>p/2$, use $g=f_n+f_n^*\neq0$, a function
of $(\chi,\sigma,\varsigma)$, and the equation
$\sum_{\chi_j=0}\tau_j(\zeta^{n-j}-\zeta^j)=0$ in the conditionally independent
signs $\tau_j$. The bound is the same. The two cases cover the event
$\{N_0\ge1,\,N_1\ge1\}$, whose complement has probability $2^{1-p}$, and
always $B_n^{\rm nc}\le n$. Therefore
\[
 \E B_n^{\rm nc}\le C_{\rm SS}\,n\Bigl(\frac p2\Bigr)^{-3/2}+n2^{1-p}\le C(n+1)^{-1/2}.
\]

\emph{Cyclotomic roots.} Let $\Phi_q$ be the $q$th cyclotomic polynomial.
It is monic with integer coefficients, irreducible over $\mathbb Q$, of
degree $\varphi(q)$, and its roots are the primitive $q$th roots of unity.
Let $M_q$ be the largest $r$ with $\Phi_q^r\mid f_n$. Since $\Phi_q$ is the
minimal polynomial of each of its roots, every primitive $q$th root of unity
is a root of $f_n$ of multiplicity exactly $M_q$, and
$B_n^{\rm cyc}=\sum_q\varphi(q)M_q$. If $r\varphi(q)\le n$, fix the signs
$\epsilon_j$ with $j\ge r\varphi(q)$. Two choices of the remaining signs with
$\Phi_q^r\mid f_n$ differ by a polynomial of degree less than $r\varphi(q)$
that is divisible by $\Phi_q^r$, so they coincide. If $r\varphi(q)>n$ the
event $M_q\ge r$ is empty. Using
$\sum_{r\ge1}(2r-1)x^r=x(1+x)/(1-x)^2\le6x$ for $0\le x\le\frac12$, we get
\begin{equation}\label{sg:cyc-tail}
 \Prob\{M_q\ge r\}\le2^{-r\varphi(q)},\qquad
 \E M_q^2=\sum_{r\ge1}(2r-1)\Prob\{M_q\ge r\}\le6\cdot2^{-\varphi(q)} .
\end{equation}
Writing $q=\prod p^a$, we have
$\varphi(q)^2/q=\prod p^{a-2}(p-1)^2\ge\frac12$, since each odd prime
contributes a factor at least $1$ and the prime $2$ a factor at least
$\frac12$. Hence $\varphi(q)\ge\sqrt{q/2}$, and Minkowski's inequality and
$\varphi(q)\le q$ give
\[
 \|B_n^{\rm cyc}\|_2\le\sum_q\varphi(q)\|M_q\|_2\le\sqrt6\sum_{q\ge1}q\,2^{-\sqrt{q/2}/2}<\infty .
\]

For the tail let $r_n=\lceil(A+2)\log_2(n+2)\rceil$ and $q_*=2r_n^2$. If
$q>q_*$, then $\varphi(q)\ge\sqrt{q/2}>r_n$. Also $M_q\ge1$ requires
$\varphi(q)\le n$, hence $q\le2n^2$. By \eqref{sg:cyc-tail},
\[
 \Prob\{\exists q>q_*:M_q\ge1\}\le2n^2(n+2)^{-A-2},\qquad
 \Prob\{\exists q\le q_*:M_q\ge r_n\}\le q_*2^{-r_n}.
\]
Off these two events,
$B_n^{\rm cyc}\le r_n\sum_{q\le q_*}\varphi(q)\le r_nq_*^2\le C_A\ell_n^5$,
and their probabilities sum to at most $C_A(n+2)^{-A}$. This proves (b).

\emph{Moments and the almost-sure bound.} Since $B_n^{\rm nc}\le n$,
$\E(B_n^{\rm nc})^2\le n\E B_n^{\rm nc}\le\sqrt{n+1}\,C$. Together with the
cyclotomic second moment and the Cauchy--Schwarz inequality this gives
$\E B_n\le C$ and $\Var B_n\le\E B_n^2\le C\sqrt{n+1}$, which completes
\eqref{sg:boundary-moments}. For (c), Markov's inequality gives
\[
 \sum_{n\ge2}\Prob\bigl\{B_n^{\rm nc}>\sqrt n\,\ell_n^{1+\eta}\bigr\}
 \le\sum_{n\ge2}\frac{C}{n\,\ell_n^{1+\eta}}<\infty ,
\]
and with (b) for $A=2$ this proves \eqref{sg:all-degree} for $\eta=1/k$.
The intersection of these events over $k\ge1$ covers every $\eta>0$.
\end{proof}

\subsection{Small balls for sign sums}\label{sg:smallball-section}
A sum of independent signs whose vector weights are bounded and spread in
all $r$ directions falls in a fixed small ball with probability
$O(n^{-r/2})$, as a Gaussian vector with covariance of order $n$ would.
This anticoncentration enters twice, for the zeros of $V_n$ and for the
values of $U_n$ at them.

\begin{lemma}[Small balls]\label{sg:smallball}
Fix $r\ge1$ and $C_0,K,\kappa>0$. There are $\rho,C>0$, depending only on
these numbers, with the following property. Let $n\ge1$, $m\le C_0n$, and
$a_1,\dots,a_m\in\R^r$ with
\[
 |a_j|\le K,\qquad \sum_{j=1}^m\langle a_j,u\rangle^2\ge\kappa n|u|^2\quad(u\in\R^r).
\]
Then, for independent random signs $\sigma_1,\dots,\sigma_m$,
\[
 \sup_{z\in\R^r}\Prob\Bigl\{\Bigl|\sum_{j=1}^m\sigma_ja_j-z\Bigr|\le\rho\Bigr\}\le Cn^{-r/2}.
\]
In particular, if $r=1$, then for all $u\in\R$ and $s\ge0$,
\begin{equation}\label{sg:scalar}
 \Prob\Bigl\{\Bigl|\sum_j\sigma_ja_j-u\Bigr|\le s\Bigr\}\le C\Bigl(\frac s\rho+2\Bigr)n^{-1/2}.
\end{equation}
\end{lemma}

\begin{proof}
Put $d=\sqrt{\kappa/(2C_0)}$ and let $u$ be a unit vector. The indices with
$|\langle a_j,u\rangle|<d$ contribute at most $md^2\le\kappa n/2$ to
$\sum_j\langle a_j,u\rangle^2\ge\kappa n$, and every other index contributes
at most $K^2$. So at least $k=\lceil\kappa n/(2K^2)\rceil$ indices satisfy
$|\langle a_j,u\rangle|\ge d$. In other words, every hyperplane through the
origin lies at distance at least $1$ from at least $k$ of the vectors
$a_j/d$. For such vectors Tao and Vu \cite[Theorem~1.4]{TaoVu2012} show that the sum
$\sum_j\sigma_ja_j/d$ lies in a given ball of radius $1$ with probability at
most $C_rk^{-r/2}$, where $C_r$ depends only on $r$. Take $\rho=d/2$. The
closed ball of radius $\rho$ about $z$ lies in the open ball of radius $d$,
so the probability in question is at most
$C_rk^{-r/2}\le C_r(2K^2/\kappa)^{r/2}n^{-r/2}$. For \eqref{sg:scalar}, cover
the interval of radius $s$ about $u$ by at most $s/\rho+2$ intervals of
radius $\rho$.
\end{proof}

\subsection{Zeros of the imaginary coordinate}\label{sg:occupancy-section}
Outside an event of probability $O(n^{-3})$, the imaginary coordinate has
only $O(\log n)$ zeros near any point, so its zeros are well spread. In the
scaled variable $x=nt$ put
\[
 Y_n(x)=n^{-1/2}V_n(x/n)=n^{-1/2}\sum_{j=0}^n\epsilon_j\sin(\omega_jx),
 \qquad\omega_j=\frac{j-n/2}n ,
\]
the sine field $S_\epsilon$ of Lemma~\ref{pre:sampling}, which is the field
$Y_n$ of Section~\ref{sec:gaussian} with sign coefficients. It is entire and
real on $\R$, and $|\omega_j|\le\frac12$ gives
\begin{equation}\label{sg:growth}
 |Y_n(z)|\le(n+1)n^{-1/2}e^{|\Im z|/2}\le2\sqrt n\,e^{|\Im z|/2}.
\end{equation}
We fix $L_0=32\pi$ and $L_1=L_0+1$, and we put
$\bar\ell_n=2\ell_n/\log2=2\log_2(n+2)$, which will bound the number of
zeros of $Y_n$ near any point. The \emph{bulk cells} are the intervals
$[k,k+1]$, $k\in\Z$, with $L_0\le k+\frac12\le\pi n-L_0$. Correlations of
$Y_n$ are governed by the kernel
\[
 R_n(u)=\frac1n\sum_{j=0}^n\cos(\omega_ju)=\frac{\sin((n+1)u/2n)}{n\sin(u/2n)} .
\]
It satisfies $R_n(0)=1+\frac1n$, $R_n(2\pi k)=(-1)^k/n$ for
$0<|k|<n$, and $R_n(2\pi n+u)=(-1)^nR_n(u)$, because
$2\pi n\omega_j=2\pi j-\pi n$. Since $|\sin\theta|\ge\frac2\pi\min(\theta,\pi-\theta)$
on $[0,\pi]$, it is small in the bulk:
\begin{equation}\label{sg:kernel-bulk}
 |R_n(u)|\le\frac\pi{L_0}=\frac1{32}\qquad(L_0\le u\le2\pi n-L_0).
\end{equation}

\begin{lemma}[Occupancy]\label{sg:occupancy}
Let $\mathcal O_n$ be the event that $Y_n\not\equiv0$ and that each closed
disk of radius $1$ about the midpoint of a bulk cell, and each of the
closed disks of radius $L_1$ about $0$ and about $\pi n$, contains at most
$\bar\ell_n$ zeros of $Y_n$, counted with multiplicity. There are absolute
constants $C,n_0$ such that $\Prob(\mathcal O_n^c)\le Cn^{-3}$ for
$n\ge n_0$. On $\mathcal O_n$, the real zeros in each bulk cell have total
order at most $\bar\ell_n$, and the real zeros in
$[0,L_1]\cup[\pi n-L_1,\pi n]$ have total order $N_{\rm end}\le2\bar\ell_n$.
\end{lemma}

\begin{proof}
The small-ball bound of Lemma~\ref{sg:smallball} in dimension $r$ costs
$n^{-r/2}$, and the union bound below runs over about $\pi n$ disks, so
eight evaluations of $Y_n$ give the failure probability $n^{-3}$. Put
$d_i=2\pi i$ for $1\le i\le8$ and $R_*=L_1+16\pi$.

\emph{Eight evaluations.} For real $x$ let $u_j(x)\in\R^8$ have coordinates
$\sin(\omega_j(x+d_i))$, $1\le i\le8$. Then
$\sum_j\epsilon_ju_j(x)=(\sqrt nY_n(x+d_i))_{i\le8}$ and $|u_j(x)|\le\sqrt8$.
By the product formula for sines, the Gram matrix is
\[
 G_{il}(x)=\frac1n\sum_{j=0}^n u_j(x)_iu_j(x)_l
 =\frac12\bigl(R_n(d_i-d_l)-R_n(2x+d_i+d_l)\bigr).
\]
If $L_0\le x\le\pi n-L_0$, then $L_0\le2x+d_i+d_l\le2\pi n-L_0$. If $x=0$ or
$x=\pi n$, then $R_n(2x+d_i+d_l)=\pm R_n(2\pi(i+l))=\pm1/n$. In
either case, for $n\ge32$, $|G_{il}(x)-\frac12\delta_{il}|\le\frac1{32}$, so
\[
 \Bigl|u^{\rm T}\bigl(G(x)-\tfrac12I\bigr)u\Bigr|\le\frac1{32}\Bigl(\sum_i|u_i|\Bigr)^2\le\frac14|u|^2,
 \qquad\text{that is,}\qquad
 \sum_j\langle u_j(x),u\rangle^2\ge\frac n4|u|^2 .
\]
Lemma~\ref{sg:smallball} with $r=8$, $C_0=2$, $K=\sqrt8$, $\kappa=\frac14$
gives absolute constants $\rho,C$ with
\begin{equation}\label{sg:eight}
 \Prob\Bigl\{\Bigl|\sum_j\epsilon_ju_j(x)\Bigr|\le\rho\Bigr\}\le Cn^{-4}
 \qquad(x\in\{0,\pi n\}\cup[L_0,\pi n-L_0]).
\end{equation}

\emph{Jensen.} Let $x$ be a bulk midpoint, or $0$, or $\pi n$, and let
$\Delta$ be the corresponding disk of radius $1$ or $L_1$. Then $\Delta$
lies in the disk of radius $R_*$ about each point $x+d_i$, and by
\eqref{sg:growth} $|Y_n|\le M=2\sqrt ne^{R_*}$ on the disk of radius $2R_*$
about $x+d_i$. If $Y_n(x+d_i)\neq0$, the local zero count of
Lemma~\ref{pre:local-jensen}, with radii $R_*$ and $2R_*$, gives
\[
 \#\{\text{zeros in }\Delta\}\le\frac{\log M-\log|Y_n(x+d_i)|}{\log2}.
\]
Hence, if $\Delta$ contains more than $\bar\ell_n$ zeros, or $Y_n\equiv0$, then for
every $i$
\[
 \sqrt n\,|Y_n(x+d_i)|<\sqrt nM2^{-\bar\ell_n}=2e^{R_*}n(n+2)^{-2},
\]
and the vector in \eqref{sg:eight} has norm at most $\rho$ once $n\ge n_0$.
There are at most $\pi n+3$ disks, so $\Prob(\mathcal O_n^c)\le C(\pi n+3)n^{-4}$.
Finally, a bulk cell lies in the disk of radius $1$ about its midpoint, and
$[0,L_1]$ and $[\pi n-L_1,\pi n]$ lie in the endpoint disks.
\end{proof}

\subsection{The real coordinate on activity fibers}\label{sg:frame-section}
Given the activity data, the real coordinate is a linear form in
independent signs. Its normalized values $\Theta_i$ at the zeros of $Y_n$
decide the signs in the crossing sum, and we call them rows. Outside an
event of probability $O(n^{-3})$ the rows are well conditioned. Their variances are of order
one, they satisfy an upper frame bound, and each is rarely small.

Fix $n$ and realize the signs as in Lemma~\ref{sg:activity}. Let $\mathcal A$
be the $\sigma$-field generated by $(\chi,\tau)$. Let $\mathfrak F$ consist of
the active pairs $\{j:\chi_j=1\}$, together with the middle index $n/2$ when $n$
is even, and let $s=(s_f)_{f\in\mathfrak F}$ be the \emph{free signs},
$s_j=\sigma_j$ and $s_{n/2}=\varsigma$. Given $\mathcal A$, the free signs
are uniform on $\{-1,1\}^{\mathfrak F}$ and $Y_n$ is fixed. We write
$\Prob_{\mathcal A}$ and $\E_{\mathcal A}$ for averages over the free signs
with $\mathcal A$ fixed. By Lemma~\ref{sg:activity}, the scaled real
coordinate is a linear form in the free signs,
\[
 n^{-1/2}U_n(x/n)=\sum_{f\in\mathfrak F}a_f(x)s_f,\qquad
 a_j(x)=\frac{2\cos(\lambda_jx/n)}{\sqrt n},\quad a_{n/2}(x)=\frac1{\sqrt n},
\]
with conditional variance
\[
 v(x)=\sum_fa_f(x)^2=\frac1n\Bigl(\ind_{n\ \rm even}+4\sum_j\chi_j\cos^2(\lambda_jx/n)\Bigr).
\]

\begin{lemma}[Rows and frame]\label{sg:frame}
There are absolute constants $C,n_0$ and, for $n\ge n_0$, events
$\mathcal G_n\in\mathcal A$ with $\mathcal G_n\subset\mathcal O_n$ and
$\Prob(\mathcal G_n^c)\le Cn^{-3}$, such that on $\mathcal G_n$:
\begin{enumerate}[label=\textup{(\alph*)}]
 \item $\frac14\le v(x)\le5$ for $L_0\le x\le\pi n-L_0$;
 \item let $x_1,\dots,x_N$ be the zeros of $Y_n$ in $(L_1,\pi n-L_1)$, each
 listed as often as its order, and let
 $w_i=a(x_i)/\sqrt{v(x_i)}\in\R^{\mathfrak F}$. Then $N\le n$, $|w_i|=1$, and
 \begin{equation}\label{sg:frame-bound}
  \sum_{i=1}^N\langle w_i,y\rangle^2\le64\pi\bar\ell_n|y|^2\qquad(y\in\R^{\mathfrak F});
 \end{equation}
 \item the row $\Theta_i=\langle w_i,s\rangle$ has the sign of $U_n(x_i/n)$, and
 \begin{equation}\label{sg:row-ball}
  \Prob_{\mathcal A}\{|\Theta_i-u|\le a\}\le C\bigl(a+n^{-1/2}\bigr)
  \qquad(u\in\R,\ a\ge0,\ i\le N).
 \end{equation}
\end{enumerate}
\end{lemma}

\begin{proof}
(a) Averaging over the activities, each pair contributes $2\cos^2$ and the
middle index contributes $\cos^20$, so
\[
 \mu_n(x)=\E v(x)=\frac1n\sum_{j=0}^n\cos^2(\omega_jx)=\frac{n+1}{2n}+\frac12R_n(2x),
\]
which is at least $\frac{31}{64}$ on $[L_0,\pi n-L_0]$ by
\eqref{sg:kernel-bulk}. The deviation
$v(x)-\mu_n(x)=\frac2n\sum_j(2\chi_j-1)\cos^2(\lambda_jx/n)$ is a sum of at most
$p_n\le n$ independent symmetric terms bounded by $2/n$, so by Hoeffding's
inequality
\[
 \Prob\Bigl\{|v(x)-\mu_n(x)|\ge\frac1{16}\Bigr\}\le2e^{-n/2048}
\]
at each fixed $x$. The derivative of the deviation is bounded by
$p_n/n\le1$. On a grid of mesh $\frac1{32}$ in $[L_0,\pi n-L_0]$, with at
most $32\pi n+1$ points, all deviations are below $\frac1{16}$ except with
probability $Cne^{-n/2048}$, and then $v\ge\frac{31}{64}-\frac3{32}>\frac14$
throughout. Always $v\le(1+4p_n)/n\le5$. Let $\mathcal G_n$ be the
intersection of this event with $\mathcal O_n$. Both depend only on
$(\chi,\tau)$.

(b) By Lemma~\ref{pre:carrier}(b), $V_n$ has at most $n$ zeros, with
multiplicity, in every half-open interval of length $2\pi$, so $N\le n$. For $y\in\R^{\mathfrak F}$, the function
$\sum_fa_f(x)y_f$ is the cosine field $C_\beta(x)$ of
Lemma~\ref{pre:sampling}, where $\beta_j=\beta_{n-j}=y_j$ for $\chi_j=1$,
$\beta_{n/2}=y_{n/2}$ for even $n$, and $\beta$ vanishes elsewhere, so that
$|\beta|^2\le2|y|^2$. If $x_i\in[k,k+1)$, then
$k+\frac12\in(x_i-\frac12,x_i+\frac12]\subset[L_0,\pi n-L_0]$, so $[k,k+1]$
is a bulk cell and contains at most $\bar\ell_n$ of the points $x_i$ by
Lemma~\ref{sg:occupancy}. The sampling bound \eqref{pre:sampling-bound}
gives $\sum_iC_\beta(x_i)^2\le8\pi\bar\ell_n|\beta|^2\le16\pi\bar\ell_n|y|^2$. Finally
$\langle w_i,y\rangle=C_\beta(x_i)/\sqrt{v(x_i)}$ and $v(x_i)\ge\frac14$ by (a), so
$\sum_i\langle w_i,y\rangle^2\le4\sum_iC_\beta(x_i)^2$.

(c) The sign statement holds because $v(x_i)>0$. Put
$\bar a_f=\sqrt n\,a_f(x_i)$. Then $|\bar a_f|\le2$,
$\sum_f\bar a_f^2=nv(x_i)\ge n/4$ and $|\mathfrak F|\le n+1\le2n$, and
$|\Theta_i-u|\le a$ means
$|\sum_f\bar a_fs_f-u\sqrt{nv(x_i)}|\le a\sqrt{nv(x_i)}\le a\sqrt{5n}$. Apply
\eqref{sg:scalar} with $C_0=2$, $K=2$, $\kappa=\frac14$.
\end{proof}

On $\mathcal G_n$ the zeros $x_i$, their number $N$ and the unit vectors
$w_i$ are fixed given $\mathcal A$. For a zero $x$ of $Y_n$ let
$\iota(x)=\iota_{V_n}(x/n)\in\{-1,0,1\}$ be its crossing index, defined
before Lemma~\ref{pre:half-turns}. Give the weight $o_i=\iota(x_i)$ to one
copy of each distinct $x_i$ and $o_i=0$ to the other copies, and put
\[
 Q=\sum_{i=1}^No_i\sgn\Theta_i .
\]
By Lemma~\ref{sg:frame}(c), $Q$ is the crossing sum of
Proposition~\ref{pre:crossing} restricted to the zeros of $V_n$ in
$(L_1/n,\pi-L_1/n)$. Replacing $D_n$ by $Q$ costs the boundary roots and
the zeros of $V_n$ near $0$ and $\pi$, and the next lemma bounds both.

\begin{lemma}[Crossings and boundary charge]\label{sg:charging}
For $n\ge n_0$, on $\mathcal G_n$ and for every value of the free signs,
\begin{equation}\label{sg:charging-bound}
 |D_n-Q|\le1+B_n+N_{\rm end},\qquad
 B_n\le2\#\{i:\Theta_i=0\}+2N_{\rm end},
\end{equation}
where $N_{\rm end}\le2\bar\ell_n$ is the total order of the zeros of $Y_n$ in
$[0,L_1]\cup[\pi n-L_1,\pi n]$.
\end{lemma}

\begin{proof}
On $\mathcal G_n$, $V_n\not\equiv0$. The bulk form \eqref{pre:bulk-crossing}
of the crossing bound, applied to the set $(L_1/n,\pi-L_1/n)$, bounds
$|D_n-Q|$ by $1+B_n$ plus the number of zeros of $V_n$ in
$(0,\pi)$ outside this set. Each such zero corresponds to a zero of $Y_n$ in
$(0,L_1]\cup[\pi n-L_1,\pi n)$, so their number is at most $N_{\rm end}$.
This is the first bound.

For the second, let $e^{i\theta}$, $0\le\theta\le\pi$, be a boundary root of
multiplicity $m(\theta)$. By Lemma~\ref{pre:carrier}(c), $\theta$ is a zero of
$U_n$ and of $V_n$ of order at least $m(\theta)$. If
$n\theta\in(L_1,\pi n-L_1)$, then $n\theta$ occurs at least $m(\theta)$ times
among the $x_i$, each time with $\Theta_i=0$. Otherwise $m(\theta)$ is at
most the contribution of $n\theta$ to $N_{\rm end}$. Conjugate roots have
equal multiplicities, so $B_n=m(0)+m(\pi)+2\sum_{0<\theta<\pi}m(\theta)$,
which gives the second bound.
\end{proof}

\subsection{Concentration}\label{sg:rate-section}
The signed sum $Q$ is not a Lipschitz function of the free signs, because
$\sgn$ jumps at $0$. We smooth it at a width $\delta$ and pay for the rows
within $\delta$ of $0$. For $\delta>0$ and $R(u)=\max(u,0)$ put
\[
 \sgn_\delta(u)=\frac{R(u+\delta)-R(u-\delta)}\delta-1,\quad
 k_\delta(u)=\frac{R(u+2\delta)-R(u+\delta)-R(u-\delta)+R(u-2\delta)}\delta .
\]
Then $\sgn_\delta$ is odd, $|\sgn_\delta|\le1$, and $\sgn_\delta(u)=\sgn u$ for
$|u|\ge\delta$. Also $0\le k_\delta\le1$, $k_\delta=1$ on $[-\delta,\delta]$,
$k_\delta=0$ off $(-2\delta,2\delta)$, and $|\sgn u-\sgn_\delta(u)|\le k_\delta(u)$.

\begin{lemma}[Convex concentration]\label{sg:convex}
Let $s$ be uniform on $\{-1,1\}^m$, and let $\Pi:\R^m\to\R$ be convex and
$L$-Lipschitz. Then, for $t\ge0$,
\[
 \Prob\{|\Pi(s)-\E\Pi(s)|\ge t\}\le4e^{4\pi}e^{-t^2/(64L^2)},\qquad
 \Var\Pi(s)\le256e^{4\pi}L^2 .
\]
\end{lemma}

\begin{proof}
Talagrand's convex concentration inequality
\cite[Theorem~6.6 and (6.18)]{Talagrand}, for independent coordinates in
$[-1,1]$, gives $\Prob\{|\Pi-M|\ge t\}\le4e^{-t^2/(16L^2)}$ for a median $M$
of $\Pi(s)$. Integrating this tail, $|\E\Pi-M|\le8\sqrt\pi L$. So for
$t\ge16\sqrt\pi L$ the event in question lies in $\{|\Pi-M|\ge t/2\}$, while
for smaller $t$ the right side exceeds one. The variance bound follows by
integrating $2t$ against the tail.
\end{proof}

The core of the fixed-degree bound is a conditional estimate. Given the
activity data, $|D_n|$ is at most a deterministic term plus a sum of
convex Lipschitz functions of the free signs, and Lemma~\ref{sg:convex}
applies to each of them.

\begin{lemma}[Conditional master bound]\label{sg:master}
There are absolute constants $C_1,C_2,c_2,n_0$ such that for $n\ge n_0$, on
$\mathcal G_n$, for every $\delta>0$ and $t\ge0$,
\begin{align*}
 \Prob_{\mathcal A}\Bigl\{|D_n|>C_1\bigl(\bar\ell_n+n\delta+\sqrt n\bigr)+t\Bigr\}
 &\le C_2\exp\Bigl(-\frac{c_2t^2\delta^2}{n\bar\ell_n}\Bigr),\\
 \E_{\mathcal A}D_n^2&\le C_2\Bigl[\bigl(\bar\ell_n+n\delta+\sqrt n\bigr)^2+\frac{n\bar\ell_n}{\delta^2}\Bigr].
\end{align*}
\end{lemma}

\begin{proof}
We bound $|D_n|$ by the smoothed crossing sum and a count of small rows,
and write both as combinations of convex Lipschitz functions. Put
\[
 Q_\delta=\sum_io_i\sgn_\delta(\Theta_i),\qquad K_\delta=\sum_ik_\delta(\Theta_i).
\]
Since $|Q-Q_\delta|\le K_\delta$, $\#\{i:\Theta_i=0\}\le K_\delta$ and
$N_{\rm end}\le2\bar\ell_n$, the two bounds \eqref{sg:charging-bound} give the
master inequality
\begin{equation}\label{sg:deterministic}
 |D_n|\le|Q_\delta|+3K_\delta+7\bar\ell_n .
\end{equation}
The map $s\mapsto-s$ preserves $\Prob_{\mathcal A}$ and changes the sign of
every $\Theta_i$, so $\E_{\mathcal A}Q_\delta=0$. The row small-ball bound
\eqref{sg:row-ball} gives
\[
 \E_{\mathcal A}K_\delta\le\sum_i\Prob_{\mathcal A}\{|\Theta_i|\le2\delta\}\le C_4(n\delta+\sqrt n).
\]

Both $Q_\delta$ and $K_\delta$ are combinations of convex functions. For
$\omega\in[0,1]^N$ and $\gamma\in\R$ let
\[
 \Pi_{\omega,\gamma}(y)=\frac1\delta\sum_i\omega_iR(\langle w_i,y\rangle+\gamma),
 \qquad y\in\R^{\mathfrak F}.
\]
It is convex, and since $R$ is $1$-Lipschitz, the Cauchy--Schwarz
inequality, $N\le n$ and the frame bound \eqref{sg:frame-bound} show that it
is $L$-Lipschitz with $L^2=64\pi n\bar\ell_n/\delta^2$. With $o^\pm_i=\max(\pm o_i,0)$,
\[
 Q_\delta=\Pi_{o^+,\delta}-\Pi_{o^+,-\delta}-\Pi_{o^-,\delta}+\Pi_{o^-,-\delta}-\sum_io_i,\qquad
 K_\delta=\Pi_{1,2\delta}-\Pi_{1,\delta}-\Pi_{1,-\delta}+\Pi_{1,-2\delta},
\]
evaluated at $s$. Let $\tilde\Pi_1,\dots,\tilde\Pi_8$ be these eight convex
pieces minus their $\E_{\mathcal A}$-means. Since $\E_{\mathcal A}Q_\delta=0$,
\eqref{sg:deterministic} gives
\[
 |D_n|\le7\bar\ell_n+3C_4(n\delta+\sqrt n)+\sum_{k\le4}|\tilde\Pi_k|+3\sum_{4<k\le8}|\tilde\Pi_k| .
\]
If $|\tilde\Pi_k|\le t/16$ for all $k$, the right side is at most
$7\bar\ell_n+3C_4(n\delta+\sqrt n)+t$, and a union bound with
Lemma~\ref{sg:convex} proves the tail. For the second moment use
$(\sum_kc_k|\tilde\Pi_k|)^2\le40\sum_k\tilde\Pi_k^2$, where $c_k\in\{1,3\}$,
and the variance bound of Lemma~\ref{sg:convex}.
\end{proof}

\subsection{Proof of Theorem~\ref{thm:C}}\label{sg:lil-section}

\begin{proof}[Proof of Theorem~\ref{thm:C}]
\emph{The fixed-degree bounds.} Both follow from the master bound
(Lemma~\ref{sg:master}) by a choice of the smoothing width. Let
$n\ge n_0$. For the tail take
\[
 \delta=n^{-1/4}\sqrt{\ell_n},\qquad t=A_1n^{3/4}\sqrt{\ell_n},
\]
where $A_1$ is an absolute constant with $\frac12c_2A_1^2\log2\ge3$. Then
$n\delta=n^{3/4}\sqrt{\ell_n}$ and
$t^2\delta^2/(n\bar\ell_n)=\frac12(\log2)A_1^2\ell_n$, and
$C_1(\bar\ell_n+n\delta+\sqrt n)+t\le An^{3/4}\sqrt{\ell_n}$ for an absolute
$A$. So on $\mathcal G_n$
\[
 \Prob_{\mathcal A}\bigl\{|D_n|>An^{3/4}\sqrt{\ell_n}\bigr\}
 \le C_2\exp\Bigl(-\frac{c_2t^2\delta^2}{n\bar\ell_n}\Bigr)
 =C_2\exp\bigl(-\tfrac12c_2A_1^2(\log2)\,\ell_n\bigr)\le C_2n^{-3}.
\]
Averaging over $\mathcal A$ on $\mathcal G_n$ and adding
$\Prob(\mathcal G_n^c)\le Cn^{-3}$ gives
\begin{equation}\label{sg:fixed-tail}
 \Prob\bigl\{|D_n|>An^{3/4}\sqrt{\ell_n}\bigr\}\le Cn^{-3}.
\end{equation}
For the second moment take $\delta=(\bar\ell_n/n)^{1/4}$, which balances
the smoothing cost $n\delta$ against the concentration scale
$\sqrt{n\bar\ell_n}/\delta$. Then
$(n\delta)^2=n\bar\ell_n/\delta^2=n^{3/2}\bar\ell_n^{1/2}$, and
$\bar\ell_n^2+n\le2n^{3/2}\bar\ell_n^{1/2}$ for large $n$, so
$\E_{\mathcal A}D_n^2\le Cn^{3/2}\sqrt{\ell_n}$ on $\mathcal G_n$. Off
$\mathcal G_n$ the bound $|D_n|\le n/2$ adds $O(n^{-1})$, and the same bound
covers $n<n_0$ after enlarging $C$.

\emph{Almost-sure bounds.} By \eqref{sg:fixed-tail},
\[
 \sum_{n\ge1}\Prob\bigl\{|D_n|>An^{3/4}\sqrt{\ell_n}\bigr\}\le C\sum_{n\ge1}n^{-3}<\infty,
\]
so the first Borel--Cantelli lemma gives $|D_n|\le An^{3/4}\sqrt{\ell_n}$
for all large $n$, almost surely, and $\ell_n\le2\log n$ for $n\ge2$. The statements about $B_n$ are
Proposition~\ref{sg:boundary}(a) and (c), with
$\ell_n^{1+\eta}\le(2\log n)^{1+\eta}$.

\emph{Input~\ref{fw:I4}.} Since $\epsilon_0\ne0$ and
$\Var B_n\le C(n+1)^{1/2}$ by \eqref{sg:boundary-moments},
Lemma~\ref{fw:block-proxies} applies with $\beta=\frac12$ and gives
Input~\ref{fw:I4} for $Z_n=D_n-\E D_n$, with $\theta_0=2$ and coupling
errors $CL_k^{1/4}$ and $Cn_k^{1/4}$.
\end{proof}

\begin{proof}[Proof of Corollary~\ref{thm:E}]
For each $n$ the coefficients of $f_n$ have the law of
$(\epsilon_0,\dots,\epsilon_n)$, so \eqref{sg:fixed-tail} holds. It
concerns one degree at a time, so the Borel--Cantelli argument above
applies whatever the joint law of the $f_n$, and almost surely
$|D_n|\le An^{3/4}\sqrt{\ell_n}$ for all large $n$. Dividing by $n$ gives
$\nu_n(1)/n\to\frac12$.
\end{proof}

\section{The fourth-cumulant variance formula}\label{sec:cumulant}

The count $D_n$ is a signed sum over the crossings of the real axis by
the curve $t\mapsto e^{-int/2}f_n(e^{it})$
(Proposition~\ref{pre:crossing}). To find its variance for a general
coefficient law, we smooth this sum into a function $J$ of the
coefficients and compare with Gaussian coefficients. Lindeberg
replacement changes the variance of $J$ by $\kappa_4$ times a Gaussian
fourth insertion, and this insertion equals $n/12$ up to smaller terms.
It remains to show that $D_n$ and $J$ are close in $L^2$, which we do
cell by cell along the curve. On a cell where the two coordinates of the
curve are never small together, an integration by parts turns the
crossing sum into $J$. Coarse small-ball bounds for sums of coefficients
make the other cells rare, and distant cells are nearly uncorrelated
where the characteristic functions of these sums decay. For laws with a
bounded density the decay holds at every point and also gives the small
balls. For random signs the small balls come from
Section~\ref{sec:signs}, and the decay holds at Diophantine points
(Section~\ref{sec:sign-variance}).

\begin{definition}\label{cu:bd}
A real probability law satisfies \textup{(BD)} if it is symmetric, has a
bounded Lebesgue density, has variance one, and has finite moments of
every order. For such a law $\kappa_4=\E\xi_0^4-3\ge-2$.
\end{definition}

The standard Gaussian law, the uniform law on $[-\sqrt3,\sqrt3]$ and the
Laplace law with density $2^{-1/2}e^{-\sqrt2|x|}$ satisfy (BD).
Throughout, $(G_j)_{j\ge0}$ are iid standard real Gaussians. For a
function $F$ of a coefficient vector, $F(\xi)$ and $F(G)$ denote its
values at $(\xi_0,\ldots,\xi_n)$ and $(G_0,\ldots,G_n)$, and $D_n^G$ is
the count for Gaussian coefficients. Constants $C,c$ and thresholds
$n_0$ depend only on the coefficient law and on the cutoff $\chi$ of
Subsection~\ref{cu:fields-section}.

\subsection{Fields and the current}\label{cu:fields-section}
Let $n\ge1$, and let $\omega_j=(j-n/2)/n$ and the trigonometric fields
$C_b,S_b$ be as in Lemma~\ref{pre:sampling}. For $a\in\R^{n+1}$ and
$x\in\C$ put
\begin{equation}\label{cu:fields}
\begin{gathered}
 Z(a,x)=\sum_{j=0}^na_jv_j(x)
 =\Bigl(\frac{2n}{n+1}\Bigr)^{1/2}\bigl(C_a(x),S_a(x),S_a'(x)\bigr),\\
 v_j(x)=\Bigl(\frac2{n+1}\Bigr)^{1/2}w_j(x),\qquad
 w_j(x)=\bigl(\cos\omega_jx,\ \sin\omega_jx,\ \omega_j\cos\omega_jx\bigr).
\end{gathered}
\end{equation}
These are entire functions of $x$, and $|v_j(x)|\le2n^{-1/2}$ for real
$x$ since $|w_j(x)|^2\le\frac54$. We write
$Z_n(x)=Z(\xi,x)=(X_n,Y_n,Y_n')(x)$. Thus $X_n$ and $Y_n$ are the fields
$C_\xi$ and $S_\xi$ of Section~\ref{sec:gaussian} multiplied by
$(2n/(n+1))^{1/2}\in[1,\sqrt2)$. The factor brings their variances close
to one and changes neither zeros nor signs. In the notation of
Lemma~\ref{pre:carrier}, $X_n$ and $Y_n$ are multiples of $U_{f_n}(x/n)$
and $V_{f_n}(x/n)$, where $x=nt\in[0,\pi n]$ is the scaled half-circle
angle. By Lemma~\ref{pre:sampling}(i) for $C_c$ and for $S_c$, whose
derivative is $C_{\omega c}$,
\begin{equation}\label{cu:energy}
 \int_0^{\pi n}|Z(c,x)|^2\dd x\le\frac{2n}{n+1}\cdot4\pi|c|^2\le8\pi|c|^2
 \qquad(c\in\R^{n+1}).
\end{equation}

\begin{lemma}[Gaussian geometry]\label{cu:geometry}
For $k\in\{0,1,2\}$ let
\[
 \vartheta_k(u)=\frac1{n+1}\sum_{j=0}^n\omega_j^k\,e^{i\omega_ju}.
\]
\begin{enumerate}[label=\textup{(\roman*)}]
 \item $\vartheta_0,\vartheta_2$ are real and even, $\vartheta_1$ is
 purely imaginary and odd, and for real $u$
 \[
  |\vartheta_k(u)|\le\min\Bigl\{1,\frac{2\pi}{\dist(u,2\pi n\Z)}\Bigr\}.
 \]
 \item Let $\Gamma(s,t)=\sum_jv_j(s)v_j(t)^T$, the covariance of $Z(G,s)$
 and $Z(G,t)$. Then
 \[
 \begin{gathered}
  \Gamma_{11}=\vartheta_0(s-t)+\vartheta_0(s+t),\qquad
  \Gamma_{22}=\vartheta_0(s-t)-\vartheta_0(s+t),\\
  \Gamma_{33}=\vartheta_2(s-t)+\vartheta_2(s+t),\qquad
  \Gamma_{12}=\Gamma_{21}=\Gamma_{13}=\Gamma_{31}=0,\\
  \Gamma_{23}=\Im\vartheta_1(s+t)+\Im\vartheta_1(s-t),\qquad
  \Gamma_{32}=\Im\vartheta_1(s+t)-\Im\vartheta_1(s-t).
 \end{gathered}
 \]
 In particular $\Sigma(x)=\Gamma(x,x)$ has $\Sigma_{11}=1+\vartheta_0(2x)$,
 $\Sigma_{22}=1-\vartheta_0(2x)$,
 $\Sigma_{33}=\frac{n+2}{12n}+\vartheta_2(2x)$,
 $\Sigma_{23}=\Im\vartheta_1(2x)$ and $\Sigma_{12}=\Sigma_{13}=0$, and
 for Gaussian coefficients the processes $X_n$ and $Y_n$ are
 independent.
\end{enumerate}
\end{lemma}

\begin{proof}
(i) The map $j\mapsto n-j$ sends $\omega_j$ to $-\omega_j$, which gives
the parity claims. Write $e^{i\omega_ju}=e^{-iu/2}e^{iju/n}$ and
$E_m=\sum_{j=0}^me^{iju/n}$, so $|E_m|\le1/|\sin(u/2n)|$ when
$u\notin2\pi n\Z$. Summation by parts with $c_j=\omega_j^k$, for which
$|c_n|\le1$ and $\sum_j|c_{j+1}-c_j|\le1$, gives
$|\sum_jc_je^{iju/n}|\le2/|\sin(u/2n)|$. Since
$|\sin\theta|\ge\frac2\pi\dist(\theta,\pi\Z)$, we get
$|\sin(u/2n)|\ge\dist(u,2\pi n\Z)/(\pi n)$, hence
$|\vartheta_k(u)|\le2\pi n/((n+1)\dist(u,2\pi n\Z))$. The bound by one
is trivial.

(ii) Since each coordinate of $v_j$ carries the factor
$(2/(n+1))^{1/2}$, product-to-sum formulas give, for instance,
\[
 \Gamma_{11}=\frac1{n+1}\sum_j\bigl[\cos\omega_j(s-t)+\cos\omega_j(s+t)\bigr],\qquad
 \Gamma_{12}=\frac1{n+1}\sum_j\bigl[\sin\omega_j(s+t)-\sin\omega_j(s-t)\bigr].
\]
The other entries are analogous, and (i) removes the vanishing real or
imaginary parts. At $s=t=x$ use $\vartheta_0(0)=1$, $\vartheta_1(0)=0$
and $\vartheta_2(0)=(n+1)^{-1}n^{-2}\sum_j(j-n/2)^2=(n+2)/(12n)$. For
Gaussian coefficients the processes $X_n$ and $Y_n$ are jointly Gaussian
and uncorrelated, hence independent.
\end{proof}

Near the ends of the window $[1,\pi n-1]$ the terms $\vartheta_k(2x)$
are not small. The next lemma controls the covariances on the whole
window.

\begin{lemma}[The covariance on the window]\label{cu:window}
There are $\gamma_*\in(0,\frac1{24}]$, $R_0\ge1$ and $n_0$ such that for
$n\ge n_0$ the following hold.
\begin{enumerate}[label=\textup{(\roman*)}]
 \item $\Sigma(x)\ge\gamma_*I$ for $x\in[1,\pi n-1]$.
 \item For $s,t\in[0,\pi n]$ with $|s-t|\ge1$, every entry of
 $\Gamma(s,t)$ is at most $12\pi/(1+|s-t|)$ in absolute value.
 \item For $s,t\in[1,\pi n-1]$ with $|s-t|\ge R_0$ and $0\le\tau'\le1$,
 the matrix
 \[
  S_{\tau'}=\begin{pmatrix}\Sigma(s)&\tau'\Gamma(s,t)\\\tau'\Gamma(s,t)^T&\Sigma(t)\end{pmatrix}
 \]
 satisfies $\frac{\gamma_*}2I\le S_{\tau'}\le5I$.
\end{enumerate}
\end{lemma}

\begin{proof}
(i) Let $L_0=72\pi$. For $x\in[L_0,\pi n-L_0]$, $\dist(2x,2\pi n\Z)\ge2L_0$,
so by Lemma~\ref{cu:geometry} the entries of
$\Sigma(x)-\operatorname{diag}(1,1,\frac{n+2}{12n})$ are at most
$\pi/L_0=\frac1{72}$ in absolute value and its norm is at most
$\frac1{24}$. The diagonal entries lie in $[\frac1{12},1]$, so
$\Sigma(x)\ge\frac1{24}I$ there. Now let $x\in[1,L_0]$ and let
$q\in\R^3$ be a unit vector. Then
$q^T\Sigma(x)q=\frac2{n+1}\sum_jf_{x,q}(\omega_j)^2$ with
$f_{x,q}(\omega)=q_1\cos\omega x+q_2\sin\omega x+q_3\,\omega\cos\omega x$.
The points $\omega_j$ have spacing $1/n$ in $[-\frac12,\frac12]$, and the
derivative of $f_{x,q}^2$ in $\omega$ is bounded uniformly in $x\le L_0$
and $q$. So $\frac1n\sum_jf_{x,q}(\omega_j)^2=\int_{-1/2}^{1/2}f_{x,q}^2\dd\omega+O(1/n)$
uniformly. The integral is continuous in $(x,q)$ on the compact set
$[1,L_0]\times S^2$ and positive there, because for $x\ne0$ the functions
$\cos\omega x$, $\sin\omega x$ and $\omega\cos\omega x$ are linearly
independent on $[-\frac12,\frac12]$. As $2n/(n+1)\ge1$, this gives
$\Sigma(x)\ge\gamma_{\rm l}I$ on $[1,L_0]$ for some $\gamma_{\rm l}>0$ and $n\ge n_0$.
For $x'=\pi n-x$ we have $\omega_jx'=\pi(j-n/2)-\omega_jx$, so $w_j(x')$
is $\pm(\cos\omega_jx,-\sin\omega_jx,\omega_j\cos\omega_jx)$ if $n$ is
even and $\pm(\sin\omega_jx,\cos\omega_jx,\omega_j\sin\omega_jx)$ if $n$
is odd. The same argument with these three functions gives a bound
$\gamma_{\rm r}>0$ on $[\pi n-L_0,\pi n-1]$. Put
$\gamma_*=\min(\frac1{24},\gamma_{\rm l},\gamma_{\rm r})$.

(ii) For $s,t\in[0,\pi n]$, $\dist(s-t,2\pi n\Z)=|s-t|$ and
$\dist(s+t,2\pi n\Z)=\min(s+t,2\pi n-s-t)\ge|s-t|$. By
Lemma~\ref{cu:geometry}, every entry of $\Gamma(s,t)$ is a sum of at
most two terms $\pm\vartheta_k(s\pm t)$ or $\pm\Im\vartheta_1(s\pm t)$,
each at most $2\pi/|s-t|\le4\pi/(1+|s-t|)$ as $|s-t|\ge1$.

(iii) Let $R_0=\max(1,216\pi/\gamma_*)$. By (ii),
$|a^T\Gamma(s,t)b|\le108\pi|a||b|/(1+|s-t|)$ for $a,b\in\R^3$. For
$q=(a,b)$, (i) gives
\[
 q^TS_{\tau'}q\ge\gamma_*|q|^2-\frac{2\cdot108\pi|a||b|}{R_0}
 \ge\Bigl(\gamma_*-\frac{108\pi}{R_0}\Bigr)|q|^2\ge\frac{\gamma_*}2|q|^2 .
\]
The upper bound holds because $\Sigma(x)\le\sum_j|v_j(x)|^2I\le\frac52I$
and $\|\Gamma(s,t)\|\le\frac52$ by the Cauchy--Schwarz inequality.
\end{proof}

\begin{definition}\label{cu:well-conditioned}
A sum $W=\sum_{j=0}^n\xi_ju_j$ with $u_j\in\R^d$ is \emph{well
conditioned} if $|u_j|\le7n^{-1/2}$ for all $j$ and
$\sum_ju_ju_j^T\ge\frac{\gamma_*}2I$.
\end{definition}

The sums of coefficients that we meet below are well conditioned. For
$x\in[1,\pi n-1]$ the sums $Z_n(x)$, $Y_n(x)$ and $(X_n(x),Y_n(x))$
are, since their rows are coordinates of $v_j(x)$ and their Gram
matrices are principal submatrices of $\Sigma(x)\ge\gamma_*I$. For
$s,t\in[1,\pi n-1]$ with $|s-t|\ge R_0$ so is $(Z_n(s),Z_n(t))$. Its
rows are $u_j=(v_j(s),v_j(t))$, and its Gram matrix is
the matrix $S_1$ of Lemma~\ref{cu:window}(iii), the covariance of
$(Z(G,s),Z(G,t))$.

\begin{lemma}[Field moments]\label{cu:moments}
Let $\xi_0$ be symmetric with finite moments of every order. For
$p<\infty$ and $R\ge1$ there is $C_{p,R}$ such that for every $n\ge1$,
every real $c$ and $F\in\{X_n,X_n',Y_n,Y_n'\}$,
\[
 \E\sup_{|z-c|\le R}|F(z)|^p\le C_{p,R}.
\]
\end{lemma}

\begin{proof}
By H\"older's inequality we may take $p\ge2$. For complex $c_j$,
\[
 \E\Bigl|\sum_jc_j\xi_j\Bigr|^p\le C_p\Bigl(\sum_j|c_j|^2\Bigr)^{p/2}\E|\xi_0|^p .
\]
Indeed, for real $c_j$ the symmetric $\xi_j$ have the law of
$\epsilon_j|\xi_j|$ with independent signs, so Khintchine's inequality
conditionally on $(|\xi_j|)$ and Minkowski's inequality in $L^{p/2}$
give the bound, and complex $c_j$ are split into real and imaginary
parts. The coefficients
of $F(z)$ are $(2/(n+1))^{1/2}\omega_j^k$ times $\pm\cos(\omega_jz)$ or
$\pm\sin(\omega_jz)$, $k\le2$, and their squares sum to at most
$2e^{|\Im z|}$, so $\E|F(z)|^p\le C_pe^{p|\Im z|/2}$. By Cauchy's
formula on the circle $|\zeta-c|=2R$,
\[
 \sup_{|z-c|\le R}|F(z)|\le\frac1\pi\int_0^{2\pi}|F(c+2Re^{i\theta})|\dd\theta,
\]
and H\"older's inequality gives
$\E\sup_{|z-c|\le R}|F|^p\le2^p\sup_\theta\E|F(c+2Re^{i\theta})|^p\le C_{p,R}$.
\end{proof}

\emph{The current.} Let $\Phi$ and $\varphi$ be the standard normal
distribution function and density. For $\delta>0$ let
$\rho_\delta(v)=\delta^{-1}\varphi(v/\delta)$, and let
$h_\delta(u)=2\Phi(u/\delta)-1$ be the smoothed sign of
Section~\ref{sec:gaussian} at width $\delta$, so that
$h_\delta'=2\rho_\delta$, $|h_\delta|\le1$ and
$h_\delta(u)=\E\sgn(u-\delta G_0)$. Fix an even smooth
$\chi:\R\to[0,1]$ equal to one on $[-1,1]$ and to zero outside $[-2,2]$,
and for $M>0$ put $\ell_M(w)=w\chi(w/M)$. Then $|\ell_M(w)|\le|w|$,
$\ell_M(w)=w$ for $|w|\le M$, $|\ell_M'|\le C_\chi:=1+\sup_s|s\chi'(s)|$,
and $\|\ell_M^{(r)}\|_\infty\le C_rM^{1-r}$ for $r\ge0$. Define
\begin{equation}\label{cu:current}
\begin{gathered}
 \Psi_{\delta,M}(u,v,w)=h_\delta(u)\rho_\delta(v)\ell_M(w),\qquad
 t_M=\E\,\ell_M'(W_*),\\
 J_{n,\delta,M}(a)=\int_1^{\pi n-1}\Psi_{\delta,M}\bigl(Z(a,x)\bigr)\dd x,
\end{gathered}
\end{equation}
where $W_*$ is centered normal with variance $1/12$, and write
$J_{n,\delta,M}=J_{n,\delta,M}(\xi)$. As $h_\delta(X_n)$ smooths
$\sgn X_n$ and $\rho_\delta(Y_n)Y_n'\dd x$ smooths the signed unit masses
at the zeros of $Y_n$, $J_{n,\delta,M}$ smooths the bulk part
$\sum\sgn X_n(x)\,\iota_{Y_n}(x)$, over the zeros of $Y_n$ in
$[1,\pi n-1)$, of the crossing sum of Proposition~\ref{pre:crossing}.

\begin{lemma}[Jets of the current]\label{cu:jets}
Let $n\ge1$, $0<\delta\le1\le M$ and $J=J_{n,\delta,M}$.
\begin{enumerate}[label=\textup{(\roman*)}]
 \item $\|D^r\Psi_{\delta,M}\|_\infty\le C_rM\delta^{-r-1}$ for $r\ge0$,
 and $|\Psi_{\delta,M}(u,v,w)|\le\rho_\delta(v)|w|$.
 \item $\Psi=\Psi_{\delta,M}$ satisfies $\Psi(-z)=\Psi(z)$,
 $\Psi(-u,v,w)=-\Psi(u,v,w)$ and $\Psi(u,-v,-w)=-\Psi(u,v,w)$.
 \item $J$ is smooth on $\R^{n+1}$, with
 \[
 \begin{gathered}
  \partial_{j_1}\cdots\partial_{j_r}J(a)
  =\int_1^{\pi n-1}D^r\Psi_{\delta,M}(Z(a,x))\bigl[v_{j_1}(x),\ldots,v_{j_r}(x)\bigr]\dd x,\\
  \|\partial_{j_1}\cdots\partial_{j_r}J\|_\infty\le C_rM\delta^{-r-1}n^{1-r/2}.
 \end{gathered}
 \]
 \item $J(-a)=J(a)$ and $J(a^{\rm rev})=-J(a)$, where
 $a^{\rm rev}_j=a_{n-j}$. Hence $\E J_{n,\delta,M}(\xi)=0$ for iid
 coefficients with any law, and $\E\,\partial_jJ(G)=0$.
\end{enumerate}
\end{lemma}

\begin{proof}
(i) A partial derivative of order $r$ of $\Psi_{\delta,M}$ is a product
$h_\delta^{(p)}(u)\rho_\delta^{(q)}(v)\ell_M^{(s)}(w)$ with $p+q+s=r$.
Since $h_\delta^{(p)}(u)=\delta^{-p}h_1^{(p)}(u/\delta)$,
$\rho_\delta^{(q)}(v)=\delta^{-1-q}\varphi^{(q)}(v/\delta)$,
$\|\ell_M\|_\infty\le2M$ and $\|\ell_M^{(s)}\|_\infty\le C_s$ for
$s\ge1$, it is at most $C_rM\delta^{-r-1}$. The second bound uses
$|h_\delta|\le1$ and $|\ell_M(w)|\le|w|$.
(ii) $h_\delta$ and $\ell_M$ are odd and $\rho_\delta$ is even.
(iii) $Z(a,x)$ is linear in $a$ with $\partial_{a_j}Z=v_j(x)$, and the
integrand is smooth with bounded derivatives on a compact interval of
length at most $\pi n$. Use (i) and $|v_j|\le2n^{-1/2}$.
(iv) $Z(-a,x)=-Z(a,x)$, and $\omega_{n-j}=-\omega_j$ gives
$Z(a^{\rm rev},x)=(z_1,-z_2,-z_3)$ when $Z(a,x)=(z_1,z_2,z_3)$. Apply
(ii). The iid law of the coefficient vector is invariant under reversal
and $J$ is bounded, so $\E J(\xi)=-\E J(\xi)$. Finally $\partial_jJ$ is
odd and $G$ is symmetric.
\end{proof}

\subsection{Replacement and the fourth insertion}
Replacing the coefficients by Gaussians one at a time changes the mean
of a smooth function by its fourth derivatives, weighted by $\kappa_4$.

\begin{lemma}[Lindeberg replacement]\label{cu:lindeberg}
Let $\xi_0$ be symmetric with $\E\xi_0^2=1$ and $\E\xi_0^6<\infty$, let
$m\ge1$, and let $F\in C^\infty(\R^m)$ have bounded partial derivatives
of all orders. For $\xi=(\xi_1,\ldots,\xi_m)$ iid and $G$ standard
Gaussian in $\R^m$,
\begin{align*}
 &\Bigl|\E F(\xi)-\E F(G)-\frac{\kappa_4}{24}\sum_j\E\,\partial_j^4F(G)\Bigr|
 \le C\Bigl(\sum_j\|\partial_j^6F\|_\infty+\sum_{i\ne j}\|\partial_i^4\partial_j^4F\|_\infty\Bigr),\\
 &\Bigl|\E\bigl[(\xi_j^2-1)F(\xi)\bigr]-\frac{\kappa_4+2}2\E\,\partial_j^2F(G)\Bigr|
 \le C\Bigl(\|\partial_j^4F\|_\infty+\sum_{i\ne j}\|\partial_i^4\partial_j^2F\|_\infty\Bigr)
 \quad(1\le j\le m),
\end{align*}
with $C$ depending only on $\E\xi_0^6$.
\end{lemma}

\begin{proof}
The laws of $\xi_1$ and $G_1$ have equal moments of orders $0,1,2,3,5$,
and their fourth moments differ by $\kappa_4$. Let $Y^{(k)}$ have
Gaussian coordinates $G_i$ for $i<k$ and coordinates $\xi_i$ for
$i\ge k$, and let $Z^{(k)}$ be $Y^{(k)}$ with coordinate $k$ set to
zero. Taylor's formula in coordinate $k$ to order five, with remainder
at most $\|\partial_k^6F\|_\infty y^6/720$, and independence give
\[
 \E F(Y^{(k)})-\E F(Y^{(k+1)})=\frac{\kappa_4}{24}\E\,\partial_k^4F(Z^{(k)})
 +O\bigl(\|\partial_k^6F\|_\infty\bigr).
\]
Replacing the coordinates $\xi_i$, $i>k$, of $Z^{(k)}$ by $G_i$ one at
a time costs $O(\|\partial_i^4\partial_k^4F\|_\infty)$ for each $i$, by
Taylor's formula to order three, since moments match through order three.
Restoring coordinate $k$ costs $O(\|\partial_k^6F\|_\infty)$ by Taylor's
formula to order one. Summing over $k$ gives the first bound.

For the second, let $Z$ be $\xi$ with coordinate $j$ set to zero.
Taylor's formula in coordinate $j$ to order three gives
\[
 \E\bigl[(\xi_j^2-1)F(\xi)\bigr]=\frac{\kappa_4+2}2\,\E\,\partial_j^2F(Z)
 +O\bigl(\|\partial_j^4F\|_\infty\bigr),
\]
because $\E[(\xi_0^2-1)\xi_0^r]$ vanishes for $r=0,1,3$ and equals
$\E\xi_0^4-1=\kappa_4+2$ for $r=2$. Replacing the other coordinates of
$Z$ by Gaussians costs $O(\sum_{i\ne j}\|\partial_i^4\partial_j^2F\|_\infty)$
as above, and restoring $G_j$ costs $O(\|\partial_j^4F\|_\infty)$.
\end{proof}

\begin{lemma}[Gaussian interpolation]\label{cu:price}
Let $X_0,X_1$ be centered Gaussian vectors in $\R^d$ with covariance
matrices $\Sigma_0,\Sigma_1$, possibly singular, and let $F$ be smooth
with bounded derivatives of all orders. Then
\[
 \bigl|\E F(X_1)-\E F(X_0)\bigr|
 \le\frac12\sum_{a,b=1}^d\bigl|(\Sigma_1-\Sigma_0)_{ab}\bigr|\,
 \sup|\partial_{ab}F| .
\]
\end{lemma}

\begin{proof}
Take $X_0,X_1$ independent and put $X_t=\sqrt t\,X_1+\sqrt{1-t}\,X_0$.
The Gaussian interpolation formula
\cite[Lemma~1.3.1 and (1.44)]{TalagrandMF2011}, which holds for singular
covariances, gives
\[
 \E F(X_1)-\E F(X_0)=\frac12\sum_{a,b=1}^d(\Sigma_1-\Sigma_0)_{ab}
 \int_0^1\E\,\partial_{ab}F(X_t)\dd t ,
\]
and the bound follows.
\end{proof}

For $n\ge1$, $0<\delta\le1\le M$, $J=J_{n,\delta,M}$ and $0\le j\le n$
put
\[
 S_n=S_n(\delta,M)=\frac1{24}\sum_{j=0}^n\E\bigl[\partial_j^4(J^2)(G)\bigr],\qquad
 h_j=\frac{j-n/2}{n+1}.
\]

\begin{proposition}[The fourth insertion]\label{cu:insertion}
For $n\ge1$, $0<\delta\le1\le M$, $0\le k\le4$ and $0\le j\le n$,
\begin{align}
 \Var\,\partial_j^kJ(G)&\le CM^2\delta^{-2k-4}n^{1-k},\label{cu:jet-variance}\\
 \Bigl|\E\,\partial_j^2J(G)-\frac{2t_Mh_j}{1+\delta^2}\Bigr|
 &\le\frac Cn\bigl(1+M\delta^{-5}\log(n+2)\bigr),\label{cu:second-jet}\\
 \Bigl|S_n-\frac{t_M^2\,n}{12(1+\delta^2)^2}\Bigr|
 &\le CM^2\delta^{-10}\log^2(n+2).\label{cu:insertion-bound}
\end{align}
Moreover $|t_M|\le C_\chi$ and $|t_M-1|\le2(C_\chi+1)e^{-6M^2}$.
\end{proposition}

\begin{proof}
\emph{Variances.} Let $f=\partial_j^kJ$. By Lemma~\ref{cu:jets}(iii),
for every $a$ and $i$, $\partial_if(a)=\int_Iv_i(x)\cdot H(x)\dd x$ with
$I=[1,\pi n-1]$ and
$H(x)_p=D^{k+1}\Psi_{\delta,M}(Z(a,x))[e_p,v_j(x),\ldots,v_j(x)]$, so
$|H(x)|^2\le CM^2\delta^{-2k-4}n^{-k}$. Put $c_i=\partial_if(a)$. By
Cauchy--Schwarz and \eqref{cu:energy},
\[
 |c|^2=\int_IZ(c,x)\cdot H(x)\dd x
 \le(8\pi)^{1/2}|c|\Bigl(\int_I|H|^2\Bigr)^{1/2},
\]
hence $|\nabla f(a)|^2\le8\pi\int_I|H|^2\le CM^2\delta^{-2k-4}n^{1-k}$
uniformly in $a$. The Poincar\'e inequality, Lemma~\ref{pre:gaussian},
gives \eqref{cu:jet-variance}.

\emph{Second-jet mean.} By Lemma~\ref{cu:jets}(iii) and Fubini,
$\E\,\partial_j^2J(G)=\int_I\E f_x(Z(G,x))\dd x$ with
$f_x(z)=D^2\Psi_{\delta,M}(z)[v,v]$, $v=v_j(x)$. The second derivatives
of $f_x$ are bounded by $\|D^4\Psi_{\delta,M}\|_\infty|v|^2\le CM\delta^{-5}/n$.
Let $Z_*=(U_*,V_*,W_*)$ have independent centered normal coordinates
with variances $1,1,1/12$. Lemmas~\ref{cu:price} and~\ref{cu:geometry}(ii)
give
\[
 \bigl|\E f_x(Z(G,x))-\E f_x(Z_*)\bigr|
 \le\frac{CM\delta^{-5}}n\Bigl(|\vartheta_0(2x)|+|\vartheta_1(2x)|
 +|\vartheta_2(2x)|+\frac1n\Bigr).
\]
Since $\dist(2x,2\pi n\Z)=2\min(x,\pi n-x)$ on $[0,\pi n]$,
Lemma~\ref{cu:geometry}(i) gives
$\int_I|\vartheta_k(2x)|\dd x\le2\int_0^{\pi n/2}\min(1,\pi/x)\dd x
\le C\log(n+2)$. At $Z_*$ the six terms of
$D^2\Psi[v,v]=\sum_{p,q}v_pv_q\partial_{pq}\Psi$ factor into products
of one-dimensional expectations. Every term containing an
undifferentiated or twice differentiated $h_\delta(U_*)$ vanishes since
$h_\delta,h_\delta''$ are odd, and so does every term containing
$\ell_M(W_*)$ or $\ell_M''(W_*)$. The survivor is
\[
 \E f_x(Z_*)=2v_1v_3\,\E h_\delta'(U_*)\,\E\rho_\delta(V_*)\,t_M
 =\frac{2v_1v_3\,t_M}{\pi(1+\delta^2)},
\]
since $\E\rho_\delta(V_*)=(2\pi(1+\delta^2))^{-1/2}$ is the
$N(0,1+\delta^2)$ density at zero and $h_\delta'=2\rho_\delta$. Here
$v_1v_3=2\omega_j\cos^2(\omega_jx)/(n+1)$, and
$\omega_j\int_0^{\pi n}\cos^2(\omega_jx)\dd x=\omega_j\pi n/2$ because
$\sin(2\omega_j\pi n)=\sin(\pi(2j-n))=0$. The two end pieces of
$[0,\pi n]\setminus I$ change this by at most $2|\omega_j|\le1$. With
$2n\omega_j/(n+1)=2h_j$ this proves \eqref{cu:second-jet}.

\emph{Insertion.} Write $g_k=\partial_j^kJ(G)$. By Leibniz,
$\partial_j^4(J^2)=2J\partial_j^4J+8\partial_jJ\,\partial_j^3J+6(\partial_j^2J)^2$.
Lemma~\ref{cu:jets}(iv) gives $\E g_0=\E g_1=0$, so by Cauchy--Schwarz
and \eqref{cu:jet-variance},
\[
 |\E g_0g_4|=|\E g_0(g_4-\E g_4)|,\quad
 |\E g_1g_3|,\quad
 \E g_2^2-(\E g_2)^2\quad\text{are all at most }CM^2\delta^{-8}n^{-1}.
\]
Summing over $j$, $|S_n-\frac14\sum_j(\E g_2)^2|\le CM^2\delta^{-8}$. Write
$\E g_2=2t_Mh_j/(1+\delta^2)+e_j$ with $e_j$ bounded by
\eqref{cu:second-jet}. Then
\begin{gather*}
 \frac14\sum_j(\E g_2)^2=\frac{t_M^2}{(1+\delta^2)^2}\sum_jh_j^2
 +O\Bigl(\sum_j|h_je_j|+\sum_je_j^2\Bigr),\\
 \sum_jh_j^2=\frac{n(n+2)}{12(n+1)}=\frac n{12}+O(1),
\end{gather*}
the second because $\sum_{j=0}^n(j-n/2)^2=n(n+1)(n+2)/12$. Since
$|h_j|\le\frac12$ and $|t_M|\le C_\chi$, the error terms are at most
$C(1+M\delta^{-5}\log(n+2))+CM^2\delta^{-10}\log^2(n+2)/n$, which proves
\eqref{cu:insertion-bound}.

Finally $\ell_M'(w)=\chi(w/M)+(w/M)\chi'(w/M)$ is bounded by $C_\chi$
and equals one for $|w|<M$, and $\Prob(|W_*|\ge M)\le2e^{-6M^2}$.
\end{proof}

\subsection{The unsmoothing proposition}\label{cu:unsmooth-section}
We compare $D_n$ with $J_{n,\delta,M}$ in $L^2$, cell by cell. On a
cell where $X_n$ and $Y_n$ are never small together, the crossing sum and
the current agree up to a small error (Lemma~\ref{cu:cell-identity}), and
the coarse small balls \ref{cu:U1} make the other cells rare. This
bounds the errors on nearby cells. On two distant cells we sharpen the
smoothing to a width $\delta'$ below $n^{-1}$, and the differences of the
two smoothings are nearly uncorrelated. This step needs the
characteristic functions of the coefficient sums at the two cells to
decay. The number $b_n$ counts the pairs of distant cells where they do
not, and we bound these pairs like nearby ones.

\emph{Parameters and cells.} In the rest of this section,
\[
\begin{gathered}
 \delta=n^{-1/25},\qquad\delta'=n^{-103/100},\qquad M=n^{1/200},\\
 k_n=\frac{2\log n}{\log2},\qquad
 \varepsilon_{\delta,M}=\delta+M^{-8}=2n^{-1/25},
\end{gathered}
\]
and $\psi=\Psi_{\delta',M}-\Psi_{\delta,M}$. Let $m=\lceil\pi n-2\rceil$
and $\ell=(\pi n-2)/m\in[\frac12,1]$. The cells are
$I_i=[1+i\ell,1+(i+1)\ell)$, $0\le i<m$, with closures $\bar I_i$ and
midpoints $c_i$. They partition $[1,\pi n-1)$, and each $\bar I_i$ lies in
$[1,\pi n-1]$ and in $(c_i-2,c_i+2)$. For a cell $I=[\alpha,\beta)$ let
\[
 K_I=\sum_{x\in[\alpha,\beta),\ Y_n(x)=0}\sgn X_n(x)\,\iota_{Y_n}(x)
\]
if $Y_n\not\equiv0$, and $K_I=0$ otherwise. For $0<\kappa\le1$ let
$J_I^\kappa=\int_I\Psi_{\kappa,M}(Z_n(x))\dd x$, and put
\[
 \Delta_I=K_I-J_I^\delta,\qquad \Delta_I^{\rm f}=K_I-J_I^{\delta'},\qquad
 \Delta_I^{\rm c}=J_I^{\delta'}-J_I^\delta=\Delta_I-\Delta_I^{\rm f}.
\]
We abbreviate $\Delta_i=\Delta_{I_i}$, and similarly for the other cell
quantities. Since the cells partition $[1,\pi n-1)$,
$J_{n,\delta,M}=\sum_{i<m}J_{I_i}^\delta$, and
$\Delta_I^{\rm c}=\int_I\psi(Z_n(x))\dd x$. Let $m_0=\lceil2R_0\rceil$.
A pair $(i,i')$ of cell indices is \emph{separated} if $|i-i'|>m_0$.
Then all $s\in\bar I_i$ and $t\in\bar I_{i'}$ satisfy
$|s-t|\ge(|i-i'|-1)\ell\ge m_0/2\ge R_0$.

\emph{Decay.} For a random vector $W$ in $\R^d$ write
$\widehat\mu_W(\theta)=\E e^{i\theta\cdot W}$. We say that
$\widehat\mu_W$ \emph{decays} if
\begin{equation}\label{cu:decay}
 |\widehat\mu_W(\theta)|\le n^{-20}\qquad(n^{1/10}\le|\theta|\le n^{26/25}).
\end{equation}
Below $n^{1/10}$ the moments of the law control $\widehat\mu_W$ for every
well-conditioned sum (Lemma~\ref{cu:charfun}), and above $n^{26/25}=n^{1/100}/\delta'$
the smoothing at the width $\delta'$ takes over.

\emph{Hypotheses.} Let $\xi_0$ be a symmetric real law with $\E\xi_0^2=1$
and finite moments of every order. For each $n$ let a set of separated
pairs be given, called \emph{good}, and let $b_n$ be the number of
separated pairs that are not good. The hypotheses are the following, for
$n\ge n_0$ and a constant $C_U$.
\begin{enumerate}[label=\textup{(U\arabic*)},ref=\textup{(U\arabic*)},leftmargin=*]
\item\label{cu:U1} \emph{Coarse small balls.} Every well-conditioned sum $W$
of dimension $d\le12$ satisfies
\[
 \Prob(|W-w|\le r)\le C_U(r+n^{-1/2})^d\qquad(w\in\R^d,\ r>0).
\]
\item\label{cu:U2} \emph{Decay on good pairs.} If $(i,i')$ is good, then
$\widehat\mu_{Z_n(x)}$ decays for $x\in\bar I_i\cup\bar I_{i'}$, and
$\widehat\mu_{(Z_n(s),Z_n(t))}$ decays for $s\in\bar I_i$ and
$t\in\bar I_{i'}$.
\item\label{cu:U3} \emph{Roots on the circle.} $\E B_n^2\le C_Un^{1/2}$.
\end{enumerate}

\begin{proposition}[Unsmoothing]\label{cu:unsmooth}
Let $\xi_0$, the good pairs and $b_n$ be as above, and assume
\ref{cu:U1}--\ref{cu:U3}. There are $C$ and $n_0$, depending only on the
law of $\xi_0$, on $\chi$, on $C_U$ and on the threshold of the hypotheses, such that
for $n\ge n_0$
\[
 \E\bigl(D_n-J_{n,\delta,M}(\xi)\bigr)^2\le C\bigl(n^{199/200}+n^{-1/200}b_n\bigr).
\]
\end{proposition}

For laws with a bounded density every separated pair is good, so
$b_n=0$ (Corollary~\ref{cu:bd-unsmoothing}). For random signs
$b_n\le Cn^{1+1/400}$ (Lemma~\ref{sv:counting}).

\subsection{Characteristic functions and distant cells}\label{cu:fourier-section}
We turn the decay of characteristic functions into small balls at the
fine width and into the near independence of distant cells. Throughout
the subsection $\xi_0$ is symmetric with $\E\xi_0^2=1$ and finite moments
of every order, and $\kappa_4=\E\xi_0^4-3$.

\begin{lemma}[Characteristic functions of well-conditioned sums]\label{cu:charfun}
There are $u_0\in(0,1]$, $C$ and $n_0$ such that for $n\ge n_0$ every
well-conditioned sum $W=\sum_j\xi_ju_j$, with $\Sigma_W=\sum_ju_ju_j^T$, satisfies
the following.
\begin{enumerate}[label=\textup{(\alph*)}]
 \item $|\widehat\mu_W(\theta)|\le e^{-\gamma_*|\theta|^2/8}$ for
 $|\theta|\le u_0\sqrt n/7$.
 \item With $B(\theta)=\frac{\kappa_4}{24}\sum_j(u_j\cdot\theta)^4$, for
 $|\theta|\le n^{1/10}$,
 \[
  \bigl|\widehat\mu_W(\theta)-e^{-\theta^T\Sigma_W\theta/2}\bigl(1+B(\theta)\bigr)\bigr|
  \le Cn^{-2}(1+|\theta|)^8e^{-\gamma_*|\theta|^2/4}.
 \]
\end{enumerate}
\end{lemma}

\begin{proof}
Let $f(u)=\E e^{iu\xi_0}$. It is real, even and six times differentiable,
and the cumulants of $\xi_0$ of orders one to five are $0,1,0,\kappa_4,0$.
So there is $u_0\in(0,1]$ such that on $[-u_0,u_0]$ we have $f\ge\frac12$,
$f(u)\le e^{-u^2/4}$, and $E_0(u)=\log f(u)+u^2/2-\kappa_4u^4/24$
satisfies $|E_0(u)|\le Cu^6$.

(a) Here $|\theta\cdot u_j|\le7|\theta|n^{-1/2}\le u_0$, so
$|\widehat\mu_W(\theta)|=\prod_jf(\theta\cdot u_j)\le e^{-\theta^T\Sigma_W\theta/4}\le e^{-\gamma_*|\theta|^2/8}$.

(b) For $n\ge n_0$ all $|u_j\cdot\theta|\le7n^{-2/5}\le u_0$, so
$\widehat\mu_W(\theta)=e^{-\theta^T\Sigma_W\theta/2}e^{B(\theta)+E(\theta)}$
with $E(\theta)=\sum_jE_0(u_j\cdot\theta)$. Since
$\sum_j|u_j|^4\le Cn^{-1}$ and $\sum_j|u_j|^6\le Cn^{-2}$,
\[
 |B(\theta)|\le Cn^{-1}|\theta|^4\le Cn^{-3/5},\qquad
 |E(\theta)|\le Cn^{-2}|\theta|^6\le Cn^{-7/5}.
\]
Hence $|e^{B+E}-1-B|\le|e^B-1-B|+e^{|B|}|e^E-1|\le C(|B|^2+|E|)\le Cn^{-2}(1+|\theta|)^8$,
and $\theta^T\Sigma_W\theta\ge\gamma_*|\theta|^2/2$ gives (b).
\end{proof}

\begin{lemma}[Small balls from decay]\label{cu:smallball}
Let $n\ge n_0$, and let $W$ be a well-conditioned sum of dimension $d\le12$
such that $\widehat\mu_W$ decays. Then
\[
 \Prob(|W-w|\le r)\le Cr^d\qquad(w\in\R^d,\ r\ge n^{-26/25}).
\]
\end{lemma}

\begin{proof}
By Esseen's inequality, for every random vector $W$ in $\R^d$,
$w\in\R^d$ and $r>0$,
\begin{equation}\label{cu:esseen}
 \Prob(|W-w|\le r)\le C_dr^d\int_{|\theta|\le1/r}|\widehat\mu_W(\theta)|\dd\theta .
\end{equation}
To prove it, let $b$ be the indicator of the ball of radius $\frac12$
about the origin, $V_d=\int b$, and
$H(x)=\int(b*b)(\theta)e^{i\theta\cdot x}\dd\theta=(\int b(y)e^{iy\cdot x}\dd y)^2$.
It is real, nonnegative and continuous with $H(0)=V_d^2$, so
$H\ge V_d^2/2$ on a ball of radius $c_d\in(0,1]$. The function $b*b$ is
supported in the unit ball and bounded by $V_d$. For $r'>0$, Fubini's
theorem gives
\[
 \E H\Bigl(\frac{W-w}{r'}\Bigr)=\int(b*b)(\theta)\,\E e^{i\theta\cdot(W-w)/r'}\dd\theta
 \le V_dr'^d\int_{|\theta|\le1/r'}|\widehat\mu_W(\theta)|\dd\theta .
\]
With $r'=r/c_d$, $\Prob(|W-w|\le r)\le(2/V_d^2)\E H((W-w)/r')$, and
$\{|\theta|\le c_d/r\}\subset\{|\theta|\le1/r\}$.

By \eqref{cu:esseen} it suffices to bound $\int_{|\theta|\le1/r}|\widehat\mu_W|$
by a constant. On $|\theta|\le n^{1/10}$, Lemma~\ref{cu:charfun}(a)
bounds the integrand by $e^{-\gamma_*|\theta|^2/8}$. On
$n^{1/10}\le|\theta|\le1/r\le n^{26/25}$, the decay \eqref{cu:decay}
bounds it by $n^{-20}$, and this range has volume at most
$Cn^{26d/25}\le Cn^{13}$. Both parts contribute at most $C$.
\end{proof}

For the rest of this subsection let $s,t\in[1,\pi n-1]$ with
$|s-t|\ge R_0$, let $u_j=(v_j(s),v_j(t))$ be the rows of
$W=(Z_n(s),Z_n(t))$, let $S=\sum_ju_ju_j^T$ be the matrix $S_1$ of
Lemma~\ref{cu:window}(iii), and let $\varphi_S$ be the centered normal
density on $\R^6$ with covariance $S$. The expansion of
Lemma~\ref{cu:charfun}(b) is the Fourier transform of
\[
 p_S=\varphi_S+\frac{\kappa_4}{24}\sum_{j=0}^n(u_j\cdot\nabla)^4\varphi_S ,
\]
and decay extends it to all frequencies that a smooth test function
sees.

\begin{lemma}[Edgeworth expansion against test functions]\label{cu:edgeworth}
Let $n\ge n_0$, and suppose that $\widehat\mu_W$ decays. Let $A\ge0$, let
$\Theta$ be measurable on $\R^6$ with
$|\Theta(\theta)|\le Ae^{-\delta'^2|\theta|^2/2}$, and let
$F(\zeta)=(2\pi)^{-6}\int\Theta(\theta)e^{i\theta\cdot\zeta}\dd\theta$. Then
\[
 \Bigl|\E F(W)-\int F\,p_S\Bigr|\le CAn^{-2}.
\]
\end{lemma}

\begin{proof}
The function $p_S$ is integrable, so Fubini's theorem gives
$\E F(W)=(2\pi)^{-6}\int\Theta(\theta)\widehat\mu_W(\theta)\dd\theta$ and
$\int Fp_S=(2\pi)^{-6}\int\Theta(\theta)\tilde p(\theta)\dd\theta$, where,
integrating by parts four times,
\[
 \tilde p(\theta)=\int p_S(z)e^{i\theta\cdot z}\dd z
 =e^{-\theta^TS\theta/2}\Bigl(1+\frac{\kappa_4}{24}\sum_j(u_j\cdot\theta)^4\Bigr).
\]
So the difference is at most
$(2\pi)^{-6}\int|\Theta||\widehat\mu_W-\tilde p|$. Put
$y=\theta^TS\theta=\sum_j(u_j\cdot\theta)^2$. Since
$\sum_j(u_j\cdot\theta)^4\le y^2$,
\begin{equation}\label{cu:ptilde}
 |\tilde p(\theta)|\le e^{-y/2}\Bigl(1+\frac{|\kappa_4|y^2}{24}\Bigr)\le Ce^{-y/4},
 \qquad y\ge\frac{\gamma_*}2|\theta|^2 .
\end{equation}
Put $T=n^{26/25}$. On $|\theta|\le n^{1/10}$, Lemma~\ref{cu:charfun}(b)
gives $|\widehat\mu_W-\tilde p|\le Cn^{-2}(1+|\theta|)^8e^{-\gamma_*|\theta|^2/4}$,
and the integral over this range is at most $CAn^{-2}$. On
$n^{1/10}<|\theta|\le T$, the decay gives $|\widehat\mu_W|\le n^{-20}$,
and $|\tilde p|\le Ce^{-\gamma_*n^{1/5}/8}$ by \eqref{cu:ptilde}. This
range has volume at most $Cn^{156/25}$, so the integral is at most
$CAn^{156/25}(n^{-20}+e^{-\gamma_*n^{1/5}/8})\le CAn^{-2}$. On
$|\theta|>T$, $|\widehat\mu_W-\tilde p|\le C$ by \eqref{cu:ptilde}, and
since $\delta'T=n^{1/100}$,
\[
 \int_{|\theta|>T}|\Theta|\dd\theta
 \le Ae^{-\delta'^2T^2/4}\int e^{-\delta'^2|\theta|^2/4}\dd\theta
 =(4\pi)^3A\delta'^{-6}e^{-n^{1/50}/4}\le An^{-2}.
 \qedhere
\]
\end{proof}

Lemma~\ref{cu:edgeworth} and a truncation of $\psi\otimes\psi$ will
reduce $\E\,\psi(Z_n(s))\psi(Z_n(t))$ to the Gaussian integral
$\int\psi\otimes\psi\,p_S$. This integral is small because $\psi$ is odd
in its first coordinate while the Gaussian coordinates at $s$ and $t$
are nearly independent.

\begin{lemma}[Two-point Gaussian integrals]\label{cu:two-point}
Let $n\ge n_0$, and let $s,t$, $S$, $u_j$ and $\varphi_S$ be as above.
\begin{enumerate}[label=\textup{(\alph*)}]
 \item $\bigl|\iint\psi(z)\psi(z')\varphi_S(z,z')\dd z\dd z'\bigr|\le C\varepsilon_{\delta,M}(1+|s-t|)^{-2}$.
 \item $\bigl|\iint\psi(z)\psi(z')(u_j\cdot\nabla)^4\varphi_S(z,z')\dd z\dd z'\bigr|\le C\varepsilon_{\delta,M}|u_j|^4$ for every $j$.
 \item $|\partial^\beta\varphi_S(\zeta)|\le C_\beta e^{-|\zeta|^2/20}$ for every multi-index $\beta$ and $\zeta\in\R^6$.
 \item $\int_{\R^3}|\psi(z)|(1+|z|)^{-10}\dd z\le C$.
\end{enumerate}
\end{lemma}

\begin{proof}
\emph{Kernel bounds.} Let $f$ be measurable on $\R^3$ with
$|f(z)|\le Ke^{-c|z|^2}$ and
$|f(u,v,w)-f(u,0,w)|\le K|v|e^{-c(u^2+w^2)}$. For $\kappa\in(0,1]$ and
$L\ge1$ we compare $\int h_\kappa(u)\rho_\kappa(v)\ell_L(w)f$ with
$\iint\sgn(u)\,w\,f(u,0,w)\dd u\dd w$ in three steps. Replacing
$\ell_L(w)$ by $w$ costs at most $CKL^{-8}$, since
$|\ell_L(w)-w|\le|w|^9L^{-8}$. Replacing the average over
$\rho_\kappa(v)\dd v$ by the value at $v=0$ costs at most
$CK\int|v|\rho_\kappa(v)\dd v\le CK\kappa$. Replacing $h_\kappa$ by
$\sgn$ costs at most $CK\int|h_\kappa-\sgn|=CK\kappa$. Applying this to
both terms of $\psi$, with $\delta'\le\delta$,
\begin{equation}\label{cu:test}
 \Bigl|\int_{\R^3}\psi f\Bigr|\le C_cK\varepsilon_{\delta,M},\qquad
 \int_{\R^3}|\psi(z)|\bigl(e^{-c|z|^2}+(1+|z|)^{-10}\bigr)\dd z\le C_c,
\end{equation}
the second because $|\psi(u,v,w)|\le(\rho_{\delta'}(v)+\rho_\delta(v))|w|$.
The second bound gives (d).

\emph{Gaussian envelopes.} If $\lambda_0I\le S'\le\lambda_1I$ in dimension
$r$, then $\partial^\beta\varphi_{S'}=P_\beta\varphi_{S'}$ for a polynomial
$P_\beta$ of degree $|\beta|$ whose coefficients are bounded in terms of
$\|S'^{-1}\|\le\lambda_0^{-1}$, and
$\varphi_{S'}(\zeta)\le(2\pi\lambda_0)^{-r/2}e^{-|\zeta|^2/2\lambda_1}$.
Hence $|\partial^\beta\varphi_{S'}(\zeta)|\le C_\beta e^{-|\zeta|^2/4\lambda_1}$,
uniformly over such $S'$. By Lemma~\ref{cu:window}(iii) the matrices
$S_{\tau'}$, $0\le\tau'\le1$, satisfy $\frac{\gamma_*}2I\le S_{\tau'}\le5I$,
which gives (c). Combining these envelopes with the first bound of
\eqref{cu:test} in the first block and the second bound in the other
block, for $|\beta|=4$ and $0\le\tau'\le1$,
\begin{equation}\label{cu:two-block}
 \Bigl|\iint\psi(z)\psi(z')\,\partial^\beta\varphi_{S_{\tau'}}(z,z')\dd z\dd z'\Bigr|\le C\varepsilon_{\delta,M}.
\end{equation}

\emph{Gaussian term.} Write $\Gamma=\Gamma(s,t)$. In Fourier variables
$(\theta,\theta')$,
$\partial_{\tau'}e^{-(\theta,\theta')S_{\tau'}(\theta,\theta')^T/2}
=-\theta^T\Gamma\theta'e^{-(\theta,\theta')S_{\tau'}(\theta,\theta')^T/2}$, and
$-\theta_p\theta'_q$ corresponds to $\partial_{z_p}\partial_{z'_q}$.
Differentiating the inverse Fourier integral, which is justified by the
uniform Gaussian decay,
\[
 \partial_{\tau'}\varphi_{S_{\tau'}}=\sum_{p,q}\Gamma_{pq}\partial_{z_p}\partial_{z'_q}\varphi_{S_{\tau'}},\qquad
 \partial_{\tau'}^2\varphi_{S_{\tau'}}
 =\sum_{p,q,p',q'}\Gamma_{pq}\Gamma_{p'q'}\partial_{z_p}\partial_{z_{p'}}\partial_{z'_q}\partial_{z'_{q'}}\varphi_{S_{\tau'}}.
\]
Let $\gamma(\tau')=\iint\psi(z)\psi(z')\varphi_{S_{\tau'}}(z,z')$. At
$\tau'=0$ the density factors. Since $\Sigma_{12}=\Sigma_{13}=0$
(Lemma~\ref{cu:geometry}(ii)), $\varphi_{\Sigma(s)}$ is even in its first
coordinate, while $\psi$ is odd in it by Lemma~\ref{cu:jets}(ii), so
$\gamma(0)=0$. Since $\psi$ is even and $\partial_p\varphi_{\Sigma}$ is
odd, $\gamma'(0)=0$. By \eqref{cu:two-block},
$|\gamma''(\tau')|\le C\varepsilon_{\delta,M}\|\Gamma\|^2$, and
$\|\Gamma\|\le C(1+|s-t|)^{-1}$ by Lemma~\ref{cu:window}(ii). Taylor's
formula gives (a).

\emph{Quartic term.} Expanding $(u_j\cdot\nabla)^4$ into fourth partial
derivatives, \eqref{cu:two-block} with $\tau'=1$ gives (b).
\end{proof}

To apply Lemma~\ref{cu:edgeworth} to $\psi\otimes\psi$, we replace $\psi$
by a function whose Fourier transform has Gaussian decay. Let
$\sigma=\delta'$ and $L=n^{1/5}$, and for $\kappa>0$ put
\[
 \tilde h_{\kappa,L}(u)=\E\bigl[\sgn(u-\kappa G_0)\ind_{|u-\kappa G_0|\le L}\bigr],\qquad
 \tilde\ell_{\sigma,M}(w)=\E\,\ell_M(w-\sigma G_0).
\]
Thus $\tilde h_{\kappa,L}=g_L*\rho_\kappa$ with $g_L=\sgn\cdot\ind_{[-L,L]}$,
and $\tilde\ell_{\sigma,M}=\ell_M*\rho_\sigma$. Also
$|\tilde h_{\kappa,L}|\le1$ and $|\tilde\ell_{\sigma,M}|\le2M$. Since
$h_\kappa(u)=\E\sgn(u-\kappa G_0)$,
\begin{gather}
 |h_\kappa(u)-\tilde h_{\kappa,L}(u)|\le\ind_{|u|\ge L/2}+\Prob\bigl(\kappa|G_0|>L/2\bigr),\label{cu:trunc}\\
 |\ell_M(w)-\tilde\ell_{\sigma,M}(w)|\le C\sigma^2/M,\qquad
 |\ell_M(w)-\tilde\ell_{\sigma,M}(w)|\le48\sigma^4/M^3\quad(|w|\le M/2).\label{cu:wcut}
\end{gather}
Indeed, $\sgn y$ and $\sgn y\,\ind_{|y|\le L}$ differ by at most one, and
only for $|y|>L$, which for $|u|<L/2$ and $y=u-\kappa G_0$ requires
$\kappa|G_0|>L/2$. Since $\E G_0=0$, Taylor's formula to second order
gives $|\tilde\ell_{\sigma,M}(w)-\ell_M(w)|\le\frac12\|\ell_M''\|_\infty\sigma^2\le C\sigma^2/M$.
If $|w|\le M/2$, then $\ell_M(w)=w$ and
$\tilde\ell_{\sigma,M}(w)-w=\E[\ell_M(y)-y]$ with $y=w-\sigma G_0$. The
integrand vanishes unless $|y|>M$, which forces $2\sigma|G_0|>M$, and
then $|\ell_M(y)-y|\le|y|<2\sigma|G_0|\le16\sigma^4G_0^4/M^3$. As
$\E G_0^4=3$, this proves \eqref{cu:wcut}. Put
\[
 \tilde\psi(u,v,w)=\bigl(\tilde h_{\delta',L}(u)\rho_{\delta'}(v)-\tilde h_{\delta,L}(u)\rho_\delta(v)\bigr)\tilde\ell_{\sigma,M}(w),
 \qquad F(z,z')=\tilde\psi(z)\tilde\psi(z').
\]
Then there is an integrable $\Theta$ on $\R^6$ with
\begin{equation}\label{cu:testfn}
 |\Theta(\theta)|\le256L^2M^4e^{-\delta'^2|\theta|^2/2},\qquad
 F(\zeta)=(2\pi)^{-6}\int\Theta(\theta)e^{i\theta\cdot\zeta}\dd\theta\quad(\zeta\in\R^6).
\end{equation}
Indeed, with $\hat g(\theta)=\int g(x)e^{-i\theta\cdot x}\dd x$, we have
$|\hat g_L|\le\|g_L\|_{L^1}=2L$,
$|\hat\ell_M|\le\int_{|w|\le2M}|w|\dd w=4M^2$ and
$\hat\rho_\kappa(\theta)=e^{-\kappa^2\theta^2/2}$. So the continuous
functions $\tilde h_{\kappa,L}$, $\rho_\kappa$ and $\tilde\ell_{\sigma,M}$
are the inverse transforms of functions bounded by
$2Le^{-\delta'^2\theta^2/2}$, $e^{-\delta'^2\theta^2/2}$ and
$4M^2e^{-\delta'^2\theta^2/2}$, because $\kappa,\sigma\ge\delta'$. The
function $F$ is a sum of four products of six such functions, one in each
coordinate, so it is the inverse transform of a sum $\Theta$ of four
products of transforms, each at most
$(2L\cdot4M^2)^2e^{-\delta'^2|\theta|^2/2}$.

\begin{lemma}[Distant cells]\label{cu:far}
Let $n\ge n_0$, let $s,t$ and $W$ be as above, and suppose that
$\widehat\mu_W$ decays. Then
\[
 \bigl|\E\,\psi(Z_n(s))\psi(Z_n(t))\bigr|
 \le C\varepsilon_{\delta,M}\bigl((1+|s-t|)^{-2}+n^{-1}\bigr)+Cn^{-3/2}.
\]
\end{lemma}

\begin{proof}
Write $z=(u,v,w)$ and $z'=(u',v',w')$. We split
\[
 \E\,\psi\otimes\psi(W)=\E(\psi\otimes\psi-F)(W)+\Bigl(\E F(W)-\int Fp_S\Bigr)
 +\int(F-\psi\otimes\psi)\,p_S+\int\psi\otimes\psi\,p_S .
\]

\emph{Gaussian main term.} The last term equals
$\int\psi\otimes\psi\,\varphi_S+\frac{\kappa_4}{24}\sum_j\int\psi\otimes\psi\,(u_j\cdot\nabla)^4\varphi_S$,
and Lemma~\ref{cu:two-point}(a),(b) bounds its two parts by
$C\varepsilon_{\delta,M}(1+|s-t|)^{-2}$ and
$C\varepsilon_{\delta,M}\sum_j|u_j|^4\le C\varepsilon_{\delta,M}/n$.

\emph{Edgeworth term.} By Lemma~\ref{cu:edgeworth} and \eqref{cu:testfn},
$|\E F(W)-\int Fp_S|\le CL^2M^4n^{-2}=Cn^{-79/50}$.

\emph{Truncation terms.} Let $r=\rho_{\delta'}+\rho_\delta\le1/\delta'$,
$a_L(u)=\ind_{|u|\ge L/2}+\Prob(\delta|G_0|>L/2)$ and
$d(w)=|\ell_M(w)-\tilde\ell_{\sigma,M}(w)|$. Since $\delta'\le\delta$,
\eqref{cu:trunc} gives $|h_\eta-\tilde h_{\eta,L}|\le a_L$ for
$\eta\in\{\delta',\delta\}$, and therefore
\[
 |\psi-\tilde\psi|(u,v,w)\le r(v)\bigl(2Ma_L(u)+d(w)\bigr),\qquad
 |\psi|,|\tilde\psi|\le2Mr(v).
\]
Hence
\[
 |\psi\otimes\psi-F|(z,z')\le2M\,r(v)r(v')\bigl(2Ma_L(u)+d(w)+2Ma_L(u')+d(w')\bigr).
\]
In the first term we use the moments of the fields. By
Lemma~\ref{cu:moments} with $p=1000$ and Markov's inequality,
$\Prob(|X_n(s)|\ge L/2)\le CL^{-1000}$ and
$\Prob(|Y_n'(s)|>M/2)\le CM^{-1000}=Cn^{-5}$, and likewise at $t$. By
\eqref{cu:wcut}, $d(w)\le48\sigma^4/M^3+4M\ind_{|w|>M/2}$. With
$r\le1/\delta'$ and $\Prob(\delta|G_0|>L/2)\le\sqrt2e^{-L^2/(16\delta^2)}$,
the first term is at most
\[
 \frac{4M}{\delta'^2}\Bigl(2M\bigl(CL^{-1000}+\sqrt2e^{-L^2/(16\delta^2)}\bigr)
 +\frac{48\sigma^4}{M^3}+4CMn^{-5}\Bigr)\le Cn^{-2},
\]
since $\sigma=\delta'$ and $M^2\delta'^{-2}=n^{207/100}$. In the third
term we use the decay of $p_S$. By Lemma~\ref{cu:two-point}(c) and
$\sum_j|u_j|^4\le C/n$, we have
$|p_S(\zeta)|\le Ce^{-|\zeta|^2/20}\le C(1+|\zeta|)^{-28}$ on $\R^6$, and
\[
 (1+|\zeta|)^{-28}\le(1+u^2)^{-5}(1+w^2)^{-1}(1+u'^2)^{-5}(1+w'^2)^{-1}.
\]
We use $d\le C\sigma^2/M$, $\int r=2$,
$\int_{|u|\ge L/2}(1+u^2)^{-5}\dd u\le\pi(1+L^2/4)^{-4}$,
$\int(1+u^2)^{-5}\dd u\le\pi$ and $\int(1+w^2)^{-1}\dd w=\pi$. This bounds
the third term by $C(M^2L^{-8}+M^2e^{-L^2/(16\delta^2)}+\sigma^2)\le Cn^{-3/2}$,
since $M^2L^{-8}=n^{-159/100}$.
\end{proof}

\subsection{Proof of the unsmoothing proposition}\label{cu:proof-section}
We prove Proposition~\ref{cu:unsmooth} cell by cell. The coarse small
balls keep the local zero counts logarithmic, and this controls every
cell. In this subsection $\xi_0$, the good pairs and $b_n$ are as in
Proposition~\ref{cu:unsmooth}, and \ref{cu:U1}--\ref{cu:U3} hold. We
write $\E[V;A]=\E[V\ind_A]$. For a real function $F$ on $\R$ and
$c\in\R$, let $N_F(c)\in[0,\infty]$ be the number of distinct zeros of
$F$ in $[c-2,c+2]$.

\begin{lemma}[Local zero counts]\label{cu:zeros}
For $n\ge n_0$, $c\in[1,\pi n-1]$ and $F\in\{Y_n,Y_n'\}$,
$\Prob\bigl(N_F(c)>k_n\bigr)\le Cn^{-6}$.
\end{lemma}

\begin{proof}
Let $r=12$ and $L_1=\max(1,24\pi r/\gamma_*)$. Put $x_l=c\pm lL_1$ for
$0\le l<r$, with the sign $+$ if $c\le\pi n/2$. For $n\ge n_0$ all $x_l$
lie in $[1,\pi n-1]$, at mutual distances at least $L_1$. Let
$W=(F(x_0),\ldots,F(x_{r-1}))=\sum_j\xi_ju_j$. Each coordinate of
$u_j$ is a coordinate of some $v_j(x_l)$, so
$|u_j|\le2\sqrt r\,n^{-1/2}\le7n^{-1/2}$. The matrix $\sum_ju_ju_j^T$
has the entries $\Gamma_{22}(x_l,x_{l'})$ for $F=Y_n$, and
$\Gamma_{33}(x_l,x_{l'})$ for $F=Y_n'$. By Lemma~\ref{cu:window}(i),(ii) its diagonal
entries are at least $\gamma_*$ and its other entries are at most
$12\pi/L_1$ in absolute value. So for $q\in\R^r$ its quadratic form is
at least
\[
 \gamma_*|q|^2-\frac{12\pi}{L_1}\Bigl(\sum_l|q_l|\Bigr)^2
 \ge\Bigl(\gamma_*-\frac{12\pi r}{L_1}\Bigr)|q|^2\ge\frac{\gamma_*}2|q|^2,
\]
so $W$ is well conditioned, and \ref{cu:U1} applies to it in dimension $12$.

Let $R_1=rL_1+2$ and $M_*=\sup_{|z-c|\le3R_1}|F(z)|$. For each $l$,
$[c-2,c+2]$ lies in the disk $|z-x_l|\le R_1$, and the disk
$|z-x_l|\le2R_1$ lies in $|z-c|\le3R_1$. If $F(x_l)\ne0$, the local
Jensen bound (Lemma~\ref{pre:local-jensen}) with radii $R_1<2R_1$ gives
\[
 N_F(c)\le\frac{\log M_*-\log|F(x_l)|}{\log2}.
\]
So $N_F(c)\le2\log n/\log2=k_n$ if $M_*<n$ and $|F(x_l)|>1/n$ for
some $l$, and therefore
\[
 \{N_F(c)>k_n\}\subset\{M_*\ge n\}\cup\{|W|\le\sqrt{12}/n\}.
\]
By Lemma~\ref{cu:moments} with $p=6$ and Markov's inequality, the first
event has probability at most $Cn^{-6}$. By \ref{cu:U1}, the second has
probability at most $C(\sqrt{12}/n+n^{-1/2})^{12}\le Cn^{-6}$.
\end{proof}

Let $\mathcal C=\{c_0,\ldots,c_{m-1},1,\pi n-1\}$, and let $\mathcal L$ be
the event that $N_{Y_n}(c)\le k_n$ and $N_{Y_n'}(c)\le k_n$ for
every $c\in\mathcal C$. Let $\mathcal L_i\supset\mathcal L$ be the same
event for $c=c_i$ alone. By Lemma~\ref{cu:zeros},
\begin{equation}\label{cu:L}
 \Prob(\mathcal L^c)\le2(m+2)\,Cn^{-6}\le Cn^{-5}.
\end{equation}
On $\mathcal L_i$, and hence on $\mathcal L$, $Y_n\not\equiv0$.

The current on an interval is controlled by the zeros of $Y_n'$. Let
$\kappa>0$, and let $I$ be a compact interval on which $Y_n'$ has finitely
many zeros. Then $Y_n$ is strictly monotone between consecutive zeros of
$Y_n'$, and the change of variables $v=Y_n(x)$ on each such piece, with
the bound $|\Psi_{\kappa,M}(u,v,w)|\le\rho_\kappa(v)|w|$ of
Lemma~\ref{cu:jets}(i), gives the occupation bound
\begin{equation}\label{cu:occupation}
 \Bigl|\int_I\Psi_{\kappa,M}(Z_n(x))\dd x\Bigr|
 \le\int_I\rho_\kappa(Y_n)|Y_n'|\dd x\le1+\#\{x\in I:Y_n'(x)=0\}.
\end{equation}
Always $|\Psi_{\kappa,M}|\le\rho_\kappa(v)|\ell_M(w)|\le2M/\kappa$, since
$\rho_\kappa\le1/\kappa$ and $|\ell_M|\le2M$, and cells have length at
most one. On $\mathcal L_i$, $Y_n$ and $Y_n'$ have at most $k_n$
zeros in $\bar I_i\subset(c_i-2,c_i+2)$, so $|K_{I_i}|\le k_n$ and
$|J_{I_i}^\kappa|\le k_n+1$ by \eqref{cu:occupation}. Hence
\begin{equation}\label{cu:crude}
 |\Delta_i^{\rm c}|\le\frac{4M}{\delta'},\qquad
 |\Delta_i|,\ |\Delta_i^{\rm f}|\le2k_n+1\ \text{ on }\mathcal L_i .
\end{equation}

On a cell where $X_n$ and $Y_n$ are never small together, an integration
by parts compares the crossing sum with the current. Zeros of any
multiplicity are allowed.

\begin{lemma}[Cell identity]\label{cu:cell-identity}
Let $I=[\alpha,\beta]$ with $\beta-\alpha\le1$, let $X,Y$ be real $C^1$
functions on $I$, let $Y$ have a finite set $\mathcal Z$ of zeros in $I$,
and let $\kappa,y_*,b,M>0$. Suppose that
\begin{enumerate}[label=\textup{(\roman*)}]
 \item $|X'|\le b$ and $|Y'|\le M$ on $I$,
 \item $|Y(\alpha)|>y_*$ and $|Y(\beta)|>y_*$,
 \item every $x\in I$ has $|X(x)|>y_*$ or $|Y(x)|>y_*$.
\end{enumerate}
Then, with $K=\sum_{t\in\mathcal Z}\sgn X(t)\,\iota_Y(t)$,
\[
 \Bigl|K-\int_I\Psi_{\kappa,M}(X,Y,Y')\dd x\Bigr|
 \le2e^{-y_*^2/(2\kappa^2)}\Bigl(\#\mathcal Z+1+\frac b\kappa\Bigr).
\]
\end{lemma}

\begin{proof}
Put $e_*=e^{-y_*^2/(2\kappa^2)}$ and $u=h_\kappa(X)$, so that $u'=2\rho_\kappa(X)X'$.
For $y\ne0$, $|h_\kappa(y)-\sgn y|=2(1-\Phi(|y|/\kappa))\le e^{-y^2/(2\kappa^2)}$,
and $|h_\kappa(y)-\sgn y|\le1$ for all $y$.

\emph{The crossing sum.} The function $\sgn Y$ is constant between
consecutive points of $\mathcal Z$, none of which is an endpoint by (ii),
and jumps by $2\iota_Y(t)$ at $t\in\mathcal Z$. Summation by parts against
the $C^1$ function $u$ gives
\[
 \sum_{t\in\mathcal Z}u(t)\,\iota_Y(t)=\tfrac12\bigl[u\sgn Y\bigr]_\alpha^\beta
 -\tfrac12\int_Iu'\sgn Y\dd x .
\]
At each $t\in\mathcal Z$, $|X(t)|>y_*$ by (iii), so
$|K-\sum_tu(t)\iota_Y(t)|\le e_*\#\mathcal Z$.

\emph{The current.} By (i), $\ell_M(Y')=Y'$ on $I$, and
$\rho_\kappa(Y)Y'=\frac12\frac{\dd}{\dd x}h_\kappa(Y)$. Integration by parts gives
\[
 \int_I\Psi_{\kappa,M}(X,Y,Y')\dd x=\tfrac12\bigl[u\,h_\kappa(Y)\bigr]_\alpha^\beta
 -\tfrac12\int_Iu'h_\kappa(Y)\dd x .
\]
The difference of the two displays is
$\frac12[u(h_\kappa(Y)-\sgn Y)]_\alpha^\beta-\frac12\int_Iu'(h_\kappa(Y)-\sgn Y)$.
By (ii) the boundary term is at most $e_*$. Where $|X|>y_*$,
$|u'|\le2b\kappa^{-1}\varphi(y_*/\kappa)\le b\kappa^{-1}e_*$. Where $|X|\le y_*$,
$|Y|>y_*$ by (iii), so $|h_\kappa(Y)-\sgn Y|\le e_*$, while
$|u'|\le2b\rho_\kappa(0)\le b/\kappa$. Since $|I|\le1$, the integral term is
at most $b\kappa^{-1}e_*$. Altogether
$|K-\int_I\Psi_{\kappa,M}|\le e_*(\#\mathcal Z+1+b/\kappa)$.
\end{proof}

The cell identity applies on the following event. For a cell $I$ with
closure $[\alpha,\beta]$ and $y_*,b>0$, let $q=\lceil b/y_*\rceil$, let
$x_k=\alpha+k(\beta-\alpha)/q$ for $0\le k\le q$, and let
$\mathcal G_I(y_*,b)$ be the event that
\begin{enumerate}[label=\textup{(G\arabic*)}]
 \item $|X_n'|\le b$, $|Y_n'|\le b$ and $|Y_n'|\le M$ on $[\alpha,\beta]$,
 \item $|Y_n(\alpha)|>y_*$ and $|Y_n(\beta)|>y_*$,
 \item $|X_n(x_k)|>2y_*$ or $|Y_n(x_k)|>2y_*$ for $0\le k\le q$.
\end{enumerate}
On (G1) and (G3), every $x\in[\alpha,\beta]$ has $|X_n(x)|>y_*$ or
$|Y_n(x)|>y_*$. Indeed, the nearest grid point $x_k$ has $|x_k-x|\le y_*/(2b)$,
so $|X_n(x)|\le y_*$ and $|Y_n(x)|\le y_*$ would give $|X_n(x_k)|\le\frac32y_*$
and $|Y_n(x_k)|\le\frac32y_*$. The suprema of $|X_n'|$ and $|Y_n'|$ on
$[\alpha,\beta]$ are at most their suprema on the disk $|z-\alpha|\le1$, so
Lemma~\ref{cu:moments} with $p=204$ and Markov's inequality give
\begin{equation}\label{cu:cell-event}
 \Prob\bigl(\mathcal G_I(y_*,b)^c\bigr)\le C\bigl(b^{-204}+M^{-204}\bigr)
 +\sum_{x\in\{\alpha,\beta\}}\Prob\bigl(|Y_n(x)|\le y_*\bigr)
 +\sum_{k=0}^q\Prob\bigl(|(X_n,Y_n)(x_k)|\le2\sqrt2\,y_*\bigr).
\end{equation}
On $\mathcal L_i\cap\mathcal G_{I_i}(y_*,b)$, Lemma~\ref{cu:cell-identity}
applies to $X_n,Y_n$ on $\bar I_i$ with $\#\mathcal Z\le k_n$, and its
sum $K$ is $K_{I_i}$ by (G2). On $\mathcal L_i\setminus\mathcal G_{I_i}(y_*,b)$,
\eqref{cu:crude} applies. Hence, for $\kappa\in\{\delta,\delta'\}$,
\begin{equation}\label{cu:local-error}
 \E\bigl[|K_{I_i}-J_{I_i}^\kappa|;\mathcal L_i\bigr]
 \le2e^{-y_*^2/(2\kappa^2)}\Bigl(k_n+1+\frac b\kappa\Bigr)
 +(2k_n+1)\,\Prob\bigl(\mathcal G_{I_i}(y_*,b)^c\bigr).
\end{equation}

\begin{lemma}[Coarse and fine cells]\label{cu:cells}
Let $n\ge n_0$.
\begin{enumerate}[label=\textup{(\alph*)}]
 \item For every cell $I_i$, $\E[\Delta_i^2;\mathcal L_i]\le Cn^{-1/200}$.
 \item For every cell $I_i$ of a good pair,
 $\E[|\Delta_i^{\rm f}|;\mathcal L_i]\le Cn^{-126/125}$.
\end{enumerate}
\end{lemma}

\begin{proof}
(a) Take $y_*=n^{-1/50}$ and $b=n^{1/100}$, so that $q+1\le3n^{3/100}$,
$n^{-1/2}\le y_*/2$, $y_*^2/\delta^2=n^{1/25}$ and $b/\delta=n^{1/20}$. By
\eqref{cu:cell-event} and \ref{cu:U1},
\[
 \Prob(\mathcal G_{I_i}^c)\le C\bigl(n^{-51/25}+n^{-51/50}\bigr)+3Cy_*
 +3n^{3/100}\cdot C(4y_*)^2\le Cn^{-1/100}.
\]
By \eqref{cu:local-error},
$\E[|\Delta_i|;\mathcal L_i]\le2e^{-n^{1/25}/2}(k_n+1+n^{1/20})+C(2k_n+1)n^{-1/100}$.
On $\mathcal L_i$, $|\Delta_i|\le2k_n+1$ by \eqref{cu:crude}, so
$\E[\Delta_i^2;\mathcal L_i]\le(2k_n+1)\E[|\Delta_i|;\mathcal L_i]\le Cn^{-1/200}$.

(b) Take $y_*=n^{-51/50}$ and $b=n^{1/100}$, so that $q+1\le2n^{103/100}$,
$y_*^2/\delta'^2=n^{1/50}$ and $b/\delta'=n^{26/25}$. For $x\in\bar I_i$
the sums $Y_n(x)$ and $(X_n(x),Y_n(x))$ are well conditioned, and their
characteristic functions are restrictions of $\widehat\mu_{Z_n(x)}$,
which decays by \ref{cu:U2}. Since $y_*$ and $2\sqrt2y_*$ are at least
$n^{-26/25}$, Lemma~\ref{cu:smallball} bounds the probabilities in
\eqref{cu:cell-event}, and
\[
 \Prob(\mathcal G_{I_i}^c)\le C\bigl(n^{-51/25}+n^{-51/50}\bigr)+2Cy_*
 +2n^{103/100}\cdot8Cy_*^2\le Cn^{-101/100}.
\]
By \eqref{cu:local-error},
$\E[|\Delta_i^{\rm f}|;\mathcal L_i]\le2e^{-n^{1/50}/2}(k_n+1+n^{26/25})
+C(2k_n+1)n^{-101/100}\le Cn^{-126/125}$.
\end{proof}

\begin{proof}[Proof of Proposition~\textup{\ref{cu:unsmooth}}]
Let $n\ge n_0$ and $J=J_{n,\delta,M}(\xi)$.

\emph{Crossings.} On $\mathcal L$, $Y_n\not\equiv0$, and by the
crossing bound \eqref{pre:bulk-crossing} with $A=(0,\pi)$, rescaled to
$x=nt$, $D_n$ differs by at most $1+B_n$ from the sum of
$\sgn X_n(x)\,\iota_{Y_n}(x)$ over the zeros $x\in(0,\pi n)$ of $Y_n$.
This sum is $\sum_{i<m}K_{I_i}$ plus at most $2k_n$ terms of
absolute value at most one, from the zeros in $(0,1)$ and in
$[\pi n-1,\pi n)$, which lie in $[-1,3]$ and in $[\pi n-3,\pi n+1]$.
Hence, on $\mathcal L$,
\[
 \Bigl|D_n-\sum_{i<m}K_{I_i}\Bigr|\le1+B_n+2k_n .
\]

\emph{The split along $\mathcal L$.} Since $J=\sum_iJ_{I_i}^\delta$, on
$\mathcal L$ we have $D_n-J=(D_n-\sum_iK_{I_i})+\sum_i\Delta_i$. Always
$|D_n|\le n/2$, because $0\le\nu_n(1)\le n$, and $|J|\le2\pi nM/\delta$,
because $|\Psi_{\delta,M}|\le2M/\delta$. Hence
\[
 (D_n-J)^2\le\Bigl(\frac n2+\frac{2\pi nM}\delta\Bigr)^2\ind_{\mathcal L^c}
 +2(1+B_n+2k_n)^2+2\Bigl(\sum_i\Delta_i\Bigr)^2\ind_{\mathcal L}.
\]
By \eqref{cu:L}, the first term has expectation at most
$(n/2+2\pi nM/\delta)^2Cn^{-5}\le C$. By \ref{cu:U3}, the second has
expectation at most $6(1+\E B_n^2+4k_n^2)\le C\sqrt n$. It remains
to show
\[
 \sum_{i,i'}\E[\Delta_i\Delta_{i'};\mathcal L]\le C\bigl(n^{199/200}+n^{-1/200}b_n\bigr).
\]

\emph{Near pairs and pairs that are not good.} There are at most
$(2m_0+1)m\le Cn$ pairs with $|i-i'|\le m_0$, and $b_n$ separated pairs
that are not good. For each such pair,
$\mathcal L\subset\mathcal L_i\cap\mathcal L_{i'}$ and
$|\Delta_i\Delta_{i'}|\le\frac12(\Delta_i^2+\Delta_{i'}^2)$, so
Lemma~\ref{cu:cells}(a) gives
\[
 \bigl|\E[\Delta_i\Delta_{i'};\mathcal L]\bigr|
 \le\tfrac12\bigl(\E[\Delta_i^2;\mathcal L_i]+\E[\Delta_{i'}^2;\mathcal L_{i'}]\bigr)
 \le Cn^{-1/200}.
\]
Their total is at most $C(n^{199/200}+n^{-1/200}b_n)$.

\emph{Good pairs.} Let $(i,i')$ be a good pair. Pointwise,
\[
 \Delta_i\Delta_{i'}-\Delta^{\rm c}_i\Delta^{\rm c}_{i'}
 =\Delta^{\rm f}_i\Delta_{i'}+\Delta^{\rm c}_i\Delta^{\rm f}_{i'} .
\]
On $\mathcal L$, \eqref{cu:crude} gives $|\Delta_{i'}|\le2k_n+1$ and
$|\Delta^{\rm c}_i|\le|\Delta_i|+|\Delta^{\rm f}_i|\le2(2k_n+1)$.
Lemma~\ref{cu:cells}(b) applies to both cells and gives
\[
 \bigl|\E[\Delta_i\Delta_{i'};\mathcal L]-\E[\Delta^{\rm c}_i\Delta^{\rm c}_{i'};\mathcal L]\bigr|
 \le3(2k_n+1)\,Cn^{-126/125}.
\]
Over the at most $m^2\le10n^2$ pairs these errors sum to at most
$Cn^{124/125}\log n$. By \eqref{cu:crude} and \eqref{cu:L},
$|\E[\Delta^{\rm c}_i\Delta^{\rm c}_{i'};\mathcal L^c]|\le16M^2\delta'^{-2}Cn^{-5}\le Cn^{-2}$,
and these sum to at most $C$. By Fubini's theorem,
$\E\Delta^{\rm c}_i\Delta^{\rm c}_{i'}=\int_{I_i}\int_{I_{i'}}\E\,\psi(Z_n(s))\psi(Z_n(t))\dd s\dd t$.
For $s\in\bar I_i$ and $t\in\bar I_{i'}$, \ref{cu:U2} and
Lemma~\ref{cu:far} bound the integrand, and the sum over all good pairs
is at most
\[
 C\varepsilon_{\delta,M}\int_0^{\pi n}\!\!\int_\R\frac{\dd s\dd t}{(1+|s-t|)^2}
 +(\pi n)^2\Bigl(\frac{C\varepsilon_{\delta,M}}n+Cn^{-3/2}\Bigr)
 \le C\bigl(n\varepsilon_{\delta,M}+n^{1/2}\bigr)\le Cn^{24/25}.
\]

\emph{Conclusion.} Altogether
$\E(D_n-J)^2\le C(\sqrt n+n^{199/200}+n^{-1/200}b_n+n^{124/125}\log n+n^{24/25})
\le C(n^{199/200}+n^{-1/200}b_n)$.
\end{proof}

\subsection{Laws with a bounded density}\label{cu:bd-section}
A bounded density keeps the characteristic function of $\xi_0$ away from
modulus one outside every neighborhood of the origin, and then every
well-conditioned sum decays.

\begin{lemma}[Decay for bounded densities]\label{cu:bd-decay}
Let $\xi_0$ satisfy \textup{(BD)}, and let $u_0$ be as in
Lemma~\ref{cu:charfun}. There are $c>0$ and $n_0$ such that for
$n\ge n_0$ every well-conditioned sum $W$ satisfies
\[
 |\widehat\mu_W(\theta)|\le e^{-cn}\qquad(|\theta|\ge u_0\sqrt n/7).
\]
In particular $\widehat\mu_W$ decays.
\end{lemma}

\begin{proof}
Let $f(u)=\E e^{iu\xi_0}$. For $u\ne0$ we have $|f(u)|<1$, since
$|f(u)|=1$ would confine $\xi_0$ to a translate of $(2\pi/u)\Z$, a null
set. Also $f$ is continuous and tends to $0$ at infinity by the
Riemann--Lebesgue lemma. Hence $q=\sup_{|u|\ge b}|f(u)|<1$ for
$b=c_2u_0/7$, where $c_2=(\gamma_*/8)^{1/2}$.

Let $e$ be a unit vector, and let $A_e$ be the set of indices $j$ with
$|u_j\cdot e|\ge c_2n^{-1/2}$. The other rows contribute at most
$(n+1)c_2^2/n\le\gamma_*/4$ to $\sum_j(u_j\cdot e)^2\ge\gamma_*/2$, and
each row contributes at most $49/n$, so $|A_e|\ge\gamma_*n/196$. Let
$\theta=|\theta|e$ with $|\theta|\ge u_0\sqrt n/7$. For $j\in A_e$,
$|\theta\cdot u_j|\ge c_2u_0/7=b$, so $|f(\theta\cdot u_j)|\le q$, and
\[
 |\widehat\mu_W(\theta)|=\prod_j|f(\theta\cdot u_j)|\le q^{|A_e|}\le q^{\gamma_*n/196}.
\]
This is the bound with $c=(\gamma_*/196)\log(1/q)$. For
$n^{1/10}\le|\theta|\le u_0\sqrt n/7$, Lemma~\ref{cu:charfun}(a) gives
$|\widehat\mu_W(\theta)|\le e^{-\gamma_*n^{1/5}/8}$. Both bounds are at
most $n^{-20}$ for $n\ge n_0$.
\end{proof}

\begin{corollary}[Unsmoothing for (BD) laws]\label{cu:bd-unsmoothing}
Let $\xi_0$ satisfy \textup{(BD)}, and let every separated pair be good.
Then \ref{cu:U1}--\ref{cu:U3} hold, and for $n\ge n_0$
\[
 \E\bigl(D_n-J_{n,\delta,M}(\xi)\bigr)^2\le Cn^{199/200}.
\]
\end{corollary}

\begin{proof}
A (BD) law is symmetric with variance one and finite moments of every
order. The sums $Z_n(x)$, and $(Z_n(s),Z_n(t))$ for $s,t$ in the cells
of a separated pair, are well conditioned, so Lemma~\ref{cu:bd-decay}
gives \ref{cu:U2}. Let $W$ be a well-conditioned sum of dimension $d\le12$. By
Lemmas~\ref{cu:bd-decay} and~\ref{cu:smallball},
$\Prob(|W-w|\le r)\le Cr^d$ for $r\ge n^{-26/25}$, and for smaller $r$
the case $r=n^{-26/25}$ gives $\Prob(|W-w|\le r)\le Cn^{-26d/25}\le Cn^{-d/2}$.
This is \ref{cu:U1}. By Lemma~\ref{pre:no-circle-roots}, $B_n=0$ almost
surely, which gives \ref{cu:U3}. Proposition~\ref{cu:unsmooth} applies
with $b_n=0$.
\end{proof}

\subsection{The variance}\label{cu:variance-section}
The next proposition turns an unsmoothing bound into the variance. It
compares the current with its Gaussian counterpart by Lindeberg
replacement, and uses the Gaussian count of Theorem~\ref{thm:B}(i) as
the reference.

\begin{proposition}[Variance transfer]\label{cu:transfer}
Let $\xi_0$ be symmetric with $\E\xi_0^2=1$ and $\E\xi_0^6<\infty$, let
$\kappa_4=\E\xi_0^4-3$, and let $J=J_{n,\delta,M}$ with $\delta$ and $M$
as in Subsection~\ref{cu:unsmooth-section}. There are $C$ and $n_0$
such that for $n\ge n_0$ the following hold.
\begin{enumerate}[label=\textup{(\alph*)}]
 \item $\bigl|\Var J(\xi)-(c_G+\kappa_4/12)n\bigr|\le Cn^{399/400}$.
 \item With $e_n=\|D_n-J(\xi)\|_2$,
 $\bigl|\Var D_n-(c_G+\kappa_4/12)n\bigr|\le C\bigl(n^{399/400}+\sqrt n\,e_n+e_n^2\bigr)$.
\end{enumerate}
\end{proposition}

\begin{proof}
\emph{Smooth comparison.} By the jet bounds of Lemma~\ref{cu:jets}(iii)
and Leibniz' rule,
\[
 \|\partial_j^6(J^2)\|_\infty\le CM^2\delta^{-8}n^{-1},\qquad
 \|\partial_i^4\partial_j^4(J^2)\|_\infty\le CM^2\delta^{-10}n^{-2}.
\]
Both $J(\xi)$ and $J(G)$ have mean zero by Lemma~\ref{cu:jets}(iv), so
Lindeberg replacement (Lemma~\ref{cu:lindeberg}) applied to $F=J^2$ gives
$\Var J(\xi)-\Var J(G)=\kappa_4S_n+O(M^2\delta^{-10})$. By the fourth
insertion (Proposition~\ref{cu:insertion}),
$|t_M^2(1+\delta^2)^{-2}-1|\le C(\delta^2+e^{-6M^2})$ and
\begin{equation}\label{cu:smooth-comparison}
 \Bigl|\Var J(\xi)-\Var J(G)-\frac{\kappa_4n}{12}\Bigr|
 \le C\bigl(M^2\delta^{-10}\log^2(n+2)+n\delta^2+ne^{-6M^2}\bigr)\le Cn^{23/25},
\end{equation}
since $M^2\delta^{-10}=n^{41/100}$ and $n\delta^2=n^{23/25}$.

\emph{Gaussian reference.} Gaussian coefficients satisfy (BD), so
Corollary~\ref{cu:bd-unsmoothing} gives $\|D_n^G-J(G)\|_2\le Cn^{199/400}$.
By Theorem~\ref{thm:B}(i), $\E D_n^G=0$ and $\Var D_n^G=c_Gn+O(n^{99/100})$,
so $\|D_n^G\|_2\le C\sqrt n$. Hence
\begin{align*}
 |\Var D_n^G-\Var J(G)|&=\bigl|\|D_n^G\|_2^2-\|J(G)\|_2^2\bigr|\\
 &\le\|D_n^G-J(G)\|_2\bigl(2\|D_n^G\|_2+\|D_n^G-J(G)\|_2\bigr)\le Cn^{399/400}.
\end{align*}
With \eqref{cu:smooth-comparison} and $\Var D_n^G=c_Gn+O(n^{99/100})$,
this proves (a).

\emph{Unsmoothing.} By (a) and $\E J(\xi)=0$, $\|J(\xi)\|_2\le C\sqrt n$.
Centering is an orthogonal projection in $L^2$ and $J(\xi)$ is
centered, so $\bar D=D_n-\E D_n$ satisfies $\|\bar D-J(\xi)\|_2\le e_n$, and
\[
 |\Var D_n-\Var J(\xi)|=\bigl|\|\bar D\|_2^2-\|J(\xi)\|_2^2\bigr|
 \le e_n\bigl(2\|J(\xi)\|_2+e_n\bigr)\le C\bigl(\sqrt n\,e_n+e_n^2\bigr),
\]
and (a) gives (b).
\end{proof}

\begin{proof}[Proof of the variance formula in Theorem~\textup{\ref{thm:D}}]
By Lemma~\ref{pre:no-circle-roots}, $\E D_n=0$ for every $n$. By
Corollary~\ref{cu:bd-unsmoothing}, $e_n\le Cn^{199/400}$, and
Proposition~\ref{cu:transfer}(b) gives
$\Var D_n=(c_G+\kappa_4/12)n+O(n^{399/400})$.
\end{proof}

\emph{The examples.} The uniform law on $[-\sqrt3,\sqrt3]$ has
$\E\xi_0^2=1$ and $\E\xi_0^4=9/5$, so $\kappa_4=-6/5$. The Laplace law
with density $2^{-1/2}e^{-\sqrt2|x|}$ has $\E\xi_0^{2k}=(2k)!\,2^{-k}$, so
$\E\xi_0^2=1$, $\E\xi_0^4=6$ and $\kappa_4=3$.

\emph{The quadratic statistic.} The statistic of Theorem~\ref{thm:D} is
\[
 H_n=\sum_{j=0}^nh_j(\xi_j^2-1),\qquad h_j=\frac{j-n/2}{n+1}.
\]
If the $\xi_j$ are iid with $\E\xi_0=0$, $\E\xi_0^2=1$ and
$\E\xi_0^4=3+\kappa_4<\infty$, the variables $\xi_j^2-1$ are centered and
uncorrelated with variance $\kappa_4+2$, so $\E H_n=0$ and
\begin{equation}\label{cu:quadratic}
 \Var H_n=(\kappa_4+2)\sum_{j=0}^nh_j^2=(\kappa_4+2)\frac{n(n+2)}{12(n+1)} .
\end{equation}

\begin{proof}[Proof of the assertions about $H_n$ in Theorem~\textup{\ref{thm:D}}]
Let $J=J_{n,\delta,M}$. Both $J(\xi)$ and $H_n$ have mean zero, so
\[
 \Cov(J(\xi),H_n)=\sum_jh_j\E\bigl[(\xi_j^2-1)J(\xi)\bigr].
\]
By the second bound of Lemma~\ref{cu:lindeberg} and
Lemma~\ref{cu:jets}(iii),
\[
 \E\bigl[(\xi_j^2-1)J(\xi)\bigr]=\frac{\kappa_4+2}2\E\,\partial_j^2J(G)
 +O\bigl(M\delta^{-5}n^{-1}+nM\delta^{-7}n^{-2}\bigr).
\]
Insert \eqref{cu:second-jet}, multiply by $h_j$, use
$|h_j|\le\frac12$, and sum over $j$ with \eqref{cu:quadratic}. This
gives
\begin{equation}\label{cu:mixed}
 \Bigl|\Cov\bigl(J(\xi),H_n\bigr)
 -\frac{(\kappa_4+2)\,t_M}{1+\delta^2}\cdot\frac{n(n+2)}{12(n+1)}\Bigr|
 \le CM\delta^{-7}\log(n+2)\le Cn^{57/200}\log n .
\end{equation}
By \eqref{cu:quadratic}, $\|H_n\|_2^2\le(\kappa_4+2)n/6$, so
Corollary~\ref{cu:bd-unsmoothing} gives
$|\Cov(D_n,H_n)-\Cov(J(\xi),H_n)|\le\|D_n-J(\xi)\|_2\|H_n\|_2\le Cn^{399/400}$.
With $t_M\to1$ and $\delta\to0$, \eqref{cu:mixed} gives
$\Cov(D_n,H_n)/n\to(\kappa_4+2)/12$, and by \eqref{cu:quadratic}
$\Var H_n/n\to(\kappa_4+2)/12$. Hence $\Cov(D_n-H_n,H_n)/n\to0$ and
\[
 \frac{\Var(D_n-H_n)}n=\frac{\Var D_n-2\Cov(D_n,H_n)+\Var H_n}n
 \to c_G+\frac{\kappa_4}{12}-\frac{\kappa_4+2}{12}=c_G-\frac16 .
 \qedhere
\]
\end{proof}
\section{The variance for random signs}\label{sec:sign-variance}

We prove Theorem~\ref{thm:S} through Propositions~\ref{cu:unsmooth}
and~\ref{cu:transfer}. A sum of signs $W=\sum_j\epsilon_ju_j$ has the
characteristic function $\prod_j\cos(\theta\cdot u_j)$, which has modulus
one wherever every $\theta\cdot u_j$ lies in $\pi\Z$. For the sums
$Z_n(x)$ the vectors $u_j$ sample the curve
$\omega\mapsto(\cos\omega x,\sin\omega x,\omega\cos\omega x)$ at equally
spaced frequencies. We show that enough of the numbers $\theta\cdot u_j$
stay away from $\pi\Z$, at all frequencies up to $n^{26/25}$, to make
this product at most $n^{-20}$ when $x/2\pi n$ is far from fractions with
small denominators. The same holds for pairs of points. The
other pairs of cells are few, and they enter
Proposition~\ref{cu:unsmooth} through $b_n$.

Throughout, $(\epsilon_j)_{j\ge0}$ are independent random signs, and we
use the fields, the current and the cells of Section~\ref{sec:cumulant}.
For real $s,t$ put $w_j(s,t)=(w_j(s),w_j(t))\in\R^6$. For real $y$, $\nint y$ is
the distance from $y$ to the nearest integer, and $\gamma_x=x/(2\pi n)$.
The fixed exponents are $\tau=1/4000$ and $\tau_*=\tau/100$. Constants
$C,c$ and thresholds $n_0$ depend only on $\tau$, $\tau_*$ and $\chi$.

\subsection{Diophantine conditions}\label{sv:conditions-section}

\begin{definition}\label{sv:conditions}
A point $x$ satisfies \textup{(A1)} if $\nint{l\gamma_x}\ge n^{-1+8\tau}$ for
all integers $1\le l\le n^\tau$. A pair $(s,t)$ satisfies \textup{(A2)} if
$\nint{l_1\gamma_s+l_2\gamma_t}\ge n^{-1+8\tau}$ for all integers $l_1,l_2$
with $|l_1|,|l_2|\le n^\tau$ and $(l_1,l_2)\ne(0,0)$. A cell $I_i$ is
\emph{good} if every point of $\bar I_i$ satisfies \textup{(A1)}, and a pair
$(i,i')$ of cells is \emph{good} if every $(s,t)\in\bar I_i\times\bar I_{i'}$
satisfies \textup{(A2)}. Cells and pairs that are not good are \emph{bad}.
\end{definition}

Taking $l_2=0$ or $l_1=0$ shows that \textup{(A2)} at $(s,t)$ implies
\textup{(A1)} at $s$ and at $t$, so both cells of a good pair are good. In
Proposition~\ref{cu:unsmooth} the good pairs are the separated pairs that
are good in this sense, so $b_n$ is at most the number of bad pairs.

\begin{lemma}[Counting]\label{sv:counting}
For $n\ge n_0$, at most $Cn^{9\tau}$ cells are bad, and for every good
cell $I_i$ at most $Cn^{10\tau}$ indices $i'$ make $(i,i')$ a bad pair.
Hence at most $Cn^{1+10\tau}$ pairs of cells are bad.
\end{lemma}

\begin{proof}
An interval of length $r$ meets the closures of at most $r/\ell+2\le2r+2$
cells. If $x\in[0,\pi n]$ fails \textup{(A1)}, there are $1\le l\le n^\tau$ and an
integer $a$ with $|l\gamma_x-a|<n^{-1+8\tau}$, that is,
$|x-2\pi na/l|<2\pi n^{8\tau}/l$. Since $0\le l\gamma_x\le l/2$, at most
$l/2+2$ integers $a$ occur. For fixed $l$ the resulting intervals, of length
$4\pi n^{8\tau}/l$, meet at most $(l/2+2)(8\pi n^{8\tau}/l+2)\le25\pi n^{8\tau}$
cells, because $l\le n^{8\tau}$. Summing over $l\le n^\tau$ gives at most
$25\pi n^{9\tau}$ bad cells.

Let $I_i$ be good, and let $(s,t)\in\bar I_i\times\bar I_{i'}$
fail \textup{(A2)} with $(l_1,l_2)$. If $l_2=0$ then $s$ fails \textup{(A1)},
which is impossible, so $1\le|l_2|\le n^\tau$. Since $|s-c_i|\le\frac12$, we
have $|l_1\gamma_s-l_1\gamma_{c_i}|\le n^{\tau-1}\le n^{-1+8\tau}$, so some
integer $a$ satisfies $|l_1\gamma_{c_i}+l_2\gamma_t-a|<2n^{-1+8\tau}$. Thus $t$
lies within $4\pi n^{8\tau}/|l_2|$ of $2\pi n(a-l_1\gamma_{c_i})/l_2$, and as
$|l_2\gamma_t|\le|l_2|/2$, at most $|l_2|+3$ integers $a$ occur. These
intervals meet at most $(|l_2|+3)(16\pi n^{8\tau}/|l_2|+2)\le70\pi n^{8\tau}$
cells. Summing over the at most $9n^{2\tau}$ pairs $(l_1,l_2)$ proves the
second claim. The number of bad pairs is at most
$2m\cdot25\pi n^{9\tau}+m\cdot630\pi n^{10\tau}$.
\end{proof}

\subsection{Decay of characteristic sums}\label{sv:decay-section}

\begin{proposition}[Characteristic sums]\label{sv:decay}
There is $n_0$ such that for $n\ge n_0$ the following hold.
\begin{enumerate}[label=\textup{(\alph*)}]
 \item If $(s,t)$ satisfies \textup{(A2)}, then
 $\sum_{j=0}^n\nint{\langle\Xi,w_j(s,t)\rangle}^2\ge\lfloor n^{\tau_*}\rfloor$
 for all $\Xi\in\R^6$ with $n^{-1/2+\tau}\le|\Xi|\le n^{27/50}$.
 \item If $x$ satisfies \textup{(A1)}, then
 $\sum_{j=0}^n\nint{\langle\Xi,w_j(x)\rangle}^2\ge\lfloor n^{\tau_*}\rfloor$
 for all $\Xi\in\R^3$ with the same bounds.
\end{enumerate}
\end{proposition}

In both cases
\begin{equation}\label{sv:signal}
 \langle\Xi,w_j\rangle=g(\omega_j),\qquad
 g(\omega)=\Re\sum_{\nu=1}^d(a_\nu+b_\nu\omega)e^{i\omega x_\nu},
\end{equation}
with $d=2$, $(x_1,x_2)=(s,t)$, $a_1=\Xi_1-i\Xi_2$, $b_1=\Xi_3$,
$a_2=\Xi_4-i\Xi_5$, $b_2=\Xi_6$ in case (a), and $d=1$, $x_1=x$,
$a_1=\Xi_1-i\Xi_2$, $b_1=\Xi_3$ in case (b). In both cases
$\sum_\nu(|a_\nu|^2+|b_\nu|^2)=|\Xi|^2$.

\begin{lemma}[Two rotations]\label{sv:rotations}
Let $\eta>0$, let $Q\ge8/\eta$ be an integer, and let $y_1,y_2$ be real with
$\nint{y_\nu}\ge\eta$ and $\nint{2y_\nu}\ge\eta$ for $\nu=1,2$. Then some
$1\le q\le Q$ has $\nint{qy_1}\ge\frac1{16}$ and $\nint{qy_2}\ge\frac1{16}$.
\end{lemma}

\begin{proof}
For real $y$ with $\nint y>0$,
$|\sum_{q=1}^Qe^{2\pi iqy}|\le2/|1-e^{2\pi iy}|\le1/(2\nint y)$, since
$|1-e^{2\pi iy}|=2|\sin\pi y|\ge4\nint y$. By the identity
$\cos^4\pi x=\frac38+\frac12\cos2\pi x+\frac18\cos4\pi x$,
\[
 \sum_{q=1}^Q\cos^4(\pi qy_\nu)\le\frac{3Q}8+\frac12\cdot\frac1{2\eta}+\frac18\cdot\frac1{2\eta}
 =\frac{3Q}8+\frac5{16\eta}.
\]
If $\nint{qy_\nu}<\frac1{16}$, then $\cos^4(\pi qy_\nu)\ge\cos^4(\pi/16)\ge0.92$.
So at most $(\frac{3Q}8+\frac5{16\eta})/0.92\le0.41Q+0.34/\eta$ integers
$q\le Q$ have $\nint{qy_\nu}<\frac1{16}$. For $\nu=1,2$ together these are at
most $0.82Q+0.68/\eta<Q$ integers, because $0.68/\eta\le0.085Q$.
\end{proof}

\begin{lemma}[Annihilators]\label{sv:annihilator}
Let $d\ge1$, let $\zeta_1,\ldots,\zeta_d\in\C$ have modulus one, let
$a_\nu,b_\nu\in\C$, and let
$f(r)=\Re\sum_{\nu=1}^d(a_\nu+b_\nu r)\zeta_\nu^r$ for integers $r\ge0$. If
$\mu\in\R$ and $|f(r)-\mu|\le1$ for $0\le r\le4d$, then
\[
 |\mu|\prod_{\nu=1}^d|1-\zeta_\nu|^4\le16^d .
\]
\end{lemma}

\begin{proof}
For $z\in\C$ and a complex sequence $y$ put $T_zy(r)=y(r+1)-zy(r)$. These
operators commute. If $y(r)=(a+br)z'^r$ with $a,b\in\C$, then $T_zy$ has the
same form with the same $z'$, and $T_{z'}y(r)=bz'^{r+1}$, so $T_{z'}^2y=0$.
Let $T=\prod_\nu T_{\zeta_\nu}^2T_{\bar\zeta_\nu}^2$, a product of $4d$
operators. Since $2\Re w=w+\bar w$, the sequence $f$ is a sum of $2d$
sequences of this form with $z'\in\{\zeta_\nu,\bar\zeta_\nu\}$, so $Tf=0$. A
constant sequence $\mu$ has $T_z\mu=(1-z)\mu$, so
$T\mu=\mu\prod_\nu(1-\zeta_\nu)^2(1-\bar\zeta_\nu)^2=\mu\prod_\nu|1-\zeta_\nu|^4$.
For $|z|=1$, $T_z$ takes a sequence on $\{0,\ldots,l\}$ to a sequence on
$\{0,\ldots,l-1\}$ and at most doubles its largest modulus. Applying the
$4d$ factors of $T$ to $\mu-f$, whose modulus is at most one on
$\{0,\ldots,4d\}$, gives $|T\mu|=|T(\mu-f)(0)|\le2^{4d}=16^d$.
\end{proof}

\begin{lemma}[Constancy]\label{sv:constancy}
Let $k,B,\lambda\ge0$ be integers with $k+2^k(B+1)\le\lambda$, and let
$f_0,\ldots,f_\lambda$ be integers with $|f_i|\le B$ and
$\sum_{r=0}^k(-1)^{k-r}\binom krf_{i+r}=0$ for $0\le i\le\lambda-k$. Then $f$ is
constant.
\end{lemma}

\begin{proof}
We induct on $k$. For $k=0$ the hypothesis says that every $f_i$ vanishes.
Let $k\ge1$ and $f'_i=f_{i+1}-f_i$ for $0\le i\le\lambda-1$. Then
$|f'_i|\le2B$, the $(k-1)$-th differences of $f'$ are the $k$-th differences
of $f$ and vanish for $0\le i\le\lambda-k$, and
$(k-1)+2^{k-1}(2B+1)\le k+2^k(B+1)-1\le\lambda-1$. By induction $f'_i=f'_0$
for all $i$, so $f_\lambda-f_0=\lambda f'_0$. Since
$|f_\lambda-f_0|\le2B<\lambda$, we get $f'_0=0$.
\end{proof}

\begin{lemma}[Gram sums along progressions]\label{sv:gram}
Let $g$ be as in \eqref{sv:signal} with $d\in\{1,2\}$, real $x_\nu$ and
complex $a_\nu,b_\nu$. Let $\omega'_i=\omega'_0+ih$ for $0\le i\le\lambda$,
with $h\ge0$ and $\omega'_0,\omega'_\lambda\in[-\frac12,\frac12]$. Suppose
that $\eta>0$, $\nint{h(x_\nu+x_{\nu'})/2\pi}\ge\eta$ for all $\nu,\nu'$, and
$\nint{h(x_\nu-x_{\nu'})/2\pi}\ge\eta$ for $\nu\ne\nu'$. Then
\[
 \sum_{i=0}^\lambda g(\omega'_i)^2\ge\Bigl(\frac{(\lambda+1)(\lambda h)^2}{36}-\frac{3d}\eta\Bigr)
 \sum_{\nu=1}^d\bigl(|a_\nu|^2+|b_\nu|^2\bigr).
\]
\end{lemma}

\begin{proof}
Put $q_\nu(\omega)=a_\nu+b_\nu\omega$ and
$\mathcal S(\omega)=\sum_\nu q_\nu(\omega)e^{i\omega x_\nu}$, so that
$g=\Re\mathcal S$ and $g^2=\frac12|\mathcal S|^2+\frac12\Re\mathcal S^2$, with
\[
 |\mathcal S|^2=\sum_\nu|q_\nu|^2+\sum_{\nu\ne\nu'}q_\nu\bar q_{\nu'}e^{i\omega(x_\nu-x_{\nu'})},\qquad
 \mathcal S^2=\sum_{\nu,\nu'}q_\nu q_{\nu'}e^{i\omega(x_\nu+x_{\nu'})}.
\]

\emph{Main term.} Let $\bar\omega$ and $V$ be the mean and the variance of
$\omega'_0,\ldots,\omega'_\lambda$, so that
$V=h^2\lambda(\lambda+2)/12\ge(\lambda h)^2/12$ and $V\le\frac14$. Put
$u_\nu=a_\nu+b_\nu\bar\omega$. Then
$\sum_i|q_\nu(\omega'_i)|^2=(\lambda+1)(|u_\nu|^2+V|b_\nu|^2)$, and
$|a_\nu|^2\le2|u_\nu|^2+\frac12|b_\nu|^2$ because $|\bar\omega|\le\frac12$.
Hence
$\frac{2V}3(|a_\nu|^2+|b_\nu|^2)\le\frac{4V}3|u_\nu|^2+V|b_\nu|^2\le|u_\nu|^2+V|b_\nu|^2$,
and
\[
 \frac12\sum_\nu\sum_{i=0}^\lambda|q_\nu(\omega'_i)|^2
 \ge\frac{(\lambda+1)V}3\sum_\nu\bigl(|a_\nu|^2+|b_\nu|^2\bigr)
 \ge\frac{(\lambda+1)(\lambda h)^2}{36}\sum_\nu\bigl(|a_\nu|^2+|b_\nu|^2\bigr).
\]

\emph{Oscillatory terms.} Each remaining term is a unimodular constant times
$\sum_iP(\omega'_i)z^i$, with $P\in\{q_\nu\bar q_{\nu'},q_\nu q_{\nu'}\}$ and
$z=e^{ihy}$, $y\in\{x_\nu-x_{\nu'},x_\nu+x_{\nu'}\}$. By hypothesis
$|1-z|=2|\sin(hy/2)|\ge4\nint{hy/2\pi}\ge4\eta$, so the partial sums of
$z^i$ are bounded by $1/(2\eta)$, and summation by parts gives
\[
 \Bigl|\sum_iP(\omega'_i)z^i\Bigr|\le\frac1{2\eta}\Bigl(|P(\omega'_\lambda)|
 +\sum_{i<\lambda}|P(\omega'_{i+1})-P(\omega'_i)|\Bigr).
\]
With $s_\nu=|a_\nu|+|b_\nu|$ we have $|q_\nu|\le s_\nu$ and
$|P'|\le2s_\nu s_{\nu'}$ on $[-\frac12,\frac12]$, so this is at most
$3s_\nu s_{\nu'}/(2\eta)$. The oscillatory terms of
$\frac12|\mathcal S|^2+\frac12\Re\mathcal S^2$ therefore sum to at most
$\frac{3}{2\eta}\bigl(\sum_\nu s_\nu\bigr)^2\le\frac{3d}{2\eta}\sum_\nu s_\nu^2
\le\frac{3d}\eta\sum_\nu(|a_\nu|^2+|b_\nu|^2)$.
\end{proof}

Steps 1--4 of the next proof follow the inverse argument of Do, Nguyen and
Nguyen \cite{DoNguyenNguyen2022} for random trigonometric polynomials.

\begin{proof}[Proof of Proposition~\textup{\ref{sv:decay}}]
Let $g$ be as in \eqref{sv:signal} and $g_j=g(\omega_j)$. Put $u=n^{\tau_*}$,
$A=\lfloor u\rfloor$, $k=400000$, $R=n^{27/50}$, $B=\lceil4R+1\rceil$,
$\eta=n^{-1+8\tau}$,
\[
 H=\Bigl\lfloor\frac{n+1}{A^3+1}\Bigr\rfloor,\qquad P=\Bigl\lfloor\frac H{A^2}\Bigr\rfloor-1 .
\]
For $n\ge n_0$ we have $A\ge u/2$,
\[
 H-A^2\ge\frac{n^{1-3\tau_*}}4,\qquad P+1\ge\frac{n^{1-5\tau_*}}4 .
\]
Suppose that $\sum_{j=0}^n\nint{g_j}^2<A$. We derive a contradiction for
$n\ge n_0$. In case (b) read $x_2=x_1$.

\emph{Step 1: a long run.} Fewer than $A^3$ indices have $\nint{g_j}\ge1/A$.
The integers $0,\ldots,(A^3+1)H-1$ lie in $\{0,\ldots,n\}$ and split into
$A^3+1$ runs of $H$ consecutive integers, so one run
$\{j_0,\ldots,j_0+H-1\}$ has $\nint{g_j}<1/A$ throughout. Let
$\lfloor g_j\rceil$ be the integer nearest to $g_j$.

\emph{Step 2: Dirichlet.} Among the $A^2+1$ points
$(q\gamma_{x_1},q\gamma_{x_2})$ modulo one, $0\le q\le A^2$, two lie in the
same square of side $1/A$. Their difference gives $1\le p\le A^2$ with
$\nint{p\gamma_{x_\nu}}<1/A$ for $\nu=1,2$.

\emph{Step 3: differences.} Let $\lambda=\lfloor(H-1)/p\rfloor$. The
progression $j_0+ip$, $0\le i\le\lambda$, lies in the run, $\lambda\ge P$, and
$\lambda p\ge H-A^2$. With $z_\nu=e^{ipx_\nu/n}$ we have
$e^{i\omega_{j+rp}x_\nu}=e^{i\omega_jx_\nu}z_\nu^r$ and
$\omega_{j+rp}=\omega_j+rp/n$, so
\[
 g_{j+rp}=\Re\sum_\nu\Bigl[(a_\nu+b_\nu\omega_j)z_\nu^r+b_\nu\frac pn\,rz_\nu^r\Bigr]e^{i\omega_jx_\nu}.
\]
Let $\Delta^k$ denote the $k$-th difference in $r$ at $r=0$. Since
$\Delta^kz^r=(z-1)^k$ and $\Delta^k(rz^r)=kz(z-1)^{k-1}$, and since
$|1-z_\nu|=2|\sin\pi p\gamma_{x_\nu}|<2\pi/A$, $|a_\nu|+|b_\nu|\le2R$ and
$p\le A^2$,
\[
 |\Delta^kg_{j+rp}|\le2dR\Bigl(\frac{2\pi}A\Bigr)^k+2dRk\,2^k\frac{A^2}n .
\]
Also $|\Delta^k(g_{j+rp}-\lfloor g_{j+rp}\rceil)|\le2^k/A$ when the indices
lie in the run. As $A\ge u/2$ and $k\tau_*=1$, the right sides are at most
$(4\pi)^k4Rn^{-1}$, $4Rk2^kn^{2\tau_*-1}$ and $2^{k+1}n^{-\tau_*}$, so
$|\Delta^k\lfloor g_{j+rp}\rceil|<1$ for $n\ge n_0$. Hence the $k$-th
differences of $\mu_i=\lfloor g_{j_0+ip}\rceil$ vanish for
$0\le i\le\lambda-k$.

\emph{Step 4: constancy.} By \eqref{sv:signal},
$|g_j|\le\sum_\nu(|a_\nu|+|b_\nu|)\le2dR$, so $|\mu_i|\le4R+1\le B$. Since
$k+2^k(B+1)\le n^{1-5\tau_*}/4-1\le P\le\lambda$ for $n\ge n_0$,
Lemma~\ref{sv:constancy} shows that $\mu_i=\mu$ is constant for
$0\le i\le\lambda$.

\emph{Step 5: $\mu=0$.} Let $Q=\lfloor P/8\rfloor$. For $n\ge n_0$,
$Q\ge n^{1-5\tau_*}/40\ge8n^{1-8\tau}=8/\eta$. Since $p$ and $2p$ are at most
$2A^2\le n^\tau$, \textup{(A2)}, respectively \textup{(A1)}, gives
$\nint{p\gamma_{x_\nu}}\ge\eta$ and $\nint{2p\gamma_{x_\nu}}\ge\eta$.
Lemma~\ref{sv:rotations} gives $1\le q\le Q$ with
$\nint{qp\gamma_{x_\nu}}\ge\frac1{16}$ for $\nu=1,2$. Put
$\zeta_\nu=e^{iqpx_\nu/n}$ and $f(r)=Ag_{j_0+rqp}$ for $0\le r\le8$. These
indices lie in the progression because $8q\le P\le\lambda$, so
$|f(r)-A\mu|<1$. As in Step~3,
$f(r)=\Re\sum_\nu(a'_\nu+b'_\nu r)\zeta_\nu^r$ for some complex
$a'_\nu,b'_\nu$. Since $4d\le8$, Lemma~\ref{sv:annihilator} with $A\mu$ in
place of $\mu$ gives $A|\mu|\prod_\nu|1-\zeta_\nu|^4\le16^d$. Here
$|1-\zeta_\nu|=2|\sin\pi qp\gamma_{x_\nu}|\ge4\nint{qp\gamma_{x_\nu}}\ge\frac14$,
so $|\mu|\le4096^d/A<1$, and $\mu=0$.

\emph{Step 6: contradiction.} Now $|g_{j_0+ip}|<1/A\le\frac12$, so
$\nint{g_{j_0+ip}}=|g_{j_0+ip}|$ for $0\le i\le\lambda$. Apply
Lemma~\ref{sv:gram} to $\omega'_i=\omega_{j_0+ip}$, with $h=p/n$. Its phase
conditions are $\nint{p(\gamma_{x_\nu}\pm\gamma_{x_{\nu'}})}\ge\eta$, given by
\textup{(A2)}, respectively \textup{(A1)}, since $2p\le n^\tau$. As
$\lambda\ge P$ and $\lambda p\ge H-A^2$,
\[
 A>\sum_{i=0}^\lambda g_{j_0+ip}^2\ge|\Xi|^2\,\mathcal B,\qquad
 \mathcal B=\frac{(P+1)\bigl((H-A^2)/n\bigr)^2}{36}-\frac{3d}\eta
 \ge\frac{n^{1-11\tau_*}}{2304}-6n^{1-8\tau}.
\]
Since $\mathcal B>0$ and $|\Xi|^2\ge n^{-1+2\tau}$, the right side is at least
$n^{2\tau-11\tau_*}/2304-6n^{-6\tau}$. As $2\tau-11\tau_*=189\tau_*$, this
exceeds $u\ge A$ for $n\ge n_0$.
\end{proof}

\subsection{Proof of Theorem~\ref{thm:S}}\label{sv:proof-section}
The coarse small balls follow from Lemma~\ref{sg:smallball} by a
covering argument.

\begin{lemma}[Coarse small balls]\label{sv:coarse-signs}
For $d\le12$ there is $C$ such that for $n\ge1$ every well-conditioned sum
$W=\sum_j\epsilon_ju_j$ of dimension $d$ satisfies
\[
 \Prob(|W-w|\le r)\le C(r+n^{-1/2})^d\qquad(w\in\R^d,\ r>0).
\]
\end{lemma}

\begin{proof}
The $n+1\le2n$ rows $a_j=\sqrt n\,u_j$ satisfy $|a_j|\le7$ and
$\sum_j\langle a_j,u\rangle^2\ge\frac{\gamma_*}2n|u|^2$. The small-ball
bound for sign sums (Lemma~\ref{sg:smallball}) with $C_0=2$, $K=7$ and
$\kappa=\gamma_*/2$ gives $\rho,C>0$ with
$\Prob(|\sqrt nW-z|\le\rho)\le Cn^{-d/2}$ for every $z\in\R^d$. A ball of
radius $\sqrt n\,r$ meets at most $(2\sqrt d\sqrt n\,r/\rho+2)^d$ cubes of
a grid of side $\rho/\sqrt d$, each in a ball of radius $\rho$, so
\[
 \Prob(|W-w|\le r)\le\Bigl(\frac{2\sqrt d\sqrt n\,r}\rho+2\Bigr)^dCn^{-d/2}
 =C\Bigl(\frac{2\sqrt d\,r}\rho+2n^{-1/2}\Bigr)^d\le C'(r+n^{-1/2})^d.
 \qedhere
\]
\end{proof}

At Diophantine points Proposition~\ref{sv:decay} gives the decay
required by \ref{cu:U2}.

\begin{lemma}[Decay at Diophantine points]\label{sv:charfun}
Let $n\ge n_0$. If $x$ satisfies \textup{(A1)}, then
$\widehat\mu_{Z_n(x)}$ decays. If $(s,t)$ satisfies \textup{(A2)}, then
$\widehat\mu_{(Z_n(s),Z_n(t))}$ decays.
\end{lemma}

\begin{proof}
The rows of $W=(Z_n(s),Z_n(t))$ are $u_j=(2/(n+1))^{1/2}w_j$ with
$w_j=w_j(s,t)$, and those of $W=Z_n(x)$ are the same with $w_j=w_j(x)$.
Put $\Xi=(2/(n+1))^{1/2}\theta/\pi$, so that
$\theta\cdot u_j=\pi\langle\Xi,w_j\rangle$. Since
$n^{-1/2}\le(2/(n+1))^{1/2}\le\sqrt2\,n^{-1/2}$, for
$n^{1/10}\le|\theta|\le n^{26/25}$ we have $|\Xi|\ge n^{-1/2}|\theta|/\pi\ge n^{-1/2+\tau}$
and $|\Xi|\le\sqrt2n^{27/50}/\pi\le n^{27/50}$. For $|y|\le\pi/2$,
$\cos y\le e^{-y^2/2}$, since $\log\cos y+y^2/2$ vanishes at $0$ and has
derivative $y-\tan y\le0$ on $[0,\pi/2)$. Hence
$|\cos\pi y|=\cos(\pi\nint y)\le e^{-\pi^2\nint y^2/2}\le e^{-2\nint y^2}$,
and Proposition~\ref{sv:decay} gives
\[
 |\widehat\mu_W(\theta)|=\prod_j|\cos(\pi\langle\Xi,w_j\rangle)|
 \le\exp\bigl(-2\lfloor n^{\tau_*}\rfloor\bigr)\le n^{-20}
\]
for $n\ge n_0$.
\end{proof}

\begin{proof}[Proof of Theorem~\textup{\ref{thm:S}}]
Signs have no atom at zero, so \eqref{pre:centering} gives
$\E D_n=\frac12\E B_n$, which is bounded by
Proposition~\ref{sg:boundary}.

We apply Proposition~\ref{cu:unsmooth} with the good pairs of
Definition~\ref{sv:conditions} that are separated. The sign law is
symmetric with variance one and bounded. Lemma~\ref{sv:coarse-signs}
gives \ref{cu:U1}. Both cells of a good pair are good, so every point of
their closures satisfies \textup{(A1)}, and every pair of points from the
two closures satisfies \textup{(A2)}. Lemma~\ref{sv:charfun} therefore
gives \ref{cu:U2}. By \eqref{sg:boundary-moments},
$\E B_n^2=\Var B_n+(\E B_n)^2\le C\sqrt n$, which is \ref{cu:U3}. By
Lemma~\ref{sv:counting}, $b_n\le Cn^{1+10\tau}=Cn^{1+1/400}$. Hence
\[
 \E\bigl(D_n-J_{n,\delta,M}(\epsilon)\bigr)^2\le C\bigl(n^{199/200}+n^{-1/200}n^{1+1/400}\bigr)\le Cn^{399/400}.
\]
The sign law has $\kappa_4=\E\epsilon_0^4-3=-2$, and
Proposition~\ref{cu:transfer}(b) with $e_n\le Cn^{399/800}$ gives
\[
 \Bigl|\Var D_n-\Bigl(c_G-\frac16\Bigr)n\Bigr|\le C\bigl(n^{399/400}+n^{1/2+399/800}\bigr)\le Cn^{799/800}.
 \qedhere
\]
\end{proof}

\appendix

\section{Numerical comparisons}\label{sec:numerics}\label{ap:section}

We test two predictions against computation, the variance constant
$c_G+\kappa_4/12$ across coefficient laws and the radial profile
$\Phi_{\rm rad}$ for random signs. All intervals are pointwise $95\%$ bootstrap intervals
that resample whole coefficient sequences. The data of these computations,
with the seeds of every sequence, are in the repository~\cite{FormalCompanion2026}.

\subsection{Certified counts}\label{ap:certificates}
The closed-disk counts of Subsection~\ref{ap:variance-section} are
certified for the exact coefficient arrays, integers for the atomic laws
and stored binary64 values for the others, by the following
certificate.

\begin{proposition}[Chord and Taylor certificate]\label{ap:certificate}
Let $F=QP$, where $Q$ is a product of cyclotomic polynomials, $b=\deg Q$,
and $P(z)=\sum_{k=0}^da_kz^k$ has real coefficients. Fix $c\in\Z$ and put
$g(t)=e^{-ict}P(e^{it})$. Partition $[0,2\pi]$ into arcs $[u,v]$ of
length $h$, and suppose that on each arc, with $A=g(u)$ and $B=g(v)$,
there are numbers $L,E$ with
\[
 L\le\dist(0,[A,B]),\qquad E\ge\tfrac18h^2\sup_{[u,v]}|g''|,\qquad L>E,
 \qquad\Re(B\bar A)>0 .
\]
If an interval enclosure of $(2\pi)^{-1}\sum\arg(B/A)$, summed over the arcs
with principal arguments, contains exactly one integer $w$, then $F$ has
exactly $w+c+b$ zeros in the closed unit disk, counted with multiplicity,
and $\min_{[0,2\pi]}|g|\ge\min(L-E)>0$.
\end{proposition}

\begin{proof}
On an arc let $\lambda$ be the linear interpolant of $g$. Since $g-\lambda$
vanishes at $u$ and $v$, $g(t)-\lambda(t)=-\int_u^vG(t,s)g''(s)\dd s$ with the
Dirichlet Green kernel $G\ge0$ of $-\dd^2/\dd t^2$ on $[u,v]$, and
$\int_u^vG(t,s)\dd s=(t-u)(v-t)/2\le h^2/8$. Hence $|g-\lambda|\le E<L\le|\lambda|$
on the arc, the homotopy from $\lambda$ to $g$ avoids $0$, and the argument
increment of $g$ over the arc is that of the chord, the principal argument of
$B/A$ because $\Re(B\bar A)>0$. Since $c\in\Z$, $g$ is a closed curve with
winding number $w$, and $|g|\ge L-E$ on each arc. Thus $P$ has
no zero on the circle and, by the argument principle, $w+c$ zeros in the open
disk. The $b$ zeros of $Q$ lie on the circle.
\end{proof}

The bounds $E$ come from Taylor jets. For $p\ge2$ let $J_r=g^{(r)}(u)/r!$ and
$C_{p+1}=\sum_k|a_k||k-c|^{p+1}/(p+1)!$. Since $|g^{(p+1)}|\le(p+1)!\,C_{p+1}$,
Taylor's theorem for $g$ and for $g''$ gives
\[
 \Bigl|g(u+s)-\sum_{r=0}^pJ_rs^r\Bigr|\le C_{p+1}s^{p+1},\qquad
 \sup_{[u,u+h]}|g''|\le\sum_{r=2}^pr(r-1)|J_r|h^{r-2}+p(p+1)C_{p+1}h^{p-1}.
\]
Appending a term $a_mz^m$ adds $a_m(i(m-c))^re^{i(m-c)u}/r!$ to $J_r$ and
$|a_m||m-c|^{p+1}/(p+1)!$ to $C_{p+1}$, so along a nested sequence the
jets are updated degree by degree. All quantities are enclosed in ball
arithmetic. A ball inverse discrete Fourier transform on a grid larger
than the degree encloses the initial jets and each new harmonic, and the
arcs are refined until the strict inequalities of
Proposition~\ref{ap:certificate} hold. Cyclotomic factors and a zero of
order $k$ at the origin are divided out exactly, the latter adding $k$ to
the count.

\subsection{Variance across coefficient laws}\label{ap:variance-section}
Table~\ref{ap:variance} compares the sample variance of $D_n$ with
$c_G+\kappa_4/12$, where $c_G\simeq0.305898$, for six symmetric laws of
variance one. The atomic laws are random signs, the four-point law
uniform on $\{\pm1,\pm2\}$ after normalization, and the three-point law
with masses $\frac16$, $\frac23$ and $\frac16$ at $-\sqrt3$, $0$ and
$\sqrt3$. The uniform law on $[-\sqrt3,\sqrt3]$, the standard Gaussian law
and the Laplace law have densities. Theorems~\ref{thm:D} and~\ref{thm:S}
give the limit $c_G+\kappa_4/12$ for signs and for the three laws with
densities. There are $16384$ sequences per law at $n=10^5$ and $2048$ at
$n=10^6$.

\begin{table}[ht]
\centering\small
\caption{Sample variances of the certified closed-disk counts.}
\label{ap:variance}
\setlength{\tabcolsep}{3.5pt}
\begin{tabular}{@{}lrrrlrl@{}}
\toprule
&&&\multicolumn{2}{c}{$n=10^5$}&\multicolumn{2}{c}{$n=10^6$}\\
\cmidrule(lr){4-5}\cmidrule(l){6-7}
Law&$\kappa_4$&$c_G+\frac{\kappa_4}{12}$&$\widehat{\Var}D_n/n$&$95\%$ interval&$\widehat{\Var}D_n/n$&$95\%$ interval\\
\midrule
Signs&$-2$&$0.139231$&$0.139568$&$[0.136579,0.142700]$&$0.140699$&$[0.131832,0.149901]$\\
Four point&$-1.64$&$0.169231$&$0.168808$&$[0.165274,0.172531]$&$0.170751$&$[0.160621,0.181000]$\\
Uniform&$-1.2$&$0.205898$&$0.206954$&$[0.202336,0.211628]$&$0.202672$&$[0.190398,0.215183]$\\
Three point&$0$&$0.305898$&$0.300426$&$[0.293881,0.306978]$&$0.323200$&$[0.304773,0.341834]$\\
Gaussian&$0$&$0.305898$&$0.302210$&$[0.295605,0.308884]$&$0.291928$&$[0.274027,0.310311]$\\
Laplace&$3$&$0.555898$&$0.549719$&$[0.538038,0.561997]$&$0.556819$&$[0.524236,0.590949]$\\
\bottomrule
\end{tabular}
\end{table}

Every interval contains $c_G+\kappa_4/12$. A weighted linear fit of
$\widehat{\Var}D_n/n$ against $\kappa_4$ over the six laws, with free
intercept, has slope $0.081573$, interval $[0.079617,0.083531]$, at
$n=10^5$, and slope $0.083108$, interval $[0.077610,0.088500]$, at
$n=10^6$. Both intervals contain $1/12$.

\subsection{The radial profile}\label{ap:radial}
For $16384$ sequences of random signs we count the zeros of $f_N$ in the
closed disks of radius $1+x/N$ at the $321$ points
$x=-16,-15.9,\dots,16$, at the degrees $N=10^3,\,3\cdot10^3,\,10^4,\,
3\cdot10^4$ and $10^5$, and for the first $2048$ sequences also at
$N=3\cdot10^5$ and $10^6$. Each count is the winding number of
$g(t)=f_N(re^{it})$, computed from $g$ and $g''$ on a grid of at least
$8(N+1)$ angles. A step of the grid is accepted when it passes the test of
Proposition~\ref{ap:certificate} with $\sup|g''|$ replaced by four times
the largest of $|g''|$ at its ends and a quarter of the root mean square
of $g''$, and every other step is subdivided until its parts pass. The
counts agree with the zeros computed directly for $92$ of the sequences,
at all $142524$ radii compared.

The largest differences of the averaged profile $\nu_N(1+x/N)/N$ from
$\Phi_{\rm rad}(x)$ over the $321$ points are $9.5\cdot10^{-4}$,
$3.1\cdot10^{-4}$, $1.2\cdot10^{-4}$, $3.5\cdot10^{-5}$ and
$2.3\cdot10^{-5}$ at the five lower degrees, and $3.3\cdot10^{-5}$ and $1.3\cdot10^{-5}$ at
$N=3\cdot10^5$ and $10^6$. They decay like $1/N$ until they reach the
sampling error of the averages. Figure~\ref{fig:radial-law}(c) shows the
density of zeros at $N=10^6$ from these counts. For the same sequences and
degrees, Figure~\ref{rc:fig-secants}(b) shows $J_N(1)$, the mean of
$\log|f_N|$ over at least $32(N+1)$ equally spaced angles, shifted off $0$
and $\pi$.
\section{Formal statements}\label{sec:formal}

The results listed below are proved in Lean~4~\cite{LeanFour2021}
(version 4.34.0) with Mathlib~\cite{Mathlib2020} (commit
\texttt{5ed2965}), in release~2.0.0 of~\cite{FormalCompanion2026}, whose
Lean sources have the Software Heritage identifier
\begin{center}
  \path{swh:1:dir:d9bcf23f7688fc0dacde344ee59003a0f48924e1}
\end{center}
The radial results are in the library \texttt{Erdos522} and the fluctuation results
in the library \texttt{Fluctuations}. The formalization also uses the
Lean libraries SLT~\cite{SLTLean} and
ProbabilityApproximation~\cite{ProbabilityApproximationLean}. For every
declaration named here, \texttt{\#print axioms} reports only
\texttt{propext}, \texttt{Classical.choice} and \texttt{Quot.sound}, and
\texttt{leanchecker} replays every module of both libraries in the Lean
kernel.

Problem~\#522 is the theorem \texttt{Erdos522.erdos\_522} of Google
DeepMind's Formal Conjectures project~\cite{FormalConjectures}. Our declaration
\texttt{FormalConjectures.Erdos522Mirror.erdos\_522} proves that
statement, copied from that project, and our Lean~4 proof has been
accepted there as a solution to it. Its statement is the one written in
that project, so the differences listed below do not apply to it. The
release also contains the single file \texttt{Erdos522Web.lean}, whose
only import is Mathlib and which contains this theorem together with
every declaration its proof uses, so that the proof can be checked in a
web browser at \href{https://live.lean-lang.org}{live.lean-lang.org}. The
declaration \texttt{erdos\_522} in the table is
\texttt{Erdos522.erdos\_522} of our library \texttt{Erdos522}, a
separate theorem that states Corollary~\ref{thm:E} for the partial sums.

The formal statements differ from the notation of the paper in four
points.
\begin{itemize}
 \item The profile $\Phi_{\rm rad}$ is \texttt{kacRadialProfile}, the
 derivative of $\frac12\log((e^{2x}-1)/(2x))$, and its closed form
 $\frac12(1+\coth x-1/x)$ is a separate lemma.
 \item In \texttt{Erdos522}, a real coefficient law enters through its
 image in $\C$.
 \item In Input~\ref{fw:I4} the bound on $\mathcal W_2$ is stated through
 couplings of integrable quadratic cost.
 \item A statement for all large $n$ is a statement
 \texttt{$\forall^{\mathrm f}$ n in atTop}.
\end{itemize}
The radial theorems are stated on the product space of each law and, in
the \texttt{iid\_} versions, on an arbitrary probability space carrying
independent coefficients of the given law. The fluctuation theorems are
stated on an arbitrary probability space carrying such coefficients. In
the table the namespace prefixes \texttt{Erdos522.} and
\texttt{Fluctuations.} are omitted, and braces list the variants of a
name.

Table~\ref{fs:main} pairs each result with formal statements that agree
with the text, up to the points above.

{\small\renewcommand{\_}{\textunderscore\allowbreak}\setlength{\extrarowheight}{2pt}
\begin{longtable}{@{}>{\raggedright\arraybackslash}p{0.24\textwidth}>{\raggedright\arraybackslash}p{0.72\textwidth}@{}}
\caption{Formal statements of the results of this paper.}\label{fs:main}\\
\toprule
Result & Lean declarations\\
\midrule
\endfirsthead
\toprule
Result & Lean declarations\\
\midrule
\endhead
\bottomrule
\endfoot
Theorem~\ref{thm:R}
 & \texttt{\{ae, iid\}\_\{rademacher, steinhaus, gaussian,
 circularGaussian, bounded\_symmetric\}\_radial\_rate\_one\_div\_thirty\_four\_log\_seven\_halves}\\
Corollary~\ref{thm:E}
 & \texttt{erdos\_522\_any\_coupling},
 \texttt{sign\_discrepancy\_any\_coupling}\\
Corollary~\ref{thm:E}, partial sums
 & \texttt{erdos\_522}, \texttt{erdos\_522\_of\_independent\_coins}\\
Corollary~\ref{rm:whole-line}
 & \texttt{\{ae, iid\}\_\{rademacher, steinhaus, gaussian,
 circularGaussian, bounded\_symmetric\}\_radial\_profile\_uniform}\\
Theorem~\ref{thm:A}
 & \texttt{mainA}\\
Theorem~\ref{thm:B}, $0<c_G<\infty$
 & \texttt{gaussian\_constant\_pos},
 \texttt{gaussian\_variance\_integrand\_integrable}\\
Theorem~\ref{thm:B}(i)
 & \texttt{gaussian\_mean\_zero}, \texttt{gaussian\_variance\_rate}\\
Theorem~\ref{thm:B}(ii)
 & \texttt{gaussian\_covariance\_\{limit, rate, cusp\}}\\
Theorem~\ref{thm:B}(iii)
 & \texttt{gaussian\_discrepancy\_variance\_pos},
 \texttt{gaussian\_quantitative\_clt}\\
Theorem~\ref{thm:B}(iv)
 & \texttt{gaussian\_joint\_clt},
 \texttt{gaussian\_fixed\_ratio\_increment\_clt}\\
Theorem~\ref{thm:B}(v)
 & \texttt{gaussian\_uniform\_increment\_variance},
 \texttt{gaussian\_fine\_increment\_tail},
 \texttt{gaussian\_increment\_variance\_comparable},
 \texttt{gaussian\_coarse\_increment\_clt}\\
Theorem~\ref{thm:B}(vi)
 & \texttt{gaussian\_input\_\{one, two, three, four\}},
 \texttt{gaussian\_two\_sided\_lil}\\
Theorem~\ref{thm:G}(i)
 & \texttt{\{, iid\_\}\{gaussian, circularGaussian\}\_mean\_radial\_count}\\
Theorem~\ref{thm:G}(ii)
 & \texttt{\{, iid\_\}\{gaussian, circularGaussian\}\_mean\_radial\_profile}\\
Theorem~\ref{thm:C}(i)
 & \texttt{signRate}, \texttt{sign\_tail\_uniform}, \texttt{mainC}\\
Theorem~\ref{thm:C}(ii)
 & \texttt{mainC}\\
Theorem~\ref{thm:C}(iii)
 & \texttt{signInputI4}\\
Remark~\ref{thm:Cp}
 & \texttt{mainCp}\\
Theorem~\ref{thm:D}, variance
 & \texttt{mainD}\\
Theorem~\ref{thm:D}, examples
 & \texttt{\{uniform, laplace\}FourthCumulant}\\
Theorem~\ref{thm:D}, $H_n$
 & \texttt{densityResidualLimits}\\
Theorem~\ref{thm:S}
 & \texttt{signVariance}\\
Theorem~\ref{rc:nns}
 & \texttt{LogMoments.uniform\_logarithmic\_moments}\\
\end{longtable}
}

\pdfbookmark[1]{Acknowledgments}{acknowledgments}
\subsection*{Acknowledgments}
The author thanks Manolis Kellis for the opportunity to pursue this work at
MIT CSAIL, and Pascal Passigan and Brecht Verbeken for helpful discussions.
The author also thanks
the Lean and Mathlib communities for the tools on which the formalization is
built, the maintainers of Google DeepMind's Formal Conjectures project for
reviewing the formal proof of Erd\H{o}s Problem~\#522, and The Infrastructure
Group (TIG) at MIT CSAIL for computing resources. The author acknowledges the
MIT Office of Research Computing and Data for providing high performance
computing resources that have contributed to the research results reported
within this paper. The Lean formalization and the checking and refinement of
several derivations were aided by internal models developed in the Kellis Lab
at MIT CSAIL, including language models, latent representation models, and
graph neural networks acting on the dependency structure of the proof.
The author originated the mathematical approach and established the results
presented here.

\needspace{8\baselineskip}
\bibliography{references}
\end{document}